\documentclass[11pt]{article}

\usepackage{a4wide,color,eucal,mathrsfs,dsfont}
\usepackage{makeidx}
\usepackage{mathtools}
\usepackage[normalem]{ulem}
\usepackage{amsmath,amsthm,amssymb,epsfig,bbm,graphicx}
\numberwithin{equation}{section}
\usepackage[nottoc,notlot,notlof]{tocbibind}
\usepackage[colorlinks=true]{hyperref}
\hypersetup{linkcolor=blue, urlcolor=blue, citecolor=red}

\usepackage{xcolor}

\usepackage{pdfsync}

\usepackage{times}
\makeindex

\usepackage[shortlabels]{enumitem}

\setlist[enumerate]{font={\upshape}}% global settings, for all lists
\setlist[enumerate,1]{label={(\alph*)}}

\makeatletter
\def\author#1{\gdef\autrun{\def\and{\unskip, }#1}\gdef\@author{#1}}

\makeatother

\newtheorem{theorem}{Theorem}[section]
\newtheorem{corollary}[theorem]{Corollary}
\newtheorem{lemma}[theorem]{Lemma}

\newtheorem{proposition}[theorem]{Proposition}
\newtheorem{definition}[theorem]{Definition}

\theoremstyle{definition}
\newtheorem{remark}[theorem]{Remark}
\newtheorem{example}[theorem]{Example}

\newtheorem*{Assumption}{Assumption}
\numberwithin{equation}{section}

\newcommand{\Md}{\operatorname{Mod}_{p}}
\newcommand{\Mdm}{\mathsf{Mod}_{p}}
\newcommand{\tMd}{\operatorname{\widetilde{Mod}}_{p}}
\newcommand{\Mdc}{\operatorname{Mod}_{p,c}}
\newcommand{\Mdmc}{\mathsf{Mod}_{p,c}}
\newcommand{\tMdc}{\operatorname{\widetilde{Mod}}_{p,c}}

\newcommand{\A}{\mathbb{A}}
\newcommand{\B}{\mathbb{B}}

\newcommand{\D}{\mathbb{D}}

\newcommand{\N}{\mathbb{N}}
\newcommand{\Q}{\mathbb{Q}}
\newcommand{\R}{\mathbb{R}}
\newcommand{\X}{\mathbb{X}}
\newcommand{\Y}{\mathbb{Y}}

\renewcommand{\AA}{\mathscr{A}}
\newcommand{\sAA}{{\scriptscriptstyle \AA}}
\newcommand{\BB}{\mathscr{B}}

\newcommand{\DD}{\mathscr{D}}

\newcommand{\FF}{\mathscr{F}}

\newcommand{\LL}{\mathscr{L}}

\renewcommand{\SS}{\mathscr{S}}
\newcommand{\TT}{\mathscr{T}}
\newcommand{\UU}{\mathscr{U}}

\newcommand{\cA}{{\ensuremath{\mathcal A}}}
\newcommand{\cB}{{\ensuremath{\mathcal B}}}
\newcommand{\calC}{{\ensuremath{\mathcal C}}}

\newcommand{\cE}{{\ensuremath{\mathcal E}}}
\newcommand{\cF}{{\ensuremath{\mathcal F}}}

\newcommand{\cK}{{\ensuremath{\mathcal K}}}

\newcommand{\cL}{{\ensuremath{\mathcal L}}}
\newcommand{\cM}{{\ensuremath{\mathcal M}}}
\newcommand{\cMp}{\mathcal M_+}

\newcommand{\cP}{{\ensuremath{\mathcal P}}}
\newcommand{\cT}{{\ensuremath{\mathcal T}}}

\newcommand{\calT}{\cT}

\newcommand{\cU}{{\ensuremath{\mathcal U}}}

\newcommand{\ff}{{\mbox{\boldmath$f$}}}
\renewcommand{\gg}{{\mbox{\boldmath$g$}}}
\newcommand{\hh}{{\mbox{\boldmath$h$}}}

\newcommand{\mm}{{\mbox{\boldmath$m$}}}

\newcommand{\shh}{{\mbox{\scriptsize\boldmath$h$}}}

\newcommand{\smm}{{\mbox{\scriptsize\boldmath$m$}}}

\newcommand{\mmu}{{\mbox{\boldmath$\mu$}}}
\newcommand{\nnu}{{\mbox{\boldmath$\nu$}}}
\newcommand{\ppi}{{\mbox{\boldmath$\pi$}}}

\newcommand{\ssigma}{{\mbox{\boldmath$\sigma$}}}

\newcommand{\sssigma}{{\mbox{\scriptsize\boldmath$\sigma$}}}

\newcommand{\sfd}{{\sf d}}
\newcommand{\sfe}{{\sf e}}

\newcommand{\sfh}{{\sf h}}

\newcommand{\sfq}{{\sf q}}

\newcommand{\sft}{{\sf t}}

\newcommand{\sfD}{{\sf D}}

\newcommand{\sfK}{{\sf K}}
\newcommand{\sfJ}{{\sf J}}

\newcommand{\sfQ}{{\mathsf Q}}

\newcommand{\fri}{{\mathfrak i}}

\newcommand{\rmd}{{\mathrm d}}

\newcommand{\rmp}{{\mathrm p}}

\newcommand{\rmx}{{\mathrm x}}

\newcommand{\rmA}{{\mathrm A}}
\newcommand{\rmN}{{\mathrm N}}

\newcommand{\rmK}{{\mathrm K}}
\newcommand{\rmH}{{\mathrm H}}
\newcommand{\rmC}{{\mathrm C}}
\newcommand{\rmE}{{\mathrm E}}
\newcommand{\rmB}{{\mathrm B}}
\newcommand{\rmD}{{\mathrm D}}
\newcommand{\rmF}{{\mathrm F}}
\newcommand{\rmG}{{\mathrm G}}

\newcommand{\rmM}{{\mathrm M}}
\newcommand{\rmS}{{\mathrm S}}
\newcommand{\rmT}{{\mathrm T}}
\newcommand{\rmU}{{\mathrm U}}

\newcommand{\Kliminf}{K\kern-3pt-\kern-2pt\mathop{\rm lim\,inf}\limits}  % Kuratowski liminf di insiemi
\newcommand{\supp}{\mathop{\rm supp}\nolimits}   % supporto 
\newcommand{\Lip}{\mathop{\rm Lip}\nolimits}          %Lipschitz
\newcommand{\Lipb}{\mathop{\rm Lip}_b\nolimits}          %Lipschitz
\newcommand{\Lipbu}[1]{\mathop{\rm Lip}_{b,#1}\nolimits}          %Lipschitz
\newcommand{\lip}{\mathop{\rm lip}\nolimits}          %Lipschitz
\newcommand{\lipdlt}{\mathop{\rm lip}\nolimits_\delta}
\newcommand{\lipd}{\mathop{\rm lip}}  
\renewcommand{\d}{{\mathrm d}}
\newcommand{\dt}{{\d t}}

\newcommand{\restr}[1]{\lower3pt\hbox{$|_{#1}$}}
\newcommand{\topref}[2]{\stackrel{\eqref{#1}}#2}
\newcommand{\Leb}[1]{{\mathscr L}^{#1}}      % Misura di Lebesgue
\newcommand{\la}{{\langle}}                  % brackets
\newcommand{\ra}{{\rangle}}
\newcommand{\down}{\downarrow}              %frecce in su e in giu nei limiti
\newcommand{\up}{\uparrow}
\newcommand{\eps}{\varepsilon}  
\newcommand{\nchi}{{\raise.3ex\hbox{$\chi$}}}
\newcommand{\weakto}{\rightharpoonup}
\newcommand{\BorelSets}[1]{\BB(#1)}

\newcommand{\e}{{\rm{e}}}                           % mappa di valutazione, bisogna mettere `a mano' il tempo t
\renewcommand{\mm}{\mathfrak m}
\renewcommand{\smm}{{\mbox{\scriptsize$\mm$}}}

\newcommand{\ep}{\varepsilon}
\newcommand{\Ldp}{\mathcal{L}^p_+(X, \mm)}

\newcommand{\res}{\mathop{\hbox{\vrule height 7pt width .5pt depth 0pt
\vrule height .5pt width 6pt depth 0pt}}\nolimits}

\newcommand{\GGG}{\relax}  %{\color{blue}}
\newcommand{\nc}{\normalcolor}

\newcommand{\KK}{\mathscr K}
\newcommand{\AC}{\mathrm{AC}}
\newcommand{\BV}{\mathrm{BV}}
\newcommand{\BVC}{\mathrm{BVC}}
\newcommand{\Var}{\mathrm{Var}}
\newcommand{\RectArc}{\mathrm{RA}}
\newcommand{\RA}{\RectArc}

\newcommand{\Arc}{\mathrm{A}}

\newcommand{\sppi}{{\mbox{\scriptsize\boldmath$\pi$}}}

\newcommand{\J}I
\newcommand{\cI}{{\ensuremath{\mathcal I}}}

\newcommand{\Osc}[2]{\operatorname{Osc}(#1,{#2})}

\newcommand{\Ch}{\mathsf{Ch}}
\newcommand{\CE}{\mathsf{C\kern-1pt E}}
\newcommand{\wCE}{\mathsf{wC\kern-1pt E}}

\newcommand{\pCE}{\mathsf{pC\kern-1pt E}}
\newcommand{\Mod}{\operatorname{Mod}}

\newcommand{\brq}[1]{\operatorname{Bar}_{#1}}
\newcommand{\tbrq}[1]{\operatorname{\widetilde{Bar}}_{#1}}

\renewcommand{\Bar}[2]{\cB_{#1}}
\newcommand{\Cont}[1]{\operatorname{Cont}_{#1}}
\newcommand{\tCont}[1]{\operatorname{\widetilde{Cont}}_{#1}}

\newcommand{\unit}{\mathds 1}
\newcommand{\proj}[1]{\operatorname{Proj}(#1)}

\newcommand{\Quot}{\mathfrak{q}}

\newcommand{\Al}{R}
\newcommand{\kkappa}{\kappa}
\newcommand{\Sob}{H}
\newcommand{\relgrad}[1]{|\rmD #1|_\star}
\newcommand{\weakgrad}[1]{|\rmD #1|_w}
\newcommand{\weakgradTq}[1]{|\rmD #1|_{w,\calT_q}}
\newcommand{\weakgradTqhat}[1]{|\rmD #1|_{w,\calT_q'}}
\newcommand{\relgradA}[1]{|\rmD #1|_{\star,{\scriptscriptstyle \AA}}}
\newcommand{\gsfe}{\hat \sfe}
\newcommand{\Restr}[2]{\operatorname{Restr}_{#1}^{#2}}

\title{Sobolev spaces in extended metric-measure spaces}
\author{Giuseppe Savar\'e\\
\small Universit\`a di Pavia, Dipartimento di Matematica
  ``F. Casorati''\\
  \small email: \textsl{giuseppe.savare@unipv.it}}

\begin{document}

\maketitle
\begin{abstract}
  These lecture notes contain an extended version of the material
presented in the C.I.M.E.~summer course in 2017.
The aim is to give a detailed introduction to the
metric Sobolev theory.

The notes are divided in four main parts.
The first one is devoted to a preliminary and detailed study of the
underlying topological, metric, and measure-theoretic aspects
needed for the development of the theory
in a general \emph{extended metric-topological measure space}
$\X=(X,\tau,\sfd,\mm)$.

The \emph{second part} is devoted to the construction of the
Cheeger energy,
initially defined on a distinguished unital algebra $\AA$ of bounded,
$\tau$-continuous and $\sfd$-Lipschitz
functions.

The \emph{third part} 
deals with the basic tools needed for the dual characterization of the
Sobolev spaces: the notion of $p$-Modulus
of a collection
of (nonparametric) rectifiable arcs
and its duality with the class 
of nonparametric dynamic plans, i.e.~Radon measures
on the space of rectifiable arcs
with finite $q$-barycentric entropy with respect to
$\mm$.

The \emph{final part} of the notes is devoted to the dual/weak formulation
of the Sobolev spaces $W^{1,p}(\X)$ in terms of
nonparametric dynamic plans and to their relations with the
Newtonian spaces $N^{1,p}(\X)$ and with the
spaces $\Sob^{1,p}(\X)$ obtained by the Cheeger
construction.
In particular, when $(X,\sfd)$ is complete,
a new proof of the equivalence between these different approaches
is given by a direct duality argument.

A substantial part of these Lecture notes relies on well established
theories.
New contributions concern the extended metric setting, the role of general compatible algebras
of Lipschitz functions and their density w.r.t.~the Sobolev energy,
a general embedding/compactification trick,
the study of reflexivity and infinitesimal Hilbertianity inherited
from the underlying space, 
and the use of nonparametric dynamic plans for the
definition of weak upper gradients.
\end{abstract}

\tableofcontents
\newpage
\section{Introduction}
These lecture notes contain an extended version of the material
presented in the C.I.M.E.~summer course in 2017.
The aim is to give a detailed introduction to the
metric Sobolev theory,
trying to unify at least two of the main approaches
leading to the construction of the Sobolev spaces 
in general metric-measure spaces.
\medskip

The notes are divided in four main parts.
The first one is devoted to a preliminary and detailed study of the
underlying topological, metric, and measure-theoretic aspects
needed for the development of the general theory.
In order to cover a wide class of examples, including genuinely
infinite dimensional cases, we consider a
general \emph{extended metric-topological measure space}
\cite{AES16} 
$\X=(X,\tau,\sfd,\mm)$, where $\sfd$ is an extended distance on $X$,
$\tau$ is an auxiliary weaker topology compatible with $\sfd$ and $\mm$ is
a Radon measure in $(X,\tau)$. The simplest example is a complete and
separable metric space $(X,\sfd)$ where $\tau$ is the
topology induced by the distance, but more general situations
as duals of separable Banach spaces or Wiener spaces can be included
as well.

The use of an auxiliary weaker (usually Polish or Souslin) topology $\tau$
has many technical advantages: first of all,
it is easier to check the Radon property of the finite Borel measure
$\mm$,
one of our crucial structural assumptions.
A second advantage is to add more flexibility
in the choice of well behaved sub-algebras
of Lipschitz functions and to allow for a powerful compactification
method. As a reward, roughly speaking,
many results which can be proved for a compact topology, can be
extended to the case of a complete metric space $(X,\sfd)$ without too much effort.
Therefore, for a first reading, it would not be too restrictive to
assume compactness of the
underlying topology, in order to avoid cumbersome technicalities.

The first part also includes a careful analysis of the
topological-metric properties of the path space (Section \ref{sec:curves})
in particular concerning invariant properties with respect to
parametrizations.
We first recall the compact-open topology of $\rmC([0,1];(X,\tau))$
and the induced quotient space of arcs, obtained by identifying two
curves $\gamma_1,\gamma_2\in \rmC([0,1];(X,\tau))$ if
there exist continuous, nondecreasing and surjective maps
$\sigma_1,\sigma_2:
[0,1]\to[0,1]$ such that
$\gamma_1\circ\sigma_1=\gamma_2\circ\sigma_2$. 
This provides the natural quotient topology for the
space of \emph{continuous and $\sfd$-rectifiable arcs} $\RA(X)$, for
which the length
\begin{equation}
  \label{eq:622}
  \ell(\gamma):=
    \sup\Big\{\sum_{j=1}^N
    \sfd(\gamma(t_j),\gamma(t_{j-1})): \{t_j\}_{j=0}^N\subset [0,1],\quad
  t_0<t_1<\cdots<t_N\ \Big\}
\end{equation}
is finite.
It results a natural metric-topological structure for
the space $\RA(X)$,
where the distance $\sfd$ characterizes the length and the integrals,
whereas the topology $\tau$ induces the appropriate notion of convergence.
This analysis plays a crucial role, since one of
the main tools for studying Sobolev spaces involves \emph{dynamic
  plans}, i.e.~Radon measures on $\RA(X)$. 
It is also the natural setting to study
the properties of length-conformal distances (Section
\ref{sec:length-Finsler}).
\medskip

\noindent
The \emph{second part} is devoted to the construction of the
Cheeger energy \cite{Cheeger00,AGS14I} (Section \ref{sec:Cheeger}), 
the $L^p(X,\mm)$-relaxation of the energy functional
\begin{equation}
  \label{eq:611}
  \int_X \big(\lip f(x)\big)^p\,\d\mm(x),\quad
  \lip f(x):=\limsup_{y,z\to x}\frac{|f(y)-f(x)|}{\sfd(y,z)},
\end{equation}
initially defined on a distinguished unital algebra $\AA$ of bounded,
$\tau$-continuous and $\sfd$-Lipschitz
functions satisfying the approximation property
\begin{equation}
  \label{eq:612}
  \sfd(x,y)=\sup\Big\{f(x)-f(y):f\in \AA,\ |f(x')-f(y')|\le
  \sfd(x',y')
  \text{ for every }x',y'\in X\Big\}
\end{equation}
for every couple of points $x,y\in X$.
This gives raise to the Cheeger energy
\begin{equation}
  \label{eq:613}
  \CE_{p,\sAA}(f):=\inf\Big\{\int_X \big(\lip f_n\big)^p\,\d\mm:
  f_n\in \AA,\ f_n\to f\ \text{in }L^p(X,\mm)\Big\},
\end{equation}
whose proper domain 
characterizes the strongest Sobolev space (deeply inspired by the
Cheeger approach \cite{Cheeger00})
\begin{equation}
  H^{1,p}(\X,\AA):=
  \Big\{f\in L^p(X,\mm): \CE_{p,\sAA}(f)<\infty\Big\}.
  \label{eq:617}
\end{equation}
We will discuss various useful properties of the Cheeger energy, in
particular
its local representation in terms of the minimal relaxed gradient
$\relgradA f$ as
\begin{equation}
  \label{eq:614}
    \CE_{p,\sAA}(f)=\int_X \relgradA f^p(x)\,\d\mm(x),
\end{equation}
the non-smooth first-order calculus properties of $\relgradA f$, and
the invariance properties of $\CE_{p,\sAA}$ with respect to
measure-preserving isometric imbedding of $X$.

A first, non obvious, important result is the independence of
$\CE_{p,\sAA}$
with respect to $\AA$, at least when $(X,\tau)$ is compact.
It is a consequence of a delicate and powerful approximation method
based on the metric Hopf-Lax flow
\begin{equation}
  \label{eq:615}
  \sfQ_t f(x):=\inf_{y\in X}f(y)+\frac 1{qt^{q-1}}\sfd^q(x,y),\quad t>0,
\end{equation}
which we will discuss in great detail in Section \ref{sec:Invariance}.
\medskip

\noindent
The \emph{third part} of these notes
deals with the basic tools needed for the dual characterization of the
Sobolev spaces. First of all the notion of $p$-Modulus
\cite{Fuglede57,Heinonen-Koskela98,Koskela-MacManus98,Shanmugalingam00} (Section
\ref{sec:Modulus}) of a collection
of (nonparametric) rectifiable arcs $\Gamma\subset \RA(X)$,
\begin{equation}
  \label{eq:616}
  \Md (\Gamma):=\inf\Big\{\int_X f^p\,\d\mm:\ 
  f:X\to [0,\infty]\text{ Borel},
  \quad
  \int_\gamma f\ge 1\text{ for every }\gamma\in \Gamma\Big\},
\end{equation}
which is mainly used to give a precise meaning to negligible sets.
It will be put in duality with the class $\Bar q\mm$
of Radon measures
$\ppi$ on $\RA(X)$ 
with finite $q$-barycentric entropy $\brq q(\ppi)$ with respect to
$\mm$ \cite{ADS15}
(Section \ref{sec:DynamicPlans}). Every
$\ppi\in \Bar q\mm$ induces a measure $\mu_\sppi=h_\sppi \mm$
with density $h_\sppi\in L^q(X,\mm)$ such that
for every bounded Borel function $\zeta:X\to \R$
\begin{equation}
  \label{eq:618}
  \int_{\RA(X)}\int_\gamma \zeta\,\d\ppi(\gamma)=\int_X \zeta\,\d\mu_\sppi=
  \int_X \zeta\,h_\sppi\,\d\mm,\quad
  \brq q^q(\ppi):=\int_X h_\sppi^q\,\d\mm.
\end{equation}
At least when $(X,\sfd)$ is complete, we will show (Section
\ref{sec:Equivalence}) that the Modulus of
a Borel subset $\Gamma\subset \RA(X)$ can be essentially identified
with
the conjugate of the $q$-barycentric entropy:
\begin{equation}
  \label{eq:619}
  \frac 1p\Md(\Gamma)=\sup_{\sppi\in \Bar q\mm}\ppi(\Gamma)-\frac
  1q\brq q^q(\ppi).
\end{equation}
The duality formula shows that a Borel set $\Gamma\subset \RA(X)$ is $\Md$-negligible 
if and only if it is $\ppi$-negligible for every dynamic plan $\ppi$ with finite
$q$-barycentric entropy.
\medskip

\noindent
The \emph{final part} of the notes is devoted to the dual/weak formulation
of the Sobolev spaces $W^{1,p}(\X)$ and to their relations with the
spaces $\Sob^{1,p}(\X)$ obtained by the Cheeger construction.
The crucial concept here is the notion of \emph{upper gradient}
\cite{Heinonen-Koskela98,Koskela-MacManus98,Cheeger00} of a function
$f:X\to \R$:
it is a nonnegative Borel function $g:X\to[0,+\infty]$ such
that
\begin{equation}
  \label{eq:620}
  |f(\gamma_1)-f(\gamma_0)|\le \int_\gamma g
\end{equation}
for every rectifiable arc $\gamma\in \RA(X)$;
$\gamma_0$ and $\gamma_1$ in \eqref{eq:620} denote the initial and final points
of $\gamma$. As suggested by the theory of Newtonian spaces \cite{Shanmugalingam00},
it is possible to adapt the notion of upper gradient to Sobolev functions
by asking that $g\in L^p(X,\mm)$,
by selecting a corresponding notion of ``exceptional'' or
``negligible''
sets of rectifiable arcs
in $\RA(X)$, 
and by imposing that the set of curves
where \eqref{eq:620} does not hold is
exceptional.

According to the classic approach leading to Newtonian spaces, a subset
$\Gamma\subset \RA(X)$
is negligible if $\Md(\Gamma)=0$. This important notion, however, is
not
invariant with respect to modification of $f$ and $g$ in
$\mm$-negligible sets. Here we present a different construction,
based on the new class $\calT_q$ of
dynamic plans $\ppi$ with finite $q$-barycentric entropy and
with finite $q$-entropy of the initial and final distribution of
points,
\begin{equation}
  \label{eq:623}
  \ppi\in \calT_q\quad\Leftrightarrow\quad
  \brq q(\ppi)<\infty,\quad
  (\sfe_i)_{\sharp}\ppi=h_i\mm\quad\text{for some }h_i\in L^q(X,\mm),\quad i=0,1,
\end{equation}
where $\sfe_i(\gamma)=\gamma_i$. The last condition requires that
there exist functions $h_i\in L^q(X,\mm)$ such that
\begin{equation}
  \label{eq:624}
  \int_{\RA(X)} \zeta(\gamma_i)\,\d\ppi(\gamma)=\int_X \zeta(x)\,h_i\,\d\mm\quad
  \text{for every bounded Borel function $\zeta:X\to \R$}.
\end{equation}
A collection $\Gamma\subset \RA(X)$ is $\calT_q$-negligible if it is
$\ppi$-negligible
for every $\ppi\in \calT_q$.
The Sobolev space $W^{1,p}(\X.\calT_q)$ precisely contains 
all the functions $f\in L^p(X,\mm)$ with a $\calT_q$-weak upper gradient $g\in
L^p(X,\mm)$, so that
\begin{equation}
  \label{eq:620bis}
  |f(\gamma_1)-f(\gamma_0)|\le \int_\gamma g\quad
  \text{ for $\calT_q$-a.e.~arc }\gamma\in \RA(X).
\end{equation}
Among all the $\calT_q$-weak upper gradient $g$ of $f$ it is possible
to select the minimal one, denoted by $\weakgradTq f$, such that
$\weakgradTq f\le g$ for every $\calT_q$-weak upper gradient $g$.
The norm of $W^{1,p}(\X,\calT_q)$ is then given by
\begin{equation}
  \label{eq:625}
  \|f\|_{W^{1,p}(\X,\calT_q)}^p:=
  \int_X\Big(|f|^p+\weakgradTq f^p\Big)\,\d\mm.
\end{equation}
Differently from the Newtonian weak upper gradient, the notion of
$\calT_q$-weak upper gradient
is invariant w.r.t.~modifications of $f$ and $g$ in $\mm$-negligible
sets; moreover it is possible to prove that functions in
$W^{1,p}(\X,\calT_q)$ are Sobolev along $\calT_q$-a.e.~arc $\gamma$
with distributional derivative bounded by $g\circ \gamma$.
The link with the Newtonian theory appears more clearly by a further
properties of
functions in $W^{1,p}(\X,\calT_q)$, at least when $(X,\sfd)$ is
complete:
for every function $f\in W^{1,p}(\X,\calT_q)$
it is possible to find a ``good representative'' $\tilde f$ 
(so that $\{\tilde f\neq f\}$ is $\mm$-negligible), so that the modified function $\tilde f$ is
absolutely continuous along $\Md$-a.e.~rectifiable curve. In this way,
the a-priori weaker approach by $\calT_q$-dynamic plans
is equivalent to the Newtonian one and it is possible to identify $
W^{1,p}(\X,\calT_q)$ with $N^{1,p}(\X)$.
We will also show in that the approach by nonparametric dynamic plans
is equivalent
to the definition by parametric $q$-test plans of \cite{AGS14I,AGS13}.

A further main identification result is stated in Section
\ref{sec:Identification}:
when $(X,\sfd)$ is complete, we can show that
$W^{1,p}(\X,\calT_q)$ coincides with $\Sob^{1,p}(\X,\AA)$.
This fact (originally proved by \cite{Cheeger00,Shanmugalingam00}
in the case of doubling-Poincar\'e spaces)
can be interpreted as a density result of
a compatible algebra of functions $\AA$ in $W^{1,p}(\X,\calT_q)$
and has important consequences, some of them recalled in the last
section of the notes.
Differently from other recent approaches \cite{AGS14I,AGS13}
the proof arises from a direct application of the Von Neumann min-max
principle
and relies on two equivalent characterizations of the dual Cheeger
energy $\CE_p^*(h)$ 
for functions $h\in L^q(X,\mm)$
\begin{equation}
  \label{eq:590}
  \frac 1q\CE_p^*(h)=\sup_{h\in \Sob^{1,p}(\X)} \int_X fh\,\d\mm-\frac
  1p\CE_p(f)\quad
  h\in L^q(X,\mm),\ \int_X h\,\d\mm=0.
\end{equation}
When $(X,\tau)$ is compact we can prove that 
\begin{equation}
  \label{eq:621}
  \CE_p^*(h)=\sup\Big\{\brq q^q(\ppi):
  (\sfe_0)_\sharp\ppi=h_-\mm,\quad
  (\sfe_1)_\sharp\ppi=h_+\mm\Big\},
\end{equation}
$h_-,h_+$ being the negative and positive parts of $h$, and
\begin{equation}
  \label{eq:621}
  \frac 1q\CE_p^*(h)=\sup\Big\{\sfK_{\sfd_g}(h_-\mm,h_+\mm)-\frac 1q
    \int_X g^q\,\d\mm:g\in \rmC_b(X),\ \inf_X g>0\Big\},
\end{equation}
where $\sfK_{\sfd_g}$ is the Kantorovich-Rubinstein distance induced
by the cost
\begin{equation}
  \label{eq:626}
  \sfd_g(x_0,x_1):=\inf\Big\{\int_\gamma g:\gamma\in \RA(X),\
  \gamma_0=x_0,\
  \gamma_1=x_1\Big\}.
\end{equation}
Thanks to the identification Theorem and the compactification method,
we obtain that for a general complete space $(X,\sfd)$
\begin{equation}
  \label{eq:627}
  \Sob^{1,p}(\X,\AA)=\Sob^{1,p}(\X)=W^{1,p}(\X,\calT_q)=N^{1,p}(\X)
\end{equation}
(the last identity holds up to the selection
of a good representative)
with equality of the corresponding minimal gradients.
As a consequence, all the approaches lead to one canonical object
and this property does not rely on the validity of doubling properties or Poincar\'e
inequalities for $\mm$.

In the last Section \ref{sec:Examples} we will show
various invariance properties of the Cheeger energy and the metric
Sobolev spaces.
In particular, when the underlying space $X$ has a linear structure, we
show that the metric approach coincides with
more classic definitions of Sobolev spaces
(e.g.~the weighted Sobolev spaces in $\R^d$ \cite{HKM93}
or the Sobolev spaces associated to a log-concave measure
in a Banach-Hilbert space), obtaining the reflexivity
(resp.~the Hilbertianity) of $W^{1,p}(\X)$ whenever $X$ is a reflexive
Banach (resp.~Hilbert) space.
\medskip

\noindent
A substantial part of these Lecture notes relies on well established
theories:
our main sources have been
\cite{AGS14I,AGS13,AGS15} (for the parts concerning the Cheeger energy,
the weak upper gradients, 
and the properties of the Hopf-Lax flow),
\cite{Bjorn-Bjorn11,HKST15,ADS15} (for the notion of the $p$-Modulus
and the Newtonian spaces),
\cite{ADS15} (for the notion of nonparametric dynamic plans
in $\Bar q\mm$, the dual characterization of the $p$-Modulus and
the selection of a good representative of a Sobolev function),
\cite{AES16} (for the extended
metric-topological structures), \cite{AGS08,Villani09} (for the
results involving the Kantorovich-Rubinstein distances of Optimal
Transport),
\cite{Schwartz73} for the theoretic aspects of Radon measures.
Further bibliographical notes are added to each Section
with more detailed comments.
We also refer to the overviews and
lecture notes \cite{Ambrosio-Tilli04,Heinonen07,Ambrosio-Ghezzi16,Gigli18II}.

New contributions concern the role of general compatible algebras
of Lipschitz functions and their invariance in the construction of
the Cheeger energy, the embedding/compactification tricks,
the use of nonparametric dynamic plans for the
definition of weak upper gradients,
the characterization of the dual Cheeger energy and the proof
of the identification theorem $H=W$ by a direct duality argument.

Of course, there are many important 
aspects that we did not include in these notes:
just to name a few of them at the
level of the Sobolev construction we quote
\begin{itemize}[-]\itemsep-3pt
\item the Haj\l asz's Sobolev spaces \cite{Hajlasz96},
\item the theoretical aspects related to the doubling
  and to the Poincar\'e inequality assumptions \cite{Bjorn-Bjorn11,HKST15},
\item the point of view of 
  \emph{parametric} dynamic plans (i.e.~Radon measures on the space of
  parametric curves with finite $q$-energy) \cite{AGS14I,AGS13}
  (but see the discussion in Section \ref{subsec:parametric}),
\item the properties of the
  $L^2$-gradient flow of the Cheeger energy,
\item the original proof of the ``$H=W$'' 
  Theorem \cite{AGS14I,AGS13}
  by a dynamic approach based on the identification of the
  $L^2$-gradient flow of the Cheeger energy with the
  Kantorovich-Wasserstein gradient flow of the Shannon-R\'eny
  entropies,
\item the approach \cite{DiMarino14,Ambrosio-Ghezzi16} by
  derivations and integration by parts,
\item the Gigli's nonsmooth differential structures
  \cite{Gigli15,Gigli18,Gigli18II} (see also
  \cite{Gigli-Pasqualetto18,Gigli-Pasqualetto-Soultanis18}),
\item the applications to metric measure spaces satisfying a lower Ricci
  curvature bounds
  \cite{Sturm06I,Sturm06II,Lott-Villani09,AGS14D,AGMR15,AGS15,EKS15,AMS15preprint}.
\end{itemize}

\newpage

\subsection{Main notation}

\smallskip
\halign{$#$\hfil\ &#\hfil
  \cr
  (X,\tau)&Hausdorff topological space\cr
  (X,\tau,\sfd)&Extended metric-topological (e.m.t.) space, see \S
  \ref{subsec:extended-def} and Definition \ref{def:luft1}
  \cr
  \X=(X,\tau,\sfd,\mm)&Extended metric-topological measure (e.m.t.m.) space, see \S
  \ref{subsec:extended-def}
  \cr
  \cM(X),\ \cMp(X)&Signed and positive Radon measures on a Hausdorff topological space
  $X$
  \cr
% \cMp(X)&Finite positive Radon measures on a Hausdorff topological space
% $X$
% \cr
\cP(X)&Radon probability measures on $X$ % (with finite
% quadratic moment)
\cr
\FF(X),\ \KK(X),\ \BorelSets X,\ \SS(X)&Closed, compact, Borel and
Souslin subsets of $X$
\cr
\supp(\mu)&Support of a Radon measure, see p.~\pageref{page:support}
\cr
f_\sharp\mu&Push forward of $\mu\in \cM(X)$ by a (Lusin
$\mu$-measurable) map $f:X\to Y$,
\eqref{eq:487}
\cr 
% \gamma=\sigma\mu{+}\mu^\perp,\ \mu=\varrho\gamma{+}\gamma^\perp&
% Lebesgue decompositions of $\gamma$ and $\mu$, Lemma \ref{le:Lebesgue}
% \cr
\rmC_b(X,\tau),\ \rmC_b(X)&$\tau$-continuous and bounded real functions on $X$\cr
\rmB_b(X,\tau),\ \rmB_b(X)&Bounded $\tau$-Borel %(resp.~bounded Borel)
real functions
\cr
\Lip(f,A,\delta)&Lipschitz constant of $f$ on $A$ w.r.t.~the extended
semidistance $\delta$, \eqref{eq:204}
\cr
\Lip_b(X,\tau,\delta)&Bounded, $\tau$-continuous and $\delta$-Lipschitz 
real functions on $X$, \eqref{eq:221}
\cr
\Lip_{b,\kappa}(X,\tau,\delta)&Functions in $\Lip_b(X,\tau,\delta)$
with Lipschitz constant bounded by $\kappa$, \eqref{eq:230}
\cr
\lipdlt f(x)&Asymptotic Lipschitz constant of $f$ at a point $x$,
\S\,\ref{subsec:aslip}
\cr
\AA,\AA_1&Compatible unital sub-algebra of $\Lip_b(X,\tau,\sfd)$,
\S\,\ref{subsec:compatible-A}
\cr
% {\rmL}^p(X,\mm),\ \rmL^p(X,\mm;\R^d) & Borel $\mm$-integrable real (or
% $\R^d$-valued) functions
% \cr
\cL^p(X,\mm) &Space of $p$-summable Borel functions
\cr 
\LL^q(\gamma|\mu)&Entropy
functionals on Radon measures
\cr
\sfK_\delta(\mu_1,\mu_2)&Kantorovich-Rubinstein extended distance in
$\cMp(X)$, \S\,\ref{subsec:KR}
\cr
\rmC([a,b];(X,\tau)),\ \rmC([a,b];X)&$\tau$-continuous curves defined
in $[a,b]$ with values in $X$, \ref{subsec:continuous_curves}
\cr
\tau_\rmC,\ \sfd_\rmC&Compact open topology and
extended distance on $\rmC([a,b];X)$, \ref{subsec:continuous_curves}
\cr
\Arc(X,\tau),\ \Arc(X)&Space of arcs, classes of curves
equivalent up to a reparametrization, \ref{subsec:arcs}
\cr
\Arc(X,\sfd)&Space of arcs with a $\sfd$-continuous reparametrization,
\ref{subsec:arcs}
\cr
\tau_\rmA,\ \sfd_\rmA&Quotient topology and
extended distance on $\Arc(X,\sfd)$, \ref{subsec:arcs}
\cr
\mathrm{BV}([a,b];(X,\delta))&Curves $\gamma:[a,b]\to X$ with finite
total variation w.r.t.~$\delta$,
\ref{subsec:rectifiable_arcs}
\cr
\mathrm{BVC}([a,b];(X,\sfd))&Continuous curves in
$\BV([a,b];(X,\sfd))$
\ref{subsec:rectifiable_arcs}
\cr
\RA(X,\sfd),\ \RA(X)&Continuous and rectifiable arcs,
\ref{subsec:rectifiable_arcs}
\cr
\Al_\gamma&Arc-length reparametrization of a rectifiable arc $\gamma$,
\ref{subsec:rectifiable_arcs}
\cr
\int_\gamma f&Integral of a function $f$ along a rectifiable curve (or
arc) $\gamma$, 
\ref{subsec:rectifiable_arcs}
\cr
\ell(\gamma)&length of $\gamma$
\ref{subsec:rectifiable_arcs}
\cr
\nu_\gamma&Radon measure in $\cMp(X)$ induced
by integration along a rectifiable arc $\gamma$,
\ref{subsec:rectifiable_arcs}
\cr
\sfd_\ell,\ \sfd_g&Length and conformal distances generated by $\sfd$,
\S\, \ref{sec:length-Finsler}
\cr
\pCE_{p},\ \CE_{p,\sAA}&(pre)Cheeger energy, Definition
\ref{def:Cheeger}
\cr
\Sob^{1,p}(\X,\AA)&Metric Sobolev space induced by the Cheeger energy,
Definition \ref{def:SobolevH}
\cr
\relgradA f&Minimal $(p,\AA)$-relaxed gradient,
\S\,\ref{subsec:relaxed}
\cr
\sfQ_t^{K,\delta}(f),\sfQ_t(f)&(Generalized) Hopf-Lax flow,
\S\,\ref{subsec:HL}
\cr
\Mdm(\Sigma),\ \Mdmc
(\Sigma)&$p$-Modulus of a collection
of measures $\Sigma\subset
\cMp(X)$,
\eqref{eqn:mod2} and \eqref{eqn:mod2c}
\cr
\Md(\Gamma),\ \tMd(\Gamma)&$p$-Moduli of a collection $\Gamma\subset
\RA(X)$,
\eqref{eqn:mod2gamma} and \eqref{eqn:mod2gammabis}
\cr
\brq q(\ppi)&$q$-barycentric entropy of a dynamic plan, Definition
\ref{def:barycenter}
\cr
\Bar q\mm&Plans with barycenter in $L^q(X,\mm)$, Definition
\ref{def:barycenter}
\cr
\Cont p(\Gamma)&$p$-Content of a family of arcs, Definition
\ref{def:pcontent}
\cr
\calT_q,\ \calT_q^*&nonparametric $q$-test plans, Definition \ref{def:testplan}
\cr
\weakgrad f,\ \weakgradTq f&Minimal $\calT_q$-weak upper gradient,
Definition \ref{def:minimalwug}
\cr
\wCE_{p,\calT_q},\ W^{1,p}(\X,\calT_q)&Weak $(p,\calT_q)$-energy and weak
Sobolev space, Definition \ref{def:weakCheeger}
\cr
\DD_q(\mu_0,\mu_1)&Dual dynamic cost, \eqref{eq:191}
\cr
% \mathrm{AC}^p([0,1];X)&Space of curves $\rmx:[0,1]\to
% X$ with $p$-integrable metric speed
% \cr
% |\rmx'|_\sfd&metric speed of a curve $\rmx\in
% \mathrm{AC}([a,b];(X,\sfd))$
% \cr
}

\part{Topological and Metric-Measure structures}
\section{Metric-measure structures}
\label{sec:Mms}
\GGG
In this section we will recall the main notion and facts we will use
in the sequel. Our main ingredients are
\begin{itemize}
\item a Hausdorff topological space $(X,\tau)$,
\item an extended distance $\sfd:X\times X\to[0,\infty]$,
\item a finite Radon measure $\mm$ on $(X,\tau)$,
\item an algebra $\cA$ of $\tau$-continuous, $\sfd$-Lipschitz, 
  bounded real functions defined in $X.$
\end{itemize}
All these objects will satisfy suitable compatibility conditions,
which we are going to explain. We will call the system
\index{Extended metric-topological measure (e.m.t.m.) space}
\begin{equation}
\text{$\X=
(X,\tau,\sfd,\mm)$,
an \emph{extended metric-topological measure (e.m.t.m.) space.}}
\label{eq:512}
\end{equation}
The choice of $\cA$ will
play a role in the construction of the Cheeger energy.

Let us first consider
the topological and measurable side of this structure.

\subsection{Topological and measure theoretic notions}
\label{subsec:measure}

Let $(X,\tau)$ be a Hausdorff topological space. We will denote 
by $\rmC_b(X,\tau)$ (resp.~$\rmB_b(X,\tau)$)
% , $\LSC(X,\tau)$ 
the space of $\tau$-continuous (resp.~Borel) and bounded
real functions defined on $X$.
$\BB(X,\tau)$ is the collection of
the Borel subsets of $X$. For every $x\in X$, 
$\UU_x$ will denote the system of neighborhoods of $x$.
We will often omit the explicit indication of the topology $\tau$, when
it will be clear from the context.

We will always deal with a \emph{completely regular} \index{Completely
  regular topology, \eqref{eq:CR}} topology, i.e. 
\begin{equation}
  \label{eq:CR}
  \begin{gathered}
    \text{for any closed set $F\subset X$ and any $x_0\in X\setminus F$}\\
    \text{there exists $f\in \rmC_b(X,\tau)$ with $f(x_0)>0$ and $f\equiv 0$ on $F$.} 
  \end{gathered}
\end{equation}
We can always assume that $f$ takes values in $[0,1]$ and
$f(x_0)=1$. An immediate consequence of \eqref{eq:CR} is that 
for every open subset $G\subset X$ its characteristic function
$\nchi_G$ can be represented as
\begin{equation}
  \label{eq:324}
  \nchi_G(x)=\sup\Big\{\varphi(x):\varphi\in \rmC_b(X,\tau),\ 
  0\le \varphi\le \nchi_G\Big\},
\end{equation}
and the same representation holds for every nonnegative 
lower semicontinuous (l.s.c.) $f:X\to [0,+\infty]$:
\begin{equation}
  \label{eq:324lsc}
  f(x)=\sup\Big\{\varphi(x):\varphi\in \rmC_b(X,\tau),\ 
  0\le \varphi\le f\Big\},\quad
  f: X\to [0,+\infty]\ \text{l.s.c.}.
\end{equation}

\begin{definition}[Radon measures {\cite[Chap.~I,
    Sect. 2]{Schwartz73}}]
  \label{def:Radon}
  A finite Radon measure
  \index{Radon measure, Definition \ref{def:Radon}} $\mu:\BB(X,\tau)\to[0,+\infty)$ is a
  Borel nonnegative   $\sigma$-additive  finite measure satisfying the following
  inner regularity property:
    \begin{equation}
      \label{eq:214}
      \forall\,B\in \BB(X,\tau):\quad
      \mu(B)=\sup\Big\{\mu(K):K\subset B,\ \text{$K$ compact}\Big\}.
    \end{equation}
\end{definition}
\noindent
A finite Radon measure $\mu$ is
also outer regular:
\begin{equation}
  \label{eq:241}
  \forall\,B\in \BB(X,\tau):\quad
  \mu(B)=\inf\Big\{\mu(O):O\subset X,\  \text{$O$ open}\Big\}.
\end{equation}
We will denote by $\cMp(X)$ (resp.~$\cP(X)$) the collection of all
finite (resp.~Probability) Radon measures on $X$.
By the very definition of Radon topological space \cite[Ch.~II, Sect. 3]{Schwartz73}, 
every Borel measure in a Radon space is Radon: such class of spaces includes
locally compact spaces with a countable base of open sets,
Polish, Lusin and Souslin spaces.
In particular the notation of $\cP(X)$
is consistent with the standard one adopted e.g.~in
\cite{AGS11a,Ambrosio-Gigli-Savare08,Villani03}, where Polish or second countable locally compact spaces are
considered.

\eqref{eq:214} implies in particular that a Radon measure is tight:
\index{Tightness of a measure}
\begin{equation}
  \label{eq:279}
  \forall\,\eps>0\quad \exists\,K_\eps\subset X\quad\text{compact such
    that}\quad
  \mu(X\setminus K_\eps)\le \eps.
\end{equation}
\label{page:support}
We can also define in the usual way the support $\supp\mu$ of a Radon
measure as the set of points $x\in X$ such that every neighborhood $U\in
\UU_x$ has strictly positive measure $\mu(U)>0$. Thanks to
\eqref{eq:214},
one can verify that $\mu(X\setminus \supp(\mu))=0$ \cite[p.~60]{Schwartz73}.

Radon measures have stronger additivity and continuity properties in connection with open
sets and lower semicontinuous functions; in particular we shall use this version of the monotone convergence theorem
(see \cite[Lemma~7.2.6]{Bogachev07})
\begin{equation}\label{eq:Beppo_Levi_general}
\lim_{i\in I}\int f_i\,\d\mu=\int\lim_{i\in I}f_i\,\d\mu
\end{equation}
valid for Radon measures $\mu$ and for nondecreasing nets $i\mapsto
f_i$, $i\in I$, of
$\tau$-lower semicontinuous and equibounded functions
$f_i:X\to[0,\infty]$.
Here $I$ is a directed set with a partial order $\preceq$ satisfying
$i\preceq j\ \Rightarrow\ f_i\le f_j$, see the Appendix \ref{subsec:nets}.

The weak (or narrow) topology
\index{Weak topology in $\cMp(X)$}
\index{Narrow topology in $\cMp(X)$}
$\tau_{\cMp}$ on $\cMp(X)$ can be defined as the coarsest topology for which all maps 
\begin{equation}
  \label{eq:5}
  \mu\mapsto \int h\,\d\mu\qquad\text{from $\cMp(X)$ into $\R$}
\end{equation}
are continuous as $h:X\to \R$ varies in $\rmC_b(X,\tau)$ \cite[p. 370,
371]{Schwartz73}.

Prokhorov Theorem provides a sufficient condition for compactness
w.r.t.~the weak topology:
\cite[Theorem 3, p.~379]{Schwartz73}.
\begin{theorem}[Prokhorov]
  \label{thm:compa-Riesz}
  Let $(X,\tau)$ be a completely regular 
  Hausdorff topological space. Assume that a collection ${\mathcal K}\subset\cMp(X)$ 
  is uniformly bounded and equi-tight, i.e. \index{Equi-tightness of a collection of measures}
\begin{gather}
  \label{eq:511}
  \sup_{\mu\in \cK}\mu(X)<\infty,\\
    \label{eq:22}
  \text{for every $\eps>0$ there exists a compact set $K_\eps\subset
    X$ such that}\quad
  \sup_{\mu\in \mathcal K}\mu(X\setminus K_\eps)\le \eps\; .
\end{gather}
Then ${\mathcal K}$ has limit points in the class $\cMp(X)$ w.r.t.~the 
 weak topology.
\end{theorem}
Recall that a set $A\subset X$ is $\mm$-measurable
\index{Measurable sets}
if there
exist Borel sets $B_1,B_2\in \BorelSets{X,\tau}$ such that
$B_1\subset A\subset B_2$ and $\mm(B_2\setminus B_1)=0$.
$\mm$-measurable sets form a $\sigma$-algebra $\mathscr B_\mm(X).$ 
A set is called \emph{universally (Radon) measurable}
\index{Universally measurable sets} if it is
$\mu$-measurable
for every Radon measure $\mu\in \cMp(X)$.

Let $(Y,\tau_Y)$ be a Hausdorff topological space.
A map $f:X\to Y$ is Borel (resp.~Borel $\mm$-measurable)
\index{Borel maps, Borel $\mm$-measurable maps}
if for every $B\in \BorelSets Y$
$f^{-1}(B)\in \BorelSets X$ (resp.~$f^{-1}(B)$ is $\mm$-measurable).
$f$ is Lusin $\mm$-measurable
\index{Lusin $\mm$-measurable maps}
if
for every $\eps>0$ there exists a compact set $K_\eps\subset X$
such that $\mm(X\setminus K_\eps)\le \eps$ and
the restriction of $f$ to $K_\eps$ is continuous.
A map $f:X\to Y$ is called \emph{universally measurable}
\index{Universally measurable maps} if it is
Lusin $\mu$-measurable
for every Radon measure $\mu\in \cMp(X)$.

Every Lusin $\mm$-measurable map is also Borel $\mm$-measurable;
the converse is true if, e.g., the topology $\tau_Y$ is metrizable and
separable
\cite[Chap.~I Section 1.5, Theorem 5]{Schwartz73}.
Whenever $f$ is Lusin $\mm$-measurable, its push-forward
\index{Push forward of a measure}
\begin{equation}
  \label{eq:487}
  f_\sharp\mm\in \cMp(Y),\quad
  f_\sharp\mm(B):=\mm(f^{-1}(B))\quad
  \text{for every Borel subset }B\subset \BorelSets Y
\end{equation}
induces a Radon measure in $Y$.

Given a power $p\in (1,\infty)$ and a Radon measure $\mm$ in
$(X,\tau)$
we will denote by $L^p(X,\mm)$ the usual Lebesgue space
of class of $p$-summable $\mm$-measurable functions
defined up to $\mm$-negligible sets.
We will also set
\begin{equation}\label{eq:callp}
\mathcal{L}^p_+(X, \mm) := \left\{ f:X \to [0,\infty] \; : \;
  \text{$f$ is Borel,    $\int_X f^p \,\d\mm < \infty$} \right\};
\end{equation}
this space is not quotiented under any equivalence relation.
We will
keep using the notation
$$ \| f\|_p=
\|f\|_{L^p(X,\mm)}:= \left( \int_X |f|^p \, \d \mm \right)^{1/p}  $$
as a seminorm on $\Ldp$ and a norm in $L^p(X,\mm)$.

\subsection{Extended metric-topological (measure) spaces}
\label{subsec:extended-def}
Let $(X,\tau)$ be a Hausdorff topological space.

An extended semidistance \index{Extended distances and semidistances}
is a symmetric map $\delta:X\times X\to[0,\infty]$
satisfying the triangle inequality; $\delta$ is an extended distance if
it also satisfies the
property 
$\delta(x,y)=0$ iff $x=y$ in $X$: in this case, we call $(X,\delta)$ 
an extended metric space. 
We will omit the adjective
``extended'' if $\delta$ takes real values.

Whenever $f:X\to \R$ is a given function, $A\subset X$, and $\delta$
is an extended semidistance on $X$, we set
\begin{equation}
  \label{eq:204}
  \Lip(f,A,\delta):=\inf\Big\{L\in [0,\infty]:|f(y)-f(z)|\le L\delta(y,z)\quad
  \text{for every }y,z\in A\Big\}.
\end{equation}
We adopt the convention to omit the set $A$ when $A=X$.
%For every $\lambda>0$ 
We consider 
the class of $\tau$-continuous and $\delta$-Lipschitz
functions
% \begin{equation}
%   \label{eq:221}
%   \begin{aligned}
%     \Lipb\lambda(X,\tau,\delta):={}&\Big\{f\in \rmC_b(X,\tau):
%     \Lip(f,\delta)\le \lambda\Big\},\\
%     \Lipbuloc\lambda(X,\tau,\delta):={}&\Big\{f\in \rmC_b(X,\tau):
%     \forall\,x\in X\ \exists U\in \UU_x:
%     \Lip(X,U,\delta)\le \lambda\Big\}
%     % \Big\{f:X\to\R:\ \text{$f$ is bounded,
%     % $\delta$-Lipschitz, $\tau$-continuous}\Big\}
%     % \bigcup_{\lambda >0} \Lip_{b,\lambda}(X,\tau,\delta).
%   \end{aligned}
% \end{equation}
%
\begin{equation}
  \label{eq:221}
  \begin{aligned}
    \Lipb(X,\tau,\delta):={}&\Big\{f\in \rmC_b(X,\tau):
    \Lip(f,\delta)<\infty\Big\},
    % \\
    % \Lipbloc(X,\tau,\delta):={}&\Big\{f\in \rmC_b(X,\tau):
    % \forall\,x\in X\ \exists U\in \UU_x:
    % \Lip(f,U,\delta)<\infty\Big\};
    % \Big\{f:X\to\R:\ \text{$f$ is bounded,
    % $\delta$-Lipschitz, $\tau$-continuous}\Big\}
    % \bigcup_{\lambda >0} \Lip_{b,\lambda}(X,\tau,\delta).
  \end{aligned}
\end{equation}
and for every $\kappa>0$ we will also consider the subsets
\begin{equation}
  \label{eq:230}
  \begin{aligned}
      \Lipbu\kappa(X,\tau,\delta):={}&
    \Big\{f\in \rmC_b(X,\tau):
    \Lip(f,\delta)\le \kappa\Big\}.
    % \\
    % \Lipbuloc\lambda(X,\tau,\delta):={}&\Big\{f\in \rmC_b(X,\tau): \forall\,x\in
    % X\ \exists U\in \UU_x:
    % \Lip(f,U,\delta)\le \lambda\Big\}.
    % \bigcup_{\lambda >0}
    % \Lipbu\lambda(X,\tau,\delta),\quad
    % \Lipbloc(X,\tau,\delta):=\bigcup_{\lambda >0}
    % \Lipbuloc\lambda(X,\tau,\delta)
  \end{aligned}
\end{equation}
% and
% \begin{equation}\label{eq:defliptau}
% {\rm Lip}_b(X,\tau,\delta):=% \Big\{f:X\to\R:\ \text{$f$ is bounded,
%   % $\delta$-Lipschitz, $\tau$-continuous}\Big\}
% \bigcup_{\lambda >0}  \cL_\lambda(X,\tau,\delta).
% \end{equation}
%
%
A particular role will be played by $\Lipbu1(X,\tau,\delta)$.
%and its local version $\Lipbuloc1(X,\tau,\delta)$.
We will sometimes omit to indicate the explicit dependence on $\tau$
and $\delta$ whenever it will be clear from the context.
It is easy to check that $\Lipb(X,\tau,\delta)$
%and $\Lipbloc(X,\tau,\delta)$ 
is a real and
commutative sub-algebras of $\rmC_b(X,\tau)$ with unit.
%  endowed with
% the norm
% \begin{equation}
%   \label{eq:189}
%   \|f\|_{\Lip_b}:=\sup_{X} |f|+\Lip(f,\sfd).
% \end{equation}
% %
% x

According to \cite[Definition 4.1]{AES16}, 
an extended metric-topological space (e.m.t.~space) 
$(X,\tau,\sfd)$ is
characterized by a Hausdorff topology $\tau$ and 
an extended distance $\sfd$
satisfying a suitable compatibility condition.
\index{Extended metric-topological (e.m.t.) space}
%and a Radon$\cMp(X,\tau)$ will denote the space of finite and nonnegative Radon
%
\begin{definition}[Extended metric-topological spaces]\label{def:luft1}
Let $(X,\sfd)$ be an extended metric space, let
$\tau$ be a Hausdorff topology in $X$.
% , and let $\mm$ be a $\tau$-Borel
% measure in $X$.
We say that $(X,\tau,\sfd)$ is an extended
metric-topological (e.m.t.) space if:
\index{Initial topology}
\begin{enumerate}[\rm (X1)]
\item the topology $\tau$ is generated by the family of functions
  $\Lipb(X,\tau,\sfd)$ %(or, equivalently, by $\Lipbu1(X,\tau,\sfd)$;
  (see the Appendix \ref{subsec:initial});
\item
  the distance $\sfd$ can be recovered by the functions in
  $\Lipbu1(X,\tau,\sfd)$ through the formula
  \begin{equation}
    \label{eq:220}
    \sfd(x,y)=\sup_{f\in \Lipbu1(X,\tau,\sfd)}|f(x)-f(y)|
    \quad\text{for every }x,y\in X.
  \end{equation}
  % there exists a family of $(\tau\times\tau)$-continuous bounded semidistances $\sfd_i:X\times X\to [0,\infty)$, $i\in I$, 
% with  $\sfd=\sup_i\sfd_i$.
%\item[(c)] $\mm\in \cMp(X,\tau)$ (i.e.~it is a Radon measure) with $\mm(X)>0$.
\end{enumerate}
\index{Polish space}
\index{Lusin space}
\index{Souslin space}
We will say that $(X,\tau,\sfd)$ is \emph{complete} if $\sfd$-Cauchy
sequences are $\sfd$-convergent. All the other topological properties
(as compactness, separability, metrizability, Borel,
Polish-Lusin-Souslin, etc) usually refers to $(X,\tau)$.
\end{definition}
The previous assumptions guarantee that $(X,\tau)$ is completely
regular, according to \eqref{eq:CR} (see the Appendix
\ref{subsec:initial}).
As in \eqref{eq:512}, when an e.m.t.~space $(X,\tau,\sfd)$ is provided
by a positive Radon measure $\mm\in \cMp(X,\tau)$
we will call the system $\X=(X,\tau,\sfd,\mm)$
an extended metric-topological measure (e.m.t.m.) space.
\index{Extended metric-topological measure (e.m.t.m.) space}

Definition \ref{def:luft1} yields two important properties linking
$\sfd$ and $\tau$: first of all
\begin{equation}
  \label{eq:243}
  \sfd\text{ is $\tau\times\tau$-lower semicontinuous in $X\times X$},
\end{equation}
since it is the supremum of a family of continuo
us maps by
\eqref{eq:220}.
On the other hand, every $\sfd$-converging net $(x_j)_{j\in J}$
indexed by a directed set $J$ is also $\tau$-convergent: 
\begin{equation}
  \label{eq:244}
  \lim_{j\in J}\sfd(x_j,x)=0\quad\Rightarrow\quad
  \lim_{j\in J}x_j=x\quad\text{w.r.t.~$\tau$.}
\end{equation}
It is sufficient to observe that $\tau$ is the initial topology
generated by $\Lipb(X,\tau,\sfd)$ so that a net $(x_j)$ is convergent to
a point $x$ if
and only if 
\begin{equation}
  \label{eq:252}
  \lim_{j\in J}f(x_j)=f(x)\quad\text{for every }f\in \Lipb(X,\tau,\sfd).
\end{equation}
A basis of neighborhoods for the
$\tau$-topology at a point $x\in X$ is given by the sets of the form
\begin{equation}
  \label{eq:245}
  U_{F,\eps}(x):=\Big\{y\in X:\sup_{f\in F}|f(y)-f(x)|<\eps\Big\}\quad
  F\subset \Lipbu1(X,\tau,\sfd)\text{ finite, }\eps>0.
\end{equation}
Definition \ref{def:luft1} is in fact equivalent to other seemingly
stronger assumptions, 
as we discuss in the following Lemma.
\begin{lemma}[Monotone approximations of the distance]
  \label{rem:monotonicity}
  Let $(X,\tau,\sfd)$ be an e.m.t.~space,
  let us denote by $\Lambda$ the collection of all the finite subsets in 
  $\Lipbu1(X,\tau,\sfd)$, a directed set ordered by
  inclusion, and let us define
  \begin{equation}
    \label{eq:222}
    \sfd_\lambda(x,y):=\sup_{f\in \lambda}|f(x)-f(y)|,\quad \lambda\in \Lambda,\ x,y\in X.
  \end{equation}
  The family $(\sfd_\lambda)_{\lambda\in \Lambda}$ is a monotone
  collection of $\tau$ continuous and bounded semidistances
  on $X$
  generating the $\tau$-topology and the extended distance $\sfd$, in
  the sense that for every net $(x_j)_{j\in J}$ in $X$
  \begin{subequations}
  \begin{gather}
    \label{eq:328}
    x_j\stackrel\tau\to x\quad\Leftrightarrow\quad
    \lim_{j\in J}\sfd_\lambda(x_j,x)=0\quad\text{for every }\lambda\in
    \Lambda,
    \intertext{and}
    \label{eq:329}
    \sfd(x,y)=\sup_{\lambda\in \Lambda} \sfd_\lambda(x,y)=\lim_{\lambda\in
      \Lambda}\sfd_\lambda(x,y)\quad
    \text{for every }x,y\in X.
  \end{gather}
  \end{subequations}
  Conversely, suppose that  $(\sfd_i)_{i\in I}$ is a directed family
  of real functions on $X\times X$ satisfying
  \begin{subequations}
    \label{eq:monotone}
  \begin{gather}
    \label{eq:330}
    \sfd_i:X\times X\to [0,+\infty)
    \quad\text{is a bounded and continuous semidistance for
      every $i\in I$},\\
    \label{eq:326}
    i\preceq j\quad\Rightarrow\quad
    \sfd_i\le \sfd_j,\\
    \label{eq:259}
    x_j\stackrel\tau\to x\quad\Leftrightarrow\quad
    \lim_{j\in J}\sfd_i(x_j,x)=0\quad\text{for every }i\in I,\\
    \label{eq:260}
    \sfd(x,y)=\sup_{i\in I}\sfd_i(x,y)=\lim_{i\in I}\sfd_i(x,y)\quad
    \text{for every }x,y\in X,
  \end{gather}
  \end{subequations}
  then $(X,\tau,\sfd)$ is an extended
  metric-topological space.
\end{lemma}
\begin{proof}
  \eqref{eq:328} and \eqref{eq:329} are immediate consequence of the
  Definition
  \eqref{def:luft1}. In order to prove the second statement, we simply
  observe that the collection of functions
  $\cF:=\{\sfd_i(y,\cdot):i\in I,\ y\in X\}$
  is included in $\Lipbu1(X,\tau,\sfd)$ and generates the topology
  $\tau$
  thanks to \eqref{eq:259}. A fortiori, $\Lipbu1(X,\tau,\sfd) $
  satisfies conditions (X1) and (X2) of 
  Definition \ref{def:luft1}.
\end{proof}
We will often use the following simple and useful property
involving a directed family of semidistances $(\sfd_i)_{i\in I}$ 
satisfying (\ref{eq:monotone}a,b,c,d):
whenever $\fri:J\to I$ is a subnet and $x_j,y_j$, $j\in J$, are $\tau$-converging 
to $x,y$ respectively, we have
\begin{equation}
  \label{eq:321}
  \liminf_{j\in J}\sfd_{\fri(j)}(x_j,y_j)\ge \sfd(x,y).
\end{equation}
It follows easily by the continuity of $\sfd_i$ and \eqref{eq:326},
since for every $i\in I$
\begin{displaymath}
  \liminf_{j\in J}\sfd_{\fri(j)}(x_j,y_j)\ge
  \liminf_{j\in J}\sfd_{i}(x_j,y_j)\ge
  \sfd_i(x,y);
\end{displaymath}
\eqref{eq:321} then follows by taking the supremum w.r.t.~$i\in I$.
\begin{remark}
  \label{rem:compact-case}
  \upshape
  Notice that if $K$ is a Souslin subset of $X$
  (in particular $K=X$ if $(X,\tau)$ is Souslin) then $K\times K$ is
  Souslin as well, so that by Lemma \ref{le:Souslin-useful}(b)
  there exists a countable collection $\cF=(f_n)_{n\in \N}\subset
  \Lip_{b,1}(X,\tau,\sfd)$ such that
  \begin{equation}
    \label{eq:553}
    \sfd(x,y)=\sup_{n\in \N}|f_n(x)-f_n(y)|\quad
    \text{for every }x,y\in K.
  \end{equation}
  If $\tau'$ is the initial topology generated by $\cF$,
  $(K,\tau',\sfd)$ is an e.m.t.~space
  whose topology $\tau'$ is coarser than $\tau$.
  $\tau'$ is also metrizable and separable:
  it is sufficient to choose
  an increasing 1-Lipschitz homeomorphism
  $\vartheta:\R\to ]0,1/12[$
  and setting $f_n':=\vartheta\circ f_n$;
  the family $(f_n')_{n\in \N}$
  induces the same topology $\tau'$,
  it separates the points of $K$, and
  the distance
    \begin{equation}
    \label{eq:251}
    \sfd'(x,y):=\sum_{n=1}^\infty 2^{-n}|f_n'(x)-f_n'(y)|
  \end{equation}
  is a bounded $\tau$-continuous semidistance dominated by $\sfd$ 
  whose restriction to $K\times K$ is a distance inducing the topology $\tau'$.
  If $K$ is also compact, than 
  $\tau$ coincides with the topology induced by $\sfd'$.
\end{remark}
Let us recap a useful property discussed in 
the previous Remark.
\index{Auxiliary topology}
\begin{definition}[Auxiliary topologies]
  \label{def:auxiliary}
  Let $(X,\tau,\sfd)$ be an e.m.t.~space. We say that
  $\tau'$ is an \emph{auxiliary} topology if
  there exist a countable collection $\cF=(f_n)_{n\in \N}\subset
  \Lip_b(X,\tau,\sfd)$ such that
  $\tau'$ is generated by $\cF$
  and
  \begin{equation}
    \label{eq:554}
    \sfd(x,y)=\sup_{n\in \N}|f_n(x)-f_n(y)|.
  \end{equation}
  Equivalently 
  \begin{enumerate}[({A}1)]
  \item $\tau'$ is coarser than $\tau$,
  \item $\tau'$ is separable and metrizable
    by a bounded $\tau$-continuous distance $\sfd'\le \sfd$,
  \item there exists a sequence $f_n\in
    \Lip_b(X,\tau',\sfd)$ such that
    \eqref{eq:554} holds.
  \end{enumerate}
  In particular, $(X,\tau',\sfd)$ is an e.m.t.~space.
\end{definition}
If $\tau'$ is generated by a countable collection
$\cF\subset \Lip_b(X,\tau,\sfd)$ satisfying \eqref{eq:554}
then properties (A1,2,3) obviously hold by the discussion of Remark
\ref{rem:compact-case}. Conversely,
if $\tau'$ satisfies (A1,2,3) then
one can consider the countable collection $\cF$
resulting by the union of $(f_n)_{n\in \N}$ given in (A3) and
the set $\{\sfd'(x_n,\cdot)\}_{n\in \N}$ where $(x_n)_{n\in \N}$ is a
$\tau'$ dense subset of $X$ and $\sfd'$ is given by (A2).
It is clear that $\tau'$ is the initial topology of $\cF$
and \eqref{eq:554} holds.

By setting 
\begin{displaymath}
  \sfd_n(x,y):=\sup_{1\le k\le n}|f_k(x)-f_k(y)|
\end{displaymath}
one can easily see that (A3) is in fact equivalent to 
\begin{enumerate}[(A3')]
\item There exists an increasing sequence of $\tau'$ continuous and bounded
  (semi)distances $(\sfd_n)_{n\in \N}$ such that
  \begin{equation}
    \label{eq:555}
    \tag{\ref{eq:554}'}
    \sfd(x,y)=\sup_{n\in\N}\sfd_n(x,y)=
    \lim_{n\to\infty}\sfd_n(x,y)\quad\text{for every }x,y\in X.
  \end{equation}
  It is also possible to assume $\sfd_n\ge \sfd'$ for every $n\in \N$.
\end{enumerate}
As a consequence of Remark \ref{rem:compact-case} we have:
\begin{corollary}[Auxiliary topologies for Souslin e.m.t.~spaces]
  \label{cor:auxiliary-Souslin}
  If $(X,\tau,\sfd)$ is a Souslin e.m.t.~space (i.e.~$(X,\tau)$ is
  Souslin)
  then it admits an auxiliary topology $\tau'$ according
  to Definition \ref{def:auxiliary}.
\end{corollary}
Notice that if $\tau'$ is an auxiliary topology of a Souslin e.m.t.~space
$(X,\tau,\sfd)$,
$(X,\tau')$ is Souslin as well. If $\mm$ is a Radon measure in
$(X,\tau)$
then it is Radon also w.r.t.~$\tau'$.
An important consequence of the existence of an 
auxiliary topology is the following fact:
\begin{lemma}
  \label{le:useful-auxiliary}
  If $(X,\tau,\sfd)$ admits an auxiliary topology $\tau'$
  then every $\tau$-compact set $K\subset X$
  is a Polish space (with the relative topology).
\end{lemma}
\begin{proof}
  It is sufficient to note that $\tau$ and $\tau'$ induces the same
  topology on $K$ and that $\tau'$ is metrizable and separable.
\end{proof}
\subsection{Examples}
\label{subsec:Examples}
\begin{example}[Complete and separable metric
  spaces]
  \label{ex:polish}
  The most important and common example is provided by 
  a 
  complete and separable metric space $(X,\sfd)$; in this case, 
  the canonical choice of $\tau$ is the (Polish) topology induced by
  $\sfd$.
  Any positive and finite Borel measure $\mm$ on $(X,\sfd)$ is
  a Radon measure so that $(X,\sfd,\mm)$ is a Polish metric measure
  space.
  This case cover the Euclidean spaces $\R^d$,
  the complete Riemannian or Finsler manifolds, the separable Banach
  spaces
  and their closed subsets.
  % The same natural choice of $\tau$ occurs 
  % in the case of separable and locally complete spaces
  % (i.e.~every point $x\in X$ has a neighborhood $U$ which is complete w.r.t.~$\sfd$).
\end{example}
In some situation, however, when $(X,\sfd)$
is not separable or $\sfd$ takes the value $+\infty$,
it could be useful to distinguish
between the topological and the metric aspects. This will
particularly important when a measure will be involved, since
the Radon property with respect to a coarser topology is
less restrictive.
\begin{example}[Dual of a Banach space]
  \label{ex:dual}
  A typical example is provided by the dual $X=B'$ of a separable
  Banach space: in this case the distance $\sfd$ is induced by the
  dual norm $\|\cdot\|_{B'}$ of $B'$ (which may not be separable) and
  the topology $\tau$ is the weak$^*$ topology of $B'$, which is Lusin
  \cite[Corollary 1, p.~115]{Schwartz73}.
  All the functions of the form
  $f(x;y,v,r):=\sft(\la x-y,v\ra)$ where $y\in B'$,
  $\sft:\R\to\R$ is a bounded $1$-Lipschitz map and $v\in B$
  with $\|v\|_B\le 1$
  clearly belong to $\Lip_b(X,\tau,\sfd)$ and are sufficient to
  recover the distance $\sfd$ since
  \begin{displaymath}
    \sfd(x,y)=\|x-y\|_{B'}=
    \sup_{\|v\|_B\le 1}\la x-y,v\ra
    =\sup_{\|v\|\le 1} f(x;y,v,\sft),\quad
    \sft(r):=0\lor r\land (2\|x-y\|).
  \end{displaymath}
\end{example}
A slight modification of the previous setting leads
to a somehow universal model:
we will see
in \S\,\ref{subsec:compactification}
that every e.m.t.~space can be
isometrically and continuously embedded
in such a framework and
every metric Sobolev space has an isomorphic representation in this
setting
(see Corollary \ref{cor:compact-r}).
\begin{example}
  \label{ex:model}
  Let $X$ be weakly$^*$ compact subset of
  a dual Banach space $B'$ endowed with the
  weak$^*$ topology $\tau$
  and a Radon measure $\mm$.
  We select a strongly closed
  and symmetric convex set $L\subset B$
  containing $0$ and separating the
  points of $B'$ and we set
  \begin{equation}
    \label{eq:528}
    \theta(z):=\sup_{f\in L}\la z,f\ra,\quad
    \sfd(x,y):=\theta(x-y).
  \end{equation}
  It is immediate to check that $(X,\tau,\sfd)$ is an
  e.m.t.~space.
  Notice that
  $\theta$ 
  is $1$-homogeneous and convex, therefore
  it is an ``extended'' norm 
  (possibly assuming the value $+\infty$),
  so that $\sfd$ is translation invariant.
  The previous Example \ref{ex:dual} correspond
  to the case when $L$ is the unit ball of $B$.
\end{example}
\begin{example}[Abstract Wiener spaces]
  \label{ex:Wiener}
  Let $(X,\|\cdot\|_X)$ be a separable Banach space endowed with a Radon measure $\mm$
  and let $(W,|\cdot|_W)$ be a reflexive Banach space
  (in particular an Hilbert space)
  densely and continuously
  included in $X$, so that there exists a constant $C>0$ such that
  \begin{equation}
    \label{eq:513}
    \|h\|_X\le C|h|_W\quad
    \text{for every $h\in W$.}
  \end{equation}
  We call $\tau$ the Polish topology of $X$ induced by the Banach norm
  and for every $x,y,z\in X$ we set
  \begin{equation}
    \label{eq:514}
    \phi(z):=
    \begin{cases}
      |z|_W&\text{if }z\in W,\\
      +\infty&\text{otherwise},
    \end{cases}
    \qquad
    \sfd(x,y):=\phi(x-y)=
    \begin{cases}
      |x-y|_W&\text{if }x-y\in W,\\
      +\infty&\text{otherwise},
    \end{cases}
  \end{equation}
  The functional $z\mapsto \phi(z)$ is $1$-homogenous,
  convex and lower semicontinuous
  in $X$ (thanks to the reflexivity of $W$)
  so that setting
  \begin{displaymath}
    L:=\Big\{f\in X':\la f,z\ra\le |z|_B\ \text{for every }z\in B\Big\},
  \end{displaymath}
  Fenchel duality yields
  \begin{equation}
    \label{eq:515}
    \phi(z)=\sup_{f\in L}\la f,z\ra,\quad
    \sfd(x,y)=\sup_{f\in L}\la f,x-y\ra;
  \end{equation}
  the same truncation trick of Example \ref{ex:dual} shows that
  \eqref{eq:220} is satisfied. On the other hand, the
  distance functions $x\mapsto \|x-z\|_X$, $z\in X$,
  induced by the norm in $X$ belong to $\Lip_b(X,\tau,\sfd)$ so that
  the first condition of Definition \ref{def:luft1} is satisfied as
  well.
  This setting covers the important case of an abstract Wiener space, when
  $\mm$ is a Gaussian measure in $X$ and
  $W$ is the Cameron-Martin space, see e.g.~\cite{Bogachev98}.
\end{example}
\begin{example}
  \label{ex:length-sub}
  Let $X:=\R^d$ and let $h:X\times \R^d\to
  [0,+\infty]$ be a lower semicontinuous function
  such that for every $x\in X$ 
  \begin{displaymath}
    h(x,\cdot)\quad\text{is $1$-homogeneous and convex,}\quad
    h(x,v)>h(x,0)=0\quad \text{for every }x\in X,\ v\in \R^d\setminus\{0\}.
  \end{displaymath}
  We can define the extended ``Finsler'' distance
  \begin{equation}
    \label{eq:517}
    \begin{aligned}
      \sfd(x_0,x_1):=\inf\Big \{
      &\int_0^1 h(\rmx (t),\rmx'(t))\,\d t:
      \rmx\in \Lip([0,1],\R^d),\quad \rmx(i)=x_i,\ i=0,1\Big\}
      % \\&\quad
      % \rmx(t)\in X\text{ for every
      % }t\in [0,1]\Big\},
    \end{aligned}
  \end{equation}
  with the convention that $\sfd(x_0,x_1)=+\infty$ if there is no
  Lipschitz curve connecting $x_0$ to $x_1$ with a finite cost.
  When there exist constants $C_0,C_1>0$ such that
  \begin{equation}
    \label{eq:519}
   C_0|v|\le h(x,v)\le C_1|v|\quad\text{for every }x,v\in \R^d, 
  \end{equation}
  $\sfd$ is the ``Finsler'' distance induced by the family of norms
  $\big(h(x,\cdot)\big)_{x\in \R^d}$, inducing the usual topology of
  $\R^d$.

  When
  \begin{equation}
    \label{eq:520}
    h(x,v)=
    \begin{cases}
      |v|&\text{if }x\in X_0,\\
      +\infty&\text{if }x\in \R^d\setminus X_0,\ v\neq 0,
    \end{cases}\qquad
    X_0\quad\text{is a closed subset of }\R^d,
  \end{equation}
  then $\sfd$ is the ``geodesic extended distance''
  induced by the Euclidean tensor on $X_0$.
  When
  % $X=\R^d$ and
  $h$ is expressed in terms of a
  smooth family of bounded vector fields $(X_j)_{j=1}^J$, $X_j:\R^d\to\R^d$,
  by the formula
  \begin{equation}
    \label{eq:518}
    h^2(x,v):=\inf\Big\{\sum_{j=1}^Ju_j^2:
    \sum_{j=1}^J u_jX_j(x)=v\Big\}
  \end{equation}
  we obtain the Carnot-Caratheodory distance
  induced by the vector fields $X_j$.
  In all these cases, we can approximate $h$ by its
  Yosida regularization:
  \begin{equation}
    \label{eq:521}
    h_\eps^2(x,v):=\inf_{w\in \R^d} h^2(x,w)+\frac1{2\eps}|w-v|^2
    \quad x,v\in \R^d,\ \eps>0,
  \end{equation}
  which satisfy
  \begin{equation}
    0<h^2_\eps(x,v)\le \frac 1{2\eps}|v|^2,\quad
  \lim_{\eps\downarrow0}h_\eps(x,v)=
  h(x,v)\quad\text{for every $v\in
  \R^d\setminus\{0\}$}.\label{eq:522}
\end{equation}
If we define the Finsler distance $\sfd_\eps$ as in \eqref{eq:517} in
terms of $h_\eps$ we can easily see that
\begin{equation}
  \label{eq:523}
  0<\sfd^2_\eps(x,y)\le
  \frac 1{2\eps}|x-y|^2\quad
  \lim_{\eps\down0}\sfd_\eps(x,y)=\sfd(x,y)\quad
  \text{for every }x,y\in \R^d,\ x\neq y.
\end{equation}
If $\tau$ is the usual Euclidean topology, we
obtain that $(\R^d,\tau,\sfd)$ is an extended metric-topological space.
\end{example}
\subsection{The Kantorovich-Rubinstein distance}
\label{subsec:KR}
\index{Kantorovich-Rubinstein distance and duality}
Let $(X,\tau,\sfd)$ be an extended metric-topological space.
We want to lift the same structure to the space of Radon probability
measures $\cP(X)$. We introduce the main definitions for 
of couple of measures
$\mu_0,\mu_1\in \cMp(X)$ with the same mass $\mu_0(X)=\mu_1(X)$.

We denote by $\Gamma(\mu_0,\mu_1)$ the collection of
plans $\mmu\in \cMp(X\times X)$ whose marginals are $\mu_0$ and
$\mu_1$ respectively:
\begin{equation}
  \label{eq:488}
  \Gamma(\mu_0,\mu_1):=\Big\{
  \mmu\in \cMp(X\times X):\pi^i_\sharp\mmu=\mu_i\Big\},\quad
  \pi^i(x_0,x_1):=x_i.
\end{equation}
It is not difficult to check that $\Gamma(\mu_0,\mu_1)$
is a nonempty (it always contains $\mu_0^{-1}(X)\, \mu_0\otimes
\mu_1$) and compact subset of $\cMp(X\times X)$.

Let $\delta:X\times X\to [0,+\infty]$ be a lower semicontinuous
extended semi distance. 
The Kantorovich formulation of
the optimal transport problem with cost $\delta$
induces the celebrated Kantorovich-Rubinstein (extended,
semi-)distance $\sfK_\delta$
in $\cP(X)$
\cite[Chap.~7]{Villani03}
\begin{equation}
  \label{eq:489}
  \sfK_\delta(\mu_0,\mu_1):=\inf
  \Big\{\int_{X\times X} \delta(x_0,x_1)\,\d\mmu(x_0,x_1):
  \mmu\in \Gamma(\mu_0,\mu_1)\Big\}.
\end{equation}
\begin{proposition}
  Let $\mu_0,\mu_1\in \cMp(X)$ with the same mass.
  \begin{enumerate}
  \item If $\sfK_\delta(\mu_0,\mu_1)$ is finite then the infimum in
    \eqref{eq:489}
    is attained. In particular, this holds if $\delta$ is bounded.
  \item {\upshape [Kantorovich-Rubinstein duality]}
    If $\delta$ is a bounded continuous (semi)-distance in
    $(X,\tau)$ then $\sfK_\delta$ is a bounded continuous
    (semi)-distance
    in $\cP(X)$ and
    \begin{align}
      \notag
      \sfK_\delta(\mu_0,\mu_1)
      &=
        \sup\Big\{\int \phi_0\,\d\mu_0-
        \int \phi_1\,\d\mu_1: \phi_i\in \rmC_b(X,\tau),
        \\&\qquad\qquad
      \phi_0(x_0)-\phi_1(x_1)\le \delta(x_0,x_1)\quad
      \text{for every }x_o,x_1\in X\Big\}
      \label{eq:490}
      \\&= \sup\Big\{\int \phi\,\d(\mu_0-\mu_1):
      \phi\in \Lip_{b,1}(X,\tau,\delta)\Big\}
      \label{eq:492}
    \end{align}
  \item
    If $(\sfd_i)_{i\in I}$ is a directed collection of bounded
    continuous semidistances satisfying $\lim_{i\in I}\sfd_i=\sfd$ then
    \begin{equation}
      \label{eq:491}
      \sfK_\sfd(\mu_0,\mu_1)=\lim_{i\in I}
      \sfK_{\sfd_i}(\mu_0,\mu_1).
    \end{equation}
  \item If $(X,\tau,\sfd)$ is an extended metric-topological space
    \begin{align}
      \notag
      \sfK_\sfd(\mu_0,\mu_1)
      &=
        \sup\Big\{\int \phi_0\,\d\mu_0-
        \int \phi_1\,\d\mu_1: \phi_i\in \rmC_b(X,\tau),
        \\&\qquad\qquad
      \phi_0(x_0)-\phi_1(x_1)\le \sfd(x_0,x_1)\quad
      \text{for every }x_0,x_1\in X\Big\}
      \label{eq:490d}
      \\&= \sup\Big\{\int \phi\,\d(\mu_0-\mu_1):
      \phi\in \Lip_{b,1}(X,\tau,\sfd)\Big\}
      \label{eq:493}
    \end{align}
  \end{enumerate}
\end{proposition}
\begin{proof}
  \textbf{(a)}
  follows by the lower semicontinuity of $\delta$ and the compactness
  of $\Gamma$.
  \medskip
  
  \noindent
  \textbf{(b)}
  we refer to \cite[Chap.~7]{Villani03}.
  \medskip
  
  \noindent
  \textbf{(c)} follows by the property
  \begin{equation}
    \label{eq:524}
    \liminf_{j\in J}\int \sfd_{i(j)}\,\d\ppi_j\ge
    \int \sfd\,\d\ppi
  \end{equation}
  whenever $(\ppi_j)_{j\in J}$ is a net in 
  $\Gamma(\mu_0,\mu_1)$ converging weakly to $\ppi$ and
  $j\mapsto i(j)$ is a subnet in $I$. See \cite[Theorem 5.1]{AES16}
    \medskip
  
  \noindent
    \textbf{(d)}
    is an immediate consequence of \eqref{eq:491} and Claim (b),
    which yields that $\sfK_\sfd$ is less or equal than the
    two expression in the right-hand side of \eqref{eq:490d} and
    \eqref{eq:493}. The converse inequality is obvious.
  \end{proof}
  \begin{remark}
    \upshape
    Thanks to the previous proposition, it would not be difficult to
    check that $(\cP(X),\tau_\cP,\sfK_\sfd)$ is an
    extended metric-topological space as well.
  \end{remark}
  \subsection{The asymptotic Lipschitz constant}
  \label{subsec:aslip}
  \index{Asymptotic Lipschitz constant}
  Whenever $\delta$ is an extended, $\tau$-lower semicontinuous
semidistance, 
and $f:X\to \R$, we set
\begin{equation}
  \label{eq:55}
  %\Lip (f,A):=\sup_{y,z\in A,y\neq z}
  %\frac{|f(z)-f(y)|}{\sfd(y,z)},\quad
  \lip_\delta f(x):= \lim_{U\in \UU_x}\Lip(f,U,\delta)=
  \inf_{U\in \UU_x}\Lip(f,U,\delta)\quad
  x\in X;
\end{equation}
recall that $\UU_x$ is the directed set of all the $\tau$-neighborhood
of $x$.
Notice that $\Lip(f,\{x\})=0$ and therefore $\lip f(x)=0$
if $x$ is an isolated point of $X$.
We will often omit the index $\delta$ when $\delta=\sfd$.
When $\delta$ is a distance, we can also define
$\lip_\delta$ as
\begin{equation}
  \label{eq:265}
  \lip_\delta f(x)=\limsup_{y,z\to x\atop y\neq z}\frac{|f(y)-f(z)|}{\delta(y,z)};
\end{equation}
in particular,
\begin{equation}
  \label{eq:266}
  \lip_\delta f(x)\ge |\rmD_\delta f|(x):=\limsup_{y\to x}\frac{|f(y)-f(x)|}{\delta(x,y)}.
\end{equation}
It is not difficult to check that $x\mapsto \lip_\delta f(x)$ is a $\tau$-upper
semicontinuous map and $f$ is locally $\delta$-Lipschitz in $X$ iff $\lip_\delta
f(x)<\infty$ for every $x\in X$.
When $(X,\delta)$ is a length space, 
$\lipdlt f$ coincides with the upper semicontinuous
envelope of the local Lipschitz constant \eqref{eq:266}.

We collect in the next useful lemma
the basic calculus properties of
$\lipdlt f$.
\begin{lemma}
  \label{le:Locality}
  For every $f,g,\nchi\in \rmC_b(X)$ with $\nchi(X)\subset [0,1]$ we have
  \begin{subequations}
    \begin{align}
    \label{eq:57}
    \lipdlt (\alpha f+\beta g)&\le |\alpha|\, \lipdlt f+|\beta|\, \lipdlt
                             g\quad\text{for every }\alpha,\beta\in
                                \R,\\
      \label{eq:306}
    \lipdlt(fg)&\le |f|\lipdlt g+|g|\lipdlt f,\\
      \label{eq:307}
    \lipdlt((1-\nchi)f+\nchi g)&\le (1-\nchi)\lipdlt f+\nchi\lipdlt
                              g+
                              \lipdlt \nchi|f-g|.
                              \intertext{Moreover, whenever $\phi\in\rmC^1(\R)$}
    \label{eq:309}\lipdlt(\phi\circ f)&=|\phi'\circ f|\lipdlt f
                       \intertext{and 
                       for every convex and nondecreasing function
                                        $\psi:[0,\infty)\to \R$
                                        and every map $\zeta\in
                                        \rmC^1(\R)$ with $0\le
                                        \zeta'\le 1$,
                                        the transformation}
                       \notag
    \tilde f:=f+\zeta(g-f),&\quad
    \tilde g:=g+\zeta(f-g)
                             \intertext{satisfies}
\label{eq:308}
    \psi(\lipdlt \tilde f)+
                          \psi(\lipdlt \tilde g)&\le 
                            \psi(\lipdlt f)+\psi(\lipdlt g).
  \end{align}
  \end{subequations}
\end{lemma}
\begin{proof}
  \eqref{eq:57} follows by the obvious inequalities
  \begin{displaymath}
    \Lip(\alpha f+\beta g,U,\delta)
    \le |\alpha|\Lip(f,U,\delta)+
    |\beta|\Lip(g,U,\delta)
  \end{displaymath}
  for every subset $U\subset X$. Similarly, for every $y,z\in U$
  \begin{align*}
    |f(y)g(y)-f(z)g(z)|
    &\le 
    |(f(y)-f(z))g(y)|+|(g(y)-g(z))f(z)|
    \\&\le \Big(\Lip(f,U,\delta)\sup_U|g|+
    \Lip(g,U,\delta)\sup_U|f|\Big)\delta(z,y)
  \end{align*}
  and we obtain \eqref{eq:306} passing to the limit w.r.t.~$U\in
  \UU_x$.
  Setting $\tilde\nchi:=1-\nchi$, \eqref{eq:307} 
  follows by
  \begin{align*}
    |\nchi(y)f(y)&+\tilde\nchi(y)g(y)-
    \nchi(z)f(z)+\tilde\nchi(z)g(z)|
                   \\&\le
      |\nchi(y)(f(y)-f(z))|+|\tilde\nchi(y)(g(y)-g(z))|
      +|(\nchi(y)-\nchi(z))(f(z)-g(z))|
    \\&\le
        \Big(\sup_U\nchi\,\Lip(f,U,\delta)+\sup_U\tilde\nchi\Lip(g,U,\delta)+\sup_U|f-g|
        \Lip(\nchi,U,\delta)\Big)\delta(y,z)
  \end{align*}
  and passing to the limit w.r.t.~$U\in \UU_x$.

  Concerning \eqref{eq:309},
  for every $y,z\in U$ we get
  \begin{displaymath}
    |\phi(f(y))-\phi(f(z))|\le \Lip(\phi,f(U))\,\Lip(f,U,\delta)\delta(y,z)
  \end{displaymath}
  which easily yields $\lipdlt \phi\circ f(x)\le |\phi'(f(x))|\lipdlt
  f(x)$. If $\phi'(f(x))\neq 0$, we can find a $\rmC^1$ function
  $\psi:\R\to\R$
  such that $\psi(\phi(r))=r$ in a neighborhood of $f(x)$, so that 
  the same property yields
  $\lipdlt f(x)\le \frac1{|\phi'(f(x))|}\lipdlt f\circ \phi(x)$ and
  the identity in \eqref{eq:309}.

  Let us eventually consider \eqref{eq:308}. As usual, we consider
  arbitrary points $y,z\in U$, $U\in \UU_x$ obtaining
  \begin{align*}
    |\tilde f(y)-\tilde f(z)|&=
                               |f(y)-f(z)+\zeta(g(y)-f(y))-\zeta(g(z)-f(z))|
    \\&=|f(y)-f(z)+\alpha\big((g(y)-g(z)-(f(y)-f(z))|
     \\&=
         |(1-\alpha)
         (f(y)-f(z))+\alpha (g(y)-g(z))|
    \\&\le \Big((1-\alpha)\Lip(f,U,\delta)+\alpha \Lip(g,U,\delta)\Big)\delta(y,z)
  \end{align*}
  for some $\alpha=\alpha_{y,z}=\zeta'(\theta_{y,z})\in [0,1]$,
  where $\theta_{y,z}$ is a convex combination of $g(y)-f(y)$ and
  $g(z)-f(z)$. Passing to the limit w.r.t.~$U$ and observing that
  $\alpha\to \zeta'(g(x)-f(x))$ we get
  \begin{displaymath}
    \lipdlt \tilde f(x)\le (1-\zeta'(g(x)-f(x)))\lipdlt
    f(x)+\zeta'(g(x)-f(x))\lipdlt g(x).
  \end{displaymath}
  A similar argument yields
  \begin{displaymath}
    \lipdlt \tilde g(x)\le (1-\zeta'(f(x)-g(x)))\lipdlt
    g(x)+\zeta'(f(x)-g(x))\lipdlt f(x).
  \end{displaymath}
  Since $\psi$ is convex and nondecreasing, we obtain \eqref{eq:308}.
\end{proof}

    \GGG
    \subsection{Compatible algebra of functions}
\label{subsec:compatible-A}
\index{Compatible algebra of Lipschitz functions}
\index{Adapted algebra of Lipschitz functions}
% \begin{equation}
%   \label{eq:231}
%   \Lipb(X,\tau,(\sfd_i)):=\bigcup_{\lambda>0}
%     \Lipbu\lambda(X,\tau,(\sfd_i)),\quad
%     \Lipbloc(X,\tau,(\sfd_i)):=\bigcup_{\lambda>0}
%     \Lipbuloc\lambda(X,\tau,(\sfd_i))
% \end{equation}
We have seen in Section \ref{subsec:extended-def}
the important role played by the algebra of function
$\Lipb(X,\tau,\sfd)$.
In many situations it could be useful to consider smaller subalgebras
which are however sufficiently rich to recover the
metric properties of an extended metric topological space
$(X,\tau,\sfd)$.
\begin{definition}[Compatible algebras of Lipschitz functions]
  \label{def:A-compatibility}
  Let $\AA$ be a unital subalgebra of\\
  $\Lipb(X,\tau,\sfd)$
  and let us set
  $\AA_\kappa:=\AA\cap \Lipbu\kappa(X,\tau,\sfd)$.
  
  We say that
  $\AA$ is \emph{compatible} with the metric-topological structure
  $(X,\tau,\sfd)$ if
  %if there exists a Borel set $A\subset X$ with $\mm(X\setminus A)=0$
  %such that 
  % \begin{equation}
  %   \label{eq:429}
  %   \AA_{1}:=\Big\{f\in \AA:|f(x)-f(y)|\le \sfd(x,y)\quad
  %   \text{for every }x,y\in X\},
  % \end{equation}
  \begin{gather}
    \label{eq:214bis}
    %\unit\in \AA,\quad
    % \AA \text{ is closed with respect to $\lor$ and $\land$,}\quad
    %\AA \text{\ separates the points of $B$ },\\
    \sfd(x,y)=\sup_{f\in \AA_1}|f(x)-f(y)|\quad
    \text{for every }x,y\in X.
  \end{gather}
  In particular, $\AA$ separates the points of $X$.\\
  We say that $\AA$ is \emph{adapted} to $(X,\tau,\sfd)$ if
  $\AA$ is compatible with $(X,\tau,\sfd)$ and
  it generates the topology $\tau$.
\end{definition}
If we do not make a different explicit choice, we will always assume that
an e.t.m.m.~space $\X$ is endowed with the canonical algebra
$\AA(\X):=\Lipb(X,\tau,\sfd)$.
% \begin{remark}

\begin{remark}[Coarser topologies and countably generated algebras]
  \label{rem:coarser}
  Suppose that $\AA\subset \Lip_b(X,\tau,\sfd)$ is an algebra
  compatible with $(X,\tau,\sfd)$ and let $\tau_\AA$ be the
  initial topology generated by $\AA$ (see
  \ref{subsec:initial} in the Appendix).
  Then $(X,\tau_\AA,\sfd)$ is an e.m.t.~space as well
  and $\AA$ is adapted to $(X,\tau_\AA,\sfd)$;
  a Radon measure $\mm\in \cMp(X,\tau)$ is also Radon in
  $(X,\tau_\AA)$.\\
  This property shows that there is some flexibility in the choice of
  the topology $\tau$, as long as $\tau$-continuous functions
  are sufficiently rich to generate the distance $\sfd$.
  An interesting example occurs when $(X,\tau)$ is a Souslin space.
  By Remark \ref{rem:compact-case} we can always find
  a countable collection $(f_n)_{n\in \N}$ of $\Lip_b(X,\tau,\sfd)$ (or of
  a compatible algebra $\AA$) satisfying \eqref{eq:553}.
  If we denote by $\AA'$ the
  algebra generated by the functions $f_n$, $n\in\N$,
  we obtain a countably generated algebra
  and an auxiliary topology $\tau'=\tau_{\AA'}$ according to
  Definition
  \ref{def:auxiliary}.
\end{remark}
\subsubsection{Examples}
\begin{example}[Cylindrical functions
  in Banach spaces and their dual]
  \label{ex:cylindrical}
  \index{Cylindrical functions}
  Let $(X,\|\cdot\|)$ be a Banach space
  (in particular the space $\R^d$ with any norm)
  endowed with its weak topology (or the dual of a Banach space $B$
  with the weak$^*$ topology)
  and let $\AA$
  be the set of smooth cylindrical functions:
  a function $f:X\to \R$ belongs to $\AA$ if there exists
  $\psi\in\rmC^\infty(\R^d)$ with
  bounded derivatives of every order and $d$ linear functionals
  $\sfh_1,\cdots,\sfh_d\in X'$ (resp.~in $B$ if the weak$^*$ topology is
  considered) such that
  \begin{equation}
    \label{eq:525}
    f(x)=\psi(\la \sfh_1,x\ra,\la \sfh_2,x\ra,\cdots,\la \sfh_d,x\ra).
  \end{equation}
  It is not difficult to check check that $\AA\subset
  \Lip_b(X,\tau,\sfd)$.
  In order to approximate the distance $\sfd(x,y)=\|x-y\|$
  between two points in $X$ we can argue as in
  Example \ref{ex:dual} by choosing
  functions of the form
  $f(x):=\sft_\eps(\la \sfh,x\ra-\la \sfh,y\ra)$
  where $\sfh$ belongs to the dual (resp.~predual) unit ball
  of $X'$ (resp.~$B$)
  and $\sft_\eps(r)$ is a smooth regularization
  of $\sft(r):=0\lor r\land 2\|x-y\|$ coinciding with $r$ in the interval
  $[\eps,\|x-y\|]$.
  In the case of Example \ref{ex:model} it is sufficient to choose
  $\sfh$
  in the convex set $L$.

  % A similar construction holds if $X$ is a dual of a
  % separable Banach space $B$
  % endowed with the weak$^*$ topology $\tau$
  % as in Example
  % \ref{ex:dual} and we choose $\sfh_i$
  % in the unit ball of $B$.
  The same approach can be adapted to the
  ``Wiener'' construction of Example \ref{ex:Wiener}:
  in this case one can use linear functionals in $X'$.
  
  In the case $X$ is separable (resp.~$X=B
  '$ and $B$ is separable) any Borel
  (resp.~weakly$^*$ Borel) measure
  is Radon.
\end{example}
\begin{example}
  A compatible algebra is provided by 
\begin{equation}\label{eq:defliptaui}
  \begin{aligned}
    \Lipb(X,\tau,(\sfd_i)):={}&\left\{f\in
      \rmC_b(X,\tau):\exists\, i\in I: \Lip(f,\sfd_i)<\infty
    \right\},
    % \Lipbloc(X,\tau,(\sfd_i)):={}&\left\{f\in
    %   \rmC_b(X,\tau):\forall\,x \in X\,\exists U\in \UU_x,\ \exists\, i\in I: \Lip(f,U,\sfd_i)<\infty
    %   \right\},    
  \end{aligned}
\end{equation}
whenever $(\sfd_i)_{i\in I}$ is a directed family 
satisfying (\ref{eq:monotone}a,b,c,d).
One can also consider the smaller unital algebra of functions
generated by the collection of distance functions
\begin{equation}
  \label{eq:58}
  \Big\{\sfd_i(\cdot,y):y\in X,\ i\in I\Big\}.
\end{equation}
%we set as in \eqref{eq:221} 
% \begin{equation}\label{eq:defliptaui}
%   \begin{aligned}
%     \Lipbu\lambda(X,\tau,(\sfd_i)):={}&\left\{f\in
%       \rmC_b(X,\tau):\exists\, i\in I: \Lip(f,\sfd_i)\le
%       \lambda\right\},\\
%     \Lipbuloc\lambda(X,\tau,(\sfd_i)):={}&\left\{f\in
%       \rmC_b(X,\tau):\forall\,x \in X\,\exists U\in \UU_x,\ \exists\, i\in I: \Lip(f,U,\sfd_i)\le
%       \lambda\right\},    
%   \end{aligned}
% \end{equation}
% with the corresponding definitions of 
% $\Lipbu\kappa(X,\tau,(\sfd_i))$
% %  and
% % $\Lipbuloc\lambda(X,\tau,(\sfd_i))$
% as in \eqref{eq:230}.
\end{example}
\begin{example}[Cartesian products]
  \label{ex:cartesian}
  Let us consider two e.t.m.~spaces $(X',\tau',\sfd')$
  and $(X'',\tau'',\sfd'')$ with two compatible algebras $\AA',\AA''$.
  For every $p\in [1,+\infty]$
  we can consider the product space
  $(X,\tau,\sfd_p)$ where
  $X=X'\times X''$, $\tau$ is the product topology of $\tau'$ and
  $\tau''$, and
  \begin{equation}
    \label{eq:526}
    \begin{aligned}
      \sfd_p((x',x''),(y',y'')):={}&\Big(\sfd'(x',y')^p+
      \sfd''(x'',y'')^p\Big)^{1/p}\quad\text{if }p<\infty,
      \\
      \sfd_\infty((x',x''),(y',y'')):={}&\max\Big(\sfd'(x',y'),
      \sfd''(x'',y'')\Big).
    \end{aligned}
  \end{equation}
  The algebra $\AA=\AA'\otimes\AA''$ generated by functions
  $f'\in \AA'$ and $f''\in \AA''$
  (an element of $\AA$ is
  a linear combination of functions of the form
  $f(x',x''):=f'(x')f''(x'')$)
  is compatible with $(X,\tau,\sfd_p)$.
  In order to prove that \eqref{eq:214bis} holds,
  let $q$ be the conjugate exponent of $p$ and let
  us introduce the convex subset of $\R^2$
  $C_q:=\{(\alpha,\beta)\in \R^2:
  \alpha^q+\beta^q\le 1\}$
  (with obvious modification when $q=\infty$).
  For every couple of point $(x',x''),(y',y'')$ in $X$
  we can find $(\alpha,\beta)\in C_q$ such that
  \begin{displaymath}
    \sfd_p((x',x''),(y',y''))=\alpha \sfd'(x',y')+\beta\sfd''(x'',y'').
  \end{displaymath}
  It is easy to check that for every $f'\in \AA_1'$ and
  $f''\in \AA_1''$ the function
  $f(z',z''):=\alpha f'(z')+\beta f''(z'')$ belongs to $\AA_1$.
  Since $\AA'$ and $\AA''$ are compatible in the respective spaces,
  we then get
    \begin{align*}
      \sfd_p((x',x''),(y',y''))&=
      \sup_{f'\in \AA'_1,f''\in \AA_1''}\alpha (f'(x')-f'(y'))+
                                 \beta(f''(x'')-f''(y''))
                                 \\&=
      \sup_{f'\in \AA'_1,f''\in \AA_1''}\alpha f'(x')+\beta f''(x'')-
      \big(\alpha f'(y')+\beta f''(y'')).
  \end{align*}
\end{example}
\begin{remark}
  The previous example \ref{ex:cartesian} shows in particular
  that the cartesian product of two e.t.m.~spaces
  is also an e.t.m.~space, a property that one can also directly check
  by using the approximating semidistance functions
  $(\sfd_i')_{i\in I},(\sfd_j'')_{j\in J}$.
\end{remark}
% and 
% $\Lipbloc(X,\tau,(\sfd_i) )$
%   \label{rem:B-sigma-compact}
%   \upshape
%   In the definition above, it is not restrictive to assume that $A$ is
%   a
%   $\sigma$-compact subset of $\supp(\mm)$.
%   Since $\mm$ is a Radon measure, we can find an increasing family of
%   compact sets $(K_n)_{n\in \N}$ in $\supp(\mm)$ such that
%   $\lim_{n\to\infty}\mm(B\setminus K_n)\downarrow0$; we can then
%   replace $A$ by $\tilde A:=\cup_{n\in \N}K_n$.
% \end{remark}
%

\GGG
In order to deal with functions in $\AA$ it will be useful to have
suitable polynomial approximations of the usual truncation maps.
\begin{lemma}[Polynomial approximation]
  \label{le:polynomial}
  Let $c>0$, $a_i,b_i\in \R$, and $\phi:\R\to \R$ be a Lipschitz function 
  satisfying 
  \begin{equation}
    \label{eq:335}
    a_0\le \phi\le b_0\quad \text{in $[-c,c]$},\quad
    a_1\le \phi'\le b_1\quad \text{$\LL^1$-a.e.~in $[-c,c]$}.
  \end{equation}
  There exists a sequence $(P_n)_{n\in \N}$ of polynomials
  such that 
  \begin{equation}
    \label{eq:336}
    \lim_{n\to\infty}\sup_{[-c,c]}|P_n-\phi|=0,\quad
    a_0\le P\le b_0,\ a_1\le P'\le b_1\quad\text{in $[-c,c]$},
  \end{equation}
  and
  \begin{equation}
    \label{eq:573bis}
    \lim_{n\to\infty}|P_n'(r)-\phi'(r)|=0\quad\text{for every }r\in
    [-c,c]\text{ where $\phi$ is differentiable.}
  \end{equation}
  If moreover $\phi\in\rmC^1([-c,c])$ we also have
  \begin{equation}
    \label{eq:573tris}
    \lim_{n\to\infty}\sup_{r\in [-c,c]}
    |P_n'(r)-\phi'(r)|=0.
  \end{equation}
\end{lemma}
\begin{proof}
  In order to prove the first statement of the lemma,
  it is sufficient to use the Bernstein polynomials of degree $2n$ 
  on the interval $[-c,c]$ given by the formula
  \begin{equation}
    \label{eq:337}
    P_n(r):=\frac1{(2c)^n}\sum_{k=-n}^n\phi(k/n)\left({{2n}\atop {k+n}}\right)(r+c)^{n+k}(c-r)^{n-k}
  \end{equation}
  recalling that $P_n$ uniformly converge to $\phi$ in $[-c,c]$ as
  $n\to\infty$
  and that formula \eqref{eq:337} preserves the bounds on $\phi$ and
  $\phi'$
  \cite[Sect. 1.7]{Lorentz86}.
\end{proof}
Applying the previous Lemma to the the function $\phi(r):=\alpha\lor r\land
\beta$ (with $a_0=\alpha$, $b_0=\beta$, $a_1=0$, $a_2=1$), we immediately get the following property.
  \begin{corollary}
      \label{cor:trunc1}
  For every interval $[-c,c]$, $c>0$, $\alpha,\beta\in \R$ with
  $\alpha<\beta$, and every $\eps>0$ 
  there exists a polynomial $P_\eps=P_{\eps}^{c,\alpha,\beta}$ such
  that
  \begin{equation}
    \label{eq:356}
    \begin{gathered}
      |P_\eps(r)-\alpha\lor r\land \beta|\le \eps,\quad
      \alpha\le P_\eps(r)\le \beta,\quad
      0\le P_\eps'(r)\le
      1\quad\text{for every }r\in [-c,c],\\
      \lim_{\eps\down0}P_\eps'(r)=
      \begin{cases}
        1&\text{if }\alpha<r<\beta\\
        0&\text{if }r<\alpha\text{ or }r>\beta.
      \end{cases}
    \end{gathered}
  \end{equation}
  If $\alpha=-\beta$ we can also find an odd $P_\eps$, thus satisfying $P_\eps(0)=0$.
\end{corollary}
A more refined argument yields:
\begin{corollary}
  \label{cor:trunc2}
  For every interval $[-c,c]\subset \R$ and every
  $\eps>0$, there exists a polynomial $Q_\eps=Q_\eps^c:\R\times
  \R\to \R$ such that 
  \begin{gather}
    \label{eq:339}
    r\land s\le Q_\eps(r,s)\le r\lor s,\quad
    |Q_\eps(r,s)-r\lor s|\le \eps\quad\text{for every }r,s\in
    [-c,c],\\
    \label{eq:340}
    0\le \partial_rQ_\eps\le 1,\quad
    0\le \partial_sQ_\eps\le 1\quad\text{in }[-c,c]\times [-c,c],\\
    \label{eq:341}
    |Q_\eps(r_2,s_2)-Q_\eps(r_1,s_1)|\le \max\big(|r_2-r_1|,|s_2-s_1|\big)\quad
    \text{for every }r_i,s_i\in
    [-c,c].
  \end{gather}
\end{corollary}
\begin{proof}
  We apply Lemma
  \ref{le:polynomial} to the function $\phi(r):=r_+$ in the interval
  $[-4c,4c]$
  (with $a_0=a_1=0$, $b_0=4c$ and $b_1=1$)
  obtaining a polynomial $P_\eps$ such that
  \begin{equation}
    \label{eq:334}
    |P_\eps(r)-r_+|\le \eps,\quad
    0\le P_\eps(r)\le 4c,\quad 0\le P_\eps'(r)\le 1\quad \text{for every
    }r\in [-4c,4c].
  \end{equation}
  We set $Q_\eps(r,s):=r+P_\eps(s-r)-P_\eps(0)$.
  Notice that $Q_\eps$ is increasing w.r.t.~$r,s$ in $[-2c,2c]\times [-2c,2c]$ since
  \begin{displaymath}
    \partial_r Q_\eps(r,s)=1-P_\eps'(s-r)\ge 0,\quad
    \partial_s Q_\eps(r,s)=P_\eps'(s-r)\ge 0;
  \end{displaymath}
  in particular
  \begin{displaymath}
    Q_\eps(r,s)\ge Q_\eps(r\land s,r\land s)=r\land s,\quad
    Q_\eps(r,s)\le Q_\eps(r\lor s,r\lor s)=r\lor s.
  \end{displaymath}
  By construction, if $r,s\in [-c,c]$ then
  \begin{displaymath}
    |Q_\eps(r,s)-r\lor s|=
    |r+P_\eps(s-r)-(r+(s-r)_+)|=
    |P_\eps(s-r)-(s-r)_+|\le \eps.
  \end{displaymath}
  Concerning the Lipschitz estimate, let us consider 
  points $(r_1,s_1),(r_2,s_2)\in [-c,c]$. Up to inverting the order
  of the couples, it is not restrictive to assume that
  $Q_\eps(r_1,s_1)\ge Q_\eps(r_2,s_2)$. Setting $ r_-:=r_1\land r_2$,
  $ r_+:=r_1\lor r_2$,
  $ s_-:=s_1\land s_2$,   $ s_+:=s_1\lor s_2$,
  $\bar r:=(r_+-r_-)\le 2c$,
  $\bar s=s_+-s_-\le 2c$, $\bar z:=\bar r\lor
  \bar s=\max(|r_2-r_1|,|s_2-s_1|)\le 2c$, the partial monotonicity of $Q_\eps$ yields
  \begin{align*}
    |Q_\eps(r_1,s_1)-Q_\eps(r_2,s_2)|&\le 
    Q_\eps(r_+,s_+)-Q_\eps(r_-, s_-)\le 
    Q_\eps(r_-+\bar z,s_-+\bar z)-Q_\eps(r_-, s_-)
                                 \\&=
    \int_0^{\bar z}\Big(\partial_r Q_\eps(r_-+z,s_-+z)+
    \partial_s Q_\eps(r_-+z,s_-+z)\Big)\,\d z
                                 \\&=
                                     \int_0^{\bar
                                     z}\Big(1-P'_\eps(s_--r_-)+P'_\eps(s_--r_-)\Big)\,\d
                                     z=\bar z.\qedhere
  \end{align*}
\end{proof}
\noindent
The next result shows how to obtain good approximations of the maximum of a finite
number
of functions in $\AA$.
\begin{lemma}
  \label{le:supremum}
  Let $f^1,f^2,\cdots, f^M\in \AA$ and let $f:=\max(f^1,f^2,\cdots,f^M)$.
  Then for every $\eps>0$ there exists a sequence $f_\eps\in \AA$ such that
  \begin{equation}
    \label{eq:338}
    \min_m f^m(x)\le f_\eps(x)\le \max_m f^m(x)\quad
    \text{for every }x\in X,\quad
    \sup_X |f_\eps-f|\le \eps.
  \end{equation}
  If moreover
  $\Lip (f^m,A,\delta)\le L$ for $1\le m\le M$ where $A\subset X$ and
  $\delta$ is an extended semidistance on $X$, then 
  $\Lip (f_\eps,A,\delta)\le L$ for every $n\in \N$.
\end{lemma}
\begin{proof}
  We split the proof in two steps.
  \medskip

  \noindent
  \textbf{1.} \emph{The thesis of the Lemma holds for $M=2$.}
  We set $c>0$ so that $f^m(X)\subset [-c,c]$
  and then we define
  $f_\eps:=Q_{\eps}(f^1,f^2)$, where $Q_\eps$ has been provided by
  Corollary \ref{cor:trunc2}.
  \eqref{eq:338} follows immediately by \eqref{eq:339}.
  \eqref{eq:341} yields for every $x,y\in X$
  \begin{displaymath}
    |f_\eps(x)-f_\eps(y)|=
    |Q_\eps(f^1(x),f^2(x))-Q_\eps(f^1(y),f^2(y))|
    \le \max\big(|f^1(x)-f^1(y)|,|f^2(x)-f^2(y)|\big)
  \end{displaymath}
   so that the composition with $Q_\eps$ preserve the Lipschitz constant
   w.r.t.~arbitrary sets and semidistances.
   \medskip

   \noindent
   \textbf{2.} \emph{The thesis of the Lemma holds for arbitrary $M\in
     \N$.} 
   We argue by
  induction,
  assuming that the result is true for $M-1$. 
  We fix a constant $c$ so that $f^m(X)\subset [-c,c]$ for $1\le m\le
  M$. We thus find 
  $h_{\eps/2}\subset \AA$ satisfying 
  \begin{displaymath}
    \min_{1\le m\le M-1} f^m(x)\le h_{\eps/2}(x)\le \max_{1\le m\le M-1} f^m(x)\quad
    \text{for every }x\in X,\quad
    \sup_X |h_{\eps/2}-\tilde f|\le \eps/2, 
  \end{displaymath}
  where $\tilde f:=f^1\lor \cdots\lor f^{M-1}$; in particular
  $h_{\eps/2}(X)\subset [-c,c]$.
  We then set 
  $f_\eps:=Q_{\eps/2}(h_{\eps/2},f^M)$; clearly for every $x\in X$
  \begin{displaymath}
    \min_{1\le m \le M}f^m(x)\le h_{\eps/2}(x)\land f^M(x)\le 
    f_\eps(x)\le 
    h_{\eps/2}(x)\lor f^M(x)\le 
    \max_{1\le m\le M} f^m(x);
  \end{displaymath}
  moreover
  \begin{align*}
    |f_\eps-f|&\le 
    |Q_{\eps/2}(h_{\eps/2},f^M) -f|
    \\&\le 
    |Q_{\eps/2}(h_{\eps/2},f^M) -Q_{\eps/2}(\tilde f,f^M)|
      +
    |Q_{\eps/2}(\tilde f,f^M)-\tilde f\lor f^M|\\&
    \le |h_{\eps/2}-\tilde f|
    +\eps/2\le \eps/2+\eps/2\le \eps.\qedhere
  \end{align*}  
\end{proof}
We conclude this section by a simple density results that will be useful in the
following.
% As an application of the previous compactification result, 
% let us prove that $\AA$ is dense in $L^p(X,\mm)$.
\begin{lemma}[Density of $\AA$ in $L^p(X,\mm)$]
  \label{le:density}
  Let $p\in [1,+\infty)$ and let $I$ a  closed (possibly unbounded) interval of $\R$.
  If $\AA$ is a compatible sub-algebra of $\Lip_b(X,\tau,\sfd)$,
  then for every $f\in \cL^p(X,\mm)$
  with values in $I$
  there exists a sequence $f_n\in \AA$ with values in $I$
  such that $\int_X |f-f_n|^p\,\d\mm\to0$. 
\end{lemma}
\begin{proof}
  By standard approximation,
  it is not restrictive to assume that $I=[\alpha,\beta]$ for some
  $\alpha,\beta\in \R$; we set $\gamma:=|\alpha|\lor |\beta|$.
  Since $\mm$ is Radon, every $\mm$-measurable function
  $f$ is Lusin $\mm$-measurable: thus for every $\eps>0$ 
  there exists a compact $K\subset X$ such that $f\restr{K}$ 
  is continuous and $\mm(X\setminus K)\le \eps$.
  % By a possible further restriction, thanks to Remark
  % \ref{rem:B-sigma-compact}
  % we can also suppose that $K$ is contained in the set $B$ of
  % Definition \ref{def:A-compatibility}.
  % Since $X$ is completely regular, Tietze extension theorem (for
  % compact sets, see \cite[Thm.~4.34]{Folland99}) shows that 
  % there exists a continuous extension $\tilde f$ of $f\restr{K}$
  % keeping the same bounds, so that $\int_X |f-\tilde f|^p\,\d\mm\le
  % (2M)^p\eps$. 
  
  Since $\AA$ 
  contains the constants and separates the points of $K$, 
  the restriction of $\AA$ to $K$ is uniformly
  dense in $\rmC_b(K,\tau)$ by Stone-Weierstrass Theorem:
  we thus find $\tilde f_\eps\in \AA$ such that
  $\sup_{x\in K}|f(x)-\tilde f_\eps(x)|\le \eps$. If $c:=\sup_X
  |\tilde
  f_\eps|\lor \gamma$,
  applying Corollary \ref{cor:trunc1} we can find a polynomial
  $P_\eps$ satisfying \eqref{eq:356}, so that
  $f_\eps:=P_\eps\circ \tilde f_\eps$ belongs to $\AA$,
  takes values in
  $[\alpha,\beta]$, and satisfies
  \begin{equation}
    \label{eq:357}
    | f_\eps -f|\le |P_\eps (\tilde f_\eps)-P_\eps(f)|+|P_\eps(f)-f|\le
    2\eps\quad\text{in }K,
  \end{equation}
  so that
  \begin{equation}
    \label{eq:358}
    \int_X | f_\eps -f|^p\,\d\mm
    \le (2\eps)^p \mm(X)+(\beta-\alpha)^p\eps.
  \end{equation}
  Choosing a sequence $f_n:=f_{\eps_n}$ corresponding to
  a vanishing sequence $\eps_n\downarrow0$ we conclude.
\end{proof}

\subsection{Embeddings and compactification of
  extended metric-measure spaces}
\label{subsec:compactification}

\newcommand{\hhat}[1]{#1'}
Let $\X=(X,\tau,\sfd,\mm)$ and $\hhat \X=(\hhat X,\hhat \tau,\hhat \sfd,\hhat\mm)$  be two
extended metric measure spaces, endowed with compatible algebras
$\AA,\hhat\AA$ according to definition \ref{def:A-compatibility}.
\begin{definition}[Embedding, compactification, and isomorphism]
  \label{def:embeddings}
  We say that a map $\iota:X\to \hhat X$ is a measure-preserving embedding
  of $(\X,\AA)$ into $(\hhat \X,\hhat\AA)$ if
  \begin{enumerate}[\rm (E1)]
  \item $\iota$ is a 
    continuous and injective map of $(X,\tau)$ into $(\hhat X,\hhat
    \tau)$;
    % onto its image
    % $\iota(X)\subset \hhat X$ endowed with the $\hhat\tau$-topology;
  \item $\iota$ is an isometry, in the sense that
    \begin{equation}
      \label{eq:347}
      \hhat \sfd(\iota(x),\iota(y))=\sfd(x,y)\quad\text{for every
      }x,y\in X.
    \end{equation}
  \item $\iota_\sharp\mm=\hhat \mm.$
  \item For every $\hhat f\in \AA'$ the function $\iota^* \hhat
    f:=\hhat f\circ\iota$
    belongs to $\AA$.
  \end{enumerate}
  We say that $(\hhat \X,\hhat \AA)$ is a compactification of $(\X,\AA)$ if 
  $(\hhat X,\hhat \tau)$ is compact and there exists a
  measure-preserving embedding of $(\X,\AA)$ into $(\hhat \X,\hhat
  \AA)$.\\
  We say that a measure-preserving embedding
  $\iota$ is an isomorphism of $(\X,\AA)$ onto $(\hat
  \X,\hat\AA)$
  if $\iota$ is an homeomorphism of $(X,\tau)$ onto $(\hat X,\hat
  \tau)$
  and $\iota^*(\hhat \AA)=\AA$.
\end{definition}
\begin{remark}[Canonical Lipschitz algebra]
  \label{rem:canonical}
  \upshape
  When $\AA=\Lip_b(X,\tau,\sfd)$ and $\AA'=\Lip_b(X',\tau',\sfd')$
  we simply say that $\iota$ is a measure preserving embedding of $\X$
  into $\X'$. In this case it is sufficient to check conditions
  (E1-2-3), since condition (E4) is a consequence of (E1-2).
  In this case $\iota$ is an isomorphism of
  $(\X,\AA)$ onto $(\iota(X),\hhat \tau,\hhat \sfd,\hhat \mm)$.
\end{remark}
\begin{example}
  \label{ex:embeddings}
  Let us show three simple examples of embeddings
  involving an e.m.t.m.~space $\X=(X,\tau,\sfd,\mm)$.
  \begin{enumerate}
  \item Let 
    $\tau'$ be a weaker topology than $\tau$, such that
    $(X,\tau',\sfd)$ is an e.t.m.~space.
    The identity map provides a measure-preserving embedding of
    $\X$ into $(X,\tau',\sfd,\mm)$.
  \item Let $\AA'\subset \AA$
    be two compatible sub-algebras of $\X$.
    The identity map provides a measure-preserving embedding of
    $(\X,\AA)$ into $(\X,\AA')$ (in particular
    when $\AA=\Lip_b(X,\tau,\sfd)$).
  \item By Corollary \ref{cor:auxiliary-Souslin},
    if $(X,\tau)$ is a
    Souslin space, one can always find
    a measure preserving embedding
    in the space $(X,\tau',\AA',\sfd)$ where
    $\AA'\subset \AA$ is countably generated and $\tau'$ is an
    auxiliary topology (thus
    metrizable
    and separable, coarser than $\tau$).
  \item Let $Y$ be any $\mm$-measurable subset of
    $X$ such that $\mm(X\setminus Y)=0$;
    we denote by $\tau_Y$ the relative topology of $Y$,
    $\sfd_Y$ the restriction of $\sfd$ to $Y\times Y$
    and $\mm_Y:=\mm\restr Y$. If $\AA$ is a compatible algebra for
    $\X$
    we define $\AA_Y:=\{f\restr Y:f\in \AA\}$ as
    the algebra obtained by
    the restriction to $Y$ of the elements of $\AA$.
    It is easy to check that
    $\Y=(Y,\tau_Y,\sfd,\mm\restr Y)$
    is an e.m.t.m.~space with a compatible algebra $\AA_Y$
    and the inclusion map
    $\iota:Y\to X$ is a measure-preserving embedding
    of $(\Y,\AA_Y)$ into $(\X,\AA)$.
  \end{enumerate}
\end{example}
Let us collect a few simple results concerning the corresponding
between measurable functions induced by a measure-preserving
embedding. Whenever $\hhat f:\hhat X\to \R$ we write
\begin{equation}
  \label{eq:349}
  f=\iota^* \hhat f:=\hhat f\circ\iota.
\end{equation}
Note that if $\iota(X)$ is $\hhat\tau$-dense in $\hhat X$
the pull back map $\iota^*$ is injective.  Independently from this
property, we will show that $\iota^*$ induces an isomorphism between
classes of measurable functions identified by $\hhat \mm$ and
$\mm$-a.e.~equivalence respectively.
We write $f_1\sim_\mm f_2$ if $\mm(\{f_1\neq f_2\})=0$
and we will denote by $[f]_\mm$ the equivalence class of a
$\mm$-measurable map $f$.
\begin{lemma}
  \label{le:measurable}
  Let $\iota:X\to \hhat X$ be a measure-preserving embedding of
  $(\X,\AA)$ into $(\hhat \X,\hhat \AA)$ according to the previous
  definition.
  \begin{enumerate}
  \item For every $\hhat \mm$-measurable function $f':X'\to \R$
    the function $\iota^* \hhat f$ is $\mm$-measurable and we have
    \begin{equation}
      \hhat {f_1}\sim_{\hhat \mm} \hhat {f_2}\quad
      \Leftrightarrow\quad
      \iota^*\hhat
    {f_1}
    \sim_\mm \iota^*\hhat{f_2}
    \label{eq:359}
  \end{equation}
\item The algebra $\AA^*:=\iota^*(\hhat \AA)$ is a sub-algebra of
    $\AA$ %, which induces a topology $\tau^*$ coarser than $\tau$ and
    which is compatible with the extended metric-measure space
    $\X$. % $=(X,\tau^*,\sfd,\mm)$.
    % If $\iota$ is a homeomorphism of $X$ with $\iota(X)$ then
    % $\tau^*=\tau$ and 
    % $\X^*=\X$.
  \item For every $p\in [1,+\infty]$
    $\iota^*$ induces a linear isomorphism between $L^p(\hhat
    X,\hhat\mm)$ and $L^p(X,\mm)$, whose inverse is denoted by
    $\iota_*$.
    For every $f\in \AA^*/\sim_{\mm}$
    the class $\iota_* f$ contains all the elements
    $f'\in \hhat\AA$ satisfying $\iota^*\hhat f=f$.
    % coincides with
    % $\iota_{\AA,*}$ in $\AA$
    % and it is the unique linear isometry extending $\iota_{\AA,*}$ to $L^1(X,\mm)$.
  \end{enumerate}
\end{lemma}
\begin{proof}
  The proof of \textbf{(a)} is immediate: \eqref{eq:359} is a consequence of
  the fact that
  the set $\hhat N:=\{\hhat f_1\neq \hhat f_2\}$ satisfies
  \begin{displaymath}
    \iota^{-1}\hhat N= \{\iota^* f_1\neq \iota^* f_2\},\quad
    \mm(\iota^{-1}\hhat N)=\hhat\mm(N).
  \end{displaymath}
  \textbf{(b)}
  It is immediate to check that $\AA^*$ is a unital algebra included
  in $\AA$.
  It is not difficult to check that $\AA^*$ satisfies
  \eqref{eq:214bis}:
  % suppose
  % that $\AA'$ satisfies
  % \eqref{eq:214bis} with respect to the Borel set
  % $A'\in\BorelSets {X'}$ and let $A^*:=\iota^{-1}(A')\subset
  % \BorelSets X$. 
  % Clearly $\mm(X\setminus A^*)=\mm'(X'\setminus A')=0$
  % and by
  by \eqref{eq:347} if $\hhat f\in \hhat\AA_1$ then
  $\iota^*\hhat f\in \AA^*_1$ and since $\hhat \AA$ is
  compatible with $\hhat \X$, for every $x,y\in X$
  \begin{align*}
    \sfd(x,y)\topref{eq:347}=&
    \hhat\sfd(\iota(x),\iota(y))
    =\sup_{\hhat f\in \hhat \AA_1}|\hhat f(\iota(x))-\hhat
    f(\iota(y))|
    =
    \sup_{\hhat f\in \hhat \AA_1}|\iota^*\hhat f(x)-\iota^*\hhat
    f(y)|
    \\=&
           \sup_{f\in \AA^*_1}|f(x)-f(y)|.
  \end{align*}
  % In particular $\AA^*$ separates the points of $X$ and therefore
  % induces
  % a topology $\tau^*$ weaker than $\tau$; in particular compact sets
  % w.r.t.~$\tau$ are also compact with respect to $\tau^*$ and $\mm$ is a
  % Radon measure in $(X,\tau^*)$ as well.
  %
  % If $\iota$ is a homeomorphism with its
  % image $\iota(X)$ in $\hhat X$, then $\AA^*$ 
  % generates $\tau$: if
  % $x_i$, $i\in I$, is a net indexed by a directed set $I$, we thus
  % have that $x_i$ is $\tau$-convergent to a point $x\in X$ if and only if $i\mapsto
  % \iota(x_i)$
  % is $\hhat \tau$ convergent to $\iota(x)$ and this last property is
  % equivalent to say that
  % $\lim_{i\in I}\hhat f(\iota(x_i))=\hhat f(\iota(x))$ for every
  % $\hhat f\in \hhat \AA$ since $\hhat \AA$ generates $\hhat \tau$.
  % We conclude that the collection $\AA^*=\{\hhat f\circ\iota:\hhat
  % f\in \AA\}$ generates $\tau$.
%
  \textbf{(c)}
  Thanks to property (E3) of Definition \ref{def:embeddings} we
  have
  \begin{equation}
    \label{eq:355}
    \int_{\hhat X}\Phi(\hhat f)\,\d\hhat \mm=
    \int_{X}\Phi(\hhat f\circ \iota)\,\d \mm=
    \int_{X}\Phi(\iota^*\hhat f)\,\d \mm
  \end{equation}
  for every nonnegative continuous function $\Phi:\R\to [0,+\infty)$,
  so that $\iota^*$ induces a linear isometry from each $L^p(\hhat
  X,\hhat \mm)$ into $L^p(X,\mm)$. It is therefore sufficient to prove
  that $\iota^*$ is surjective; since $\iota^*$ is an isometry with
  respect to the $L^1$-norm,
  this is equivalent to the density of the image of $\iota^*$
  in $L^1(X,\mm)$.
  Since the image contains (the equivalence classes of elements in) $\AA^*$,
  the density follows by (b) and Lemma \ref{le:density}.
  The last statement is a consequence of \eqref{eq:359}.
\end{proof}
% When two compatible algebras of functions $\AA$ and $\hhat \AA$ are
% also given, a further useful property is that
% \begin{equation}
%   \label{eq:351}
%   \forall\, \hhat f\in \hhat\AA:\quad
%   f=\hhat f\circ\iota\in \AA.
% \end{equation}
On of the most useful application
of the concept of measure-preserving embeddings
is the possibility to
construct a compactification $\X'$ of
$\X$ starting from a compatible algebra $\AA$.
As a byproduct, we will obtain a compatible algebra $\hhat \AA$ in $\hhat
\X$ such that
\begin{equation}
  \label{eq:352}
  \hhat f\in \hhat\AA\quad\Leftrightarrow\quad
  f=\hhat f\circ\iota\in \AA.
\end{equation}
As a general fact, every completely regular space $(X,\tau)$ can be
homeomorphically imbedded as a dense subset of a compact Hausdorff
space $\beta X$ (called the
Stone-Cech compactification, \cite[\S\,38]{Munkres00}),
where every function $f\in \rmC_b(X)$
admits a unique continuous extension. The Gelfand theory of Banach algebras applied to
$\rmC_b(X)$ provides one of the most effective construction of such a
compactification
and has the advantage to be well adapted to the 
setting of extended metric-topological spaces and compatible sub-algebras.

Let us briefly recall
the construction.
We consider $\AA$ as a vector subspace
of $\rmC_b(X,\tau)$
endowed with the $\sup$ norm $\|\cdot\|_\infty$
and we call $\overline \AA$ the (strong) closure of $\AA$
in $\rmC_b(X,\tau)$.
Since $(\AA,\|\cdot\|_\infty)$ is a normed space we can
consider the dual Banach space
$(\AA^*,\|\cdot\|_{\AA^*})$ endowed with the weak$^*$ topology $\hat\tau$
and the distinguished subset of characters.
\begin{definition}[Characters]
  A character of $\AA$ is an element $\varphi$ of $\AA^*\setminus \{0\}$
  satisfying
  \begin{equation}
    \label{eq:215}
    \varphi(fg)=\varphi(f)\varphi(g)\quad\text{for every }f,g\in \AA.
  \end{equation}
  We will denote by $\hat X$ the 
  subset of the characters of $\AA$. 
\end{definition}
\noindent
Let us first recall a preliminary list of useful properties of $\hat X$.
\begin{proposition}
  \label{prop:iota-description}
  Let us consider the set
\begin{equation}
  \label{eq:275}
  \Sigma:=\{\psi\in \AA^*: \|\psi\|_{\AA^*}\le 1,\
  \psi(f)\ge0
  \text{ for every }f\ge0,\ \psi(\unit)=1\}.
\end{equation}
  \begin{enumerate}
  \item $\Sigma$ is a weakly$^*$ compact convex subset of $\AA^*$
    contained in $  \{\psi\in \AA^*:\|\psi\|_{\AA^*}=1\}.$
  \item
    $\hat X$ is a (weakly$^*$) compact subset of $\Sigma$.
  \item Every point of $\hat X$ is an extremal point of $\Sigma$.
  \end{enumerate}
\end{proposition}
\begin{proof}
  \textbf{(a)}
  is an immediate consequence of Banach-Alaouglu-Bourbaki theorem.
  \medskip
  
  \noindent
  \textbf{(b)}  
  It is not difficult to check that every element $f\in \AA$ with
  $0<m\le f\le M$
  admits a nonnegative square root $g\in \overline \AA$
  such that $g^2=f$:
  it is sufficient to define $h:=1-f/M\in \AA$
  taking values in $[0,1-m/M]$
  and use the power series expansion of the square root function in $]0,2[$:
  \begin{displaymath}
    g=\sqrt M \sum_{n=0}^\infty \left({1/2\atop n}\right)
    (-1)^nh^n=
    \sum_{n=0}^\infty \left({1/2\atop n}\right)
    M^{1/2-n}(f-M)^n.
  \end{displaymath}
  The relation $\varphi(f)=(\varphi(g))^2$
  shows that $\varphi f\ge0$ and since
  every nonnegative $f\in \AA$ can be strongly approximated
  by uniformly positive elements
  we obtain that every $\varphi\in \hat X$ satisfies
  $\varphi(f)\ge0$ for every nonnegative $f\in \AA$.
  Moreover, since $\varphi\neq 0$,
  there exists an element $f\in \AA$ such that $\varphi(f)\neq0$;
  from
  $0\neq \varphi(f)=\varphi(f\unit)=\varphi(f)\varphi(\unit)$
  we deduce that $\varphi(\unit)=1$.
  By comparison we obtain
  \begin{equation}
    \label{eq:274}
    \varphi(\gamma\unit)=\gamma\quad\text{for every }\gamma\in \R;
    \quad
    -\alpha\le f\le \beta\quad\Rightarrow\quad
    -\alpha\le \varphi(f)\le \beta,
  \end{equation}
  and in particular
  \begin{equation}
    \label{eq:527}
    -\|f\|_\infty\le \varphi(f)\le \|f\|_\infty,
  \end{equation}
  so that every element of $\hat X$ is included in the weakly$^*$
  compact set $\Sigma$.
  Since condition \eqref{eq:215}
  characterize a weakly$^*$ closed set,
  we conclude that $\hat X$ is a compact Hausdorff space
  endowed with the weak$^*$ topology of $\AA^*$.
  \medskip
  
  \noindent\textbf{(c)}
  Whenever $\psi\in \Sigma$ and $f\in \AA$, the inequality
  \begin{displaymath}
    0\le \psi((f+\kappa \unit)^2)=
    \psi(f^2)+2\kappa \psi(f)+\kappa^2\psi(\unit)= 
    \psi(f^2)+2\kappa \psi(f)+\kappa^2
  \end{displaymath}
  for every $\kappa\in \R$ shows that
  \begin{displaymath}
    \big(\psi(f)\big)^2\le \psi(f^2)\quad\text{for every }\psi\in \Sigma.
  \end{displaymath}
  If $\hat X\ni \varphi=\frac 12\varphi_1+\frac 12\varphi_2$ with
  $\varphi_i\in \Sigma$, we obtain
  \begin{align*}
    \frac12 \varphi_1(f^2)+\frac12 \varphi_2(f^2)&=
                                                   \varphi(f^2)=
                                                   \big(\varphi(f)\big)^2=
                                                   \frac14 \big(\varphi_1(f)\big)^2+\frac14 \big(\varphi_2(f)\big)^2+
                                                   \frac 12 \varphi_1(f)\varphi_2(f)
    \\&\le 
    \frac14 \varphi_1(f^2)+\frac14 \varphi_2(f^2)+
    \frac 12 \varphi_1(f)\varphi_2(f)    
  \end{align*}
  thus showing that
  \begin{align*}
    \Big(\frac 12 \varphi_1(f)-\frac 12 \varphi_2(f)\Big)^2
    =\frac14 \varphi_1(f^2)+\frac14 \varphi_2(f^2)-
    \frac 12 \varphi_1(f)\varphi_2(f)
    \le 0
  \end{align*}
  for every $f\in \AA$ and therefore $\varphi_1=\varphi_2$.
\end{proof}
We can now define a canonical embedding
$\iota:X\to \hat X$ by
\begin{equation}
  \label{eq:216}
  \iota(x)=\hat x,\quad
  \hat x(f):=f(x)\quad\text{for every }x\in X,\ f\in \AA,
\end{equation}
and we call $\Gamma:\overline\AA\to \rmC_b(\hat X,\hat \tau)$ the
Gelfand transform
\index{Gelfand transform}
\index{Compactification of e.m.t.m.~spaces}
\begin{equation}
  \label{eq:218}
  \hat f=\Gamma f,\quad
  \hat f(\varphi):=\varphi(f)\quad\text{for every }\varphi\in \hat X.
\end{equation}
We will set $\hat \AA:=\Gamma(\AA)$.
\begin{theorem}[Gelfand compactification of extended metric
  topological
  measure
  spaces]\ 
  \label{thm:G-compactification}
  \begin{enumerate}
  \item $\iota$ is a continuous and injective,
    $\iota(X)$
    is a dense subset of the compact Hausdorff space
    $\hat X$ endowed with the weak$^*$ topology $\hat \tau$.
    If $\AA$ is separable w.r.t.~the uniform norm
    then
    $\hat X$ is metrizable.
  \item $\iota$ is a homeomorphism between $X$ and $\iota(X)$
    if and only if $\AA$ is adapted, i.e.~it generates the topology $\tau$.
  \item Every function $f\in \overline \AA$ admits a unique extension
    $\hat f=\Gamma f$ to $\hat X$ and the unital algebra
    $\hat \AA:=\Gamma(\AA)$ is uniformly dense in $\rmC_b(\hat X,\hat
    \tau)$.
    The pull back algebra $\AA^*$ coincides with $\AA$.

    \item The measure $\hat \mm:=\iota_\sharp \mm$ is a Radon measure on
      $\hat X$ %$\Delta(\AA)$
      concentrated on the $\hat \mm$-measurable subset
      $\iota(X)$.
    % For every Borel map $f\in \BB(X,\tau)$ the map
    % \begin{equation}
    %   \label{eq:354}
    %   \hat f(\varphi):=\varphi(f)
    % \end{equation}
    % is a 
    \item If $I$ is the directed set of all the finite collections
    of functions in $\AA_1$, setting
    \begin{align}
      \label{eq:276}
      \hat \sfd_i(\varphi_1,\varphi_2):= {}&
      \sup_{f\in i}|\varphi_1(f)-\varphi_2(f)|=
      \sup_{f\in i} |\hat f_i(\varphi_1)-\hat f_i(\varphi_2)|,\\
      \label{eq:353}
      \hat \sfd(\varphi_1,\varphi_2):={}&\sup_{f\in
        \AA_1}|\varphi_1(f)-\varphi_2(f)|
      =\sup_{i\in I}\hat \sfd_i(\varphi_1,\varphi_2),
    \end{align}
    % if $\overline\AA$ contains the functions $x\mapsto \sfd_i(x,z)$ for
    % $i\in I$ and $z\in X$, then there exists continuous functions
    % $\hat \sfd_i$, continuous with respect to
    % $\hat \tau\times \hat\tau$
    % such that $\hat \sfd_i(\hat x,\hat y)=\sfd_i(x,y)$ for every
    % $x,y\in
    % X$
    % and
    % $\hat \sfd:=\sup_i\sfd_i$
    $\hat \sfd_i$ are continuous and bounded semidistances on
    $\hat X$
    % $\Delta(\AA)$
    and $\sfd$ is an extended distance on $\hat X$ %$\Delta(\AA)$
    satisfying
    \begin{equation}
      \label{eq:277}
      \hat \sfd_i(\hat x,\hat y)=\sfd_i(x,y),
      \quad
      \hat \sfd(\hat x,\hat y)=\sfd(x,y) \quad\text{for every $x,y\in
        X$}.
    \end{equation}
  \item
    $\iota$ is a measure preserving embedding of $(\X,\AA)$
    into the compact extended metric measure space
    $\hhat \X:=(\hat X,\hat\tau,\hat\sfd,\hat\mm)$ endowed with
    the compatible algebra $\hat \AA=\Gamma (\AA)$.
  \item The map $\iota_*:L^p(X,\mm)\to L^p(\hat X,\hat\mm)$
    is the unique linear isometric extension of $\Gamma:\AA\to
    \hat\AA$ to $L^p(X,\mm)$ for every $p\in [1,\infty[$.
  \end{enumerate}
\end{theorem}
\begin{proof}
  \textbf{(a)}
  Since the weak$^*$ topology is the coarsest topology that makes all
  the maps $\varphi\mapsto \varphi(f)$ continuous for every
  $f\in \AA$, the continuity of $\iota$ is equivalent to check the
  continuity of $x\mapsto f(x)$ from $(X,\tau)$ to $\R$, which is
  guaranteed by the $\tau$-continuity of $f$.  $\iota$ is also
  injective, since $\AA$ separates the points of $X$.

  Let us now prove that $\iota(X)$ is dense in $\hat X$.  Let us
  denote by $Y$ the weak$^*$ closure of $\iota(X)$ and let us first
  consider the closed convex hull $K:=\overline{\mathrm{co}(Y)}$
  of $Y$
  in the weak$^*$ topology.  Since $K$ is bounded and weakly$^*$
  closed, $K$ is compact.  If a point $\psi\in \hat X$ does not
  belong to $K$ we can apply the second geometric form of Hahn-Banach
  Theorem \cite[Thm.~3.21]{Rudin91} to find $f\in \overline{\AA}$
  separating $\psi$ from $K$: there exists $\alpha\in \R$ such that
  $\zeta(f)\ge\alpha>\psi(f)$ for every
  $\zeta\in K$. 
  Choosing $\zeta=\hat z=\iota(z)$ for $z\in X$ we
  deduce that $f(z)\ge\alpha$ for every $z\in X$ and therefore
  $\psi(f)\ge \alpha$ since $\psi$ is
  a nonnegative functional.

  Thus $\psi\in \overline{\mathrm{co}(Y)}$; since $Y\subset \Sigma$,
  also $\overline{\mathrm{co}(Y)}\subset \Sigma$; since $\psi$ is an
  extreme point of $\Sigma$, we deduce that $\psi$ is an extreme point
  of $\overline{\mathrm{co}(Y)}$; applying Milman's theorem
  \cite[Thm.~3.25]{Rudin91} we conclude that $\psi\in Y$.

  Finally, if $\AA$ is separable then the unit ball of $\AA^*$
  endowed with the weak$^*$ topology is metrizable
  so that $\hat X$ is metrizable as well.
  \medskip
  
  \noindent
  \textbf{(b)}
  It is easy to check that $\hat f\in \rmC_b(\hat X)$ and that 
  $\hat \AA$ is an algebra. Clearly $\hat \unit$ is the unit function
  in $\rmC_b(\hat X)$ and $\hat \AA$ separates the points of
  $\hat X$: if $\varphi_i\in \hat X$ satisfies
  $\varphi_1(f)=\varphi_2(f)$ for every $f\in \AA$, then 
  $\varphi_1=\varphi_2$ in $\AA^*$. 
  Applying Stone-Weierstrass Theorem, we conclude that $\hat \AA$ 
  is dense in $\rmC_b(\hat X)$.
  By construction $\Gamma$ is a isomorphism between $\AA$ and $\hat
  \AA$,
  and $\iota^*$ is its inverse.
  \medskip
  
  \noindent
  \textbf{(c)}
  Since $(\hat X,\hat \tau)$ is compact
  and $\hat \AA$ separates the points of $\hat X$,
  $\hat \tau$ is the initial topology of $\hat \AA$.
  Thus a net $\hat x_i=\iota(x_i)$ in
  $\iota(X)$, $i\in I$, converges to $\hat x=\iota(x)$
  if and only if
  $\lim_{i\in I} \hat f(\hat x_i)=\hat f(\hat x)$ for every $\hat f\in
  \hat \AA$. By \eqref{eq:216} and \eqref{eq:218} the latter property
  is equivalent to
  $\lim_{i\in I}f(x_i)=f(x)$ for every $f\in \AA$,
  so $\iota$ is a homeomorphism between $X$ endowed with the initial
  topology of $\AA$ and $\iota(X)$.
  %
  % Let us suppose that $\tau$ is the initial topology of $\AA$
  % and let us prove that for every open subset $U\subset X$
  % $
  % \iota(U)$ is relatively open in $\iota(X)$.
  % By assumption, 
  % for every $x\in U$
  % there exists $f\in \AA$ such that 
  % \begin{equation}
  %   f(z)\ge 0\quad\text{for every }z\in X,\quad
  %   f(x)=0,\quad
  %   \{y\in X:f(y)<1\}\subset U.\label{eq:217}
  % \end{equation}
  % % In order to show
  % % that $\iota$ is an homeomorphism between $X$ and
  % % $\iota(X)\subset \hat X$ it is sufficient to check that
  % % $\iota$ is open.
  % % Whenever $U$ is $\tau$-open and $x\in U$, we can
  % % select a function $f\in \AA$ satisfying \eqref{eq:217};
  % since
  % $\hat y=\iota(y)$ yields $\hat y(f)=f(y)$, we get
  % \begin{displaymath}
  %   \iota(U)\supset \iota(X)\cap\{\psi\in \hat X:\psi(f)<1\}.
  % \end{displaymath}
  %
  \medskip
  
  \noindent
  \textbf{(d)}
  This is a general property for the push-forward of Radon measures
  through a continuous map:
  since $\mm$ is tight, we can find an increasing sequence of compact
  sets
  $K_n\subset X$ such that $\mm(X\setminus K_n)\le 1/n$.
  Since $\iota$ is continuous, $\hat K_n=\iota(K_n)\subset \iota(X)$ are compact in
  $\hat X$ and $\hat \mm(\hat X\setminus \hat K_n)\le 1/n$
  so that $\hat\mm(\hat X \setminus \hat K)=0$ where
  $\hat K=\cup_n \hat K_n\subset \iota(X)$.
  \medskip
  
  \noindent
  \textbf{(e)} and \textbf{(f)} are immediate.
  % It is easy to check that $\hat\sfd$ is an extended distance 
  % and $(\hat X,\hat\tau,\hat\sfd,\hat \mm)$ is an 
  % extended metric-topological measure space.
  % Concerning the extension of $\sfd_i$, let us 
  % set $\sfd_{i,z}(x):=\sfd_i(x,z)$; since $\sfd_{i,z}\in \overline\AA$
  % we can first define
  % $\hat \sfd_{i,z}:=\Gamma \sfd_{i,z}$.
  % Since $\sup_{x\in X}|\sfd_{i,z}(x)-\sfd_{i,w}(x)|\le \sfd_{i}(z,w)$ 
  % we deduce that 
  % $\sup_{\psi\in \hat X}|\hat
  % \sfd_{i,z}(\psi)-\hat\sfd_{i,w}(\psi)|\le \sfd_{i}(z,w)$
  \medskip
  
  \noindent\textbf{(g)}
  Thanks to Lemma \ref{le:measurable} the Gelfand isomorphism
  $\Gamma$ preserves the equivalence classes in the sense that
  for every $f_1,f_2\in \AA$,
  \begin{displaymath}
    f_1\sim_\mm f_2\quad\Leftrightarrow\quad
    \hat f_1\sim_{\hat \mm} \hat f_2
  \end{displaymath}
  so that
  \begin{displaymath}
    \iota_* ([f]_\mm)=[\hat f]_{\hat\mm}\quad\text{for every }f\in \AA.
  \end{displaymath}
  Since $\iota_*$ is a linear isometry
  and the equivalence classes of elements of
  $\AA$ are dense in $L^p(X,\mm)$ we conclude.  
\end{proof}
\begin{remark}[A universal model]
  The compactification $\hat \X=(\hat X,\hat \tau,\hat\sfd,\hat\mm)$
  with the Gelfand algebra $\hat \AA$
  is a particular case of the case considered Example \ref{ex:model},
  where $B$ is the space $\overline \AA$, $L=\overline{\AA_1}$
  and $\hat X$ is a weakly$^*$ compact subset of the dual ball of $B'$
  (in fact concentrated on its extremal set).
  It follows that any e.m.t.m.~space has a
  measure-preserving isometric immersion in a space
  with the characteristics of Example \ref{ex:model}.
\end{remark}
The previous construction is also useful to quickly get
a completion of an e.m.t.m.~space.
\index{Completion of an e.m.t.m.~space}
We start from the compactification
$\hat \X=(\hat X,\hat \tau,\hat\sfd,\hat\mm)$
of $(\X,\AA(\X))$ obtained by the canonical algebra
$\Lip_b(X,\sfd,\mm)$ and 
we set:
\begin{equation}
  \label{eq:529}
  \text{$\bar X:=$ the $\hat \sfd$-closure
  of $\iota(X)$ into $\hat X$}.
\end{equation}
We obviously define
$\bar \tau$, $\bar \sfd$ and $\bar \mm$
respectively as the restrictions of $\hat \tau$,
$\hat \sfd$, and $\hat \mm$ to $\bar X$.
\begin{corollary}[Completion]
  \label{cor:completion}
  The map $\iota:X\to \bar X$
  is a measure preserving embedding
  of $\X$ into the e.m.t.m.~space
  $\bar \X=(\bar X,\bar \tau,\bar \sfd,\bar\mm)$
  such that
  \begin{enumerate}[(C1)]
  \item $(\bar X,\bar \sfd)$ is complete,
  \item $\iota$ is a homeomorphism of
    $X$ onto $\iota(X)$,
  \item $\iota(X)$ is $\bar \sfd$-dense in $\bar X$,
  \item every function $f\in \Lip_b(X,\tau,\sfd)$ admits a unique
    extension to a function
    $\bar f\in \Lip_b(\bar X,\bar \tau,\bar \sfd)$.
    The map $f\mapsto \bar f$ is an isomorphism of
    $\Lip_b(X,\tau,\sfd)$ onto $\Lip_b(\bar X,\bar \tau,\bar \sfd)$.
  \end{enumerate}
  The space $\bar \X:=(\bar X,\hat \tau,\hat \sfd,\hat \mm)$ 
  is a completion of $\X$.
\end{corollary}
Notice that in the simple case when $\tau$ is the topology induced by
the distance $\sfd$, the previous construction coincides with the
usual
completion of a metric space.
\begin{remark}
  \label{rem:better-interpretation}
  It is not difficult to check that if another e.m.t.m.~space
  $\X'$ satisfies the properties (C1)-(C4) of Corollary
  \ref{cor:completion} then $\X'$ is isomorphic to $\bar \X$,
  so that the completion of an e.m.t.m.~space is unique
  up to isomorphisms.
  We may identify $X$ with
  the $\bar \sfd$-dense subset $\iota(X)$ in $\bar X$,
  so that for every function $f':\bar X\to \R$,
  $\iota^* f'$ is just the restriction of $f'$
  to the $\bar\mm$-measurable set $X$ and $\mm$ can be considered
  as the restriction of $\bar \mm$ to $X$.
  Since $\bar\mm(\bar X\setminus X)=0$, we can
  identify $\bar \mm$ and $\mm$.
  Every function $f:X\to \R$
  can be considered defined $\bar\mm$-a.e.
  and the trivial extension
  provides a realization of $\iota_*$
  (which of course does not coincides with the extension
  $\bar f$ 
  of $f\in \Lip_b(X,\tau,\sfd)$ by continuity).
\end{remark}
We conclude this section with an example of application of
the compactification trick to prove
a useful approximation result.
\begin{corollary}
  \label{cor:adapted-is-better}
  Let $\X=(X,\tau,\sfd,\mm)$ be an e.m.t.m.~space
  endowed with an \emph{adapted} algebra $\AA$.
  Then for every bounded upper semicontinuous function $g:X\to \R$
  \begin{align}
    \label{eq:537}
    g(x)&=\inf\Big\{h(x): h\in \AA,\ h\ge g\Big\},
          \\
    \label{eq:533}
    \int_X g\,\d\mm&=
    \inf \Big\{\int_X h(x)\,\d\mm: h\in \AA,\ h\ge g\Big\}.
  \end{align}
\end{corollary}
\begin{proof}
  Let $\iota:X\to \hat X$ be the compactification
  of $\X$ induced by $\AA$ and let $\hat g:\hat X\to \R$
  be the upper semicontinuous envelope of $g$ to $\hat X$,
  i.e.~the lowest upper semicontinuous function
  whose restriction to $X$ is greater or equal to $g$.
  We know that for every $\hat x=\iota(x)\in X$
  \begin{displaymath}
    \hat  g(\hat x)=
    \inf_{U\in \UU_{\hat x}} \sup \big\{g(y):y\in X,\ \iota(y)\in
    U\big\}
    =g(x)
  \end{displaymath}
  since $g$ is upper semicontinuous; here we use the fact that
  $\iota$ is an homeomorphism between $(X,\tau)$
  and $(\iota(X),\hat \tau)$, since $\AA$ is adapted).
  Thanks to (the u.s.c.~version of) \eqref{eq:324lsc}
  we deduce that
  \begin{displaymath}
    \hat g(\hat x)=\inf \big\{h(\hat x):h\in \rmC_b(\hat X,\hat
    \tau),\ h\ge \hat g\big\}
  \end{displaymath}
  and since $\hat \AA=\Gamma(\AA)$ is uniformly dense in $\rmC_b(\hat X,\hat
  \tau)$ we deduce the formula
  \begin{displaymath}
    \hat g(\hat x)=\inf \big\{\hat h(\hat x):h\in \AA,\ h\ge
    g\text{ in $X$}\big\}
  \end{displaymath}
  which \eqref{eq:537}.
  Since $H:=\{h\in \rmC_b(\hat X,\hat
  \tau),\ h\ge \hat g\big\}$ is a directed set
  by the order relation $h_0\prec h_1 \ \Leftrightarrow
  h_0(x)\le h_1(x)\text{ for every }x\in \hat X$,
  \eqref{eq:Beppo_Levi_general} yields
  \begin{equation}
    \label{eq:538}
    \int_{\hat X}\hat g\,\d\hat \mm=
    \lim_{h\in H}\int_{\hat X} h\,\d\hat \mm
    =
    \inf_{h\in H}\int_{\hat X}h \,\d\hat \mm.
  \end{equation}
  Since $\Gamma(\AA)$ is uniformly dense in $\rmC_b(\hat X,\hat
  \tau)$,
  every function $h\in H$ can be uniformly approximated
  from above by functions of the forms $\hat f=\Gamma(f)$ 
  for $f\in \AA$, $\Gamma(f)\ge h\ge \hat g$; we deduce that 
  \begin{displaymath}
    \int_{\hat X}\hat g\,\d\hat \mm
    =
    \inf_{f\in \AA,\ f\ge g}\int_{\hat X}\hat f \,\d\hat \mm=
    \inf_{f\in \AA,\ f\ge g}\int_{X} f \,\d \mm.\qedhere
  \end{displaymath}
\end{proof}

\nc
\newcommand{\Para}[1]{\medskip \noindent{\bfseries \S\ #1}}
\newenvironment{notes}{\small}{}
\subsection{Notes}
\label{subsec:notes2}
\begin{notes}

\Para{\ref{subsec:measure}}:
general references for measure theory are
\cite{Bogachev07,Schwartz73}; here we mainly followed the approach
to Radon measures given by \cite{Schwartz73}, trying to minimize the
topological assumptions. The main points are the \emph{complete
  regularity}
of $(X,\tau)$ (which is almost needed for the standard formulation of the weak convergence
of probability measures, see the Appendix of \cite{Schwartz73})
and the Radon property of the reference
measure $\mm$.
Complete regularity is in fact equivalent to the fact that
continuous functions characterizes the topology of $X$ (see \S\
\ref{subsec:initial} in the Appendix) and it
is also important for the formulation of the extended
metric-topological spaces and for the compactification argument of \S\ \ref{subsec:compactification}.

\Para{\ref{subsec:extended-def}}:
here we followed very closely the presentation of
\cite[Section 3]{AES16}. Extended metrics and the use of an auxiliary
topology have already been considered in \cite{AGS14I},
under a slightly different set of compatibility conditions.

\Para{\ref{subsec:KR}}:
the section just recalls the basic
properties of the
Kantorovich-Rubinstein distance,
adapted to the extended
setting, 
General references are
the books \cite{Villani03,Ambrosio-Gigli-Savare08,Villani09};
notice that $\sfK_\delta$ is sometimes called Kantorovich-Wasserstein
distance of order $1$ and denoted by $W_1$.
The approximation result is 
quite similar to \cite{AES16}, where the distance of order $2$
has been considered.

\Para{\ref{subsec:aslip}}:
most of the result are classic for local slope \eqref{eq:266}.
We adapted the same approach of \cite{AGS14I}
to the more refined (and stronger) local Lipschitz constant,
which in the present setting also depends on the topology $\tau$.
\eqref{eq:307} plays a crucial role in the locality property
of the Cheeger energy and \eqref{eq:308}
is quite useful to derive contraction estimate
for its $L^2$ gradient flow. It is in fact possible to prove
a more refined property, see \cite{Luise-Savare19}.

\Para{\ref{subsec:compatible-A}} contains
the main definition of algebras of functions compatible
with the extended metric-topological setting. The basic requirements
is that the algebra is sufficiently rich to recover the extended
distance.
We also collected a few results, mainly based on Bernstein
polynomials \cite{Lorentz86}, which will be quite useful
to replace Lipschitz truncations with smoother
polynomial maps preserving the algebraic structure.

\Para{\ref{subsec:compactification}}:
measure-preserving isometric embeddings play an important role
in the theory of metric-measure spaces, in particular
when one studies their convergence
(see \cite[Chap.~27]{Villani09}, \cite{Gromov99}).
Here we adapted this notion to the
presence of the auxiliary topology $\tau$ and of
the compatible algebra $\AA$.
Another typical application arises in regular representation
of Dirichlet spaces
(the so-called transfer method, 
\cite[Chap.~VI]{Ma-Rockner92});
the idea of using the Gelfand transform to
construct a suitable compactification is
taken from \cite{Fukushima71}, \cite[Appendix
A.4]{Fukushima-Oshima-Takeda11}
and it is based on one of the possible construction
of the Stone-Cech compactification of a completely regular
topological space.

\end{notes}

\section{Continuous curves and nonparametric arcs}
\label{sec:curves}
\GGG
This section mostly contains classic material on the topology of space
of curves, adapted to the extended metric-topological setting.
Its main goal is to construct a useful setting to deal with Radon
measures
on (non parametric) rectifiable curves. 
Differently from other approaches (see e.g.~\cite{Paolini-Stepanov12,HKST15,ADS15})
we first study class of equivalent curves (up to reparametrizations)
without assuming their
rectifiability.
In \S\,\ref{subsec:continuous_curves} we study
the natural e.m.t.~structure on $\rmC([a,b];(X,\tau))$,
and in \S\,\ref{subsec:arcs}
we consider the natural quotient space of (continuous) arcs $\Arc(X,\tau)$,
which behaves quite well with respect to topology and distance.
The last part \ref{subsec:rectifiable_arcs}
is focused in $\sfd$-rectifiable arcs $\RA(X,\sfd)$, considered as a
natural Borel subset of $\Arc(X,\tau)$. Here we have the arc-length
reparametrization
at our disposal, and we study
measurability properties of important operations, like
evaluations, integrals, length, and reparametrizations. We will also state
a useful compactness result, which
is a natural generalization of Arzel\`a-Ascoli Theorem.
Theorem \ref{thm:important-arcs} will collect
most of the main
properties we will use in the next chapters.

\subsection{Continuous curves}
\label{subsec:continuous_curves}
Let $(X,\tau)$ be a completely regular Hausdorff space.
We will denote by $\rmC([a,b];(X,\tau))$ the set of $\tau$-continuous curves
$\gamma:[a,b]\to X$ endowed with the compact-open topology
\index{Compact-open topology}
$\tau_\rmC$ (we will simply write $\rmC([a,b];X)$, when
the topology $\tau$ will be clear from the context).
By definition, a subbasis generating $\tau_C$ is given by 
the collection of sets %finite intersections 
\begin{equation}
  \label{eq:247}
  %\cap_{h=1}^H 
  S(K,V):=\Big\{\gamma\in \rmC([a,b];X):\gamma(K)\subset
  V\Big\},\quad
  K\subset [a,b]\text{ compact,}\ 
  V\text{ open in $X$.}
\end{equation}
\begin{remark}
  \label{re:monotone-partitions}
  \upshape
Thanks to the particular structure of the domain $[a,b]$, we can
also consider an equivalent basis associated to partitions
$\cP=\{a=t_0<t_1<\cdots<t_J= b\}$:
\begin{equation}
  \label{eq:246}
  \cap_{j=1}^J \Big\{\gamma\in \rmC([a,b];X):\gamma([t_{j-1},t_j])\subset
  W_j\Big\},\quad
  t_j\in \cP,\ 
  W_j\text{ open in $X$.}
\end{equation}
It is also not restrictive to
consider partitions $\cP$ induced by rational points $t_j\in \Q\cap\,
]a,b[$,
$j=1,\cdots,J.$
It is sufficient to show that if $\gamma_0$ belongs to an open set
$U=\cap_{h=1}^HS(K_h,V_h)$ 
arising from the finite intersection of elements of the subbasis
\eqref{eq:247},
we can also find a set $U'$ of the form \eqref{eq:246} such that
$\gamma_0\in U'\subset U$. 
To this aim, it is not restrictive to add to the collection
$\{K_h,V_h\}_{h=1}^H$ %of \eqref{eq:247}
 the couple $K_0=[a,b],\ V_0=X$;
we can cover each $K_h$ with a finite number of intervals
$I_{h,k}=[\alpha_{h,k},\beta_{h,k}]$, $1\le k\le k(h)$, such that
$\gamma_0(I_{h,k})\subset V_h$. We can then take the partition $\cP$ of
$[a,b]$ containing all the extrema of $I_{h,k}$
(notice that $a$ and $b$ are included).
If $t_j$, $j\ge 1$, is a point of the partition $\cP$, we set
$W_j:=\cap\big\{V_h:\exists k\in \{1,\cdots,k(h)\},\
[t_{j-1},t_j]\subset I_{h,k}\big\}. $
\end{remark}
Let us recall a few simple and useful properties of the compact-open
topology
\cite[\S\,46]{Munkres00}:
\begin{enumerate}[\rm (CO1)]
\item If the topology $\tau$ is induced by a distance $\delta$, 
then the topology $\tau_\rmC$ is induced by the uniform distance
$\sfd_\delta(\gamma,\gamma'):=\sup_{t\in
  [a,b]}\delta(\gamma(t),\gamma'(t))$ and convergence w.r.t.~the
compact-open 
topology coincides with the uniform convergence w.r.t.~$\delta$.
If moreover $\tau$ is separable then also $\tau_\rmC$ is separable.
\item 
  The evaluation map $\sfe:[a,b]\times\rmC([a,b];X)\to X$, 
$\sfe(t,\gamma):=\gamma(t)$, is continuous.
\item
If $f:X\to Y$ is a continuous function with values in 
a Hausdorff space $(Y,\tau_Y)$, the composition
map $\gamma\mapsto f\circ\gamma$ is continuous 
from $\rmC([a,b];X)$ to $\rmC([a,b];Y)$.
\item If $\sigma:[0,1]\to[0,1]$ is nondecreasing and continuous
  the map $\gamma\mapsto \gamma\circ\sigma$ 
  is continuous 
  from $\rmC([a,b];X)$ to $\rmC([a,b];X)$.
\item If $\tau$ is Polish, then
  $\tau_\rmC$ is Polish.
\end{enumerate}
The proof of the first two properties can be found, e.g., in
\cite[Thm. 46.8, 46.10]{Munkres00};
in the case when $\tau$ is metrizable and separable, the separability of
$\tau_\rmC$ follows by \cite[4.2.18]{Engelking89}. 

In order to prove (CO3) it is sufficient to show that
the inverse image of an arbitrary element $S(K,W)$ (as in
\eqref{eq:247}) 
of the subbasis
generating the compact-open topology of $\rmC([a,b];Y)$
is an element of the corresponding subbasis of the topology of
$\rmC([a,b];X)$.
In fact, $f\circ\gamma\in S(K,W)$ if and only if $\gamma\in
S(K,f^{-1}(W))$.
% CO3 immediately follows by observing that 
% a map finite intersections of
% sets as in \eqref{eq:247} forms a
% a basis of open sets for the compact-open topology.
% If $\gamma_0\in \rmC([a,b];X)$, $K_i$ are compact subsets
% of $[a,b]$ and $W_i$ are open subsets of $Y$, $i=1,\cdots,I$, with
% $(f\circ\gamma_0)(K_i)\subset W_i$, then setting $V_i:=f^{-1}(W_i)$ 
% we have $\gamma_0(K_i)\in V_i$ so that $\gamma_0\in \cap_{i=1}^I
% S_X(K_i,V_i) $ and 
% for every $\gamma\in \cap_{i=1}^I S(K_i,V_i)$ we have $f\circ
% \gamma\in \cap_{i=1}^IS(K_i,W_i).$ 
%
(CO4) can be proved by a similar argument: 
if $\gamma\circ\sigma \in S(K,V)$ if and only if $\gamma\in
S(\sigma(K),V)$.
% then $\gamma(\sigma(K))\subset V$ 
%  $K_i$ are compact subsets
% of $[a,b]$ and $V_i$ are open subsets of $X$, $i=1,\cdots,I$, with
% $\gamma_0\circ\sigma\in \cap_{i=1}^I S(K_i,V_i)$, then setting
% $K_i':=\sigma(K_i)$ we have $\gamma_0\in \cap_{i=1}^I S(K_i',V_i)$ 
% and every curve $\gamma\in \cap_{i=1}^I S(K_i',V_i)$ satisfies
% $\gamma\circ \sigma(K_i)=\gamma(K_i')\subset V_i$ so that 
% $\gamma\circ\sigma\in \cap_{i=1}^I S(K_i,V_i)$.
%
Finally (CO5) is a consequence of (CO1).

We will denote by $\sfe(\gamma)$ the image $\gamma([a,b])$ in $X$.
\subsubsection*{(Semi)distances on $\rmC([a,b];X)$}
An extended semidistance $\delta:X\times X\to[0,\infty]$ 
induces an extended semidistance $\delta_\rmC$ in $\rmC([a,b];X)$ by
\begin{equation}
  \label{eq:235}
  \delta_\rmC(\gamma_1,\gamma_2):=\sup_{t\in [a,b]}
  \delta(\gamma_1(t),\gamma_2(t)).
\end{equation}
We have
\index{Auxiliary topology}
\begin{proposition}
  \label{prop:C-extended}
  If $(X,\tau,\sfd)$ is an extended metric-topological space,
  than also $(\rmC([a,b];(X,\tau)),\tau_\rmC,\sfd_\rmC)$ is an extended
  metric-topological space.\\
  If moreover $\tau'$ is an auxiliary topology for $(X,\tau,\sfd)$
  according to Definition \ref{def:auxiliary}, then
  $\tau_\rmC'$ is an auxiliary topology for
  $(\rmC([a,b];(X,\tau)),\tau_\rmC,\sfd_\rmC)$ as well.
\end{proposition}
\begin{proof}
  Notice that for every $f\in \Lip_{b,1}(X,\tau,\sfd)$ and $t\in [a,b]$ the map
  $\mathsf f_t:=f\circ \sfe_t$ belongs to $\Lip_{b,1}
  (\rmC([a,b];X),\tau_\rmC,\sfd_\rmC)$
  and it is easy to check by \eqref{eq:220} that 
  \begin{equation}
    \label{eq:236}
    \sfd_\rmC(\gamma_1,\gamma_2)=\sup\Big\{|\mathsf
    f_t(\gamma_1)-\mathsf f_t(\gamma_2)|:
    f\in \Lipbu1(X,\tau,\sfd),\ t\in [a,b]\Big\}.
  \end{equation}
  For every finite collection $\cK:=\{K_h\}_{h=1}^H$ of compact
  subsets of $[a,b]$ 
  and for every finite collection $\cF:=(f_h)_{h=1}^H$ in
  $\Lipb(X,\tau,\sfd)$ let us consider the function $F=F_{\cK,\cF}$
  defined by
  \begin{equation}
    \label{eq:237}
    F(\gamma):=\max_{1\le h \le H}\max_{t\in K_h}f_h(\gamma(t)).
  \end{equation}
  It is easy to check that $F\in
  \Lipb(\rmC([a,b];X),\tau_\rmC,\sfd_\rmC)$; we want to show that
  this family of functions separates points from closed sets. 
  To this aim, we fix a closed set $\rmF\subset \rmC([a,b];X)$ 
  and a curve $\gamma_0\in \rmC([a,b];X)\setminus \rmF$. 

  By the definition of compact-open topology, we can find a
  collection $\cK =\{K_h\}_{h=1}^H$ of compact sets of $[a,b]$ and open sets $U_h\subset X$ such that 
  the open set $\rmU=\{\gamma:\gamma(K_h)\subset U_h,\ 1\le h\le H\}$ is
  included in $\rmC([a,b];X)\setminus \rmF$ and
  contains $\gamma_0$. By \eqref{eq:CR} and the compactness 
  we can find nonnegative functions $\cF=(f_h)_{h=1}^H\subset  \Lipb(X,\tau,\sfd)$ 
  such that $f_h\equiv 1$ on $X\setminus U_h$ and 
  $f_h\restr{\gamma_0(K_h)}\equiv 0$.
  It follows that $F=F_{\cK,\cF}$ satisfies $F(\gamma_0)=0$ and 
  $F(\gamma)\ge 1$ for every $\gamma$ in the complement of
  $\rmU$, in particular in $\rmF$.

  Let us eventually check the last statement, by checking that
  $\tau_\rmC'$ satisfies
  properties (A1,2,3) of Definition \ref{def:auxiliary}. (A1,2) are
  obvious.
  Concerning (A3) we select a sequence $f_n\in \Lip_b(X,\tau',\sfd)$
  satisfying \eqref{eq:554} and the countable collection of maps
  $\cF:=(\mathsf f_{n,t})_{n\in \N,\ t\in [a,b]\cap \Q}$,
  $\mathsf f_{n,t}:=f_n\circ \sfe_t$.
  It is clear that for every $\gamma_1,\gamma_2\in \rmC([a,b];X)$ and
  every $n\in \N$
  \begin{displaymath}
    \sup_{t\in [a,b]}|\mathsf f_{n,t}(\gamma_1)-\mathsf
    f_{n,t}(\gamma_2)|=
    \sup_{t\in [a,b]\cap \Q}|\mathsf f_{n,t}(\gamma_1)-\mathsf
    f_{n,t}(\gamma_2)|
  \end{displaymath}
  \eqref{eq:554} then yields
  \begin{displaymath}
    \sfd_\rmC(\gamma_1,\gamma_2)=
    \sup_{\mathsf f\in \cF}|\mathsf f(\gamma_1)-\mathsf f(\gamma_2)|.\qedhere
  \end{displaymath}
\end{proof}
We can state a useful compactness criterium, which is the natural
extension of Arzel\`a-Ascoli Theorem to extended metric-topological structures.
\begin{proposition}
  \label{prop:compactness}
  Let us suppose that $(X,\tau)$ is compact
  and let $\Gamma\subset \rmC([a,b];X)$ be $\sfd$-equicontinuous,
  i.e.~there exists $\omega:[0,\infty)\to[0,\infty)$ 
  concave, nondecreasing and continuous with $\omega(0)=0$ such that
  \begin{equation}
    \label{eq:228}
    \sfd(\gamma(r),\gamma(s))\le \omega(|r-s|)\quad\text{for every
    }r,s\in [a,b],\ \gamma\in \Gamma.
  \end{equation}
  Then $\Gamma$ is relatively compact with respect to $\tau_\rmC$.
\end{proposition}
\begin{proof}
  Let $\gamma_i$, $i\in I$, be a net in $\Gamma$. 
  Since $X$ is compact, we can find a subnet $h:J\to I$ and a limit
  curve $\gamma:[a,b]\to X$ such that 
  $\lim_{j\in J}\gamma_{h(j)}=\gamma$ 
  with respect to the topology of pointwise convergence.
  Passing to the limit in \eqref{eq:228} we immediately see that 
  $$\sfd(\gamma(r),\gamma(s))\le \liminf_{j\in
    J}\sfd(\gamma_{h(j)}(r),\gamma_{h(j)}(s))
  \le \omega(|r-s|)\quad\text{for every $r,s\in [a,b]$},$$
  so that 
  $\gamma\in \rmC([a,b];(X,
  \sfd))$.

  Let us now consider a neighborhood $\cU=\cap_{h=1}^H S(K_h,V_h)$ of
  $\gamma$: we will show that there exists $j_0\in J$ such that 
  $\gamma_{h(j)}\in \cU$ for every $j\succeq j_0$.
  Let us consider functions $f_h\in \Lipb(X,\tau,\sfd)$ satisfying
  $f_h\equiv 0$ on $\gamma(K_h)$ and $f_h\equiv 1$ on $X\setminus V_h$
  as in the previous Lemma and let $L>0$  be the maximum of the
  Lipschitz constants of $f_h$. We call $\ff_j:[a,b]\to \R^H$ 
  the family of curves
  $\ff_j(t):=(f_1(\gamma_{h(j)}(t)),f_2(\gamma_{h(j)}(t)),\cdot,f_H(\gamma_{h(j)}(t))$
  indexed by $j\in J$. Since $f_h$ are continuous, we have
  \begin{equation}
    \label{eq:229}
    \lim_{j\in J}\ff_j(t)=\ff(t)\quad\text{pointwise},\quad
    \ff(t):=(f_1(\gamma(t)),f_2(\gamma(t)),\cdot,f_H(\gamma(t)).
  \end{equation}
  On the other hand, we have
  \begin{equation}
    \label{eq:242}
    |\ff_j(r)-\ff_j(s)|\le HL\omega(|r-s|)
  \end{equation}
  so that Ascoli-Arzel\`a theorem (in $\rmC([a,b];\R^H)$) yields
  \begin{equation}
    \label{eq:264}
    \lim_{j\in J}\sup_{t\in [a,b]}|\ff_j(t)-\ff(t)|=0.
  \end{equation}
  Choosing $j_0\in J$ so that $\sup_{t\in [a,b]}|\ff_j(t)-\ff(t)|\le 1/2$
  for every $j\succeq j_0$,
  since $f_h(\gamma(t))\equiv 0$ whenever $t\in K_h$
  we deduce that $f_h(\gamma_j(t))\le 1/2$
  whenever $t\in K_h$ so that $\gamma_j(K_h)\subset V_h$.
\end{proof}

\subsection{Arcs}
\label{subsec:arcs}
Let us denote by $\Sigma$ 
the set of continuous, nondecreasing and
surjective map $\sigma:[0,1]\to[0,1]$ and
by $\Sigma'$ the subset of $\Sigma$ of the invertible maps,
thus increasing homeomorphisms of $[0,1]$.
We also set
\begin{equation}
  \label{eq:256}
  \Sigma_2:=\Big\{\sigma\in \Sigma:|\sigma(r)-\sigma(s)|\le 2|r-s|\Big\}.
\end{equation}
On $\rmC([0,1];(X,\tau))$
we introduce the symmetric and reflexive relation
\begin{equation}
  \label{eq:211}
  \begin{gathered}
    \gamma_1\sim \gamma_2\quad\text{if}\quad 
    \exists\,\sigma_i\in \Sigma:\ 
    {\gamma_1}\circ\sigma_1={\gamma_2}\circ\sigma_2.
  \end{gathered}
\end{equation}
Notice that $\gamma\sim\gamma\circ\sigma$ for every $\sigma\in
\Sigma$. 
It is also not difficult to check that if 
a map $\gamma:[0,1]\to X$ satisfies $\gamma\circ\sigma\in
\rmC([0,1];(X,\tau))$ for some $\sigma\in \Sigma$ then $\gamma\in
\rmC([0,1];(X,\tau))$. In fact, if $A\subset X$ is a closed set, then
$B:=(\gamma\circ\sigma)^{-1}(A)=\sigma^{-1}(\gamma^{-1}(A))$ is closed
in $[0,1]$ and therefore compact. Since $\sigma$ is surjective, for
every $Z\subset [0,1]$ we have $\sigma(\sigma^{-1}(Z))=Z$ so that 
$\gamma^{-1}(A)=\sigma(B)$ is closed since it is the continuous image
of a compact set.

We want to show that $\sim$ also satisfies the transitive
property, so that it is an equivalence relation. 

If $\delta:X\times X\to [0,+\infty]$ is a $\tau$-l.s.c.~extended
  semidistance we also introduce 
  \begin{equation}
    \delta_\rmA(\gamma_1,\gamma_2):=\inf_{\sigma_i\in \Sigma}
    \delta_\rmC(\gamma_1\circ\sigma_1,\gamma_2\circ\sigma_2)
    \quad\text{
      for every $\gamma_i\in \rmC([0,1];(X,\tau))$}\label{eq:278}
\end{equation}
and the set
\begin{equation}
  \label{eq:282}
  \rmC([0,1];(X,\tau,\delta)):=\big\{\gamma\in \rmC([0,1];(X,\tau)):
  \lim_{s\to t}\delta(\gamma(s),\gamma(t))=0\quad\ \text{for every
  }t\in [0,1]\big\}.
\end{equation}
It is not difficult to check that 
\begin{equation}
  \label{eq:282bis}
  \gamma\in \rmC([0,1];(X,\tau,\delta))
  \quad\Rightarrow\quad
  \lim_{r\down0}\sup_{|t-s|\le r}
  \delta(\gamma(s),\gamma(t))=0;
\end{equation}
in the case $\delta=\sfd$ we have
\begin{equation}
  \label{eq:530}
    \rmC([0,1];(X,\tau,\sfd))=\rmC([0,1];(X,\sfd)).
\end{equation}
\begin{theorem}[Reparametrizations and semidistances]
  \label{thm:reparam}
  Let $\delta:X\times X\to [0,+\infty]$ be a $\tau$-l.s.c.~extended
  semidistance and let
  $\gamma_i\in \rmC([0,1];(X,\tau,\delta))$, $i=1,2,3$.
  We have
\begin{subequations}
\label{eq:267}
  \begin{align}
    \label{eq:258}
    \delta_\rmA(\gamma_1,\gamma_2)                                      
    &=
      \inf_{\sigma\in
      \Sigma'}\delta_\rmC({\gamma_1},{\gamma_2}\circ\sigma)
    \\&=
          \min_{\sigma_i\in
      \Sigma_2}\delta_\rmC({\gamma_1}\circ\sigma_1,{\gamma_2}\circ\sigma_2)
          \label{eq:239tris}
    \\&=\inf_{\gamma_i'\sim\gamma_i}\delta_{\rmC}(\gamma_1',\gamma_2').
        \label{eq:280}
  \end{align}
\end{subequations}
In particular $\delta_\rmA$ satisfies the triangle inequality
\begin{equation}
  \label{eq:281}
  \delta_{\rmA}(\gamma_1,\gamma_3)\le 
    \delta_{\rmA}(\gamma_1,\gamma_2)+  \delta_{\rmA}(\gamma_2,\gamma_3).
\end{equation}
\end{theorem}
\begin{proof}  
  If $\sigma\in \Sigma$ we will still denote by $\sigma$ its extension
  to $\R$ defined by the map $r\mapsto \sigma(0\lor r \land 1)$.
  We introduce the $(1+\eps)$-Lipschitz map $j_{\sigma,\eps}:\R\to\R$
  defined by
  \begin{equation}
    s=j_{\sigma,\eps}(r)\quad \Leftrightarrow\quad
    s+\eps\sigma(s)=(1+\eps)r  \label{eq:254}
  \end{equation}
  and the maps $\hat \sigma_\eps,\sigma_\eps:\R\to \R$
    \begin{align}
    \label{eq:253}
    \hat \sigma_\eps(r):={}&(1+\eps)^{-1}(\eps r+\sigma(r)),\quad
                             |\hat \sigma_\eps(r)-\sigma(r)|\le
                             \frac{\eps}{1+\eps}|r-\sigma(r)|\le
                             \eps,\\
      \sigma_\eps(r):={}& \eps^{-1}((1+\eps) r-j_{\sigma,\eps}(r))=\sigma( j_{\sigma,\eps}(r)).\label{eq:346}
    \end{align}
  % \begin{align}
  %   \label{eq:253}
  %   \hat \sigma_\eps(r):={}&(1+\eps)^{-1}(\eps r+\sigma(r)),\quad
  %                            |\hat \sigma_\eps(r)-\sigma(r)|\le
  %                            \frac{\eps}{1+\eps}|r-\sigma(r)|\le \eps,
  %   \\
  %   \end{align}
    Notice that the restrictions to $[0,1]$ of the maps 
    $\hat \sigma_\eps,j_{\sigma,\eps},\sigma_\eps$
    operate in $[0,1]$ and
    $\hat\sigma_\eps\in \Sigma'$,
    $j_{\sigma,\eps},\sigma_\eps\in \Sigma$,
    $\sigma_{\eps}$
    is $(1+1/\eps)$-Lipschitz.
  We also denote by $\omega_\gamma$ 
  the $\delta$-modulus of continuity of $\gamma\in \rmC([0,1];(X,\tau,\delta))$, i.e.
  \begin{equation}
    \label{eq:348}
    \omega_\gamma(r):=\sup \Big\{\delta(\gamma(s),\gamma(t)):s,t\in
    [0,1],\ |s-t|\le r\Big\},
  \end{equation}
  observing that $\lim_{r\down0}\omega_{\gamma_i}(r)=0$ in the case of
  the curves $\gamma_1,\gamma_2$ considered by the Lemma.
  
  In order to prove \eqref{eq:258}, we observe that 
  that $\delta_\rmC$ is invariant w.r.t.~composition with
  arbitrary $\sigma\in \Sigma$
  \begin{equation}
    \label{eq:262}
    \delta_\rmC(\gamma_1,\gamma_2)=\delta_\rmC(\gamma_1\circ\sigma,\gamma_2\circ\sigma)
    \quad\text{for every }\gamma_i\in \rmC([0,1];X),\ \sigma\in \Sigma,
  \end{equation}
  and every $\sigma\in \Sigma$ can be uniformly approximated
  by the increasing homeomorphisms $\hat\sigma_\eps\in \Sigma'$;
  we easily get
  \begin{equation}
    \label{eq:263}
    \delta_\rmC(\gamma\circ\sigma,\gamma\circ\hat\sigma_{\eps})
    \le \omega_\gamma\big(\sup_{r\in [0,1]}
    |\sigma(r)-\hat\sigma_{\eps}(r)|\big)
    \topref{eq:253}\le \omega_\gamma(\eps),
  \end{equation}
  so that the triangle inequality for $\delta_\rmC$ yields
  \begin{align*}
    \delta_\rmC({\gamma_1}\circ\sigma_1,{\gamma_2}\circ\sigma_2)
    &\ge
      % \ell( \gamma_1)% +\ell(\gamma_2))
      \delta_\rmC({\gamma_1}\circ\hat\sigma_{1,\eps},{\gamma_2}\circ\sigma_{2})-
      \omega_{{\gamma_1}}(\eps)
      \ge 
        \delta_\rmC({\gamma_1},
      {\gamma_2}\circ\sigma_{2}')-
      \omega_{{\gamma_1}}(\eps)
    \\&\ge
         \delta_\rmC({\gamma_1},
        {\gamma_2}\circ\hat\sigma_{2,\eps}')-
        \omega_{{\gamma_1}}(\eps)-
        \omega_{{\gamma_2}}(\eps)
    \\&    \ge 
        \inf_{\sigma\in
        \Sigma'}\delta_\rmC({\gamma_1},{\gamma_2}\circ\sigma)-
        \omega_{{\gamma_1}}(\eps)-
        \omega_{{\gamma_2}}(\eps)
  \end{align*}
  where $\sigma_2':=\sigma_{2}\circ(\hat\sigma_{1,\eps})^{-1}$
  and $\hat\sigma_{2,\eps}'$ is obtained by $(\sigma_2')_\eps$ as in \eqref{eq:253}.
  Taking the infimum w.r.t.~$\sigma_1,\sigma_2\in \Sigma$ and
  passing to the limit as $\eps\down0$ we obtain 
  \begin{displaymath}
    \delta_{\rmA}(\gamma_1,\gamma_2)
    \ge
    % \Al_{\gamma_1}\circ
    % % \varrho_1,\Al_{\gamma_2}\circ\varrho_2)=
    \inf_{\sigma\in \Sigma'} \delta_{\rmC}({\gamma_1}
    ,{\gamma_2}\circ\sigma).
  \end{displaymath}
  Since the opposite inequality is obvious, we get \eqref{eq:258}.
 
  The triangle inequality is an immediate consequence of
  \eqref{eq:258}: if $\gamma_1,\gamma_2,\gamma_3\in \rmC([0,1];(X,\tau,\delta))$ and
  $\sigma_1,\sigma_3\in \Sigma'$ 
  we have
  \begin{align*}
    \delta_\rmA(\gamma_1,\gamma_3)\le 
    \delta_\rmC(\gamma_1\circ\sigma_1,\gamma_3\circ\sigma_3)
    \le 
    \delta_\rmC(\gamma_1\circ\sigma_1,\gamma_2)
    +
    \delta_\rmC(\gamma_2,\gamma_3\circ\sigma_3);
  \end{align*}
  taking the infimum w.r.t.~$\sigma_1,\sigma_3$ we obtain \eqref{eq:281}.

  Let us now prove \eqref{eq:239tris}, i.e.~the infimum in \eqref{eq:278}
  is attained
  by a couple $\varrho_i\in \Sigma$ given by $2$-Lipschitz maps.
  We observe that 
  \begin{equation}
    \label{eq:255}
    \delta_\rmC({\gamma_1},{\gamma_2}\circ\varrho)\topref{eq:262}=
    \delta_\rmC({\gamma_1}\circ
    j_{\varrho,\eps},{\gamma_2}\circ\varrho\circ j_{\varrho,\eps})\topref{eq:346}=
    \delta_\rmC({\gamma_1}\circ
    j_{\varrho,\eps},{\gamma_2}\circ\varrho_\eps)
  \end{equation}
  and $j_{\varrho,\eps}$ is $(1+\eps)$-Lipschitz, $\varrho_\eps$ is
  $(1+\eps^{-1})$-Lipschitz.
  Choosing $\eps=1$ we deduce that the infimum in \eqref{eq:278}
  can be restricted to $\Sigma_2$. Since $\Sigma_2$ 
  is compact w.r.t.~uniform convergence, ${\gamma_i}$ are 
  continuous and $\delta_\rmC$ is lower semicontinuous
  w.r.t.~$\tau_\rmC$, we obtain \eqref{eq:239tris}.

  Let us eventually consider \eqref{eq:280}; since 
  $\inf_{\gamma_i'\sim\gamma_i}\delta_{\rmC}(\gamma_1',\gamma_2')\le
  \delta_\rmA(\gamma_1,\gamma_2)$ it is sufficient to prove the
  opposite inequality. 
  By \eqref{eq:281} we have
  \begin{displaymath}
    \delta_{\rmC}(\gamma_1',\gamma_2')\ge
    \delta_{\rmA}(\gamma_1',\gamma_2')\ge
    -\delta_{\rmA}(\gamma_1',\gamma_1)
    +\delta_{\rmA}(\gamma_1,\gamma_2)
    -\delta_{\rmA}(\gamma_2,\gamma_2')=
    \delta_{\rmA}(\gamma_1,\gamma_2). \qedhere
  \end{displaymath}
\end{proof}
\begin{corollary}
  \label{cor:tedious-arcs}
  The relation $\sim$ satisfies the transitive property and it is an
  equivalent relation. 
  Moreover
  \begin{enumerate}
  \item The space $\Arc(X,\tau):=\rmC([0,1];(X,\tau))/\sim$ endowed with the
    quotient topology $\tau_\rmA$ is an Hausdorff space.
    We will denote by $[\gamma]$
    the corresponding equivalence class associated to $\gamma\in
    \rmC([0,1];(X,\tau))$ and by $\Quot: \rmC([0,1];(X,\tau))\to \Arc(X,\tau)$
    the quotient map $\gamma\mapsto [\gamma]$.
  \item If $\delta$ is a $\tau$-continuous semidistance, then
    $\delta_\rmA$ is a $\tau_\rmA$ continuous semidistance
    (considered as a function between equivalence classes of curves).
  \item If the
    topology $\tau$ is induced by the distance $\delta$ then the
    quotient topology $\tau_\rmA$ is induced by
    $\delta_\rmA$
    (considered as a distance between equivalence classes of curves).
  \item If $(X,\tau)$ is a Polish space, then
    $(\Arc(X,\tau),\tau_\rmA)$
    is a Souslin metrizable space.
  \end{enumerate}
\end{corollary}
\begin{proof}
  \textbf{(a)}
  Let us first prove the transitivity of $\sim$. Let $\gamma_i\in
  \rmC([0,1];X)$, $i=1,2,3$, such that
  $\gamma_1\sim\gamma_2$ and $\gamma_2\sim\gamma_3$
  and let $K:=\cup_{i=1}^3\gamma_i([0,1])$. 
  $K$ is a compact and separable set; applying Remark
  \ref{rem:compact-case} we can find a bounded and continuous
  semidistance $\delta$ whose restriction to $K\times K$ induces the
  $\tau$-topology.

  By the very definition \eqref{eq:278} of $\delta_\rmA$ we get $\delta_\rmA(\gamma_1,\gamma_2)=0$ and
  $\delta_{\rmA}(\gamma_2,\gamma_3)=0$, so that \eqref{eq:281} yields
  $\delta_\rmA(\gamma_1,\gamma_3)=0$ and \eqref{eq:239tris} yields
  $\gamma_1\sim\gamma_3$.

  Let us now show that $(\Arc(X),\tau_\rmA)$ is Hausdorff. 
  We fix two curves $\gamma_1,\gamma_2\in \rmC([0,1];X)$
  such that $[\gamma_1]\neq [\gamma_2]$,
  we consider the compact and separable subspace
  $K:=\gamma_1([0,1])\cup\gamma_2([0,1])$, and a bounded
  $\tau$-continuous
  semidistance $\delta$, generated as in Remark \ref{rem:compact-case},
  whose restriction to $K$ induces the $\tau$-topology.
  
  %It is not restrictive to assume that $\gamma_i=\Al_{\gamma_i}$. 
  We notice that the maps $\gamma\mapsto \delta_{\rmC}(\gamma_i,\gamma)=\sup_{t\in
    [0,1]}\delta(\gamma_i(t),\gamma(t))$ 
  are continuous in $\rmC([0,1];X)$; since the composition maps
  $\gamma\mapsto
  \gamma\circ\varrho$, $\varrho\in \Sigma'$, are continuous, we deduce
  that the maps
  $\gamma\mapsto \delta_{\rmA}([\gamma_i],[\gamma])=
  \inf_{\varrho\in \Sigma'}\delta_{\rmC}(\gamma_i,\gamma\circ\varrho)$
  are upper semicontinuous from $\rmC([0,1];X)$ to $\R$.
  By the above discussion, if $[\gamma_1]\neq[\gamma_2]$ we get
  $\delta_{\rmA}(\gamma_1,\gamma_2)=\delta>0$ so that 
  the open sets $U_i:=\big\{\gamma\in \rmC([0,1];X):
  \delta_{\rmA}(\gamma_i,\gamma)<\delta/2\big\}$ are disjoint (by the
  triangle inequality) saturated open neighborhoods
  of $\gamma_i$.
  \medskip

  \noindent\textbf{(b)}
  Since $\delta_\rmA$ is
  continuous
  w.r.t.~$\tau_\rmC$ and it is invariant w.r.t.~the equivalence
  relation, it is clear that
  $\delta_\rmA$ is continuous w.r.t.~$\tau_\rmA$.
  \medskip

  \noindent\textbf{(c)}
  If $\gamma_0\in \rmC([0,1];X)$ 
  every $\tau_\rmA$ open neighborhood $U_0$ of $[\gamma_0]$ in
  $\Arc(X)$ correspond to a saturated open set $U$ of $\rmC([0,1];X)$
  containing $\gamma_0$. In particular, there exists $\eps>0$ such
  that $U$ contains all the curves 
  $\gamma\in \rmC([0,1];X)$ with $\delta_{\rmC}(\gamma,\gamma_0)<\eps$ 
  and all the curves of the form $\gamma\circ\sigma$,
  $\gamma_0\circ\sigma_0$ for arbitrary $\sigma,\sigma_0\in \Sigma$. 
  It follows that $U_0$ contains
  $\{[\gamma]:\delta_{\rmA}([\gamma],[\gamma_0])<\eps\}$.
  \medskip

  \noindent\textbf{(d)}
  The Souslin property is an immediate consequence of the fact
  that $(\Arc(X,\tau),\tau_\rmA)$ is obtained as a quotient space
  of the Polish space $(\rmC([0,1];(X,\tau)),\tau_\rmC)$.
  The metrizability follows by the previous claim.  
\end{proof}
The image of an arc $[\gamma]$
is independent of the parametrization and
it will still be denoted by $\sfe([\gamma])$;
we will also set $\sfe_i([\gamma]):=\sfe_i(\gamma)$, $i=0,1$.
Every function $f:\Arc(X,\tau)\to Y$ is associated to 
a ``lifted'' function $\tilde f:\rmC([0,1];(X,\tau))\to Y$, $\tilde
f(\gamma):=f([\gamma])$, whose values are invariant by
reparametrizations of $\gamma$.
$f$ is continuous (w.r.t.~the quotient
topology $\tau_\rmA$) if and only if $\tilde f$ is continuous 
(w.r.t.~the compact-open
topology $\tau_\rmC$).

Let us now consider the case of an extended metric-topological space.
We denote by $\Arc(X,\sfd)$ the subset of arcs admitting
an equivalent $\sfd$-continuous parametrization (or, equivalently, 
whose equivalent parametrizations are $\sfd$-continuous).
Notice that $\Arc(X,\sfd)\subset \Arc(X,\tau)$.
\begin{proposition}
  \label{prop:A-extended}
  Let $(X,\tau,\sfd)$ be an extended metric-topological space.
  $(\Arc(X,\sfd),\tau_\rmA,\sfd_\rmA)$ is an extended
  metric-topological space.
\end{proposition}
\begin{proof}
  The fact that $\sfd_\rmA$ is an extended distance on $\Arc(X,\sfd)$ 
  is a consequence of the previous results (applied
  to the topology induced by $\sfd$). 
  % In fact, by \eqref{eq:239} there exists a
  % sequence of increasing homeomorphisms $\varrho_n\in \Sigma'$, $n\in
  % \N$, such that $\lim_{n\to\infty}\sfd_\rmC(\Al_{\gamma_1},
  % \Al_{\gamma_2}\circ\varrho_n)=0$.
  % Since $\varrho_n$ are increasing, 
  % by Helly's Theorem it is not restrictive to suppose that there
  % exists
  % a nondecreasing map $\varrho:[0,1]\to[0,1]$ such that
  % $\lim_{n\to\infty}\varrho_n(s)=\varrho(s)$
  % for every $s\in [0,1]$. It follows that 
  % $\Al_{\gamma_1}=
  % \Al_{\gamma_2}\circ\varrho$ and in particular therefore $\ell(\gamma_1)\le
  % \ell(\gamma_2)$.
  % Inverting the role of $\gamma_1$ and
  % $\gamma_2$ we may find an increasing map $\varrho':[0,1]\to[0,1]$ 
  % such that $\Al_{\gamma_2}=\Al_{\gamma_1}\circ\varrho'$, so that
  % $\ell(\gamma_1)=\ell(\gamma_2)=:\ell$. Since on every subinterval
  % $[t_0,t_1]\subset [0,1]$ 
  % the functions $\Al_{\gamma_i}$ cannot be constant, we deduce that
  % $\varrho$ and $\varrho'$ are homeomorphisms and therefore
  % $\gamma_1\sim\gamma_2$.
  Let us now consider the two properties (X1) and (X2)
  of Definition
  \ref{def:luft1} separately.
  
  \noindent
  \underline{\emph{Proof of property (X1).}}
  For every $J\in \N$, let us denote by $\cP_J\subset [0,1]^{J+1}$ the
  collection of
  all partitions 
  $P=(t_0,t_1,\cdots t_J)$ 
  with $0=t_0\le t_1\le t_2\le \cdots \le t_J=1$.
  For every $P\in \cP_J$ and every 
  finite collection of functions $\cF=(f_j)_{j=1}^J\subset
  \Lipb(X,\tau,\sfd)$ 
  %and $P\in \cP_J$ 
  we consider the functions $F_{P,\cF},\ F_{\cF}$ defined by
  \begin{equation}
    \label{eq:240}
    F_{(t_0,t_1,\cdots,t_J),\cF}(\gamma):=
    \max_{1\le j\le J}\max_{t\in [t_{j-1},t_j]}f_j(\gamma(t)),\quad 
    F_\cF(\gamma)=\min_{P \in \cP_J} F_{P,\cF}(\gamma).
  \end{equation}
  By fixing the uniform partition $P_J=(0,1/J,2/J,\cdots,j/J,\cdots,
  1)$ as a reference,
  it is not difficult to check that
  \begin{equation}
    \label{eq:249}
    F_\cF(\gamma)=\min_{\gamma'\sim\gamma}F_{P_J,\cF}(\gamma').%=F_{\cF}(\Al_\gamma);
  \end{equation}
  Since $F_\cF$ only depends on the equivalence class of $\gamma$, 
  it induces a map (denoted by $\bar F_\cF$) on $\Arc(X,\sfd)$.

  \noindent \emph{Claim: $\bar F_\cF$ is $\tau_\rmA$-continuous.}
  % In order to show that $\bar F$ is $\tau_\rmA$-continuous, 
  It is sufficient to show that $F_\cF$ is $\tau_\rmC$ continuous. 
  Let us consider the corresponding map
  $H:\cP_J\times\rmC([0,1];\R^J)\to \R$, 
  \begin{equation}
    \label{eq:250}
    H((t_0,t_1,\cdots,t_N),\ff):=\max_{1\le j\le J}\max_{t\in
      [t_{j-1},t_j]}f_j(t),\quad
    \ff:=(f_1,\cdots,f_J)\in \rmC([0,1];\R^J).
  \end{equation}
  $H$ is continuous map (where $P_J$ is endowed with
  the product topology of $[0,1]^{J+1}$);
  since $P_J$ is compact, 
  also $\tilde H(\ff):=\min_{(t_0,t_1,\cdots,t_N)\in
    P_J}H((t_0,t_1,\cdots,t_N),\ff)$ is continuous.
  Since $F_\cF(\gamma)=\tilde H(\ff\circ\gamma)$ 
  we conclude that $F_\cF$ is $\tau_\rmC$-continuous.

  \noindent\emph{Claim: $\bar F_\cF$ is $\sfd_\rmA$-Lipschitz.}
  Let $[\gamma_i]\in \Arc(X,\sfd)$, $\eps>0$, and $\varrho\in\Sigma'$
  such that $\sfd_\rmA([\gamma_1],[\gamma_2])\ge
  \sfd_\rmC({\gamma_1}\circ\varrho,{\gamma_2})-\eps$. 
  We select a sequence $P\in \cP_J$ 
  such that $\bar F([\gamma_2])=F_{P,\cF}(\gamma_2)$ 
  and we observe that 
  \begin{align*}
    \bar F_\cF([\gamma_1])-\bar F_\cF([\gamma_2])&\le 
    F_{\varrho(P),\cF}({\gamma_1})-
    F_{P,\cF}({\gamma_2})
    \le 
    F_{P,\cF}({\gamma_1}\circ\varrho)-
    F_{P,\cF}({\gamma_2})
                                 \\&\le 
    L\sfd_\rmC({\gamma_1}\circ\varrho,{\gamma_2})
    \le 
                                     L\big(\sfd_\rmA([\gamma_1],[\gamma_2])+\eps\big),
  \end{align*}
  where $L$ is the greatest Lipschitz constant of the functions
  $f_j$. 
  Since $\eps>0$ is arbitrary and we can invert the order of
  $\gamma_1$ and $\gamma_2$ we conclude that $\bar F_\cF$ is
  $\sfd_\rmA$-Lipschitz.

  \noindent \emph{Conclusion.}
  Now, if $\rmU$ is saturated open set containing ${\gamma_0}$,
  thanks to Remark \ref{re:monotone-partitions} we can
  find
  a collection of open sets $U_j\subset X$ such that  
  $$\rmU\supset \rmU'=\{\gamma\in \rmC([a,b];X):
  \text{there exists }(t_0', t_1',\cdots,t_J')\in
  \cP_J:\gamma([t'_{j-1},t'_j])\subset U_j\},\quad
  {\gamma_0}\in \rmU';$$
  in particular, there exist $(t_0, t_1,
  \cdots, t_J)\in \cP_J$ such that 
  ${\gamma_0}([t_{j-1},t_j])\subset U_j$. 
  
  Selecting $f_j\in \Lip_b(X,\tau,\sfd)$ so that 
  $f_j\restr{X\setminus U_j}\equiv 1$,
  $f_j\restr{\gamma_0([t_{j-1},t_j])}\equiv 0$, as in the proof of Proposition
  \ref{prop:C-extended},
  we conclude that 
  $\bar F_{\cF}([\gamma_0])=F_\cF(\gamma_0)=0$,
  $\bar F_{\cF}\restr{\Arc(X,\sfd)\setminus
    \rmU}\ge 1$.
  
  \noindent
  \underline{\emph{Proof of property (X2).}}
  % Let us now eventually check that the distance $\sfd_\rmA$ satisfies
  % the approximation property (X2) of Definition \ref{def:luft1}. 
  By Lemma \ref{rem:monotonicity} we can find a directed set $I$ 
  and a monotone family $(\sfd_i)_{i\in I}$ of 
  $\tau$-continuous, bounded semidistances such that $\sfd=\sup_I
  \sfd_i$. We can therefore introduce the corresponding continuous and bounded
  semidistances
  $\sfd_{\rmC,i}$ on $\rmC([0,1];X)$ and $\sfd_{\rmA,i}$ on $\Arc(X)$ by
  \begin{equation}
    \label{eq:268}
    \sfd_{\rmC,i}(\gamma_1,\gamma_2):=
    \sup_{t\in [0,1]}\sfd_i(\gamma_1(t),\gamma_2(t)),\quad
    \sfd_{\rmA,i}([\gamma_1],[\gamma_2]):=\inf\Big\{\sfd_{\rmC,i}(\gamma_1',\gamma_2'):
    \gamma_j'\sim\gamma_j\Big\}.
  \end{equation}
  Let us first notice that for every $\gamma_1,\gamma_2\in
  \rmC([0,1];X)$ we have
  \begin{equation}
    \label{eq:269}
    \sfd_\rmC(\gamma_1,\gamma_2)=\lim_{i\in
      I}\sfd_{\rmC,i}(\gamma_1,\gamma_2)=
    \sup_{i\in I}\sfd_{\rmC,i}(\gamma_1,\gamma_2).
  \end{equation}
  By claim (b) of Corollary \ref{cor:tedious-arcs} the semidistances $\sfd_{\rmA,i}$ are
  $\tau_\rmA$-continuous.
  Let us now fix $\gamma_1,\gamma_2\in \rmC([0,1];(X,\sfd))$
  with $\sfd$-modulus of continuity
  $\omega_{\gamma_1},\omega_{\gamma_2}$ as in \eqref{eq:348}; by the
  reparametrization Theorem \ref{thm:reparam} we can find $\varrho_{1,i},\varrho_{2,i}\in
  \Sigma_2$ such that 
  \begin{displaymath}
     \sfd_{\rmA,i}([\gamma_1],[\gamma_2])=
     \sfd_{\rmC,i}(\gamma_1\circ\varrho_{1,i},\gamma_2\circ\varrho_{2,i})
     \quad\text{for every }i\in I.
  \end{displaymath}
  Since $\Sigma_2$ is compact, we can find a directed subset $J$ 
  and a monotone final map $h:J\to I$ such that the subnets
  $j\mapsto\varrho_{1,h(j)}$ and
  $j\mapsto\varrho_{2,h(j)}$ are convergent to elements
  $\varrho_1,\varrho_2\in \Sigma_2$.
  By \eqref{eq:269}, for every $\eps>0$ we may find $i_0\in I$ such
  that 
  \begin{equation}
    \label{eq:270}
    \sfd_{\rmA}([\gamma_1],[\gamma_2])\le 
    \sfd_{\rmC}({\gamma_1}\circ\varrho_1,{\gamma_2}\circ\varrho_2)\le 
    \sfd_{\rmC,i}({\gamma_1}\circ\varrho_1,{\gamma_2}\circ\varrho_2)+\eps\quad
    \text{for every }i\succeq i_0.
  \end{equation}
  On the other hand, there exists $j_0\in J$ such that 
  $h(j_0)\succeq i_0$ and for every $j\succeq j_0$ 
  \begin{displaymath}
    \sup_{t\in [0,1]}|\varrho_{1,h(j)}(t)-\varrho_1(t)|\le \eps,\quad
    \sup_{t\in [0,1]}|\varrho_{2,h(j)}(t)-\varrho_2(t)|\le \eps,
  \end{displaymath}
  so that for every $j\succeq j_0$
  \begin{equation}
    \label{eq:271}
    \sfd_{\rmC,h(j)}({\gamma_1}\circ\varrho_{1,h(j)},{\gamma_2}\circ\varrho_{2,h(j)})
    \ge
    \sfd_{\rmC,h(j)}({\gamma_1}\circ\varrho_{1},{\gamma_2}\circ\varrho_{2})
    -\omega_{\gamma_1}(\eps)-\omega_{\gamma_2}(\eps).
  \end{equation}
  Combining \eqref{eq:270} with \eqref{eq:271} we thus obtain
  \begin{equation}
    \label{eq:272}
    \sfd_{\rmA}([\gamma_1],[\gamma_2])\le 
    \sfd_{\rmA,h(j)}([\gamma_1],[\gamma_2])+\eps+\omega_{\gamma_1}(\eps)+\omega_{\gamma_2}(\eps)
  \end{equation}
  for every $j\succeq j_0$. Passing to the limit w.r.t.~$j\in J$ we get 
  \begin{align*}
    \sfd_{\rmA}([\gamma_1],[\gamma_2])&\le 
    \lim_{j\in J}\sfd_{\rmA,h(j)}([\gamma_1],[\gamma_2])
    +\eps+\omega_{\gamma_1}(\eps)+\omega_{\gamma_2}(\eps)\\
   & = \lim_{i\in I}\sfd_{\rmA,i}([\gamma_1],[\gamma_2])+\eps+\omega_{\gamma_1}(\eps)+\omega_{\gamma_2}(\eps),
  \end{align*}
  which yields $ \sfd_{\rmA}([\gamma_1],[\gamma_2])=
  \lim_{i\in I}\sfd_{\rmA,i}([\gamma_1],[\gamma_2])$ since $\eps>0$ is arbitrary.
\end{proof}
\index{Auxiliary topology}
\begin{corollary}
  \label{cor:auxiliary-arc}
    If $\tau'$ is an auxiliary topology for
  $(X,\tau,\sfd)$ according to Definition
  \ref{def:auxiliary}
  then $\tau_\rmA'$ is an auxiliary topology for
  $(\Arc(X,\sfd),\tau_\rmA,\sfd_\rmA)$.
\end{corollary}
\begin{proof}
  Properties (A1,2) of Definition \ref{def:auxiliary} are immediate.
  We now select an increasing sequence $\sfd_n$ of $\tau'$-continuous
  distances satisfying \eqref{eq:555}; arguing as in the previous
  proof
  we obtain that $\sfd_\rmA([\gamma_1],[\gamma_2])=
  \sup_{n\in\N}\sfd_{\rmA,n}([\gamma_1],[\gamma_2])$
  for every $\gamma_1,\gamma_2\in \rmC([0,1];(X,\sfd))$,
  which precisely yields (A3') for $\tau_\rmA'$.
\end{proof}

\subsection{Rectifiable arcs}
\label{subsec:rectifiable_arcs}

If $\delta$ is an extended semidistance in $X$, $\gamma:[a,b]\to X$ and $J\subset [a,b]$ we set
\begin{equation}
  \label{eq:208}
  \Var_\delta(\gamma;J):=\sup\Big\{\sum_{j=1}^N
  \delta(\gamma(t_j),\gamma(t_{j-1})): \{t_j\}_{j=0}^N\subset J,\quad
  t_0<t_1<\cdots<t_N\ \Big\}.
\end{equation}
$\BV([a,b];(X,\delta))$ will denote the space of maps
$\gamma:[a,b]\to X$
such that $\Var_\delta(\gamma;[a,b])<\infty$
(we will omit $\delta$ in the case of the distance $\sfd$).
We will set
\begin{equation}
\ell(\gamma):=\Var_\sfd(\gamma;[a,b]),\quad
V_\gamma(t):=\Var_\sfd(\gamma;[a,t])\quad t\in [a,b],\label{eq:531}
\end{equation}
 and
 \begin{equation}
   V_{\gamma,\ell}(t):=\ell(\gamma)^{-1}V_\gamma(t)\quad \text{whenever
$\ell(\gamma)>0$},\qquad
V_{\gamma,\ell}(t):=0\quad\text{if }
\ell(\gamma)=0;\label{eq:532}
\end{equation}
notice that if $\ell(\gamma)=0$ then $\gamma$ is constant.
\begin{lemma}
  \label{le:surprising-BVC}
  Let $(X,\tau,\sfd)$ be an extended metric-topological space.
  If $\gamma\in\rmC([a,b];(X,\tau))$ satisfies
  $\Var_\sfd(\gamma;[a,b])<\infty$ 
  then $\gamma\in \rmC([a,b];(X,\sfd))$,
  and its variation map $V_\gamma$ is continuous
  in $[a,b]$.
  We will set 
  $$\BVC([a,b];X):=\BV([a,b];(X,\sfd))\cap \rmC([a,b];(X,\tau))=
  \BV([a,b];(X,\sfd))\cap \rmC([a,b];(X,\sfd)).$$
\end{lemma}
\begin{proof}
  Thanks to the obvious estimate
  \begin{equation}
    \label{eq:248}
    \sfd(\gamma(r),\gamma(s))\le V_\gamma(s)-V_\gamma(r)\quad
    \text{for every }r,s\in [a,b], \ r\le s,
  \end{equation}
  it is easy to check $\gamma$ is $\sfd$-continuous at every
  continuity point of $V_\gamma$. 
  Let us fix $r,s,t\in (a,b]$ with $r<s<t$; passing to the limit in
  the inequality \eqref{eq:248} keeping $r$ fixed and letting
  $s\uparrow t$ we get by the $\tau$-continuity of $\gamma$,
  the lower semicontinuity of $\sfd$ and the monotonicity of $V_\gamma$
  \begin{displaymath}
    \sfd(\gamma(t),\gamma(r))\le V_\gamma(t-)-V_\gamma(r)\quad
    V_\gamma(t-):=\lim_{s\up t}V_\gamma(s).
  \end{displaymath}
  We thus obtain $\lim_{r\uparrow t}\sfd(\gamma(t),\gamma(r))=0$.
  A similar argument yields the right continuity of $\gamma$ and
  therefore the continuity of $V_\gamma$.
\end{proof}
Since the length functional $\gamma\mapsto \ell(\gamma)$ is lower
semicontinuous with respect to the compact-open topology of 
$\rmC([a,b];(X,\tau))$, the set 
$\BVC([a,b];X)$ is an $F_\sigma$ in $\rmC([a,b];(X,\tau))$.
We consider two other subsets:
\begin{equation}
  \label{eq:556}
  \BVC_c([a,b];X):=\big\{\gamma\in \BVC([a,b];X):
  V_\gamma(t)=\ell(\gamma)(t-a)\big\}
\end{equation}
whose curves have constant velocity, and
\begin{equation}
  \label{eq:557}
  \BVC_k([a,b];X):=\big\{\gamma\in
  \BVC([a,b];X):\ell(\gamma)>k\big\}\quad
  k\in [0,+\infty[\,,
\end{equation}
whose curves are non-constant
$\BVC_c([a,b];X)$ is a Borel subset of $\BVC([a,b];X)$ since it can be
equivalently characterized by the condition
\begin{equation}
  \label{eq:558}
  \gamma\in \BVC_c([a,b];X)\quad\Leftrightarrow\quad
  \Lip(\gamma;[a,b])\le \ell(\gamma),
\end{equation}
and the maps $\Lip$ and $\ell$ are lower semicontinuous.
$\BVC_0([a,b];X)$ is open in $\BVC_c([a,b];X)$.
\begin{corollary}
  \label{cor:just-to-fix}
  If $(X,\tau,\sfd)$ is Polish then
  $\BVC([a,b];X),\ \BVC_c([a,b];X),\ \BVC_k([a,b];X)$, $k\ge 0$, are
  Lusin spaces. If $(X,\tau,\sfd)$ admits an auxiliary topology
  $\tau'$
  (in particular if $(X,\tau)$ is Souslin) then they
  %$\BVC([a,b];X),\ \BVC_c([a,b];X),\ \BVC_0([a,b];X)$
  are $\FF(\rmC([a,b];X),\tau_\rmC)$ analytic sets.
\end{corollary}
\begin{proof}
  The first statement follows by the fact that Borel sets in a Polish
  space (in this case the space
  $(\rmC([a,b];(X,\tau),\tau_\rmC)$)
  are Lusin.

  The second claim
  is obvious for the $F_\sigma$ set $\BVC([a,b];X)$;
  in the case of 
  $\BVC_c([a,b];X)$ and $\BVC_k([a,b];X)$, it follows by the fact that they
  are Borel sets in
  the metrizable and separable space
  $(\rmC([a,b];(X,\tau'),\tau'_\rmC)$, thus are $\FF$-analytic
  (see (A3) in \S\,\ref{subsec:PLS}) for the coarser topology $\tau_\rmC'$,
  and thus $\FF$-analytic also with respect to $\tau_\rmC$.
\end{proof}

If $f\in \rmC(X)$ and $\gamma\in \BVC([a,b];X)$ then
the integral $\int_\gamma f$ is well defined by
Riemann-Stieltjes integration of $f\circ \gamma$ with respect to $\rmd
V_\gamma$; it can also be obtained as the limit of the Riemann sums
\begin{equation}
  \label{eq:209}
  \begin{aligned}
    \int_\gamma f=&\lim_{\tau(P)\down0} \sum_{j=1}^N
    f(\gamma(\xi_j))\sfd(\gamma(t_j),\gamma(t_{j-1})): \\P=&\{t_0=a\le
    \xi_1\le t_1\le\xi_2\le \cdots\le t_{N-1}\le \xi_N\le
    t_N=b\},\quad \tau(P):=\sup_j |t_j-t_{j-1}|
  \end{aligned}
\end{equation}
since in \eqref{eq:209} is equivalent to use
$\sfd(\gamma(t_j),\gamma(t_{j-1})) $ or
$V_\gamma(t_j)-V_\gamma(t_{j-1})$.

Notice that for every $\gamma\in\BVC([a,b];X)$
the map $V_\gamma:[a,b]\to [0,\ell(\gamma)]$ is continuous and surjective and
\begin{equation}
  \text{there exists a unique $\ell(\gamma)$-Lipschitz map
    $\Al_\gamma\in \BVC_c([0,1];X)$ such
  that $\gamma=\Al_\gamma\circ (V_{\gamma,\ell})$},
\label{eq:560}
\end{equation}
with $|\Al_\gamma'|(s)=\ell(\gamma)$ a.e.~and 
\begin{equation}
  \label{eq:210}
  \int_\gamma f=
  \int_0^{1} f(\Al_\gamma(s)) |\Al_\gamma'|(s)\,\d s
  =\ell(\gamma)\int_0^{1} f(\Al_\gamma(s))\,\d s.
\end{equation}
Denoting by $\vartheta_\gamma:[0,1]\to[a,b]$ the right-continuous pseudo
inverse of $V_{\gamma,\ell}$ (when $\ell(\gamma)>0$)
\begin{equation}
  \label{eq:257}
  \vartheta_\gamma(s):=\max\Big\{t\in
  [a,b]:V_{\gamma,\ell}(t)=s\Big\}\quad
  s\in [0,1],\quad
  \text{so that }V_{\gamma,\ell}(\vartheta_\gamma(s))=s\quad\text{in }[0,1],
\end{equation}
we have $\Al_\gamma=\gamma\circ \vartheta_\gamma$.
When $\ell(\gamma)=0$ we set $\vartheta_\gamma(s)\equiv b$,
and we still have $\Al_\gamma=\gamma\circ \vartheta_\gamma$.
We also notice that 
\begin{equation}
  \label{eq:232}
  \int_\gamma f=
  \int f\,\d\nu_\gamma\quad
  \text{where}\quad
  \nu_\gamma:=\ell(\gamma)(\Al_\gamma)_\sharp(\Leb 1\res{[0,1]});
\end{equation}
by \eqref{eq:232} it is possible to extend the integral to every
bounded or nonnegative Borel map $f:X
\to \R$.

  \begin{lemma}
    \label{le:trivial}
    Let $\gamma\in \BVC([0,1];X)$ and let $\vartheta:[0,1]\to[0,1]$ be
    an increasing map.
    \begin{enumerate}
    \item The map $\tilde\gamma:=\gamma\circ\vartheta$ belongs to
      $\BV([0,1];(X,\sfd))$ and
      \begin{equation}
        \label{eq:350}
        V_{\gamma\circ\vartheta,\ell}(t)\le
        V_{\gamma,\ell}(\vartheta(t))\quad\text{for every }t\in [0,1].
      \end{equation}
    \item If $\tilde\gamma\in \rmC([0,1];(X,\tau))$, $\ell(\tilde\gamma) =\ell(\gamma)$, 
      and $V_{\gamma,\ell}$ is strictly increasing, then
      $\vartheta\in \Sigma.$
    \item If $\vartheta\in \Sigma$ then $\tilde\gamma$
      still belongs to $\BVC([0,1];X)$ and 
      \begin{equation}
        \label{eq:284}
        V_{\tilde\gamma,\ell}=V_{\gamma\circ\vartheta,\ell}=V_{\gamma,\ell}\circ\vartheta
      \end{equation}
      and
% so that 
      \begin{equation}
        % \tilde\gamma=\gamma\circ\
        % \quad\Rightarrow\quad
        \ell(\tilde\gamma)=\ell(\gamma),\quad
\Al_{\tilde\gamma}=\Al_\gamma,\quad
\int_\gamma f=\int_{\tilde\gamma} f.\label{eq:285}
\end{equation}
\end{enumerate}
\end{lemma}
\begin{proof}
  Claims (a) and (c) follow easily by the definition \eqref{eq:208}, the
  characterization of $\Al_\gamma$ and \eqref{eq:210}.

  In order to check Claim (b), we choose a point
  $t\in (0,1)$ (the argument for the case $t=0$ or $t=1$ can be easily
  adapted)
  and we set
  $r_-=\lim_{s\uparrow t}\vartheta(s),$ $r_+=\lim_{s\downarrow
    t}\vartheta(s)$.
  The identity $\tilde\gamma=\gamma\circ\vartheta$ and the continuity
  of $\tilde\gamma$ yield
  $ \gamma(r_-)=\gamma(r_+)$ and
  \begin{align*}
    \ell(\tilde\gamma)&=\Var_\sfd(\tilde \gamma;[0,t])+
    \Var_\sfd(\tilde \gamma;[t,1])\le
    \Var_\sfd(\gamma;[0,r_-])+
    \Var_\sfd(\gamma;[r_+,1])\\&=
    \ell(\gamma)-\Var_\sfd(\gamma;[r_-,r_+])
    =\ell(\gamma)\big(1+V_{\gamma,\ell}(r_-)-V_{\gamma,\ell}(r_+)\big).
  \end{align*}
  We deduce that $V_{\gamma,\ell}(r_-)=V_{\gamma,\ell}(r_+)$ so that
  $r_-=r_+$. As similar argument shows that $\vartheta(i)=i$, $i=0,1$,
  so that $\vartheta$ is surjective.
\end{proof}
On $\BVC([0,1];X)$
we introduce the equivalence relation \eqref{eq:211}
% \begin{equation}
%   \label{eq:211}
%   \begin{gathered}
%     \gamma_1\sim \gamma_2\quad\text{if}\quad 
%     \Al_{\gamma_1}=\Al_{\gamma_2}.
%   \end{gathered}
% \end{equation}
% Notice that $\gamma\sim\gamma\circ\sigma$ for every $\sigma\in \Sigma$
and 
we will denote by $\RA(X,\sfd)$ (or simply $\RA(X)$)
the quotient space 
$\BVC([0,1];X)/\sim$ 
endowed with the quotient topology $\tau_\rmA$ induced by
$\rmC([0,1];(X,\tau))$ and with the extended distance
\begin{align}
  \label{eq:238}
  \sfd_\rmA([\gamma_1],[\gamma_2]):=
  {}&
      \inf\Big\{\sfd_\rmC(\gamma_1',\gamma_2'):
      \gamma_i'\sim\gamma_i\Big\}
      =\inf_{\varrho_i\in \Sigma} \sfd_{\rmC}({\gamma_1}\circ
      \varrho_1,{\gamma_2}\circ\varrho_2)
%  \label{eq:239}={}&\inf_{\varrho\in \Sigma'}\sfd_\rmC(\Al_{\gamma_1},\Al_{\gamma_2}\circ\varrho).
\end{align}
as in \eqref{eq:278}.
By Proposition \ref{prop:A-extended} the space 
$(\RA(X,\sfd),\tau_\rmA,\sfd_\rmA)$ is an extended metric-topological
space.
\begin{lemma}[Reparametrizations of rectifiable arcs]
  \label{le:reparam-arcs}
  Let $(X,\tau,\sfd)$ be an extended metric-topological space.
  We have:
  \begin{enumerate}
  \item If $\gamma\in \BVC([0,1];X)$, $\gamma'\in \rmC([0,1];X)$  and
    $\gamma'\sim\gamma$
    then $\gamma'\in \BVC([0,1];X)$.
  \item For every $\gamma,\gamma'\in \BVC([0,1];X)$ we have
    \begin{equation}
      \label{eq:283}
      \gamma\sim\gamma'\quad\Leftrightarrow\quad
      \Al_{\gamma}=\Al_{\gamma'},
    \end{equation}
    and all the curves $\gamma'$ equivalent to $\gamma$ can be
described as $\gamma'=\Al_\gamma\circ\sigma$ for some $\sigma\in
\Sigma$.
    \item
      For every $\gamma_i\in \BVC([0,1];X)$ the distance $\sfd_\rmA$
      satisfies
      {\em (\ref{eq:267}a,b,c)} and we have
\begin{subequations}
\label{eq:267bis}
  \begin{align}
    \label{eq:258bis}
    \sfd_\rmA(\gamma_1,\gamma_2)&=
                                      \inf_{\sigma\in
                                      \Sigma'}\sfd_\rmC(\Al_{\gamma_1},\Al_{\gamma_2}\circ\sigma)\\
    \label{eq:239bisbis}
                                &=\min_{\varrho_i\in
      \Sigma_2}\sfd_\rmC(\Al_{\gamma_1}\circ\varrho_1,\Al_{\gamma_2}\circ\varrho_2)
  \end{align}
\end{subequations}
\item
  The function $\ell$ and the evaluation maps $\sfe_0,\sfe_1$ are
  invariant w.r.t.~parametrizations, so that we will still denote by
  $\ell$
  and $\sfe_0,\sfe_1$ the corresponding quotient maps.
  $\ell:\Arc(X,\tau)\to [0,+\infty]$ is $\tau_\rmA$-lower
  semicontinuous
  and $\sfe_0,\sfe_1:\Arc(X,\tau)\to X$ are continuous.
  \item If
    $f:X\to[0,+\infty]$ is lower semicontinuous then the map
    $\gamma\mapsto \int_\gamma f$ only depends on $[\gamma]$ and it is
    lower semicontinuous w.r.t.~$\tau_\rmA$ in $\RA(X,\sfd)$.
\end{enumerate} 
\end{lemma}
\begin{proof}  
  \textbf{(a)} 
  Since $\gamma'\circ\sigma'=\gamma\circ\sigma$ for some
  $\sigma,\sigma'\in \Sigma$, we have
  $\ell(\gamma')=\ell(\gamma)<\infty$ by \eqref{eq:285}.
  \medskip

  \noindent
  \textbf{(b)} The right implication $\Rightarrow$ in \eqref{eq:283} follows by
  \eqref{eq:285}.
  In order to prove the converse implication it is not restrictive to
  suppose $\ell(\gamma)=\ell(\gamma')>0$; we observe that there
  exist $\sigma,\sigma'\in \Sigma$ so that 
  $\Al_\gamma\circ\sigma=\Al_{\gamma'}\circ\sigma'=\gamma''$ and therefore
  $\Al_{\gamma''}=\Al_\gamma\circ(\sigma\circ \vartheta_{\gamma''})$.
  Thanks to the second claim of Lemma \ref{le:trivial}
  we deduce that $\varrho:=\sigma\circ \vartheta_{\gamma''} \in \Sigma$
  is continuous and surjective. Recalling that for every $\gamma\in \BVC([0,1];X)$ 
  $V_{\Al_\gamma,\ell}(t)=t$ by construction, 
 \eqref{eq:284} yields
  \begin{displaymath}
    t=V_{\Al_{\gamma''},\ell}(t)=V_{\Al_\gamma\circ \varrho,\ell}(t)=
    V_{\Al_\gamma,\ell}(\varrho(t))=\varrho(t)\quad\text{for every }t\in [0,1].
  \end{displaymath}
  \textbf{(c)} is an immediate consequence of Theorem \ref{thm:reparam} for
  $\delta:=\sfd$.
  \medskip

  \noindent
  \textbf{(d)} The function $\ell$ is lower semicontinuous w.r.t.~the
  $\tau_\rmC$-topology
  being the supremum of lower semicontinuous functions by
  \eqref{eq:208}.
  Since $\ell(\gamma)$ is independent on the choice
  of a representative in $[\gamma]$, it is also lower
  semicontinuous w.r.t.~the $\tau_\rmA$ topology.
  A similar argument holds for the initial and final evaluation maps $\sfe_0,\sfe_1$.
  \medskip

  \noindent
  \textbf{(e)} If $f$ is continuous, we use the representation of the integral
  by Riemann sums
  \begin{equation}
    \label{eq:209bis}
    \begin{aligned}
      \int_\gamma f=&\sup\sum_{j=1}^N
      \Big(\inf_{t\in [t_{j-1},t_j]}f(\gamma(t))\Big)\sfd(\gamma(t_j),\gamma(t_{j-1})): t_0=0<
      t_1<\cdots<t_{N-1}<
      t_N=1
    \end{aligned}
  \end{equation}
  which exhibits $\int_\gamma f$ as the supremum of $\tau_\rmC$ lower
  semicontinuous functions. The invariance of the integral 
  w.r.t.~reparametrization yields the $\tau_\rmA$ lower
  semicontinuity.

  When $f$ is $\tau$-lower semicontinuous, we can represent it as the
  supremum
  of the (directed) set
  \begin{displaymath}
    f(x)=\sup_{g\in F} g(x),\quad
    F:=\Big\{g\in \rmC_b(X),\ 0\le g\le f\Big\}.
  \end{displaymath}
  Since $\nu_\gamma$ is a Radon measure, we have
  \begin{displaymath}
    \int_\gamma f=
    \int_X f\,\d\nu_\gamma=
    \sup_{g\in F}\int_X g\,\d\nu_\gamma=
    \sup_{g\in F}\int_\gamma g. \qedhere
  \end{displaymath}
  \end{proof}
  \begin{lemma}
    \label{le:just-to-fix2}
    $\RA(X,\sfd)$ is an $F_\sigma$-subset
    of $(\Arc(X,\tau),\tau_\rmA)$.\\
    If $(X,\tau)$ is a Polish space then
    $(\RA(X,\sfd),\tau_\rmA)$ is a Lusin space.
    If $(X,\tau,\sfd)$ admits an auxiliary topology $\tau'$
    (in particular if $(X,\tau)$ is Souslin)
    then  for every $k\ge0$ the (relatively) open subsets
    \begin{equation}
      \RA_k(X,\sfd):=
      \big\{\gamma\in \RA(X,\sfd):\ell(\gamma)>k\big\}
      \label{eq:RAk}
  \end{equation}
  are
    $\FF(\RA(X,\sfd))$-analytic set for the $\tau_\rmA'$
    and the $\tau_\rmA$-topology.
  \end{lemma}
  \begin{proof}
    Notice that $\RA(X,\sfd)$ can be equivalently identified with the
    $F_\sigma$-subset of $\Arc(X,\tau)$ and of $\Arc(X,\sfd)$ defined
    by $\{\gamma\in \Arc(X,\tau):\ell(\gamma)<\infty\}$ with the
    induced topology $\tau_\rmA$ and the extended distance $\sfd_\rmA$
    of $\Arc(X,\sfd)$.  From this point of view, $\RA(X,\sfd)$ is a
    $F_\sigma$ subset (i.e.~it is the countable union of closed sets
    and therefore it is also a Borel set) since
    $\RA(X,\sfd)=\cup_{k\in \N}\{\gamma\in
    \Arc(X,\tau):\ell(\gamma)\le k\}$ and the map
    $\ell:\Arc(X,\tau)\to[0,+\infty]$ is lower semicontinuous with
    respect to $\tau_\rmA$ thanks to (d) of the previous Lemma
    \ref{le:reparam-arcs}.
    Since $\ell$ is also $\tau_\rmA'$-l.s.c. and $\tau_\rmA'$ is metrizable,
    all the sets
    $\RA_k(X,\sfd)$ are $F_\sigma$ and thus $\FF$-analytic.

    Finally, if $(X,\tau)$ is Polish, Corollary \ref{cor:just-to-fix}
    shows
    that $(\BVC_c([0,1];X),\tau_\rmC)$ is a Lusin space. Lemma
    \ref{le:reparam-arcs}
    shows that the quotient map $\sfq$ is a continuous bijection
    of $(\BVC_c([0,1];X),\tau_\rmC)$ onto $(\RA(X,\sfd),\tau_\rmA)$,
    so that the latter is Lusin as well.
  \end{proof}

We conclude this section with a list of useful properties concerning
the compactness in $\RA(X,\sfd)$ and 
the continuity of the map $\gamma\mapsto\nu_\gamma$ 
defined by \eqref{eq:232}.
For every $t\in [0,1]$ we also introduce the
arc-length evaluation maps
\begin{equation}
\gsfe_t:\RA(X,\sfd)\to X,\quad
\gsfe_t:=\gsfe_t\circ \Al,\quad
\gsfe_t(\gamma)=\Al_\gamma(t)\quad\text{for every }\gamma\in \RA(X,\sfd).
\label{eq:431}  
\end{equation}
When $t=0,1$ we still keep the notation $\gamma_t$ for
the initial and final points
$\sfe_t(\gamma)=\hat \sfe_t(\gamma)$.
\begin{theorem}
  \label{thm:important-arcs}
  \begin{enumerate}
  \item If $\gamma_i$, $i\in I$, is a converging net in $\RA(X,\sfd)$
    with $\gamma=\lim_{i\in I}\gamma_i$ and
    $\lim_{i\in I}\ell(\gamma_i)=\ell(\gamma)$ then 
    \begin{equation}
      \label{eq:288}
      \lim_{i\in I}\Al_{\gamma_i}=\Al_\gamma\quad\text{in }\rmC([0,1];X),
      \quad
      \lim_{i\in I}\gsfe_t(\gamma_i)=\gsfe_t(\gamma)\quad\text{for
        every }t\in [0,1],
    \end{equation}
    and for every bounded and continuous function $f\in \rmC_b(X,\tau)$ we have
    \begin{equation}
      \label{eq:286}
      \lim_{i\in I}\int_{\gamma_i}f=\int_\gamma f.
    \end{equation}
    In particular, we have
    \begin{equation}
      \label{eq:200}
      \lim_{i\in I}\nu_{\gamma_i}=\nu_\gamma\quad\text{weakly in }\cMp(X).
    \end{equation}
    \item The map $\gamma\mapsto \nu_\gamma$ from $\RA(X,\sfd)$ to
      $\cMp(X)$ is
      universally Lusin measurable.
      % (i.e.~for every Radon probability measure
      % $\mu\in \cP(\RA(X,\sfd))$ there exists an increasing sequence of
      % compact sets $\cK_n\subset \RA(X,\sfd)$ such that 
      % $\nu\restr{K_n}$ is continuous and
      % $\lim_{n\to\infty}\mu(\RA(X,\sfd)\setminus \cK_n)=0$.
      \item If $i\mapsto\gamma_i$ converges to $\gamma$ 
        in $\RA(X)$, $\sup\ell(\gamma_i)<\infty$ and
        $\nu_{\gamma_i}\weakto \mu$ in $\cMp(X)$
        with $\mu(X)>0$, then
        $\supp(\mu)=\gamma([0,1])$.
      % \item 
    %       If $\gamma_i$, $i\in I$, is a converging net in $\RA(X)$
    % with $\gamma=\lim_{i\in I}\gamma_i$ and
    % $\lim_{i\in I}\ell(\gamma_i)=\ell(\gamma)$ then $\Al_{\gamma_i}\to \Al_\gamma$ in $\rmC([0,1];X)$.
        \item The map $[\gamma]\mapsto \Al_\gamma$  
           is
           universally Lusin measurable from $\RA(X,\sfd)$ to
      $\BVC_c([0,1];(X,\sfd))$
      endowed with the topology $\tau_\rmC$
      and it is also Borel if $\X$ has an auxiliary topology (in particular
      if $(X,\tau)$ is Souslin).
      For every $t\in [0,1]$ the maps $\gsfe_t:\RA(X,\sfd)\to X$
      are universally Lusin measurable
      (and Borel if $\X$ has an auxiliary topology).
    \item
      If $f\in \rmB_b(X)$ (or
$f:X\to[0,+\infty]$ Borel) 
the map $\gamma\mapsto \int_\gamma f$ is Borel.
In particular the family of measures
  $\{\nu_\gamma\}_{\gamma\in \RA(X)}$ is Borel.
    \item If $(X,\tau)$ is compact
      and $\Gamma\subset \RA(X,\sfd)$ satisfies
      $\sup_{\gamma\in \Gamma}\ell(\gamma)<+\infty$ then
      $\Gamma$ is relatively compact
      in $\RA(X,\sfd)$ w.r.t.~the
    $\tau_\rmA$ topology.
  \item If $(X,\sfd)$ is complete 
    and $\Gamma\subset \RA(X,\sfd)$ satisfies the following
    conditions:
    \begin{enumerate}[1.]
    \item
      $\sup_{\gamma\in \Gamma}\ell(\gamma)<+\infty$;
    \item there exists a $\tau$-compact set $K\subset X$ such that
      $\sfe(\gamma)\cap K\neq \emptyset$ for every $\gamma\in \Gamma$;
    \item $\{\nu_\gamma:\gamma\in \Gamma\}$ is equally tight, i.e.~for
      every $\eps>0$ there exists a $\tau$-compact set
      $K_\eps\subset X$ such that
      $\nu_\gamma(X\setminus K_\eps)\le \eps$ for every
      $\gamma\in \Gamma$,
    \end{enumerate}
    then $\Gamma$ is relatively compact in $\RA(X,\sfd)$ w.r.t.~the
    $\tau_\rmA$ topology.
  \end{enumerate}
\end{theorem}
\begin{proof}
  \textbf{(a)}   In order to prove \eqref{eq:288} 
  we consider the 
  compactification
  $(\hat X,\hat \tau,\hat \sfd)$ 
  given by Theorem \ref{thm:G-compactification}
  (here we can choose, e.g., $\AA=\Lipb(X,\tau,\sfd)$; the measure $\mm$ does not play any
  role).
  Clearly the imbedding $\iota:X\to \hat X$ extends to a corresponding
  embedding of $\RA(X,\sfd)$ in $\RA(\hat X,\hat\sfd)$, simply 
  by setting $\hat\gamma(t):=\iota\circ \gamma(t)$ 
  and considering the corresponding equivalence class. 
  We can apply Proposition \ref{prop:compactness} to 
  the net $i\mapsto \Al_{\hat \gamma_i}=\hat \Al_{\gamma_i}$ 
  and we find a limit curve $\hat\Al_
  *\in \Lip([0,1];(\hat X,\hat\sfd))$ 
  with respect to the topology $\hat\tau_\rmC$.
  Since the projection from $\rmC([0,1];\hat X)$ to $\Arc(\hat X)$ is
  continuous, we deduce that $[\Al_*]=\hat\gamma$ 
  so that $\hat \Al_*$ takes values in $\iota(X)$ and therefore can be written
  as
  $\iota\circ \Al_*$ for a curve $\Al_*\in \Lip([0,1];(X,\sfd))$ which is the limit
  of $\Al_{\gamma_i}$ in $\rmC([0,1];X)$. 
  Passing to the limit in the identities
  \begin{equation}
    \label{eq:287}
    \sfd(\Al_{\gamma_i}(r),\Al_{\gamma_i}(s))\le 
    \ell(\gamma_i)|r-s|\quad\text{we get}\quad
    \sfd(\Al_*(r),\Al_*(s))\le 
    \ell(\gamma)|r-s|.
  \end{equation}
  so that $\Al_*=\Al_\gamma$.

  Let us prove \eqref{eq:286}. We set $m:=\inf f$ and $M=\sup f$ and we observe that 
  $\ell(\gamma)=\int_\gamma \unit$ so that 
  the thesis follows by applying the lower semicontinuity property of
  Lemma \ref{le:reparam-arcs} (e),
  to the functions $f-m$ and $M-f$.
  \medskip

  \noindent
  \textbf{(b)} Let $\mu\in \cP(\RA(X))$; since the function $\ell$ 
  is lower semicontinuous in $X$, it is Lusin $\mu$-measurable and 
  there exists a sequence of compact
  sets
  $K_n\subset X$ with $\lim_{n\to\infty}\mu(X\setminus K_n)=0$ 
  such that the restriction of $\ell$ to $K_n$ is continuous.
  By the previous claim, the restriction of $\nu$ to $K_n$ is also
  continuous.
  \medskip

  \noindent
  \textbf{(c)} Let $K:=\gamma([0,1])$; if $y\not\in K$ then we can find
  a function $f\in \Lipb(X,\tau,\sfd)$ with values in
  $[0,1]$ such that
  $f\restr K\equiv 0$ and $f(y)=1$. If $U:=\{x\in X:f(x)> 1/2\}$ 
  then there exists $i_0\in I$ such that 
  $\gamma_i([0,1])\cap \overline U=\emptyset$ for $i\succeq i_0$. 
  It follows that $\nu_{\gamma_i}(U)=0$ and therefore 
  \begin{displaymath}
    \mu(U)\le \liminf_{i\in I}\nu_{\gamma_i}(U)=0.
  \end{displaymath}
  This shows that $\supp(\mu)\subset K$. If $K$ consists of an
  isolated point, the thesis then follows.
  On the other hand, if $K$ contains at least two points and
  $y\in K$ then for every open neighborhood $U$ of $y$
  $\nu_\gamma(U)>0$ and therefore $\mu(U)>0$.
  \medskip

  \noindent
  \textbf{(d)} The proof of universal measurability follows as in Claim b), by using the continuity
  property \eqref{eq:288}.
  
  Let us now suppose that
  $\X$ admits
  an auxiliary topology $\tau'$ (thus metrizable and separable) and
  let us prove that $\Al$ is Borel from $\RA(X)$ endowed with
  $\tau_\rmA'$ to $\rmC([0,1];X)$ endowed with $\tau_\rmC$ (this
  implies
  the same property for the stronger topology $\tau_\rmA$ on
  $\RA(X)$).
  We observe that the map $J:\gamma\to (\gamma,\ell(\gamma))$ is Borel
  from $(\RA(X),\tau_\rmA')$ to $(\RA(X)\times \R,\tau'_\rmA\times
  \tau_\R)$ since the latter topology has a countable base of open
  sets
  (thus the Borel $\sigma$-algebra coincides with the product of the
  Borel $\sigma$-algebra of the factors) and each component of $J$ is
  Borel. On the other hand, 
  $G:=\{(\gamma,r)\in \RA(X)\times \R: r=\ell(\gamma)\}$
  is Borel in $\RA(X)\times \R$ (with the product topology
  $\tau_\rmA\times \tau_\R$)
  \cite[Chapter II, Lemma 12]{Schwartz73}
  and 
  Claim (a) shows that
  the map $\tilde R:G\to \RA(X)$, $\tilde R(\gamma,r):=R_\gamma$ is
  continuous in $G$,
  so that $\Al=\tilde R\circ J$ is a Borel map.
  Finally, since $\hat\sfe_t=\sfe_t\circ \Al$, the maps $\hat\sfe_t$
  are Borel as well.
  \medskip

  \noindent
  \textbf{(e)}
  Let us consider the set
  $H\subset \rmB_b(X)$ of functions $f$ such that $\gamma\mapsto
  \int_\gamma f$ is Borel.
  $H$ is clearly a vector space and contains the set $C:=\{\nchi_U,\
  U\text{ open in $X$}\}$, since the map $\gamma\mapsto \int_\gamma
  \nchi_U$ is lower semicontinuous. Since $C$ is closed under
  multiplication, we can apply the criterium \cite[Chap.~I, Theorem
  21]{Dellacherie-Meyer78},
  which shows that $H=\rmB_b(X)$. A simple truncation argument 
  extends this property to arbitrary nonnegative Borel functions.
  \medskip

  \noindent
  \textbf{(f)} The image of $\Gamma_*:=\Al(\Gamma)$ through
  the arc-length reparametrization $\Al$
  is relatively compact in $\rmC([0,1];(X,\tau))$
  by Proposition \ref{prop:compactness}. Since $\Gamma$
  is the image of $\Gamma_*$ through the quotient map
  $\sfq:\rmC([0,1];X)$ to $\Arc(X)$, $\Gamma$ is relatively compact as well.
  \medskip

  \noindent  
  \textbf{(g)} Let us consider the compactification $(\hat X,\hat \tau,\hat \sfd)$ 
  as in Theorem \ref{thm:G-compactification} and claim (a),
  and let
  $[\gamma_i]$ be a net in $\Gamma$ with $\mu_i:=\nu_{\gamma_i}$.
  %Clearly the imbedding $\iota:X\to \hat X$ extends to a corresponding
  We also set $\hat \mu_i=\iota_\sharp\mu_i=\nu_{\hat \gamma_i},$
  $\hat\gamma_i=\iota\circ\gamma_i$.
  It is not restrictive to assume $\gamma_i=\Al_{\gamma_i}$ so that 
  $\gamma_i$ is uniformly Lipschitz.
  We can then apply Proposition \ref{prop:compactness} to 
  the net $\hat \gamma_i$ in $\rmC([0,1];\hat X)$ 
  and find a subnet $j\mapsto h(j)$ and a limit curve $\gamma_*\in
  \Lip([0,1];(\hat X,\hat\sfd))$ such that 
  $j\mapsto \hat\gamma_{h(j)}$ converges to $\gamma_*$ with respect to
  $\hat \tau_\rmC$. Since the total mass of $\mu_{i}=\ell(\gamma_i)$ remains
  bounded,
  we can also find a further subnet (still denoted by $h$) and a limit
  probability measure $\mu$ such that $\mu_{h(j)}\weakto \mu$.
  Since $\iota$ is continuous, we have $\hat \mu_{h(j)}\weakto \hat
  \mu=\iota_\sharp \mu$ with $m:=\mu(X)=\hat\mu(\hat X)$.

  If $m=0$ then $\ell(\gamma_*)=0$ so that
  $\gamma_*$ is constant and coincides with a point $\hat x\in \hat X $.
  Since the image of every curve $\gamma_i$
  intersects the compact set $K$ we deduce that $\hat x=\iota(x)$ for
  some
  $x\in K$, so that $\gamma_{h(j)}$ converges to the constant curve
  $\gamma$, 
  $\gamma(t)\equiv x$ w.r.t.~$\tau_\rmC$ and
  $[\gamma_{h(j)}]$ converges to $[\gamma]$ in $\Arc(X)$.
  
  If $m>0$, the uniform tightness condition shows that $\mu$ is concentrated on 
  $\cup_{n\in \N}K_{1/n}$ so that $\hat\mu(\hat X\setminus
  \iota(X))=0$.
  It follows that $\iota(X)$ is dense in $\supp(\hat\mu)=
  \gamma_*([0,1])$.
  Since $\gamma_*$ is Lipschitz and $\iota(X)$ is complete, and thus
  $\sfd$-closed,
  we conclude that $\gamma_*([0,1])\subset \iota(X)$ and therefore
  $\gamma_*=\iota\circ\gamma$ for a curve $\gamma\in \Lip([0,1];X)$. 
  We deduce that $j\mapsto \gamma_{h(j)}$ converges to $\gamma$
  w.r.t.~the compact-open topology $\tau_\rmC$ and therefore
  $\lim_{j\in J}[\gamma_{h(j)}]=[\gamma]$ w.r.t.~$\tau_\rmA$.
\end{proof}
\subsection{Notes}
\label{subsec:notes3}
\begin{notes}
  \Para{\ref{subsec:continuous_curves}} contains standard material
  on the compact-open topology (which is well adapted to deal with
  general
  topologies $\tau$ on $X$) and its
  natural role in lifting the metric-topological structure of
  $(X,\tau,\sfd)$ to the space 
  $(\rmC([a,b];X),\tau_\rmC,\sfd_\rmC)$.
  The compactness result of Proposition \ref{prop:compactness}
  combines compactness w.r.t.~$\tau$ and equicontinuity w.r.t.~$\sfd$,
  see also \cite[Prop.~3.3.1]{Ambrosio-Gigli-Savare08}.

  \Para{\ref{subsec:arcs}} devotes some effort to
  construct a natural notion of invariance by parametrizations
  for arbitrary continuous curves. Since we did not assume
  rectifiability, the existence of a canonical arc-length
  parametrization
  is not guaranteed and one has to deal with a more general notion
  where arbitrary increasing, continuous and surjective
  change of variable are
  allowed (see \cite{Paolini-Stepanov12} for a similar approach).
  Here the main properties are provided by Theorem \ref{thm:reparam}.
  The construction of an extended metric-topological
  setting is presented in Proposition \ref{prop:A-extended}:
  although very natural, it requires a detailed proof.
  Everything becomes much simpler in the case
  of Example \ref{ex:polish}.

  \Para{\ref{subsec:rectifiable_arcs}} combines
  the two previous sections to deal with
  continuous rectifiable arcs.
  The presentation here slightly differs from \cite{ADS15}.  
\end{notes}

\section{Length and conformal distances}
\label{sec:length-Finsler}
\subsection{The length property}
\label{subsec:length}
To every extended metric space $(X,\sfd)$
it is possible to associate the length distance
\index{Length distance}
\index{Length spaces}
\begin{equation}
  \label{eq:233}
  \sfd_\ell(x,y):=\inf\Big\{\ell(\gamma):\gamma\in \RA(X),\ 
  \gamma_0=x,\ \gamma_1=y\Big\};
\end{equation}
$(X,\sfd)$ is a length space if $\sfd=\sfd_\ell$.
$(X,\sfd)$ is a \emph{geodesic space}
\index{Geodesic spaces}
if for every $x,y\in X$ with $\sfd(x,y)<\infty$ there
exists an arc $\gamma\in \RA(X)$ connecting $x$ to $y$ with
$\ell(\gamma)=\sfd(x,y)$.

It is not difficult to check that the classes of rectifiable arcs for
$\sfd$ and for $\sfd_\ell$ coincide, as well as the corresponding
notion of length and integral.

When $(X,\sfd)$ is complete, it is possible to give an equivalent
characterization
of the length property in terms of 
the approximate mid-point property: 
every couple of points $x,y\in X$ with $\sfd(x,y)<\infty$ admits
approximate midpoints
\begin{equation}
  \label{eq:223}
  \forall \, \theta>\frac 12\ 
  \exists\, z_\theta\in X:\quad
  \sfd(x,z_\theta)\lor \sfd(z_\theta,y)\le \theta\sfd(x,y).
\end{equation}
By iterating the middle point construction, it is possible to show
that for every $x,y\in X$ with $\sfd(x,y)<\infty$ and for every
$D>\sfd(x,y)$ there exists a map
$\gamma:\D\to X$ defined on the set of dyadic points in $[0,1]$,
$\D:=\{k/2^n:n,k\in \N,\ 0\le k\le 2^n\}$, satisfying
\begin{equation}
  \label{eq:226}
  \sfd(\gamma(s),\gamma(t))\le D|t-s|%\sfd(x,y)
  \quad
  \text{for every }
  s,t\in \D.
\end{equation}
Thus, if $(X,\sfd)$ is complete the curve $\gamma$ admits a unique
extension to a curve $\tilde\gamma\in \BVC([0,1];X)$ with 
$\ell(\tilde\gamma)\le D$. Since $D>\sfd(x,y)$ is arbitrary, we
conclude that $\sfd_\ell=\sfd$.

Notice that if $(X,\sfd)$ satisfies the approximate mid-point property
then for every $x,y\in X$, $\eps>0,$ and $L>1$ there exists a sequence
$(x_n)_{n=0}^N\subset X$
such that 
\begin{equation}
  \label{eq:219}
  x_0=x,\ x_N=y,\quad \sup_{1\le n\le N}\sfd(x_{n-1},x_{n})\le \eps,\quad
  \sum_{n=1}^N\sfd(x_{n-1},x_n)\le L\sfd(x,y).
\end{equation}

\subsection{Conformal distances}
\label{subsec:Finsler}
More generally, 
let $g:X\to(0,\infty)$ be a continuous function satisfying
\begin{equation}
  \label{eq:311}
  m_g:=\inf_X g>0,\quad
  M_g:=\sup_X g<\infty.
\end{equation}
We can consider $g$ as a conformal metric density,
inducing the 
length distance 
\begin{equation}
  \label{eq:207}
  \sfd_g(x,y):=
  \inf\Big\{\int_\gamma g:\gamma\in \RA(X),\
  \gamma(0)=x,\ \gamma(1)=y\Big\}.
\end{equation}
It is clear that $\sfd_g$ is 
an extended distance and satisfies 
\begin{equation}
  \label{eq:310}
  m_g\,\sfd_\ell(x,y)\le \sfd_g(x,y)\le M_g\,\sfd_\ell(x,y)\quad
  \text{for every }x,y\in X.
\end{equation}
By construction, $\sfd_g$ is a length distance, i.e.~$(\sfd_g)_\ell=\sfd_g;$
when $g\equiv 1$ we clearly have $\sfd_g=\sfd_\ell$.

We can introduce different inner approximations of $ \sfd_g$.
The first one, $\sfd_g'$, arises by the
the following procedure:
first of all we set
\begin{equation}
  \label{eq:203}
  \beta(x,y):=(g(x)\lor g(y))\sfd(x,y),\quad
  \beta_i(x,y):=(g(x)\lor g(y))\sfd_i(x,y),
\end{equation}
where $(\sfd_i)_{i\in I}$ is a directed family of
$\tau$-continuous bounded
semidistances generating $\sfd$ by $\sfd(x,y)=\lim_{i\in
  I}\sfd_i(x,y)$
as in Lemma \ref{rem:monotonicity};
for every $\eps\in (0,+\infty]$ we first set 
\begin{equation}
  \label{eq:186}
  \begin{aligned}
    \sfd_{g,i,\eps}(x,y):=\inf\Big\{&\sum_{n=1}^N\beta_i(x_{n-1},x_n):
    N\in \N,\ (x_n)_{n=0}^N\in X,\\
    &x_0=x,\ x_N=y,\ \sfd_i(x_{n-1},x_n)<
    \eps\Big\}\land (M_g\,\sup \sfd_i).
  \end{aligned}
\end{equation}
It is not difficult to check that $\sfd_{g,i,\eps}$ is a bounded $\tau$-continuous
semidistance with
$\sfd_{g,i,\eps}(x,y)\le \beta_i(x,y)$ whenever $\sfd_i(x,y)<\eps$;
moreover
it is easy to check that if $0<\eps<\eps'$ and $i\prec j$ we have
$m_g\sfd_i\le \sfd_{g,i,\eps'}\le \sfd_{g,j,\eps}\le
M_g\sfd_{j,\ell}\le M_g\sfd_\ell$. 
We need a more localized estimate involving the sets
\begin{equation}
  \label{eq:319}
  D_i(x,y):=\Big\{
  \sfd_i(z,x)\lor\sfd_i(z,y)\le\sfd_i(x,y)\Big\},\quad
  D(x,y):=\Big\{
    \sfd(z,x)\lor\sfd(z,y)\le\sfd(x,y)\Big\},
\end{equation}
where $x,y\in X$. Notice that $D_i(x,y)$ and $D(x,y)$
are closed sets containing $x$ and $y$.
\begin{lemma}
  \label{le:localized}\ 
  \begin{enumerate}
  \item For every $x,y\in X$ we have
    \begin{equation}
      \label{eq:313}
      \sfd_{g,i,\eps}(x,y)\ge \sfd_i(x,y)\inf_{D_i(x,y)} g.
    \end{equation}
    \item For every $z\in X$, $i\in I$ and $\eps>0$ the map
      $h:x\mapsto \sfd_{g,i,\eps}(x,z)$ belongs to $\Lip_b(X,\tau,\sfd_i)$
    with
    \begin{equation}
      \label{eq:486}
      \lip_{\sfd_i}h\le g\quad\text{in }X.
    \end{equation}
  \item
    If moreover $(X,\tau)$ is compact, then the infimum of $g$ on
    $D_i(x,y)$ and $D(x,y)$ is attained and
    \begin{equation}
      \label{eq:320}
      \liminf_{i\in I}\min_{D_i(x,y)}g\ge 
      \min_{D(x,y)}g\quad\text{for every }x,y\in X.
    \end{equation}
  \end{enumerate}
\end{lemma}
\begin{proof}
  \textbf{(a)}
  Let $(x_n)_{n=0}^N$ be any sequence of points connecting
  $x$ to $y$ as in \eqref{eq:186}. If all the points $x_n$ belong to
  $D_i(x,y)$ then \eqref{eq:313} immediately follows by the inequality
  \begin{equation}
  \beta_i(x_{n-1},x_n)\ge \sfd_i(x_{n-1},x_n)(g(x_{n-1})\lor g(x_n))\ge
  \sfd_i(x_{n-1},x_n)\inf_{D_i(x,y)}g.\label{eq:485}
\end{equation}
  If
  not, there are indexes $n$ such that
  $\sfd_i(x_n,x)\lor \sfd_i(x_n,y)>\sfd_i(x,y)$.  Just to fix ideas,
  let us suppose that the set of indexes $n\in \{1,\cdots,N-1\}$ such
  that $\sfd_i(x_n,x)>\sfd_i(x,y)$ is not empty and let us call $\bar n$ its
  minimum, so that $x_n\in D_i(x,y)$ if $0\le n<\bar n$.  It follows
  that
  \begin{align*}
    \sum_{n=1}^N\beta_i(x_{n-1},x_n)
    &\ge 
      \sum_{n=1}^{\bar n}\beta_i(x_{n-1},x_n)
      \topref{eq:485}\ge \sum_{n=1}^{\bar n}
      \big(\inf_{D_i(x,y)}g\big)\sfd_i(x_{n-1},x_n)
    \\&\ge 
        \big(\inf_{D_i(x,y)}g\big)\sfd_i(x,x_{\bar n})
        \ge \big(\inf_{D_i(x,y)}g\big)\sfd_i(x,y).
  \end{align*}
  A similar argument holds if the set of indexes
  $n\in \{1,\cdots,N-1\}$ such that $\sfd_i(x_n,y)>\sfd_i(x,y)$ is not
  empty: in this case one can select the greatest index.
  \medskip

  \noindent
  \textbf{(b)}
  We first observe that 
  for every $z\in X$ and $\eps>0$
  the map $h:x\mapsto \sfd_{g,i,\eps}(z,x)$ 
  belongs to $\Lipb(X,\tau,\sfd_i)
  $ with Lipschitz constant bounded by $M_g\eps^{-2}$.
  In fact the triangle inequality yields
  \begin{displaymath}
    |\sfd_{g,i,\eps}(z,x)-\sfd_{g,i,\eps}(z,y)|\le \sfd_{g,i,\eps}(x,y)
  \end{displaymath}
  and
  \begin{displaymath}
    \sfd_{g,i,\eps}(x,y)\le
    \begin{cases}
      M_g \sfd_i(x,y)&\text{if }\sfd_i(x,y)<\eps;\\
      \frac{M_g \sup \sfd_i}\eps\sfd_i(x,y)&\text{if }\sfd_i(x,y)\ge \eps.
    \end{cases}
  \end{displaymath}
  On the other hand, for every $\bar x\in X$ the continuity of
  $\sfd_i$ and of $g$
  ensures that there exists a neighborhood $U\in \UU_{\bar x}$ such that 
  $\sfd_i(\bar x,y)<\eps/2$ and $g(y)\le g(\bar x)+\eps$ for every $y\in U$, so that
  \begin{displaymath}
    \sfd_{g,i,\eps}(x,y)\le \beta_i(x,y)\le \sfd_i(x,y)\sup_U g\le
    \sfd_i(x,y)(g(\bar x)+\eps)
    \quad
    \text{for every }x,y\in U
  \end{displaymath}
  and therefore
  $\Lip(h,U,\sfd_i)\le (g(\bar x)+\eps)$,
  $\lip_{\sfd_i}h(\bar x)\le g(\bar x)+\eps.$ Since $\eps>0$ is
  arbitrary we conclude.  
  \medskip

  \noindent
  \textbf{(c)}
  Concerning \eqref{eq:320},
  let $z_i\in D_i(x,y)$, $i\in I$, be a minimizer for $g$ in
  $D_i(x,y)$,
  whose existence follows by the compactness of $(X,\tau)$ (and
  therefore of $D_i(x,y)$) and the 
  continuity of $\sfd_i$.
  We can find a converging subnet $\alpha\mapsto i(\alpha)$, $\alpha\in
  A$, such that
  $z_{i(\alpha)}\to z$, $g(z_{i(\alpha)})\to g(z)=
  \liminf_{i\in I}g(z_i)$ and (recalling \eqref{eq:321})
  \begin{displaymath}
    \sfd(z,x)\lor\sfd(z,y)\le \liminf_{\alpha\in A}
    \sfd_{i(\alpha)}(z_{i(\alpha)},x)\lor\sfd_{i(\alpha)}(z_{i(\alpha)},y)
    \le 
    \liminf_{\alpha\in A}\sfd_{i(\alpha)}(x,y)=\sfd(x,y),
  \end{displaymath}
  so that $z\in D(x,y)$. It follows that 
  \begin{displaymath}
    \liminf_{i\in I}\min_{D_i(x,y)}g
    =\liminf_{i\in I}g(z_i)=g(z)\ge \min_{D(x,y)}g. \qedhere
  \end{displaymath} 
\end{proof}
We then define 
\begin{equation}
  \label{eq:188}
  \sfd_g'(x,y):=\lim_{\eps\down0,i\in I}\sfd_{g,i,\eps}(x,y)=
  \sup_{\eps>0,i\in I}\sfd_{g,i,\eps}(x,y).
\end{equation}
Different approximations of $\sfd_g$ are provided by the formula
  \begin{align}
    \notag
    \sfd_g''(x,y):={}\sup\Big\{&|f(x)-f(y)|:
                                 \exists\, i\in I\ \text{such that }
    \\&
    f\in
                                 \Lipb(X,\tau,\sfd_i)
                                 \text{ and } 
    \label{eq:234}
    \lip_{\sfd_i}f\le g \text{ in } X\Big\}\\\notag
    % \forall\, z\in X\,\exists\, U\in \UU_z:
    % \Lip(f,U,\sfd_i)\le \sup_U g\Big\},\\\notag
    \sfd_g'''(x,y):={}\sup\Big\{&|f(x)-f(y)|:
    \\&f\in
                                     \Lipb(X,\tau,\sfd)\text{ and }
    %\\&% \forall\, z\in X\,\exists\, U\in \UU_z:
    \lip_\sfd f\le g \text{ in }X\Big\}.
    %\Lip(f,U,\sfd)\le \sup_U g\Big\}.
    \label{eq:261}
  \end{align}
When $g\equiv 1$ we will also write $\sfd_\ell':=\sfd_1',$ $\sfd_\ell'':=\sfd_1''$, $\sfd_\ell''':=\sfd_1'''$.
In the next Lemma we collect a few results concerning these distances.
\newcommand{\gll}{g}
\begin{theorem}
  \label{thm:Finsler}
  \begin{enumerate}[(a)]
  \item If $(X,\tau,\sfd)$ is an extended metric-topological space,
    then also $(X,\tau,\sfd'_\gll)$, $(X,\tau,\sfd_\gll'')$
    and $(X,\tau,\sfd_\gll''')$ are extended
    metric-topological space and we have for every $x,y\in X$
    \begin{equation}
      \label{eq:224}
      \sfd_\gll'(x,y)\le       \sfd_\gll''(x,y)\le
      \sfd_\gll'''(x,y)\le \sfd_\gll(x,y).
    \end{equation}
  \item 
    \begin{equation}
      \label{eq:503}
      (\sfd_\gll')_\ell= (\sfd_\gll'')_\ell=
      (\sfd_\gll''')_\ell=\sfd_\gll.
    \end{equation}
  \item If $(X,\tau)$ is compact    
    %then $\sfd_\gll'$ satisfies the approximate mid-point property.
%  \item If $(X,\sfd)$ is complete and $(X,\sfd_\gll')$ 
 %   satisfies the approximate mid-point property
%    (in particular when 
%    $(X,\tau)$ is compact)
    then $\sfd_\gll'=\sfd_\gll''=\sfd_\gll'''=\sfd_\gll$.
    % If moreover $(X,\sfd)$ is complete, then all the above extended
    % semidistances coincide with $\sfd_\gll$.
    In particular, $(X,\tau,\sfd_\gll)$ is an extended
    metric-topological space and
    $(X,\sfd_\gll)$ is a geodesic space.
  % \item If $(X,\tau)$ is compact then 
  %   $\sfd_\gll=\sfd_\gll'=\sfd_\gll''=\sfd_\gll'''$ and $(X,\sfd_\gll)$ is a geodesic space.   
   % where $\Lip_{b,1,{\rm loc}}(X,\tau,\sfd)$ is the collection of
    % functions $f\in \Lip_{b}(X,\tau,\sfd)$
    % which are \emph{locally $1$-Lipschitz} in the sense that
    % \begin{equation}
    %   \label{eq:225}
    %   \text{for every $x\in X$ there exists $U\in \UU_x$ such that
    %   $\Lip(f,U,\sfd)\le 1$}.
    % \end{equation}
    % In particular the algebra $\AA:=\Lip_b(X,\tau,\sfd)$ is compatible
    % with the metric topological structure of $(X,\tau,\sfd_\gll)$.
  \end{enumerate}
\end{theorem}
\begin{proof}
  \textbf{(a)} The fact that we are dealing with extended metric-topological spaces
  is clear from the construction.
  
  The first inequality $\sfd_\gll'\le \sfd_\gll''$ in
  \eqref{eq:224} follows immediately by Lemma \ref{le:localized}(b).
  % we first observe that 
  % for every $z\in X$ and $\eps>0$
  % the maps $x\mapsto \sfd_{g,i,\eps}(z,x)$ 
  % belong to $\Lipb(X,\tau,\sfd_i)
  % $ with Lipschitz constant bounded by $M_g\eps^{-2}$.
  % In fact the triangle inequality yields
  % \begin{displaymath}
  %   |\sfd_{g,i,\eps}(z,x)-\sfd_{g,i,\eps}(z,y)|\le \sfd_{g,i,\eps}(x,y)
  % \end{displaymath}
  % and
  % \begin{displaymath}
  %   \sfd_{g,i,\eps}(x,y)\le
  %   \begin{cases}
  %     M_g \sfd_i(x,y)&\text{if }\sfd_i(x,y)<\eps;\\
  %     \frac{M_g \sup \sfd_i}\eps\sfd_i(x,y)&\text{if }\sfd_i(x,y)\ge \eps.
  %   \end{cases}
  % \end{displaymath}
  % On the other hand, for every $\bar x\in X$ the continuity of
  % $\sfd_i$ and of $g$
  % ensures that there exists a neighborhood $U\in \UU_{\bar x}$ such that 
  % $\sfd_i(\bar x,y)<\eps/2$ and $g(y)\le g(\bar x)+\eps$ for every $y\in U$, so that
  % \begin{displaymath}
  %   \sfd_{g,i,\eps}(x,y)\le \beta_i(x,y)\le \sfd_i(x,y)\sup_U g\le
  %   \sfd_i(x,y)(g(\bar x)+\eps)
  %   \quad
  %   \text{for every }x,y\in U
  % \end{displaymath}
  % and therefore
  % $\Lip(\sfd_{g,i,\eps}(z,\cdot),U,\sfd_i)\le (g(\bar x)+\eps)$.
  % It follows that the maps $x\mapsto \sfd_{g,i,\eps}(z,x)$ belongs to
  % the 
  % set defining the supremum in \eqref{eq:234}, so that $\sfd_\gll'\le \sfd_\gll''$.

  The inequality $\sfd_\gll''\le \sfd_\gll'''$ is obvious since the latter is
  obtained
  by taking the supremum on a bigger set.
  
  The last inequality $\sfd_\gll'''\le \sfd_\gll$ easily follows since 
  for every $\gamma\in \BVC([0,1];X)$ and every map $f\in
  \Lip(X,\tau,\sfd)$ with $\lip_\sfd f\le g$, the composition
  $\mathrm f:=f\circ \Al_\gamma$ is Lipschitz with
  \begin{equation}
    \label{eq:56}
    \big|\mathrm f'(t)\big|\le \ell(\gamma)\,\lip_\sfd
    f(\Al_\gamma(t))\le \ell(\gamma)g(\Al_\gamma(t))
    % \quad\text{for every }U\in \UU_{\Al_\gamma(t)}
    \quad
    \text{$\Leb 1$-a.e.~in $[0,1]$}.
  \end{equation}
  % Since $U$ is arbitrary and $g$ is continuous 
  % we deduce that 
  % \begin{equation}
  %   \label{eq:56}
  %   \big|\mathrm f'(t)\big|\le 
  %   \ell(\gamma)g(\Al_\gamma(t))\quad
  %   \text{$\Leb 1$-a.e.~in $[0,1]$};
  % \end{equation}
  An integration in the interval $[0,1]$ yields
  \begin{equation}
    \label{eq:201}
    \big|f(\gamma(1))-f(\gamma(0))\big|
    \le \int_\gamma g
  \end{equation}
  and a further minimization w.r.t.~all the curves $\gamma$
  connecting
  $x=\gamma(0)$ and $y=\gamma(1)$ 
  yields for every $f$ satisfying \eqref{eq:261}
  \begin{equation}
    \label{eq:312}
    |f(y)-f(x)|\le \sfd_g(x,y).
    \end{equation}
    Taking the supremum w.r.t.~$f$ we conclude.
    \medskip

    \noindent \textbf{(b)}
      Since $\sfd_g$ is a length distance, $(\sfd_g')_\ell\le \sfd_g$,
      so that it is sufficient to prove the converse inequality.
      Let $x,y\in X$ with $(\sfd_g')_\ell(x,y)<D$;
      we can find $\gamma\in \Lip([0,1];(X,\sfd))$ with $\gamma(0)=x$, 
      $\gamma(1)=y$, 
      and $\sfd_g'(\gamma(s),\gamma(t))\le  D|s-t|$
      for every $0\le s<t\le 1$. We want to show
  \begin{equation}
    \label{eq:314}
    I:=\int_\gamma g\le D.
  \end{equation}
  By \eqref{eq:310} 
  $\sfd(\gamma(s),\gamma(t))\le m_g^{-1}\,D|t-s|$ 
  so that $\gamma$ 
  is also $\sfd$-Lipschitz.
  % and therefore it is also 
  % $\tau$-continuous. 
  The map $g\circ\gamma$ is uniformly continuous as well.
  A standard compactness argument shows that 
  for every $\eps>0$ there exists $\delta>0$ such that 
  \begin{equation}
    \label{eq:316}
    \inf \Big\{g(z):z\in X,\ \sfd(z,x)\le \delta\Big\}\ge g(x)-\eps
    \quad\text{for every }x\in \gamma([0,1]).
  \end{equation}
  By \eqref{eq:209bis} for every $I_1<I$ and $\eps>0$ we can find
  a subdivision $(t_n)_{n=0}^N$ of $[0,1]$ such that
  \begin{equation}
    \label{eq:315}
    \sum_{n=1}^N \big(\inf_{[t_{n-1},t_n]}g\circ\Al_\gamma\big)
    \,\sfd(\Al_\gamma(t_{n-1}),\Al_\gamma(t_n))> I_1,\quad
    D|t_n-t_{n-1}|\le (\delta\land \eps) m_g,
  \end{equation}
  so that, in particular,
  $\sfd(\Al_\gamma(t_{n-1}),\Al_\gamma(t_n))\le \eps$.
  We set
  \begin{equation}
    \label{eq:317}
    m_n:=\min_{[t_{n-1},t_n]}g,\quad
    m_{i,n}:=\inf_{D_i(\Al_\gamma(t_{n-1}),\Al_\gamma(t_n))}g,
    \quad i\in I,\ 1\le n\le N.
  \end{equation}
  We can then find $i_0\in I$ such that 
  for every $i\succeq i_0$ 
  \begin{equation}
    \label{eq:315i}
    \sum_{n=1}^N  m_n
    \,\sfd_i(\Al_\gamma(t_{n-1}),\Al_\gamma(t_n))\ge I_1.
  \end{equation}
  Applying \eqref{eq:313}
  we obtain
  \begin{align*}
    I_1&\le 
         \sum_{n=1}^N m_{i,n}
    \,\sfd_i(\Al_\gamma(t_{n-1}),\Al_\gamma(t_n))+
         \sum_{n=1}^N (m_n-m_{i,n})
    \,\sfd_i(\Al_\gamma(t_{n-1}),\Al_\gamma(t_n))
         \\&\le \sum_{n=1}^N m_{i,n}
    \,\sfd_i(\Al_\gamma(t_{n-1}),\Al_\gamma(t_n))+
         m_g^{-1} D \sup_{1\le n\le N}
    \Big(m_n-m_{i,n}\Big)
         \\&\le
             \sum_{n=1}^N \sfd_{\gll,i,\eps}
             (\Al_\gamma(t_{n-1}),\Al_\gamma(t_n))+
             m_g^{-1}  D \sup_{1\le n\le N}
    \Big(m_n-m_{i,n}\Big)
    \\&\le D+m_g^{-1}  D \sup_{1\le n\le N}
    \Big(m_n-m_{i,n}\Big)
  \end{align*}
  % if $f\in \Lip_b(X,\tau,\sfd)$ the map
  % $t\mapsto f(\Al_\gamma(t))$ is uniformly Lipschitz in $\D$ so that it
  % admits a unique Lipschitz extension that we will denote by $\tilde
  % f$. Since $\tau$ is the initial topology generated by
  % $\Lip_b(X,\tau,\sfd)$ it follows that $\Al_\gamma$ is also continuous 
  % from $\D$ to $(X,\tau)$.
  We can now pass to the limit w.r.t.~$i\in I$, observing that by
  \eqref{eq:320}
  and \eqref{eq:316}
  \begin{equation}
    \label{eq:318}
    \liminf_{i\in I} m_{i,n}\ge 
    \min_{D(\Al_\gamma(t_{n-1}),\Al_\gamma(t_n))}g
    \ge g(\Al_\gamma(t_n))-\eps
    \ge m_n-\eps
  \end{equation}
  since $\sfd(\Al_\gamma(t_{n-1}),\Al_\gamma(t_n))\le \delta$.
  It follows that 
  \begin{displaymath}
    I_1\le D+m_g^{-1}D\eps;
  \end{displaymath}
  since $I_1<I$ and $\eps>0$ are arbitrary, we conclude.
  \medskip

  \noindent
    \textbf{(c)} We will show that
    \emph{$(X,\sfd_\gll')$ is a geodesic space.}
    Since $(X,\tau)$ is compact, it is sufficient to prove
    that $(X,\sfd_\gll')$ satisfies the
    approximate mid-point property. In particular
    $(\sfd_g')_\ell=\sfd_g'$ and
    the claim will follow by the previous point (b).
    
    Let us fix couple $x,y\in X$ with $2D:=\sfd_\gll'(x,y)<\infty$.
  By definition, for every $\eps>0$ we can find 
  $\eta_0>0$ with $(M_g\lor 1)\eta_0<\eps$  and $i_0\in I$
  such that $2D-\eps<\sfd_{g,i,\eta}(x,y)\le 2D$ for every 
  $0<\eta\le \eta_0$ and $i\succeq i_0$. 
  Therefore, we find points $(x_n)_{n=0}^N\in X$
  (depending on $i,\eta$) such that 
  $\sfd_i(x_{n-1},x_n)<\eta$ and 
  $2D-\eps<\sum_{n=1}^N\beta_i(x_{n-1},x_n)\le
  \sfd_{g,i,\eta}(x,y) +\eps\le 2D+\eps$.
  If $k_0=\max\{k\le N: \sum_{n=1}^{k}\beta_i(x_{n-1},x_n)\le D\}$ 
  and $z_{i,\eta}:=x_{k_0+1}$, we clearly have
  \begin{align*}
    \sfd_{\gll,i,\eta}(x,z_{i,\eta})
    &\le
    D+\beta_i(x_{k_0},x_{k_0+1})\le
      D+M_g\eta\le D+\eps,
    \\
    \sfd_{g,i,\eta}(y,z_{i,\eta})
    &\le
      \sum_{n=k_0+1}^N \beta_i(x_{n-1},x_n) =
      \sum_{n=1}^N \beta_i(x_{n-1},x_n)-
      \sum_{n=1}^{k_0+1} \beta_i(x_{n-1},x_n)
      \\&\le 2D+\eps-D\le D+\eps.
  \end{align*}
  Let now $(h,k):J\to \{i\in I:i\succeq i_0\}\times (0,\eta_0)$ 
  be a monotone subnet such that $z_{(h(j),k(j))}$ converges to $z\in X$. 
  Since $(i,\eta)\mapsto \sfd_{g,i,\eta}$ is monotone,
  for every $i\in I$ and $\eta>0$ we have
  \begin{align*}
    \sfd_{g,i,\eta}(x,z)&= \lim_{j\in
      J}\sfd_{g,i,\eta}(x,z_{h(j),k(j)})
    \le 
    \limsup_{j\in
      J}\sfd_{g,h(j),k(j)}(x,z_{h(j),k(j)})\le D+\eps,\\
    \sfd_{g,i,\eta}(y,z)&= \lim_{j\in
      J}\sfd_{g,i,\eta}(y,z_{h(j),k(j)})
    \le 
    \limsup_{j\in
      J}\sfd_{g,h(j),k(j)}(y,z_{h(j),k(j)})\le D+\eps.
  \end{align*}
  Taking the supremum w.r.t.~$i\in I$ and $\eta>0$ we eventually get
  \begin{displaymath}
    \sfd_\gll'(x,z)\le D+\eps,\quad
    \sfd_\gll'(y,z)\le D+\eps
  \end{displaymath}
  so that $z$ is an $\eps$-approximate midpoint between $x$ and $y$.
\end{proof}
\begin{remark}
  \label{rem:coincindence}
  Notice that when $\sfd_g$ is $\tau$-continuous, then also $\sfd$ is $\tau$-continuous
  and $\sfd_g=\sfd_g''=\sfd_g'''$. In this case
  $(X,\tau,\sfd_g)$ is an extended metric-topological space.  
\end{remark}
\subsection{Duality for Kantorovich-Rubinstein cost functionals
  induced by conformal distances}
\label{subsec:Kantorovich-duality}
\index{Kantorovich-Rubinstein distance and duality}
We apply Theorem \ref{thm:Finsler} to obtain
a useful dual representation for
Kantorovich-Rubinstein distances.
\begin{proposition}
  \label{prop:eventually}
  Let us suppose that
  the extended distances $\sfd_g$ and $\sfd_g'$
  defined by \eqref{eq:207} and \eqref{eq:188} coincide
  (in particular when $(X,\tau)$ is compact) and
  let $\sfK_{\sfd_g}$ be the Kantorovich functional induced by
  $\sfd_g$. Then for every $\mu_0,\mu_1\in \cMp(X)$ with the same mass
      \begin{align}
      \notag
        \sfK_{\sfd_g}(\mu_0,\mu_1)
      &=
        \sup\Big\{\int \phi_0\,\d\mu_0-
        \int \phi_1\,\d\mu_1: \phi_i\in \rmC_b(X,\tau),
        \\&\qquad\qquad
      \phi_0(x_0)-\phi_1(x_1)\le \sfd_g(x_0,x_1)\quad
      \text{for every }x_0,x_1\in X\Big\}
      \label{eq:490dg}
      \\&= \sup\Big\{\int \phi\,\d(\mu_0-\mu_1):
        \phi\in \Lip_{b}(X,\tau,\sfd),\ \lip_\sfd \phi\le g\Big\}.
      \label{eq:493dg}
    \end{align}
  \end{proposition}
  \begin{proof}
    \eqref{eq:490dg} is a particular case of
    \eqref{eq:490d} for the extended metric-topological space
    $(X,\tau,\sfd_g)$, thanks to Theorem \ref{thm:Finsler}(c).

    Concerning \eqref{eq:493dg}, we can first observe that the right
    hand side is dominated
    by $\sfK_{dg}$ since every function
    $\phi\in \Lip_{b}(X,\tau,\sfd)$ with $\lip_\sfd \phi\le g$
    belongs to $\Lip_{b,1}(X,\tau,\sfd_g)$
    thanks to \eqref{eq:224} and the very definition of $\sfd_g'''$
    given by \eqref{eq:261}.

    On the other hand, we know by \eqref{eq:224} that
    $\sfd_g=\sfd_g'$ so that the collection $(\sfd_{g,i,\eps})_{i\in
      I,\eps>0}$ is a direct set of continuous and bounded
    semidistances
    giving \eqref{eq:188}. We can then apply \eqref{eq:491} obtaining
    \begin{equation}
      \label{eq:495}
      \sfK_{d_g}(\mu_0,\mu_1)=\lim_{i\in I,\eps\downarrow0}
      \sfK_{d_{g,i,\eps}}(\mu_0,\mu_1),
    \end{equation}
    so that \eqref{eq:492} yields
    \begin{equation}
      \label{eq:494}
      \sfK_{d_g}(\mu_0,\mu_1)=\sup
      \Big\{\int \phi\,\d(\mu_0-\mu_1):
      \phi\in \Lip_{b,1}(X,\tau,\sfd_{g,i,\eps}),\ i\in I,\ \eps>0\Big\}.
    \end{equation}
    On the other hand, using \eqref{eq:486} one immediately sees that
    \begin{displaymath}
      \phi\in \Lip_{b,1}(X,\tau,\sfd_{g,i,\eps})
      \quad
      \Rightarrow\quad
      \lip_\sfd \phi\le \lip_{\sfd_i}\phi\le g. \qedhere
    \end{displaymath}
  \end{proof}
\subsection{Notes}
\label{subsec:notes4}
\begin{notes}
  \Para{\ref{subsec:length}} is standard,
  see e.g.~\cite{Burago-Burago-Ivanov01}

  \Para{\ref{subsec:Finsler}} will play a crucial role
  in the proof of the identification
  Theorem for metric Sobolev spaces of Section
  \ref{sec:Identification}.
  One of the main point here is that even in
  standard metric spaces the length-conformal construction
  may easily lead to extended distances.
  Theorem \ref{thm:Finsler} shows that
  at least in the compact case
  we can recover the length-conformal distances
  by inner approximation with $\tau$-continuous
  Lipschitz functions.
  Such kind of constructions and dual representations
  by local Lipschitz bounds are typical in the study of local
  properties of Dirichlet forms, see
  e.g.~\cite{Biroli-Mosco95,Sturm95,Stollmann10}.

  \Para{\ref{subsec:Kantorovich-duality}} contains
  the natural extension to the Kantorovich distance of the
  dual characterization $\sfd_g=\sfd_g'''$; it will play a crucial
  role
  in \S\,\ref{subsec:H=W}. Notice that
  if $\sfd_g$ is continuous Proposition \ref{prop:eventually}
  could be proved by a more direct argument based on the identity
  $\sfd_g=\sfd_g'''$ and on the classic representation \eqref{eq:493}
  for $\sfd_g$.
\end{notes}

\GGG
\part{The Cheeger energy}
In all this part we will always
refer to this basic setting:
\begin{Assumption}
  \label{ass:mainII}
  Let $\X=(X,\tau,\sfd,\mm)$ be an extended 
  metric-topological measure space
  as in \S\,\ref{subsec:extended-def}
  and let
  $\AA\subset \Lipb(X,\tau,\sfd)$ be a compatible algebra of functions,
  according to Definition \ref{def:A-compatibility}.
  For $f\in \Lip_b(X,\tau,\sfd)$
  $\lip f$ will always refer to the asymptotic Lipschitz constant $\lip_\sfd f$
  defined in \S\,\ref{subsec:aslip}.
  We fix an exponent $p\in (1,\infty)$
\end{Assumption}

\section{The strongest form of the Cheeger energy}
\label{sec:Cheeger}
\index{Cheeger energy (strong)}
Let us first define the notion of Cheeger energy $\CE_{p,\AA}$ associated to $(\X,\AA)$.
\begin{definition}[Cheeger energy]
  \label{def:Cheeger}
  For every $\kkappa\ge0$ and $p\in (1,\infty)$ we define the ``pre-Cheeger'' energy functionals 
\begin{equation}
  \label{eq:15}
  \pCE_{p}(f):=\int_X
  % \Big(
  \big(\lipd f(x)\big)^p
  % +\kkappa\,|f|^p\Big)
  \,\d\mm\quad \text{for every }
  f\in \Lip_b(X,\tau,\sfd).
\end{equation}
% with $\pCE_{p,\kkappa}(f)=+\infty$ if $f \in L^p(X)\setminus
% \AA$.
The $L^p$-lower semicontinuous envelope
of the restriction to $\AA$ of $\pCE_{p,\kkappa} $
is the ``strong'' Cheeger energy
\begin{equation}
    \label{eq:59}
    \CE_{p,\sAA}(f):=
    \inf \Big\{\liminf_{n\to \infty}\int_X \big(\lipd
    f_n\big)^p\,\d\mm:
    f_n\in \AA,\ 
    f_n\to f\text{ in }L^p(X,\mm)\Big\}.
  \end{equation}
  When $\AA=\Lip_b(X,\tau,\sfd)$ we will
  simply write $\CE_p(f)$.
% with
% \begin{equation}
%   \label{eq:16}
%   \CE_{p,\kkappa}(f):=\CE_p(f)+\kkappa \|f\|_{L^p(X,\mm)}^p.
% \end{equation}
  % $G\in L^p(X,\mm)$ is a relaxed
  % gradient of $f\in L^p(X,\mm)$ if there exists a sequence $f_n\in
  % \Lip_b(X)$ such that 
  % $f_n\to f$ strongly in $L^p(X,\mm)$, $\lip f_n\weakto \tilde G$ weakly in
  % $L^p(X,\mm)$ and $\tilde G\le G$ $\mm$-a.e.~in $X$.
  % 
\end{definition}
\begin{remark}[The notation $\CE$]
  \upshape 
  We used the symbol $\CE$ instead of 
  $\Ch$ (introduced by \cite{AGS14I})
  in the previous definition
  to stress three differences:
  \begin{itemize}
  \item the dependence on the strongest $\lip_d f$ instead of
    $|\rmD f|$,
  \item the restriction to functions in the algebra $\AA\subset \Lip_b(X,\tau,\sfd)$,
  \item 
    the factor $1$ instead of $1/p$ in front of the energy integral.
  \end{itemize}
\end{remark}
\noindent
It is not difficult to check that $\CE_p:L^p(X,\mm)\to [0,+\infty]$ is a convex, lower
semicontinuous and $p$-homogeneous functional;
it is the greatest $L^p$-lower semicontinuous functional ``dominated''
by $\pCE_{p}$
(extended to $+\infty$ whenever a function does not belong to $\AA$).
\index{Metric Sobolev space $\Sob^{1,p}(\X,\AA)$}
\begin{definition}
  \label{def:SobolevH}
  We denote by $\Sob^{1,p}(\X,\AA)$ the subset of $L^p(X,\mm)$ whose
  elements $f$ have finite Cheeger energy $\CE_{p,\sAA}(f)<\infty$:
  it is a Banach space with norm
  \begin{equation}
    \label{eq:60}
    \|f\|_{\Sob^{1,p}(\X,\AA)}:=\Big(\CE_{p,\sAA}(f)+\|f\|_{L^p(X,\mm)}^p\Big)^{1/p}.
  \end{equation}
  When $\AA=\Lipb(X,\tau,\sfd)$ we will simply write $\Sob^{1,p}(\X)$.
\end{definition}

\begin{remark}[$\Sob^{1,p}(\X,\AA)$ as Gagliardo completion \cite{Gagliardo60}]
  \upshape
  Recall that if $(A,\|\cdot\|_A)$ is a normed vector space
  continuously imbedded in a Banach space $(B,\|\cdot\|_B)$, 
  the \emph{Gagliardo completion} $A^{B,c}$ is the Banach
  space defined by
  \begin{equation}
    \label{eq:66}
    A^{B,c}:=\Big\{b\in B:\exists (a_n)_n\subset A, \
    \lim_{n\to\infty}\|a_n-b\|_B=0,\ \sup_n \|a_n\|_A<\infty\Big\}
  \end{equation}
  with norm
  \begin{equation}
    \label{eq:67}
    \|b\|_{A^{B,c}}:=\inf\Big\{\liminf_{n\to\infty} \|a_n\|_A:a_n\in
    A,\ \lim_{n\to\infty}\|a_n-b\|_B=0\Big\}.
  \end{equation}
  When $\supp(\mm)=X$,
  we can identify $\AA$ with a vector space $A$ with the norm induced by $\pCE_{p}$
  imbedded in $B:=L^p(X,\mm)$; it is immediate to check that 
  $\Sob^{1,p}(\X,\AA)$ coincides with the
  Gagliardo completion of $A$ in $B$.
\end{remark}
Notice that when $\mm$ has not full support, two different elements $f_1,f_2\in
\AA$ 
may give rise to the same equivalence class in $L^p(X,\mm)$. In this
case, $\CE_p$ can be equivalently defined starting
from the functional
\begin{equation}
  \label{eq:331}
  \widetilde{\pCE}_{p}(f):=
  \inf\Big\{\pCE_p(\tilde f):\tilde f\in \AA,\ \tilde f=f \text{ $\mm$-a.e.}\Big\},
\end{equation}
defined on the quotient space
\begin{equation}
  \label{eq:345}
  \tilde\AA:=\AA/\sim_{\mm},\quad
  f_1\sim_\mm f_2\text{ if }f_1=f_2\text{ $\mm$-a.e.}
\end{equation}
\subsection{Relaxed gradients and local representation of the Cheeger
  energy}
\label{subsec:relaxed}
The Cheeger energy $\CE_{p,\sAA}$ admits an integral representation in terms of the minimal
relaxed gradient $|\rmD f|_{\star,\sAA}$: we collect here a series of useful
results, which mainly follow 
by properties (\ref{eq:57}--e) of Lemma \ref{le:Locality}
arguing as in \cite[Lemma 4.3, 4.4, Prop.~4.8]{AGS14I}.
Here we have also to take into account the role of the algebra $\AA$.
\index{Relaxed gradients}
\index{Minimal relaxed gradients}
\begin{definition}[Relaxed gradients]\label{def:genuppergrad} 
We say that $G\in L^p(X,\mm)$ is a
$(p,\AA)$-relaxed gradient of $f\in L^p(X,\mm)$ if there exist
functions $f_n\in \AA$ such that:
\begin{itemize}
\item[(a)] $f_n\to f$ in $L^p(X,\mm)$ and $\lip f_n$
weakly converge to $\tilde{G}$ in $L^p(X,\mm)$;
\item[(b)] $\tilde{G}\leq G$ $\mm$-a.e. in $X$.
\end{itemize}
We say that $G$ is the minimal $(p,\AA)$-relaxed gradient of $f$ if its
$L^p(X,\mm)$ norm is minimal among relaxed gradients. We shall denote
by $\relgradA f$ the minimal relaxed gradient.
As usual, we omit the explicit dependence on $\AA$ when
$\AA=\Lip_b(X,\tau,\sfd)$.
\end{definition}
\noindent
Thanks to \eqref{eq:57} and the reflexivity of $L^p(X,\mm)$ 
one can easily check that 
\begin{equation}
  \label{eq:332}
  S:=\Big\{(f,G)\in L^p(X,\mm)\times L^p(X,\mm):\text{ $G$ is a $(p,\AA)$-relaxed gradient of $f$}\Big\}
\end{equation}
is convex.
Its closure follows by the following lemma, which also shows
that it is possible to obtain the minimal relaxed gradient as
\emph{strong} limit in $L^p$.
\begin{lemma}[Closure and strong approximation of the minimal
  relaxed gradient]\label{lem:strongchee}
  \
  \begin{enumerate}
  \item
    If $(f,G)\in S$ then there exist functions
    $f_n\in \AA$, $G_n\in \Lip_b(X,\tau,\sfd)$
    ($G_n\in \AA$ if $\AA$ is adapted)
    strongly converging to $f,\tilde G$ in $L^p(X,\mm)$ % and
      % $G_n\in L^p(X,\mm)$ strongly convergent to $\tilde G$ in
      % $L^p(X,\mm)$
    with $\lip  f_n\le G_n$ and $\tilde G\le G$.
    \item
    $S$ is weakly closed in $L^p(X,\mm)\times L^p(X,\mm)$.
  \item
    The collection of all the relaxed gradients of $f$ is closed in
    $L^p(X,\mm)$; if it is not empty, it contains a unique element of minimal norm
    and there exist functions
    $f_n\in \AA$, $G_n\in \Lip_b(X,\tau,\sfd)$
    ($G_n\in \AA$ if $\AA$ is adapted)
    such that $G_n\ge \lip f_n$ and
    \begin{equation}
      \label{eq:16bis}
      f_n\to f,\quad
      G_n\to \relgradA f,\quad
      \lip f_n\to \relgradA f\quad\text{strongly in }L^p(X,\mm).
    \end{equation}
  \end{enumerate}
\end{lemma}
\begin{proof} 
  \textbf{(a)} Since $G$ is a relaxed gradient, we can find 
  functions $h_i\in \AA$ such that $h_i\to f$ in $L^p(X,\mm)$
  and $\lip h_i$ weakly converges to $\tilde G\le G$ in
  $L^p(X,\mm)$.
  Since $\lip h_i$ are bounded, nonnegative and upper semicontinuous,
  by Corollary \ref{cor:adapted-is-better} we can find
  functions $g_i\in \AA(\X)$ ($g_i\in \AA$ if $\AA$ is adapted)
  such that
  $g_i\ge \lip h_i$ and $\|g_i-\lip h_i\|_{L^p(X,\mm)}\le 2^{-i}$ so
  that $\tilde G$ is also the weak limit of $g_i$ in $L^p(X,\mm)$.
  By Mazur's lemma we can find a sequence of convex
  combinations $G_n$ of $g_i$ (thus belonging to $\AA$), starting from an index
  $i(n)\to\infty$, strongly convergent to $\tilde G$ in $L^p(X,\mm)$;
  the corresponding convex combinations of $h_i$, that we shall denote
  by $f_n$, still belong to $\AA$, converge in $L^p(X,\mm)$ to $f$ and $\lip f_n$
  is bounded from above by $G_n$, thanks to \eqref{eq:57}.
  \medskip

\noindent  \textbf{(b)} Let us prove now the weak closure in $L^p(X,\mm)\times
L^p(X,\mm)$ of $S$.
Since $S$ is convex, it is sufficient to prove that $S$ is strongly
closed. If $S\ni(f^i,G^i)\to (f,G)$ strongly in $L^p(X,\mm)\times
L^p(X,\mm)$, we can find sequences of functions
$(f^i_n)_n\in \AA$ and of nonnegative functions $(G^i_n)_n\in
L^p(X,\mm)$ such that
$$
  f^i_n\stackrel{n\to\infty}\longrightarrow f^i, \quad
  G^i_n \stackrel{n\to\infty}\longrightarrow  \tilde G^i
  \text{ strongly in $L^p(X,\mm)$,}
  \quad
  \lip f^i_n\le G^i_n,\quad
  \tilde G^i\le G^i.
$$
Possibly extracting a suitable subsequence, we can assume that $\tilde
G^i\weakto \tilde G$ weakly in $L^p(X,\mm)$ with $\tilde G\le G$;
by a standard diagonal argument 
we can find an increasing sequence $i\mapsto n(i)$ such that
$f^i_{n(i)}\to f$, $G^i_{n(i)}\weakto \tilde G$ in $L^p(X,\mm)$ and $\lip  f^i_{n(i)}$ is bounded in
$L^p(X,\mm)$.
By the reflexivity of $L^p(X,\mm)$ we can also assume, possibly extracting a further subsequence,
that $\lip  f^i_{n(i)}\weakto H$.
It follows that $H\le \tilde G\le G$ so that $G$ is a relaxed gradient
for $f$.
\medskip

\noindent
\textbf{(c)} The closure of the collection of the relaxed gradients of $f$
follows by the previous claim. Since the $L^p$-norm is strictly
convex, if it is not empty it contains a unique element of minimal
norm.

Let us consider now the minimal relaxed gradient $G:=\relgradA f$
and let $f_n$, $G_n$ be sequences in $L^p(X,\mm)$ as in the first
part of the present Lemma. Since $\lip f_n$ is uniformly bounded
in $L^p(X,\mm)$ it is not restrictive to assume that it is weakly
convergent to some limit $H\in L^p(X,\mm)$ with $0\le H\le \tilde
G\le G$. This implies at once that  $H=\tilde G=G$ and $\lip f_n$ 
weakly converges to $\relgradA f$ (because any limit point in
the weak topology of $\lip f_n$ is a relaxed gradient with minimal
norm) and that the convergence is strong, since
\begin{displaymath}
  \limsup_{n\to\infty}\int_X (\lip  f_n)^p\,\d\mm\leq
  \limsup_{n\to\infty}\int_X G_n^p\,\d\mm=\int_X G^p\,\d\mm=
  \int_X H^p\,\d\mm. \qedhere
\end{displaymath}
\end{proof}
\index{Cheeger energy (strong)}
\begin{corollary}[Representation of the Cheeger energy]
  A function $f\in L^p(X,\mm)$ belongs to $\Sob^{1,p}(\X,\AA)$ if and only
  if
  it admits a $p$-relaxed gradients. In this case
  \begin{equation}
    \label{eq:333}
    \CE_{p,\sAA}(f)=\int_X \relgradA f^p\,\d\mm.
  \end{equation}
\end{corollary}
\begin{remark}[Dependence of $\relgrad f$ with respect to $p$]
\upshape
Notice that $\relgrad f$ may depend on $p$, even for Lipschitz
functions, see e.g.~\cite{AGS13}.
Since in these notes we will keep the exponent $p$ fixed,
we will omit to denote this dependence in the notation for $\relgrad f$.
\end{remark}
We want to show now that if $f\in \Sob^{1,p}(\X,\AA)$ 
satisfies the uniform bound $a\le f\le b$ $\mm$-a.e.~in $X$, then
there exists a sequence $f_n\in\AA$ satisfying \eqref{eq:16bis} and 
the same uniform bounds of $f$. This result is trivial if $\AA$ 
is the algebra of bounded Lipschitz functions, since 
truncations operates on $\AA$. In the general case we use the
approximated truncation polynomials of Corollary \ref{cor:trunc1}
.
\begin{corollary}
  \label{cor:poly-truncation}
  Let $f\in \Sob^{1,p}(\X,\AA)$ 
  be satisfying the uniform bounds $\alpha\le f\le \beta$ $\mm$-a.e.~in
  $X$.
  Then there exists a sequence $(f_n)\subset \AA$ satisfying
  \eqref{eq:16bis} such that $\alpha\le f_n\le
  \beta$ in $X$ for every $n\in \N$.
\end{corollary}
\begin{proof}
  Let $(f_n)_{n\in \N}$ be a sequence in $\AA$ as in \eqref{eq:16bis}. Since 
  functions in $\AA$ are bounded, we can find a sequence $c_n>0$ such
  that $f_n(X)\subset [-c_n,c_n]$.
  Let us choose a vanishing sequence $\eps_n\downarrow0$ and
  consider the truncation polynomials
  $P_n=P_{\eps_n}^{c_n,\alpha,\beta}$
  of Corollary \ref{cor:trunc1}
  corresponding to $c:=c_n$ and satisfying \eqref{eq:356}.
  % $\eps:=\eps_n$
  % $\phi(r):=a_0\lor r\land b_0$; by Lemma \ref{le:polynomial} we can
  % find 
  % polynomials $P_n$ such that 
  % \begin{displaymath}
  %   \sup_{[-M_n,M_n]}|P_n-\phi|\le 1/n,\quad
  %   P_n([-M_n,M_n])\subset [a_0,b_0],\quad
  %   |P_n'(r)|\le 1\quad\text{for every }r\in [-M_n,M_n].
  % \end{displaymath}
  We can then define the functions $\tilde f_n:=P_n\circ f_n$
  taking values in $[\alpha,\beta]$ 
  and $h_n:=-c_n\lor f\land c_n$ taking values in $[\alpha,\beta]\cap [-c_n,c_n]$; since
  $|P_n(r)-P_n(s)|\le |r-s|$ for every $r,s\in [-c_n,c_n]$  we have
  as $n\to \infty$
  \begin{align*}
    \|\tilde f_n-f\|_{L^p}
    &\le 
    \|P_n\circ f_n-P_n\circ h_n\|_{L^p}+
      \|P_n\circ h_n-h_n\|_{L^p}+
      \|h_n-f\|_{L^p}
    \\&
    \le 
      \|f_n- h_n\|_{L^p}+\mm(X)^{1/p}\eps_n+
      \|h_n-f\|_{L^p}
      \\&\le \|f_n- f\|_{L^p}+\mm(X)^{1/p}\eps_n+
    2\|h_n-f\|_{L^p}\to0\quad\text{as }n\uparrow\infty.
  \end{align*}
  On the other hand \eqref{eq:309} yields
  $\lip \tilde f_n= |P_n'\circ f|\lip f_n\le \lip f_n$
  so that 
  \begin{displaymath}
    \limsup_{n\to\infty}\int_X |\lip {\tilde f_n}|^p\,\d\mm\le 
    \limsup_{n\to\infty}\int_X |\lip f_n|^p\,\d\mm=\int_X \relgradA f^p\,\d\mm.
  \end{displaymath}
  Since $\relgradA f$ is the minimal $(p,\AA)$-relaxed gradient, we also have
  \begin{displaymath}
    \liminf_{n\to\infty}\int_X |\lip {\tilde f_n}|^p\,\d\mm\ge 
    \int_X \relgradA f^p\,\d\mm,
  \end{displaymath}
  so that the sequence $\tilde f_n$ satisfies the properties stated by
  the Lemma.
\end{proof}
\begin{corollary}[Lebnitz rule]
  For every $f,g\in \Sob^{1,p}(\X,\AA)\cap L^\infty(X,\mm)$ we have
  $fg\in \Sob^{1,p}(\X,\AA)$ and
  \begin{equation}
    \label{eq:weak-leibn}
    \relgradA {(fg)}\le |f|\,\relgradA g+|g|\,\relgradA f.
  \end{equation}
\end{corollary}
\begin{proof}
  It is sufficient to approximate $f,g$ by two uniformly bounded sequences 
  $f_n,g_n\in \AA$ 
  thanks to Corollary \ref{cor:poly-truncation}
  and then pass to the limit in \eqref{eq:306}.
\end{proof}
Let us now consider the locality property of the minimal $p$-relaxed
gradient, by adapting the proof of \cite{AGS14I} to the case of an
arbitrary
algebra $\AA$.
\begin{lemma}[Locality]\label{lem:locality}
  Let $G_1,\,G_2$ be $(p,\AA)$-relaxed gradients of $f$. Then $\min\{G_1,G_2\}$ and
$\nchi_BG_1+\nchi_{X\setminus B}G_2$,  $B\in\BorelSets{X}$, are
relaxed gradients of $f$ as well. In particular, for any
$(p,\AA)$-relaxed gradient
$G$ of $f$ it holds
\begin{equation}\label{eq:facileee}
  \relgradA  f\leq G\qquad\text{$\mm$-a.e. in $X$.}
\end{equation}
\end{lemma}
\begin{proof} 
It is sufficient to prove that if $B_i\in\BorelSets{X}$ with $B_1\cap
B_2=\emptyset$ and $B_1\cup B_2=X$ then
$\nchi_{B_1}G_1+\nchi_{B_2}G_2$ is a relaxed gradient of $f$.
If $A\in \BorelSets X$ given by Definition \ref{def:A-compatibility},
we can replace $B_2$ with $\tilde B_2:=
B_2\cap A$ (and $B_1$ by $\tilde B_1:=X\setminus B_2$) and
assume that $B_2\subset A$; moreover, by
approximation, taking into account the closure of the class of
relaxed gradients and the inner regularity of $\mm$, we can assume with no loss of generality that
$B_2$ is a compact set (and, in particular, $B_1$ is open).
We can also approximate $B_1$ by an increasing sequence of compact sets
$B_{n,1}\subset (B_1\cap A)$ such that $\mm(B_1\setminus B_{n,1})\to 0$.

Let us fix an integer $n$ and consider the compact set
$K_n:=B_{n,1}\cup B_2\subset A$;
since $\AA$ contains the constants and separates the points of $K_n$,
the restriction of $\AA$ to $K_n$ is uniformly dense in $\rmC(K_n)$ by
Stone-Weierstrass Theorem. Being $B_{n,1}$ and $B_2$ compact and
disjoint, the function
$$\nchi_n(x):=
\begin{cases}
  1&\text{if }x\in B_{n,1}\\
  0&\text{if }x\in B_2
\end{cases}
$$
belongs to $\rmC(K_n)$ 
% generates the topology of $X$, for every $y\in B_n$ we can
% find $f^y\in \AA$
% such that $f^y\restr{X\setminus B}\equiv 0$ and $f^y(y)=1$.
% 
so that for every $\eps>0$ we can find $\tilde\nchi_{n,\eps}\in \AA$
such that $\sup_{K_{n,1}}|\tilde\nchi_{n,\eps}-\nchi_n|\le
\eps/2$.
If we compose $\tilde\nchi_{n,\eps}$ with the truncation polynomial
$P=P_{\eps/2}^{c,0,1}$ of Corollary \ref{cor:trunc1}
corresponding 
$c:=1+\sup |\nchi_{n,\eps}|$, we obtain the function
$\nchi_{n,\eps}:=P\circ\tilde\nchi_{n,\eps}$  taking values in $[0,1]$
and satisfying
\begin{displaymath}
  \sup_{K_n}|\nchi_{n,\eps}-P\circ \nchi_n|\le \eps/2, \quad
  \sup_{K_n}|P\circ\nchi_n-\nchi_n|\le \eps/2
\end{displaymath}
since $0\lor \nchi_n\land 1=\nchi_n$ on $K_n$.
We deduce that 
% \begin{equation}
%   \label{eq:342}
%   0\le f^y_\eps\le 1,\quad 
%   0\le f^y_\eps\le \eps\text{ on }X\setminus B,\quad
%   f^y_\eps(y)> 1-\eps/2.
% \end{equation}
% We can cover $B_n$ with the open sets $U^y_\eps:=\{x\in X:
% f^y_\eps(x)>1-\eps/2\}$; since $B_n$ is compact, we can extract a finite 
% covering associated, say, to the functions $f^m_\eps:=f^{y_m}_\eps$,
% $1\le m\le M$, so that
% for every $x\in B_n$ there exists an integer $m$ between $1$ and $M$
% such that
% $f^m_\eps(x)>1-\eps/2$. It follows that $\max_{1\le m\le M}f^m_\eps$ 
% is $\le \eps$ on $X\setminus B$ and $> 1-\eps/2$ on $B_n$.
% We can then apply Lemma \ref{le:supremum} and find $\nchi_{n,\eps}\in \AA$ such 
\begin{equation}
  \label{eq:343}
  0\le \nchi_{n,\eps}\le 1,\quad
  0\le \nchi_{n,\eps}\le \eps\text{ on }B_2,\quad
  1-\eps\le \nchi_{n,\eps}\le 1\text{ on }B_{n,1}.
\end{equation}
Let now $h_{k,i}\in \AA$, $i=1,\,2$, 
functions converging to $f$ in $L^p$ as $k\to\infty$
with $\lip h_{k,i}$ weakly converging to $ \tilde G_i\le G_i$,
and set $f_{k,n,\eps}:=\nchi_{n,\eps}
h_{k,1}+(1-\nchi_{n,\eps})h_{k,2}\in \AA$.
Passing first to the limit as $k\up+\infty$, since $f_{k,n,\eps}\to f$,
\eqref{eq:307} immediately gives that $G_{n,\eps}:=\nchi_{n,\eps}
G_1+(1-\nchi_{n,\eps})G_2
 \ge \nchi_{n,\eps} \tilde G_1+(1-\nchi_{n,\eps})\tilde G_2$ is a
 relaxed gradient of $f$.
 
We can now select a vanishing sequence $(\eps_j)_{j\in \N}$ 
and we pass to the limit as $j\up+\infty$, obtaining (possibly extracting a further
subsequence) a limit function $\nchi_n$ taking values in $[0,1]$ 
such that 
$G_{n}:=\nchi_{n}
G_1+(1-\nchi_{n})G_2
 \ge \nchi_{n} \tilde G_1+(1-\nchi_{n})\tilde G_2$ is a
 relaxed gradient and $\nchi_n\restr{B_2}=0$,
 $\nchi_n\restr{B_{n,1}}=1$.
We can finally pass to the limit as $n\to\infty$, observing that 
$\nchi_n$ converges pointwise $\mm$-a.e.~to the characteristic
function of $B$.

For the second part of the statement we argue by contradiction: let
$G$ be a relaxed gradient of $f$ and assume that {there exists a Borel
set $B$ with $\mm(B)>0$ on which $G<\relgradA f$}. Consider the
relaxed gradient $G\nchi_B+\relgradA f\nchi_{X\setminus B}$: its $L^p$
norm is strictly less than the $L^p$ norm of $\relgradA f$, which is
a contradiction.
\end{proof}

\begin{theorem}
  \label{le:useful}
  For every $f,g\in \Sob^{1,p}(\X,\AA)$ we have
  \begin{enumerate}
   \item (Pointwise sublinearity) For %every $\alpha,\beta\ge0$ 
      $|\rmD (\alpha f+\beta g)|_{\star,\sAA}\le \alpha |\rmD f|_{\star,\sAA}+\beta
      |\rmD g|_{\star,\sAA}$.
   \item (Locality) For any Borel set $N\subset \R$ with $\Leb 1(N)=0$ we have
      \begin{equation}
        \label{eq:62}
        |\rmD f|_{\star,\sAA}=0\quad\text{$\mm$-a.e.~on }f^{-1}(N).
      \end{equation}
    In particular for every constant $c\in \R$
    \begin{equation}
    \text{$|\rmD f|_{\star,\sAA}=|\rmD g|_{\star,\sAA}$ $\mm$-a.e.~on
    $\{f-g=c\}$}.
  \label{eq:63}
  \end{equation}
\item (Chain rule) If $\phi\in \Lip(\R)$ then 
  $\phi\circ f\in \Sob^{1,p}(\X,\AA)$ with
  \begin{equation}
    \label{eq:64}
    |\rmD (\phi\circ f)|_{\star,\sAA}\le |\phi'(f)|\,|\rmD f|_{\star,\sAA}.
  \end{equation}
  Equality holds in \eqref{eq:64} if $\phi$ is monotone or $\rmC^1$.
\item (Normal contractions) If $\phi:\R\to \R$ is a nondecreasing contraction 
  and $\tilde f=f+\phi(g-f)$, $\tilde g=g+\phi(f-g)$ then
  \begin{equation}
    \label{eq:65}
    |\rmD \tilde f|_{\star,\sAA}^p+
    |\rmD \tilde g|_{\star,\sAA}^p\le 
    |\rmD  f|_{\star,\sAA}^p+
    |\rmD g|_{\star,\sAA}^p.
  \end{equation}
\end{enumerate}
\end{theorem}
\begin{proof}
\textbf{(a)} follows immediately by the convexity of the set $S$
defined by \eqref{eq:332} and \eqref{eq:facileee}.
\medskip

\noindent \textbf{(b)}
We first claim that for $\phi:\R\to\R$ continuously
differentiable whose  derivative $\phi'$ is Lipschitz on the image of $f$ it holds
\begin{equation}
\label{eq:chainc1} \relgradA {\phi(f)}\leq |\phi'\circ f|\relgradA
f,\qquad\text{$\mm$-a.e. in $X$},
\end{equation}
for any $f\in \Sob^{1,p}(\X,\AA)$. 
\eqref{eq:chainc1} easily follows by approximation from \eqref{eq:309}
whenever $f$ is bounded and $\phi$ is a polynomial: it is sufficient
to apply Corollary \ref{cor:poly-truncation}.

Still assuming the boundedness of $f$, arbitrary $\rmC^1$ functions
$\phi$ can be approximated by a sequence of polynomials
$P_n$
with respect to the $\rmC^1$-norm induced by a compact interval
containing $f(X)$.
Thanks to the weak closure of \eqref{eq:332} we can pass to the limit
in \eqref{eq:chainc1} written for $P_n$ and obtain the same bound for
$\phi$. In particular, for every $f\in \AA$ we get
\begin{equation}
  \label{eq:344}
  \relgradA {(\phi\circ f)}\le |\phi'\circ f|\relgradA f\le 
  |\phi'\circ f|\lip f\quad\text{$\mm$-a.e.~in $X$}.
\end{equation}
If now $\phi\in \rmC^1(\R)\cap \Lip(\R)$ and $f\in \Sob^{1,p}(\X,\AA)$ we can
use the approximation \eqref{eq:16bis} and \eqref{eq:344} to obtain
a sequence $f_n\in \AA$ such that
\begin{align*}
  \phi\circ f_n&\to \phi\circ f&&\text{in }L^p(X,\mm),\\
  \relgradA {(\phi\circ f_n)}&\weakto G&&\text{in }L^p(X,\mm),\\
  |\phi'\circ f_n|\relgradA {f_n}&\to |\phi'\circ f|\relgradA
                                            f&&\text{in }L^p(X,\mm),
\end{align*}
so that $\relgradA {(\phi\circ f)}\le G\le 
|\phi'\circ f|\relgradA f$.

Now, assume that $N$ is compact. In this case, let $A_n\subset\R$ be
open sets such that $A_n\downarrow N$ {\nc and $\Leb{1}(A_1)<\infty$}. Also, let $\psi_n:\R\to[0,1]$
be a continuous function satisfying $\nchi_{N}\leq\psi_n\leq
\nchi_{A_n}$, and define $\phi_n:\R\to\R$ by
\[
\left\{\begin{array}{ll}
\phi_n(0)&=0,\\
\phi_n'(z)&=1-\psi_n(z).
\end{array}
\right.
\]
The sequence $(\phi_n)$ uniformly converges to the identity map, and
each $\phi_n$ is 1-Lipschitz and $C^1$. Therefore $\phi_n\circ f$
converge to $f$ in $L^2$. Taking into account that $\phi_{n}'=0$ on $N$
and \eqref{eq:chainc1} we deduce
\[
\begin{split}
\int_X \relgradA f^p\,\d\mm
&\leq\liminf_{n\to\infty}\int_X\relgradA {\phi_n(f)}^p\,\d\mm
\leq\liminf_{n\to\infty}\int_X|\phi'_n\circ f|^p\relgradA f^p\,\d\mm\\
&=\liminf_{n\to\infty}\int_{X\setminus f^{-1}(N)}|\phi'_n\circ
f|^p\relgradA f^p\,\d\mm\leq\int_{X\setminus f^{-1}(N)}\relgradA
f^p\,\d\mm.
\end{split}
\]
It remains to deal with the case when $N$ is not compact. In this
case we consider the finite measure $\mu:=f_\sharp\mm$.
Then there exists
an increasing sequence $(K_n)$ of compact subsets of $N$ such that
$\mu(K_n)\uparrow\mu(N)$. By the result for the compact case we know
that $\relgradA f=0$ $\mm$-a.e.\ on $\cup_nf^{-1}(K_n)$, and by
definition
 of push forward we know that $\mm(f^{-1}(N\setminus
\cup_nK_n))=0$.

\eqref{eq:63} then follows if $g$ is identically 0. In the
general case we notice that $\relgradA {(f-g)}+\relgradA  g$ is a
relaxed gradient of $f$, hence on $\{f-g=c\}$ we conclude that
$\mm$-a.e. it holds $\relgradA f\leq \relgradA g$. Reversing the roles
of $f$ and $g$ we conclude.
\\* {\bf (c)} By 2. and Rademacher
Theorem we know that the right hand side is well defined, so that
the statement makes sense (with the convention to define
$|\phi'\circ f|$ arbitrarily at points $x$ such that $\phi'$ does
not exist at $f(x)$). Also, by \eqref{eq:chainc1} we know that the
thesis is true if $\phi$ is $C^1$. For the general case, just
approximate $\phi$ with a sequence $(\phi_n)$ of equi-Lipschitz and
$C^1$ functions, such that $\phi_n'\to \phi'$ a.e. on the image of
$f$.

Let us now consider the monotone case;
with no loss of generality we can assume that $0\leq\phi'\leq
1$. We know that $(1-\phi'(f)){\relgradA f}$ and $\phi'(f)\relgradA f$
are relaxed gradients of $f-\phi(f)$ and $f$ respectively. Since
$$
\relgradA f\leq \relgradA {(f-\phi(f))}+\relgradA {\phi(f)}\leq
\Big( (1-\phi'(f))+\phi'(f)\Big)\relgradA f=\relgradA f
$$
it follows that all inequalities are equalities $\mm$-a.e. in $X$.

When $\phi$ is $\rmC^1$ we can use the locality property.
\\*
{\bf (d)} Applying Lemma \ref{lem:strongchee} we find
two optimal sequences $(f_n),\,(g_n)$ of bounded
Lipschitz functions satisfying \eqref{eq:16bis}
(w.r.t.~$f$ and $g$ respectively).
When $\phi$ is of class $C^1$, 
passing to the limit in the inequality
\eqref{eq:308} written for
$f_n$ and $g_n$ we easily get \eqref{eq:65}.
In the general case, we first approximate $\phi$ by a sequence
$\phi_n$ of nondecreasing contraction of class $C^1$ converging to
$\phi$ pointwise and then pass to the limit in \eqref{eq:65} 
written for $\phi_n$.
\end{proof}
\begin{corollary}
  \label{cor:trivialsup}
  If $f_1,\cdots,f_M\in \Sob^{1,p}(\X,\AA)$ then also the functions
  $f_+:=f_1\lor f_2\lor\cdots\lor f_M$ and
  $f_-:=f_1\land f_2\land\cdots\land f_M$
  belong to $\Sob^{1,p}(\X,\AA)$ and
  \begin{equation}
    \label{eq:364}
    \begin{aligned}
      \relgradA{f_+}&=\relgradA{f_j}\quad\text{on }A_j:=\{x\in
      X:f_+=f_j\},\\
      \relgradA{f_-}&=\relgradA{f_j}\quad\text{on
      }B_j:=\{x\in X:f_-=f_j\}.
    \end{aligned}
  \end{equation}
\end{corollary}
\subsection{Invariance w.r.t.~restriction and completion}
\label{subsec:invariance1}
It is obvious that the Cheeger energy and the minimal relaxed gradient
are invariant with respect to isomorphisms of e.m.t.m.~structures
$(\X,\AA)$,
according to Definition \ref{def:embeddings}.
Here we state two simple
(and very preliminary) results
concerning the behaviour of the Cheeger energy
w.r.t.~a general measure-preserving
embedding $\iota$
of $(\X,\AA)$ into $(\hhat \X,\hhat \AA)$: we keep the same notation
of Section \ref{subsec:compactification}.
We will state a much deeper result in
the last Section of these notes, see Theorem \ref{thm:invariance-i}.
\begin{lemma}
  \label{le:preliminary-embedding}
  For every $\hhat f\in \Sob^{1,p}(\hhat \X,\hhat\AA)$
  the function $f:=\iota^*\hhat f$ belongs to $\Sob^{1,p}(\X,\AA)$
  and
  \begin{equation}
    \label{eq:362}
    \relgradA f\le \iota^*(|\rmD \hhat f|_{\star,\sAA'})\quad
    \text{$\mm$-a.e.~in $X$.}
  \end{equation}
\end{lemma}
\begin{proof}
  Let us first observe that if $\hhat f\in \Lipb(\hhat X,\hhat
  \tau,\hhat \sfd)$ and $\hhat G\ge \lip_{\sfd'} \hhat f$, then
  \begin{equation}
    \label{eq:363}
    \iota^* \hhat G\ge \lip_\sfd(\iota^* \hhat f).
  \end{equation}
  In fact, setting $f:=\iota^*(\hhat f)$ and choosing arbitrary sets
  $U\in X$ and $\hhat U\in \hhat X$ containing $\iota(U)$,
  if $L=\Lip (\hhat f,\hhat U,\hhat \sfd)$ we have
  \begin{displaymath}
    |f(x)-f(y)|=
    |\hhat f(\iota(x))-\hhat f(\iota(y))|
    \le L \hhat\sfd(\iota(x),\iota(y))=L\sfd(x,y)\quad
    \text{for every }x,y\in U
  \end{displaymath}
  so that
  \begin{equation}
    \label{eq:365}
    \Lip(f,U,\sfd)\le \Lip(\hhat f,\hhat U,\hhat \sfd).
  \end{equation}
  Recalling the definition \eqref{eq:55} and considering
  the collection of all the open neighborhood of $\iota(x)$ in $\hhat
  X$, we get \eqref{eq:363}.
  
  In order to obtain \eqref{eq:362}
  it is now sufficient to take an optimal sequence $\hhat f_n,\hhat G_n$ as in
  \eqref{eq:16bis} for
  the $(p,\AA')$-minimal relaxed gradient $|\rmD f'|_{\star,\sAA'}$
  of $\hhat f\in \Sob^{1,p}(\hhat \X,\hhat \AA)$
  observing that
  \begin{displaymath}
    \iota^*\hhat f_n\to f,\quad
    \iota^*\hhat G_n\to \iota^*\relgrad {\hhat f}\quad
    \text{strongly in }L^p(X,\mm),
  \end{displaymath}
  so that $\iota^*{|\rmD \hhat f|_{\star,\AA'}}$ is a relaxed gradient for $f$.
\end{proof}
When $\iota(X)$ is $\hhat\sfd$-dense in $\hhat X$
(in particular, when $\hhat X$ is a completion according to the 
definition given in Corollary \ref{cor:completion}),
we have a better behaviour.
\begin{proposition}
  Suppose that $\iota:X\to \hhat X$
  is a measure-preserving embedding of $(\X,\AA)$
  into $(\hhat X,\hhat \AA)$ such that
  \begin{equation}
    \label{eq:536}
    \iota(X)\text{ is $\hhat\sfd$-dense in $\hhat X$},\quad
    \iota^*(\hhat \AA)=\AA.
  \end{equation}
  Then $\iota^*$ is an
  isomorphism of $\Sob^{1,p}(\hhat\X,\hhat \AA)$
  onto $\Sob^{1,p}(\X,\AA)$
  and for every $f=\iota^* \hhat f$
  \begin{equation}
    \label{eq:362bis}
    \relgradA f= \iota^*(|\rmD  \hhat f|_{\star,\hhat \sAA})\quad
    \text{$\mm$-a.e.~in $X$.}
  \end{equation}
\end{proposition}
\begin{proof}
  Let $f\in \Sob^{1,p}(\X,\AA) $ and let
  $f_n\in \AA$ % and $G_n\in \Lip_b(X,\tau,\sfd)$,
  % $G_n\ge\lipd f_n$,
  be an optimal approximating sequence
  as in \eqref{eq:16bis}.
  
  We want to show that $f'=\iota_* f\in \Sob^{1,p}(\X',\AA')$;
  by Lemma \ref{le:measurable} and \eqref{eq:536},
  we can find $f_n'\in \AA'$ such that $f_n=\iota^* f_n'$.
  Since $\iota_*$ is an $L^p$-isometry, we know that
  $f_n'\to f'$ strongly in $L^p(X',\mm')$.
  Since $\iota$ is a homeomorphism between $X$ and $\iota(X)$,
  if $x'=\iota(x)$ and $g>\lip_\sfd f_n(x)$, we can find an open
  neighborhood
  $U'$ of $x'$ such that setting $U:=\iota^{-1}(U')$ we have
  \begin{equation}
    \label{eq:540}
      |f_n(z)-f_n(y)|=
      |\hhat f_n(\iota(z))-\hhat f_n(\iota(y))|
      \le g \hhat\sfd(\iota(z),\iota(y))=g\sfd(z,y)\quad
    \text{for every }z,y\in U.
  \end{equation}
  On the other hand, the $\sfd$-density of $\iota(X)$ in $\hhat X$
  guarantees that for every $z',y'\in U'\setminus(\iota(U))$
  there exist $\hhat\sfd$ balls $B_\delta(z'),B_\delta(y')$
  of radius $\delta$ such that
  \begin{displaymath}
    B_\delta(z')\subset U',\     B_\delta(y')\subset U',\quad
    B_\delta(z')\cap\iota(U)\neq \emptyset,\
    B_\delta(y')\cap\iota(U)\neq \emptyset,
  \end{displaymath}
  so that 
  \eqref{eq:540} extends to $U'$ as
  \begin{equation}
    \label{eq:541}
    | f_n'(z')- f_n'(y')|\le g\sfd(y',z')
    \quad \text{for every }z,y\in U.
  \end{equation}
  We deduce that
  \begin{equation}
    \label{eq:542}
    \lip_{\sfd'} f_n'(\iota(x))\le \lipd f_n(x)\quad\text{for every }x\in X
  \end{equation}
  and therefore
  \begin{align*}
    \limsup_{n\to\infty}
    \int_{\hhat X'}\big(\lip_{\sfd'} f_n'\big)^p\,\d\mm'
    &=
    \limsup_{n\to\infty}
    \int_{\hhat X'}\big(\lip_{\sfd'} f_n'(\iota(x)\big)^p\,\d\mm(x)
    \\&\le
    \lim_{n\to\infty}
    \int_X \big(\lipd f_n(x)\big)^p\,\d\mm(x)=
    \CE_{p,\AA}(f)
  \end{align*}
  We obtain that
  \begin{displaymath}
    \hhat f\in \Sob^{1,p}(\hhat\X,\hhat \AA),\quad
    \CE_{p,\hhat \sAA'}(\hhat f)\le
    \CE_{p,\sAA}(f).
  \end{displaymath}
  Thanks to \eqref{eq:362} we also get \eqref{eq:362bis}.
\end{proof}
As an immediate application we obtain that the class of complete
e.m.t.m.~spaces is the natural setting for 
the Cheeger energy.
\begin{corollary}[Invariance of the Cheeger energy by completion]
  \label{cor:completionA}
  If $\bar \X=(\bar X,\bar \tau,\bar \sfd,\bar\mm)$
  is the completion of $\X$ induced by
  $\iota:X\to \bar X$ of Corollary \ref{cor:completion}
  and $\bar \AA=\{\bar f:f\in \AA\}$, then
  $\iota^*$ is an
  isomorphism of $\Sob^{1,p}(\bar\X,\bar \AA)$
  onto $\Sob^{1,p}(\X,\AA)$
  and 
  \begin{equation}
    \label{eq:362tris}
    \relgradA f= \iota^*(|\rmD  \bar f|_{\star,\bar \sAA})\quad
    \text{$\mm$-a.e.~in $X$}\quad
    \text{for every $f=\iota^* \bar f$}.
  \end{equation}
\end{corollary}
We conclude with another easy application
of the previous results to restrictions, as in Example
\ref{ex:embeddings}(c).
Recall that
if $Y\subset X$
 is a $\tau$-dense $\mm$-measurable
 subset satisfying $\mm(X\setminus Y)=0,$ 
 the restriction to $Y$ is an isomorphism of
 $L^p(X,\mm)$ with $L^p(Y,\mm_Y)$, so that one can compare
 the Sobolev spaces
$\Sob^{1,p}(\X,\AA)$ and $\Sob^{1,p}(\Y,\AA_Y)$.
We will denote by $|\rmD f|_{\star,\Y,\AA_Y}$ the
$(p,\AA_Y)$ minimal relaxed gradient in $\Sob^{1,p}(\Y,\AA_Y)$.
\begin{corollary}[Restriction]
  \label{cor:restriction}
  Let $\X=(X,\tau,\sfd,\mm)$ be an e.m.t.m.~space
   and let $Y\subset X$
  be a $\tau$-dense $\mm$-measurable
  subset satisfying $\mm(X\setminus Y)=0.$ 
  With the above notation, (the restriction to $Y$ of) every function
  $f\in \Sob^{1,p}(\X,\AA)$ belongs to
  $\Sob^{1,p}(\Y,\AA_Y)$
  and
  \begin{equation}
    \label{eq:534}
    |\rmD f|_{\star,{\scriptscriptstyle\Y,\AA_Y}}\le \relgradA f
    \quad\text{$\mm$-a.e.}
  \end{equation}
  If moreover $Y$ is $\sfd$-dense in $X$, then
  the converse property is also true: for every $f\in L^p(X,\mm)$
  \begin{equation}
    \label{eq:535}
    f\in \Sob^{1,p}(\X,\AA)\quad\Leftrightarrow\quad
    \Sob^{1,p}(\Y,\AA_Y),\qquad
    |\rmD f|_{\star,{\scriptscriptstyle\Y,\AA_Y}}= \relgradA f
    \quad\text{$\mm$-a.e.}
  \end{equation}
\end{corollary}
\subsection{Notes}
\label{subsec:notes5}
\begin{notes}
  \Para{\ref{subsec:relaxed}}
  is strongly inspired by Cheeger's work \cite{Cheeger00}
  (where the energy is obtained starting from upper gradients instead of
  the local Lipschitz constants)
  and
  follows quite closely the
  presentation of \cite{AGS14I,AGS13},
  with the required adjustments due to the
  presence of a compatible algebra $\AA$ instead of
  $\Lip_b(X,\tau,\sfd)$.
  Corollary \ref{cor:poly-truncation}
  and the crucial locality Lemma \ref{lem:locality}
  take advantage of the approximation tools
  presented in \S\,\ref{subsec:compatible-A}.
  Even if a posteriori the Cheeger energy
  will be independent of $\AA$, the role of the algebra
  should be considered as a technique to get new
  density results. Moreover, 
  it allows for simpler constructions in many cases, where
  a distinguished algebra provides better structural properties
  of the energy, see the final Section \ref{sec:Examples}.

  \Para{\ref{subsec:invariance1}}
  contains some preliminary facts about the behaviour of the Cheeger energy
  with respect to measure-preserving embeddings of e.m.t.m.~spaces
  (in particular w.r.t.~completion). The possibility to modify the
  topological and the algebraic properties of the e.m.t.m.~setting
  is one of its strength point.
\end{notes}

\section{Invariance of the Cheeger energy with respect to
  the core algebra: the compact case}
\label{sec:Invariance}
The aim of this section is to study the
property of the Cheeger energy with respect to
the choice of the core algebra $\AA$
in the case of a compact ambient space $(X,\tau)$.
% embeddings
% according to Definition \ref{def:embeddings}  (in particular the
% compactification
% studied by Theorem \ref{thm:G-compactification}) and
% to modifications of the algebra $\AA$.
An important tool is provided by the (generalized)
Hopf-Lax flow, which we collect in the next section.

\subsection{The metric Hopf-Lax flow in compact spaces}
\label{subsec:HL}
\index{Hopf-Lax flow}
Let $(X,\tau,\sfd)$ be a compact 
extended metric-topological space and
let $\delta:X\times
X\to[0,+\infty]$
be a $\tau$-l.s.c.~continuous extended semidistance
(our main examples will be the extended distance $\sfd$
and the continuous semidistances $\sfd_i$
as in Lemma \ref{rem:monotonicity}).
% The distance from $K$ is the function 
% \begin{equation}
%   \label{eq:383}
%   \delta(x,K):=\min_{y\in K}\delta(x,y)\quad x\in X,
% \end{equation}
% We clearly have
% \begin{equation}
%   \label{eq:384}
%   \delta(x,K)=0\quad
%   \Leftrightarrow\quad
%   x\in K,\qquad
%   \delta(x,K)\le   \operatorname{diam}_\delta(X):=\max_{x,y\in X}\delta(x,y)<\infty.
% \end{equation}
For every $f\in
\rmC_b(X)$,
$x,y\in X$
and $t>0$ we
set
\begin{equation}
  \label{eq:381}
  F^{\delta}(t,x,y):=f(y)+\frac
  {\delta^q(x,y)}{q\,t^{q-1}},\quad
  F(t,x,y):=F^{\sfd}(t,x,y).
\end{equation}
$F^{\delta}$ is a l.s.c.~(continuous, if $\delta$ is continuous)
function bounded from below.

Let us also fix a compact set $K\subset X$ such that
there exists a constant
$S=S(K,\delta)\in [0,+\infty[$ 
satisfying
\begin{equation}
  \label{eq:403}
  \min_{y\in K}\delta(x,y)\le S\quad\text{for every }x\in X.
\end{equation}
\eqref{eq:403} is always satisfied if $\delta$ is continuous
or if $K=X$ (and in this case $S=0$).
%
% to avoid functions taking values in the extended real line,
% we will suppose $K=X$ when
% $\delta$ takes the value $+\infty$.

The modified Hopf-Lax evolution is defined by the formula
\begin{equation}
  \label{eq:372}
  \sfQ^{K,\delta}_{t}f(x):=\min_{y\in K}
  F^\delta(t,x,y)
  % f(y)+\frac
  % {\delta^q(x,y)}{q\,t^{q-1}},
  % \quad
  % \sfQ_{t}f(x):=\min_{y\in X}f(y)+\frac
  % {\sfd^q(x,y)}{pt^{q-1}}
  % ,
  \quad t>0,
\end{equation}
where we will omit to indicate the explicit dependence on $K$ (resp.~on $\delta$)
when $K=X$ (resp.~when $\delta=\sfd$), thus setting
\begin{equation}
\sfQ^{\delta}_{t}f:=\sfQ^{X,\delta}_tf,\qquad
\sfQ_t f:=\sfQ^{X,\sfd}_tf.\label{eq:402}
\end{equation}
Since $K$ is compact and $F^\delta(t,x,\cdot)$ takes at least one
finite value in $K$ by \eqref{eq:403}, 
the minimum in \eqref{eq:372} is attained:
for every $x\in X$ we also set
\begin{equation}
  \label{eq:374}
  \sfJ^{K,\delta}_{t}f(x):=\Big\{y\in K: f(y)+\frac
  {\delta^q(x,y)}{pt^{q-1}}=\sfQ^\delta_{t}f(x)\Big\},\quad
  \sfJ^\delta_t :=\sfJ^{X,\delta}_t,\quad
  \sfJ_t:=\sfJ^{X,\sfd}_t
  % \sfJ_{t}(x):=\Big\{y\in X: f(y)+\frac
  % {\sfd^q(x,y)}{pt^{q-1}}=\sfQ_{t}f(x)\Big\},
\end{equation}
and
\begin{equation}
  \label{eq:373}
  \sfD^{K,\delta,+}_{t}f(x):=\max_{y\in
    \sfJ^{K,\delta}_{t}f(x)}\delta(x,y),\quad
  \sfD^{K,\delta,-}_{t}f(x):=\min_{y\in
    \sfJ^{K,\delta}_{t}f(x)}\delta(x,y).
  % \quad
  % \sfD^+_{i,t}(x):=\sup_{y\in
  %   \sfJ_{t}(x)}\sfd(x,y).
  % \sfD^+_{t}(x):=\lim_{\eps\down0}\sfD^+_{i,t,\eps}(x)=
  % \inf_{\eps>0}\sfD^+_{i,t,\eps}(x).
  % \limsup_{n\up\infty}\sfd_i(x,y_n),\quad
  % \sfD^-_{i,t}(x):=\inf_{(y_n)}\liminf_{n\up\infty}\sfd_i(x,y_n)
\end{equation}
As usual, we set $\Osc fX:=\sup_X f-\inf_X f$.
\begin{lemma}[Basic estimates]
  Let $f\in \rmC_b(X)$ and let $f_t(x):=\sfQ^{K,\delta}_t f(x),$
  $J_t(x):=\ \sfJ^{K,\delta}_t f,$ 
  $D_t^\pm(x):=\sfD^{K,\delta,\pm}_t f(x)$ be defined as \eqref{eq:372}, \eqref{eq:374},
  \eqref{eq:373}
  for $t>0$.
  For every $x,y\in X$, $0<s<t$, $x'\in J_t(x)$, $y'\in J_s(y)$ we have
  \begin{equation}
    \label{eq:390}
    \min_X f\le f_t(x)\le \max_X f+\frac 1{q\,t^{q-1}}S^{q}\quad\text{for every }t>0,\ x\in X,
  \end{equation}
  \begin{equation}
    \label{eq:404}
    \Big(\frac {D_t^+(x)}t\Big)^q\le
    \min\Big(q\,t^{-1}\Osc fX, \big(q \Lip(f,X,\delta)\big)^p\Big),
    % ,\quad
   % \frac {D_t^+(x)}t\le \Big(\frac{q \Osc fX}t\Big)^{1/q}
  \end{equation}
  \begin{equation}
    \label{eq:405}
    \Big(\frac 1{qs^{q-1}}-\frac 1{qt^{q-1}}\Big)D_s^+(x)
    \le f_s(x)-f_t(x)\le \Big(\frac 1{qs^{q-1}}-\frac 1{qt^{q-1}}\Big)D_t^-(x),
  \end{equation}
  \begin{align}
    \label{eq:387}
    -\Big(\frac{\delta(x,y')}{t}\Big)^{q-1}\delta(x,y)
    &+ \frac
      1p\Big(\frac{\delta(y,y')}{t}\Big)^{q}(t-s)
      \le f_s(y)-f_t(x)
      \\&\le \label{eq:387bis}
    \Big(\frac{\delta(y,x')}{s}\Big)^{q-1}\delta(x,y)+ \frac
      1p\Big(\frac{\delta(x,x')}{s}\Big)^{q}(t-s).
  \end{align}
  % \begin{equation}
  %   \label{eq:387}
  %   |f_s(y)-f_t(x)|\le \Big(\Big(\frac{\Osc fX}{\min(s,t)}\Big)^{1/q}+p\Lip(f,X)\Big)\delta(x,y)+
  %     \Lip(f,X)|t-s|.
  %   \end{equation}
\end{lemma}
\begin{proof}
  \eqref{eq:390} is immediate.
  In order to prove \eqref{eq:404} we simply observe that for every
  $x'\in J_t(x)$
  \begin{align*}
    \frac {\delta^q(x,x')}{t^{q-1}}\le q\big(f(x)-f(x')\big)\le q\min\big(\Osc fX,\Lip(f,X,\delta)\delta(x,x')\big)
  \end{align*}
  thus obtaining
  \begin{displaymath}
    \Big(\frac {\delta(x,x')}{t}\Big)^q\le \frac qt\Osc fX,\quad
    \Big(\frac {\delta(x,x')}{t}\Big)^{q-1}\le q \Lip(f,X,\delta).
  \end{displaymath}
  Let us now check \eqref{eq:387bis}: selecting $x'\in J_t(x)$
  \begin{align}
    \notag
    f_s(y)-f_t(x)
    &\le
      F^\delta(s,y,x')-F^\delta(t,x,x')
      =\frac {\delta^q(y,x')}{qs^{q-1}}-
      \frac {\delta^q(x,x')}{qt^{q-1}}
    \\&=\notag
    \frac {\delta^q(y,x')}{qs^{q-1}}-\frac {\delta^q(x,x')}{qs^{q-1}}+
    \Big(\frac 1{qs^{q-1}}-\frac 1{qt^{q-1}}\Big)\delta^q(x,x')
    \\&\le\label{eq:406}
    \Big(\frac{\delta(y,x')}s\Big)^{q-1}|\delta(y,x')-\delta(x,x')|+
    \Big(\frac 1{qs^{q-1}}-\frac 1{qt^{q-1}}\Big)\delta^q(x,x').
 %   -\frac{q-1}{q s^q}(s-t) \delta^q(x,x');
  \end{align}
  Applying the triangle inequality for $\delta$ and
  the elementary inequality (arising from the convexity of
  $r\mapsto 1/r^{q-1}$ in $(0,\infty)$)
  \begin{equation}
    \label{eq:407}
    \frac 1{p t^q}(t-s)\le  \frac 1{qs^{q-1}}-\frac 1{qt^{q-1}}\le
    \frac 1{p s^q}(t-s)\quad
    \text{for every }s,t\in (0,+\infty),
  \end{equation}
  we obtain \eqref{eq:387bis}.
  \eqref{eq:387} will follow by switching the role of
  $(x,t)$ and $(y,s)$.

  Concerning \eqref{eq:405}, the right inequality
  can be easily obtained by choosing $y=x$ in
  \eqref{eq:406}
  and minimizing with respect to $x'\in J_t(x)$.
  Inverting the role of $s$ and $t$ (notice that \eqref{eq:406}
  does not require $s<t$) and maximizing with respect to $x'\in
  J_s(x)$ we get the left inequality of \eqref{eq:405}.
\end{proof}
We collect further properties in the next Lemma.
\begin{lemma}
  Let us assume that $\delta$ is continuous. For every
  $f\in \rmC(X)$ we have
  \label{le:HL1}
  \begin{enumerate}
  \item The map $(x,t)\mapsto \sfQ^{K,\delta}_t f(x)$ is
    continuous in $X\times (0,+\infty)$ and
    for every $0<s<t$ 
    it satisfies the estimate
% \begin{equation}
%   \label{eq:382}
%   \min_{y\in K}f\le \sfQ^{K,\delta}_{t}f(x)\le
%   \max_{y\in K}f+\frac
%   {1}{q\,t^{q-1}}\delta^q(x,K)\quad
%   \text{for every }x\in X,\ t>0;
% \end{equation}
% and $\delta$-Lipschitz in every interval $[t_0,+\infty)$,
%     $t_0>0$. It satisfy
    \begin{equation}
      \label{eq:389}
      \begin{aligned}
        %\Big(\frac1{qt^{q-1}}-\frac 1{qs^{q-1}}\Big)
        \frac 1p \Big(\frac{\sfD^{K,\delta,+}_sf(x)}t\Big)^q(t-s)
        &\le \sfQ_s^{K,\delta}f(x)
        - \sfQ_t^{K,\delta}f(x)
        \le \frac 1p
        % \Big(\frac1{qt^{q-1}}-\frac
        % 1{qs^{q-1}}\Big)
        \Big(\frac{\sfD^{K,\delta,-}_tf(x)}s\Big)^q(t-s).
      \end{aligned}
          \end{equation}
  \item The map $(x,t)\mapsto \sfD^{K,\delta,+}_{t}f(x)$
    (resp.~$(x,t)\mapsto \sfD^{K,\delta,-}_{t}f(x)$) is upper
    (resp.~lower) semicontinuous in $X\times (0,\infty)$ and there holds
    \begin{equation}
      \label{eq:386}
       \sfD^{K,\delta,-}_{s}f(x)\le
       \sfD^{K,\delta,+}_{s}f(x)\le
       \sfD^{K,\delta,-}_{t}f(x)\quad\text{if }0<s<t.
    \end{equation}
  \item If $(K_\lambda)_{\lambda\in \Lambda}$ is an increasing net
    with $\cup_{\lambda\in \Lambda} K_\lambda$ dense in $K$, then
    for every $x\in X$ the net $\lambda\mapsto \sfQ^{K_\lambda,\delta}_t f(x)$
    is decreasing and converging to $\sfQ^{K,\delta}_t f(x)$.
  \item If $(K_\lambda)_{\lambda\in \Lambda}$ is an increasing net
    with $K_\lambda\subset K$, then
    \begin{equation}
      \label{eq:385}
      \lim_{\lambda\in \Lambda}\sfQ^{K_\lambda,\delta}_t f(x)=
      \sfQ^{K,\delta}_t f(x)\quad
      \Rightarrow\quad
      \limsup_{\lambda\in \Lambda} \sfD^{K_\lambda,\delta,+}_tf(x)\le
      \sfD^{K,\delta,+}_t f(x).
    \end{equation}
  \end{enumerate}
\end{lemma}
\begin{proof}
  \textbf{(a)}
  Continuity follows by \eqref{eq:387} and \eqref{eq:387bis}.
  \eqref{eq:389} is a consequence of \eqref{eq:405} and
  \eqref{eq:407}.
  % the fact that for every $y\in K$ the maps
  % $(t,x)\mapsto \frac 1{q t^{q-1}}\delta^q(x,y)$ are uniformly
  % bounded, continuous, and $\delta$ Lipschitz in $[t_0,\infty)$, with
  % \begin{displaymath}
  %   \frac 1{q t^{q-1}}\delta^q(x,y)-
  %   \frac 1{q s^{q-1}}\delta^q(x',y)\le
  %   C(|t-s|+\delta(x,x'))
  % \end{displaymath}
  % for a constant $C$ only depending on $q$, $t_0$ and
  % $\operatorname{diam}_\delta(X)$.
  %
  % For every $y\in \sfJ^{K,\delta}_sf(x)$ we have
  % \begin{equation}
  %   \label{eq:388}
  %   \sfQ_t^{K,\delta}f(x)
  %   -    \sfQ_s^{K,\delta}f(x)\le
  %   F^{\delta}(t,x,y)-F^{\delta}(s,x,y)
  %   \le
  %   \Big(\frac1{qt^{q-1}}-\frac
  %       1{qs^{q-1}}\Big) \delta^q(x,y);
  %     \end{equation}
  %     taking the infimum w.r.t.~$y\in \sfJ^{K,\delta}_sf(x)$ we obtain
  %     the second inequality of \eqref{eq:389}. The first one
  %     follows by a similar argument, choosing $y\in \sfJ^{K,\delta}_t
  %     f(x)$.
  \medskip
  
 \noindent\textbf{(b)}
 \eqref{eq:386} follows immediately from \eqref{eq:405}.
 It is not difficult to check that
$(t,x)\mapsto \sfD^{K,\delta,+}_{t}f(x)$ is
upper semicontinuous:
if $(x_\lambda,t_\lambda)$, $\lambda\in \Lambda$, is a net converging
to $(x,t)\in X\times (0,+\infty)$ with
$\sfD^{K,\delta,+}_{t_\lambda}(x_\lambda)
\ge c$ and
$y_\lambda\in \sfJ^{K,\delta}_{t_\lambda}(x_\lambda)$ such that
\begin{equation}
  \label{eq:375}
  F^{\delta}(t_\lambda,x_\lambda,y_\lambda)=\sfQ^{K,\delta}_{t_\lambda}f(x_\lambda),
  \quad
  \delta(x_\lambda,y_\lambda)=\sfD^{K,\delta,+}_{t_\lambda}(x_\lambda)\ge
  c,
\end{equation}
we can find a subnet $j\mapsto \lambda(j)$, $j\in J$,
such that $j\mapsto y_{\lambda(j)}$ converges to a point $y\in K$ with
\begin{displaymath}
  F^{\delta}(t,x,y)=
  \lim_{j\in J}
  F^{\delta}(t_{\lambda(j)},x_{\lambda(j)},y_{\lambda(j)})=
  \lim_{j\in J}\sfQ^{K,\delta}_{t_{\lambda(j)}}f(x_{\lambda(j)})=
  \sfQ^{K,\delta}_tf(x),
\end{displaymath}
showing that $y\in \sfJ^{K,\delta}_t f(x)$. Since
\begin{displaymath}
  \delta(x,y)=\lim_{j\in J}\delta(x_{\lambda(j)},y_{\lambda(j)})\ge c,
\end{displaymath}
we obtain that $\sfD^{K,\delta,+}_t(x)\ge c$.
A similar argument holds for the lower semicontinuity of
$\sfD^{K,\delta,-}_t f$.
\medskip

\noindent\textbf{(c)}
The decreasing property of $\sfQ$ w.r.t.~$\lambda$ is obvious; 
in particular it yields
\begin{equation}
  \label{eq:391}
  \sfQ^{K_\lambda,\delta}_t f(x)\ge   \sfQ^{K,\delta}_t f(x).
\end{equation}
On the other hand, by the density of $\cup_\lambda K_\lambda$ in $K$
and the continuity of $F$, for every $y\in \sfJ^{K,\delta}_t f(x)$ and 
$\eps>0$ we can find $\lambda_\eps\in \Lambda$ and $y_\eps\in K_{\lambda_\eps}$
such that $F(t,x,y_\eps)\le F(t,x,y)+\eps$
so that for every $\lambda\succ\lambda_\eps$
\begin{displaymath}
  \sfQ^{K_\lambda,\delta}_t f(x)\le
  F(t,x,y_\eps)\le F(t,x,y)+\eps\le \sfQ^{K,\delta}_t f(x)+\eps.
\end{displaymath}
Since $\eps>0$ is arbitrary we obtain the proof of the claim.
\medskip

\noindent\textbf{(d)}
We argue as in the proof of the second claim:
we select $y_\lambda\in \sfJ^{K_\lambda,\delta}_t f(x)\subset
K_\lambda$ so that
\begin{equation}
  \label{eq:392}
  F(t,x,y_\lambda)=\sfQ^{K_\lambda,\delta}_t f(x),\quad
  \sfD^{K_\lambda,\delta,+}_t f(x)=\delta(x,y_\lambda).
\end{equation}
We can find a subnet $j\mapsto y_{\lambda(j)}$, $j\in J$,
converging to some $y\in K$ with
$S:=\limsup_{\lambda\in \Lambda}\sfD^{K_\lambda,\delta,+}_t f(x)=
\lim_{j\in J}\delta(x,y_{\lambda(j)})=\delta(x,y)$.
It follows that
\begin{align*}
  \sfQ^{K,\delta}_tf(x)\topref{eq:385}=
  \lim_{j\in J}\sfQ^{K_{\lambda(j)},\delta}_tf(x)
  =
  \lim_{j\in J}F(t,x,y_{\lambda(j)})=
  F(t,x,y)
\end{align*}
so that $y\in \sfJ^{K,\delta}_tf(x)$. This yields
\begin{displaymath}
  \sfD^{K,\delta,+}_tf(x)\ge \delta(x,y)=\limsup_{\lambda\in \Lambda}\sfD^{K_\lambda,\delta,+}_t f(x).
\end{displaymath}
% Since $\sfd_i(y_\lambda,x_\lambda)$ is definitely bounded,
% $\lim_{\lambda\in \Lambda}\sfd_i(x_\lambda,x)=0$,
% $\lim_{\lambda\in \Lambda}\sfQ_{i,t_{\lambda}}f(x_\lambda)=\sfQ_{i,t}(x)$,
% we have
% \begin{equation}
%   \label{eq:376}
%   \limsup_{\lambda\in \Lambda}
%   f(y_\lambda)+\frac
%   {\sfd_i^q(x,y_\lambda)}{pt^{q-1}}\le 
%   \sfQ_{i,t}f(x_\lambda)+\eps/2,
%   \quad
%   \limsup_{\lambda\in \Lambda}\sfd_i(x_\lambda,y_\lambda)\ge c,
% \end{equation}
% so that $\sfD^+_{i,t,\eps}(x)\ge c$; passing to the limit as
% $\eps\downarrow0$ we conclude.
%  The upper
%  semicontinuity of $\sfD^{K,\delta,+}$ 
\end{proof}
\noindent
We consider now the behaviour of $\sfQ^\delta=\sfQ^{X,\delta}$ with
respect to $\delta$. 
\begin{proposition}
  \label{prop:HL2}
  Let $(\sfd_i)_{i\in I}$ be a directed family of continuous
  semidistances as in {\upshape (\ref{eq:monotone}a,b,c,d)} and let
  $f\in \rmC_b(X)$.
  For every $x\in X$
    the net $i\mapsto \sfQ^{\sfd_i}_t f(x)$ is monotonically
    converging to $\sfQ_t f(x)$ and
    \begin{equation}
    \limsup_{i\in I}\sfD^{\sfd_i,+}_tf(x)\le \sfD^{+}_t
    f(x).\label{eq:394}
  \end{equation}
  More generally, if $j\mapsto i(j)$, $j\in J$, is  an increasing net (but not
    necessarily a subnet)
    \begin{equation}
      \label{eq:393}
      \lim_{j\in J}\sfQ^{\sfd_{i(j)}}_tf(x)=
      \sfQ_t f(x)\quad\Rightarrow\quad
      \limsup_{j\in J}\sfD^{\sfd_{i(j)},+}_tf(x)\le \sfD^{+}_t f(x)
    \end{equation}
\end{proposition}
\begin{proof}
  Since $i\prec j$ yields $\sfd_i\le \sfd_j$, it is
  clear that $i\mapsto \sfQ^{\sfd_i}_t f(x)$ is increasing.
  For every
  fixed $t,x,y$ we have $\lim_{i\in I}F^{\sfd_i}(t,x,y)=F^\sfd(t,x,y)$
  monotonically.
  The first statement then follows by a standard application of
  $\Gamma$-convergence
  of a family of increasing real functions in a compact set.
  \eqref{eq:394} is a particular case of \eqref{eq:393} for the
  identity map in the directed set $I$.
  
  Let us now assume that $\lim_{j\in J}\sfQ^{\sfd_{i(j)}}_tf(x)=
  \sfQ_t f(x) $ along an increasing net $j\mapsto i(j)$.
  We can select $y_j\in X$ such that
  \begin{equation}
    \label{eq:395}
    \sfQ^{\sfd_{i(j)}}_tf(x)=f(y_j)+\frac 1{q
      t^{q-1}}\sfd^q_{i(j)}(x,y_j),\quad
    \sfD^{\sfd_{i(j)},+}_t f(x)=\sfd_{i(j)}(x,y_j).
  \end{equation}
  We can find a further subnet $h\mapsto j(h)$, $h\in H$,
  such that $(y_{j(h)})_{h\in H}$ is convergent to $y\in X$ and
  \begin{equation}
    \label{eq:396}
    \limsup_{j\in J}\sfD^{\sfd_{i(j)},+}_t f(x)=
    \lim_{h\in H} \sfd_{i(j(h))}(x,y_{j(h)}).
\end{equation}
Passing to the limit in the first equation of \eqref{eq:395} and using
the assumption of \eqref{eq:393} we get
\begin{equation}
  \label{eq:397}
  \sfQ_t f(x)=f(y)+\frac 1{q
    t^{q-1}}\lim_{h\in H}\sfd^q_{i(j(h))}(x,y_{j(h)})
  \topref{eq:321}\ge f(y)+\frac 1{q
    t^{q-1}}\sfd^q(x,y)\ge \sfQ_t f(x)
\end{equation}
where the last inequality follows by the very definition of $\sfQ_t
f(x)$.
We deduce that
\begin{equation}
  \label{eq:398}
  y\in \sfJ_t f(x),\quad
  \lim_{h\in H}\sfd^q_{i(j(h))}(x,y_{j(h)})=\sfd^q(x,y).
\end{equation}
Since $\sfd(x,y)\le \sfD^+_t f(x)$, by \eqref{eq:396} we get \eqref{eq:393}.
\end{proof}
Notice that the upper semicontinuity property of \eqref{eq:394} and of
\eqref{eq:393} are not immediately obvious as in the case of, e.g.,
\eqref{eq:385},
since $\sfd$ is typically just lower semicontinuous along
$\tau$-converging sequences. In the proof we used in an essential way
the minimality of $y_j$ and the continuity of $f$.

We conclude this section with the main structural properties for the
Hopf-Lax evolution generated by $\sfd$.
\begin{theorem}
  \label{thm:HL3}
  Let $f\in \Lipb(X,\tau,\sfd)$ and let $\sfQ_t f,\ \sfJ_t f,\
  \sfD^\pm_t f$ be defined as \eqref{eq:402}, \eqref{eq:374},
  \eqref{eq:373}
  for $t>0$.
  \begin{enumerate}
  \item The functions $(x,t)\mapsto \sfQ_t f(x),\ \sfD_t^- f(x)$ are lower
    semicontinuous in $X\times (0,+\infty)$ and 
    \begin{equation}
      \label{eq:399}
      \min_X f\le \sfQ_t f(x)\le \max_X f\quad\text{for every }t>0,\
      x\in X.
    \end{equation}
  \item For every $x\in X$
    \begin{equation}
      \label{eq:410}
      \lim_{t\down0}\sfQ_t f(x)=\sfQ_0 f(x):=f(x),
    \end{equation}
    the map $t\mapsto \sfQ_t f(x)$ is
    Lipschitz in $[0,\infty)$ and satisfies
    \begin{equation}
      \label{eq:411}
      \frac\d{\d t}\sfQ_t f(x)=-\frac1p\Big(\frac{\sfD_t^\pm
        f(x)}t\Big)^q\quad\text{for  $t>0$ 
        with at most countable exceptions.}
    \end{equation}
    %    
  % \item For every $t>0$ $\sfQ_t f\in \Lipb(X,\sfd)$ with
  %   \begin{equation}
  %     \label{eq:400}
  %     \lip (\sfQ_tf)(x)\le \sfD_t^+ f
  %   \end{equation}
  \item
    For every $x\in X$ and $t>0$
    \begin{equation}
      \label{eq:401}
      f(x)-\sfQ_t f(x)=\frac tp\int_0^1 \Big(\frac{\sfD^+_{tr}
        f(x)}{tr}\Big)^p\,\d r,
    \end{equation}
    \begin{equation}
      \label{eq:418}
      \limsup_{t\down0}\frac{f(x)-\sfQ_t f(x)}t\le
      \frac1p|\rmD f|^p(x)\le \frac 1p \big(\lip f(x)\big)^p.
    \end{equation}
  \end{enumerate}
\end{theorem}
\begin{proof}
  \textbf{(a)} Lower semicontinuity of $\sfQ f$
  is a consequence of the joint lower
  semicontinuity of $F$ and of the compactness of $(X,\tau)$;
  it can also be obtained by Proposition \ref{prop:HL2},
  which characterizes $Q_t f$ as a supremum of continuous functions.  
  The bound \eqref{eq:399} is immediate.
  \medskip

  \noindent
  \textbf{(b)}
  As for the proof of \eqref{eq:389}, we get from
  \eqref{eq:405} and
  \eqref{eq:407}
  \begin{equation}\label{eq:416}
    % \Big(\frac1{qt^{q-1}}-\frac 1{qs^{q-1}}\Big)
      \frac 1p \Big(\frac{\sfD^{+}_sf(x)}t\Big)^q(t-s)
      \le \sfQ_sf(x)
        - \sfQ_t f(x)
        \le \frac 1p
        % \Big(\frac1{qt^{q-1}}-\frac
        % 1{qs^{q-1}}\Big)
        \Big(\frac{\sfD^{-}_tf(x)}s\Big)^q(t-s),
      \end{equation}
      and \eqref{eq:404} yields the uniform bound
      \begin{equation}
        \label{eq:408}
        \Big(\frac{\sfD^{+}_tf(x)}t\Big)^{q}
        \le \big(q\Lip(f,X)\big)^p.
      \end{equation}
      Since $t\mapsto \sfQ_t f(x)$ is decreasing, we obtain
      \begin{equation}
        \label{eq:409}
        \frac{|\sfQ_sf(x)-\sfQ_t f(x)|}{|t-s|}
        \le \frac 1p\left(\frac ts\right)^q
        \big(q\Lip(f,X)\big)^p\quad
        \text{for every }0<s<t.
      \end{equation}
      \eqref{eq:409} shows that $t\mapsto \sfQ_t f(x)$ is Lipschitz in
      every compact interval of $(0,\infty)$ with
      \begin{equation}
        \label{eq:412}
        \Big|\frac\d{\d t}\sfQ_t f(x)\Big|\le \frac 1p
        \big(q\Lip(f,X)\big)^p\quad\text{for a.e.~$t>0$},
      \end{equation}
      so that $t\mapsto \sfQ_t f(x)$ is Lipschitz in $(0,+\infty)$.
      In
      order to prove \eqref{eq:410} we simply observe that
      for every $x'\in \sfJ_t f(x)$, $x'\neq x$,
      \begin{equation}\label{eq:414}
        f(x)-\sfQ_t f(x)=
        f(x)-f(x')-\frac{\sfd^q(x,x')}{q\,t^{q-1}}
        \le \sfd(x,x')\Big(\frac{f(x)-f(x')}{\sfd(x,x')}-
        \frac 1q\frac{\sfd^{q-1}(x,x')}{t^{q-1}}\Big)
      \end{equation}
      so that
      \begin{equation}
        \label{eq:415}
       0\le  f(x)-\sfQ_t f(x)\le \sfD_t^- f(x)\Lip(f,X)
      \end{equation}
      and the right hand side vanishes as $t\down0$ thanks to
      \eqref{eq:408}.

      \eqref{eq:411} follows from \eqref{eq:417} and the monotonicity
      property (a consequence of \eqref{eq:405})
      \begin{equation}
        \label{eq:417}
        \sfD_s^- f(x)\le \sfD_s^+ f(x)\le \sfD_t^- f(x)\quad\text{for
          every }x\in X,\ 0<s<t,
      \end{equation}
      which in particular shows that $\sfD_t^- f(x)=\sfD_t^+ f(x)$
      for every $t>0$ with at most countable exceptions.
      \medskip

      \noindent\textbf{(c)}
      \eqref{eq:401} follows by integrating \eqref{eq:411}.

      Dividing \eqref{eq:414} by $t$ we get for every $x'\in \sfJ_t
      f(x)\setminus \{x\}$
      \begin{align*}
        \frac{f(x)-\sfQ_t f(x)}t=
        \frac{\sfd(x,x')}t\frac{f(x)-f(x')}{\sfd(x,x')}-
        \frac 1q\frac{\sfd^{q}(x,x')}{t^{q}}\le
        \frac 1p \Big(\frac{f(x)-f(x')}{\sfd(x,x')}\Big)^p;
      \end{align*}
      passing to the limit as $t\downarrow0$ and observing that
      $\lim_{t\down0}\sfD_t^+ f(x)=0$ we obtain \eqref{eq:418}.  
\end{proof}
We conclude this section with a discussion of the measurability
properties of the maps $\sfD^\pm f$.
In the case when $\sfd$ is continuous, $\sfD^+f$ (resp.~$\sfD^- f$)
is upper- (resp.~lower-) semicontinuous by Proposition \ref{prop:HL2}.
In the general case we can anyway prove that they are $\mm\times
\LL^1$ measurable.
\begin{lemma}[Conditional semicontinuity and measurability of $\sfD^\pm f$]
  \label{le:HL4}
  Under the same assumptions of Theorem \ref{thm:HL3}:
  \begin{enumerate}
  \item for every net $(x_\lambda,t_\lambda)_{\lambda\in \Lambda}$ in
    $X\times (0,\infty)$ such that $\lim_{\lambda\in
      \Lambda}(x_\lambda,t_\lambda)=(x,t)\in X\times (0,\infty)$ we have
    \begin{equation}
      \label{eq:424}
      \lim_{\lambda\in \Lambda}\sfQ_{t_\lambda} f(x_\lambda)=\sfQ_t
      f(x)\quad\Rightarrow\quad
      \left\{\begin{aligned}
        \limsup_{\lambda\in \Lambda}\sfD^+_{t_\lambda} f(x_\lambda)&\le
        \sfD_t^+ f(x),\\
        \liminf_{\lambda\in
          \Lambda}\sfD^-_{t_\lambda} f(x_\lambda)&\ge \sfD_t^-f(x).
      \end{aligned}
      \right.
    \end{equation}
  \item The maps $(x,t)\mapsto \sfD^\pm_t f(x)$ are Lusin
    $\mm\otimes \LL^1$-measurable in $X\times (0,\infty)$; moreover,
    for every $t>0$ the maps $x\mapsto \sfD^\pm_t f(x)$ are Lusin
    $\mm$-measurable in $X$.
  \end{enumerate}
\end{lemma}
\begin{proof}
  \textbf{(a)}
  Let us check the upper semicontinuity of $\sfD_\cdot^+ f$,
  by arguing as in the proof of Lemma \ref{le:HL1}(b)
  (the proof of the conditional lower semicontinuity of $\sfD_\cdot ^-
  f$ is
  completely analogous).
  We fix $c<\sfD^+_t f(x)$ and 
  we suppose that for some $\lambda_0\in \Lambda$
  $\sfD_{t_\lambda}^+(x_\lambda)
  \ge c$ for every $\lambda\succ \lambda_0$.
  We pick 
  $y_\lambda\in \sfJ_{t_\lambda}(x_\lambda)$ such that
\begin{equation}
  \label{eq:375bis}
  F(t_\lambda,x_\lambda,y_\lambda)=\sfQ_{t_\lambda}f(x_\lambda),
  \quad
  \sfd(x_\lambda,y_\lambda)=\sfD_{t_\lambda}^+(x_\lambda)\ge
  c.
\end{equation}
We can find a subnet $j\mapsto \lambda(j)$, $j\in J$,
such that $j\mapsto y_{\lambda(j)}$ converges to a point $y\in X$ with
\begin{displaymath}
  F(t,x,y)\le 
  \liminf_{j\in J}
  F(t_{\lambda(j)},x_{\lambda(j)},y_{\lambda(j)})=
  \liminf_{j\in J}\sfQ_{t_{\lambda(j)}}f(x_{\lambda(j)})
  \topref{eq:424}=
  \sfQ_tf(x),
\end{displaymath}
showing that $y\in \sfJ_t f(x)$ and $\lim_{j\in
  J}\sfd(x_{\lambda(j)},y_{\lambda(j)})=
\sfd(x,y)\ge c$.
It follows that $\sfD^+_t f(x)\ge \sfd(x,y)
\ge c.$
\medskip

\noindent\textbf{(b)}
Since the map $\sfQ f$ is lower semicontinuous, it is Lusin
$\mm\otimes \LL^1$ measurable
in $X\times (0,\infty)$ \cite[I.1.5, Theorem 5]{Schwartz73}.
For every compact set $K\subset (0,\infty)$ and
every $\eps>0$ we can find a compact subset $H_\eps\subset X\times K$
such that the restriction of $\sfQ f$
to $H_\eps$ is continuous and $\mm\otimes \LL^1\big((X\times
K)\setminus H_\eps\big)\le \eps/2$.
By the previous claim, we deduce that
the restriction of $\sfD^\pm f$ to $H_\eps$ are semicontinuous, and
thus Lusin $\mm\otimes \LL^1$-measurable:
therefore we can find a further
compact subset
$H_\eps'\subset H_\eps$ such that the restriction of $\sfD^\pm f$ to
$H_\eps'$ are continuous and $\mm\otimes \LL^1\big(H_\eps\setminus
H_\eps'\big)\le \eps/2$, so that
$\mm\otimes \LL^1\big((X\times
K)\setminus H_\eps'\big)\le \eps$. We conclude that $\sfD^\pm f$ are
Lusin $\mm\otimes \LL^1$-measurable.
The second statement can be proved by the same argument.
\end{proof}

\subsection{Invariance of the Cheeger energy with respect to $\AA$
when $(X,\tau)$ is compact}
\label{subsec:invariance-A-compact}
As a preliminary obvious remark, we observe that
if $\AA'\subset \AA''\subset\Lip_b(X,\tau,\sfd)$ are two algebras of Lipschitz functions
compatible with the metric-topological measure structure $\X=(X,\tau,\sfd,\mm)$ we
have
\begin{equation}
  \label{eq:360}
  \Sob^{1,p}(\X,\AA')\subset \Sob^{1,p}(\X,\AA'')\subset \Sob^{1,p}(\X),
\end{equation}
and for every $f\in \Sob^{1,p}(\X,\AA')$ 
\begin{equation}
  \label{eq:361}
  \CE_{p,\AA'}(f)\ge \CE_{p,\AA''}(f)\ge \CE_{p}(f),\quad
  |\rmD f|_{\star,\AA'}\ge
  |\rmD f|_{\star,\AA''}\ge |\rmD f|_\star\quad\text{$\mm$-a.e.~in }X.
\end{equation}
We will see that \eqref{eq:360} and \eqref{eq:361} can be considerably
refined,
obtaining the complete independence of the choice of $\AA$.
In this section we will focus on the case when $(X,\tau)$ is a compact
topological space; we suppose that $\AA$ is an algebra
compatible with $\X$,
we denote by $I$ the directed set of all the finite collections
$i\subset \AA_1$ satisfying $f\in i\ \Rightarrow\  -f\in i$
and we set
\begin{equation}
  \label{eq:370}
  \sfd_i(x,y):=\sup_{f\in i}f(x)-f(y),\quad
  \text{$i$ is a finite subset of $\AA_1$
    satisfying $f\in i\ \Rightarrow\  -f\in i$.}
\end{equation}
\begin{lemma}
  Let us suppose that $(X,\tau)$ is compact
  and $\AA$ is an algebra compatible with $\X$ generating the
  bounded continuous semidistances $(\sfd_i)_{i\in I}$
  as in \eqref{eq:370}.
  \label{le:step1}
  \begin{enumerate}
  \item For every $y\in X$ and $i\in I$ the map
    $h_i^y:x\to \sfd_i(x,y)$ belongs to $\Sob^{1,p}(\X,\AA)$ and
    $\relgradA {h_i^y}\le 1$ $\mm$-a.e.
  \item For every $f\in \rmC(X)$, $t>0$, $i\in I$, % and $K\subset X$
    % finite,
    the function $\sfQ_{t}^{\sfd_i} f$ belongs to
    $\Sob^{1,p}(\X,\AA)$ and
    \begin{equation}
      \label{eq:371}
      \relgradA{\sfQ_{t}^{\sfd_i} f}\le t^{-1} \sfD_{t}^{\sfd_i,+} f\quad\text{$\mm$-a.e.}.
    \end{equation}
    \item For every $f\in \rmC(X)$ and $t>0$
    % finite,
    the function $\sfQ_{t} f$ belongs to
    $\Sob^{1,p}(\X,\AA)$ and
    \begin{equation}
      \label{eq:371bis}
      \relgradA{\sfQ_{t} f}\le t^{-1}\sfD_{t}^{+} f\quad\text{$\mm$-a.e.}.
    \end{equation}
  \end{enumerate}
\end{lemma}
\begin{proof}
  Claim \textbf{(a)} immediately follows from Corollary \ref{cor:trivialsup}.

  \noindent \textbf{(b)}
  Let us denote by
  $\Pi$ the directed family of finite subset of $X$.
  For every $\pi\in \Pi$, the definition of $\sfQ_t^{\pi,\sfd_i}$ given in
  \eqref{eq:372},
  the chain rule \eqref{eq:64}, the previous claim, and
  % we can define
  % \begin{equation}
  %   \label{eq:377}
  %   f_{t,\pi}(x)=\min_{y\in \pi} f(y)+\frac
  %   {\sfd_i^q(x,y)}{q t^{q-1}}
  % \end{equation}
  % We also set
  % \begin{equation}
  %   \label{eq:378}
  %   \sfJ_{t,\pi}(x):=\Big\{y\in \pi: f(y)+\frac
  %   {\sfd_i^q(x,y)}{q t^{q-1}}=f_{t,\pi}(x)\Big\},\quad
  %   \sfD^+_{t,\pi}(x)=\max_{y\in \sfJ_{t,\pi}}\sfd_i(x,y).
  % \end{equation}
  Corollary \ref{cor:trivialsup} yield
  \begin{equation}
    \label{eq:379}
    \sfQ_t^{\pi,\sfd_i} f\in \Sob^{1,p}(\X,\AA),\quad
    \relgradA{\sfQ_t^{\pi,\sfd_i} f}\le
    g_{t}^{\pi,\sfd_i}:=\Big(t^{-1}\sfD^{\pi,\sfd_i+}_{t} f\Big)^{q-1}
  \end{equation}
  By Lemma \ref{le:HL1}(c), for every $x\in X$ and $t>0$
  $\pi\mapsto
  \sfQ_t^{\pi,\sfd_i} f(x)$ is a decreasing net, converging
  to
  $\sfQ_{t}^{\sfd_i}f(x)$.
  Thanks to \ref{eq:Beppo_Levi_general} and to the uniform bound
  \eqref{eq:390} (where $S:=\max_{x,y\in X}\sfd_i(x,y))$
  we have
  \begin{equation}
    \label{eq:419}
    \lim_{\pi\in \Pi}\int_X \big|\sfQ_t^{\pi,\sfd_i}
    f-\sfQ_{t}^{\sfd_i}f\big|\,\d\mm=
    \lim_{\pi\in \Pi}\int_X \sfQ_t^{\pi,\sfd_i}
    f\,\d\mm-\int_X \sfQ_{t}^{\sfd_i}f\,\d\mm=0.
  \end{equation}
  We can thus find an increasing sequence $n\mapsto \pi_n\in \Pi$ such
  that
  \begin{equation}
    \label{eq:420}
    \lim_{n\to\infty}\int_X \big|\sfQ_t^{\pi_n,\sfd_i}
    f-\sfQ_{t}^{\sfd_i}f\big|\,\d\mm
    =\lim_{n\to\infty}\int_X \big|\sfQ_t^{\pi_n,\sfd_i}
    f-\sfQ_{t}^{\sfd_i}f\big|^p\,\d\mm=0
  \end{equation}
  and a $\mm$-negligible subset $N\subset X$ such that
  \begin{equation}
    \label{eq:421}
    \lim_{n\to\infty}\sfQ_t^{\pi_n,\sfd_i}
    f(x)=\sfQ_{t}^{\sfd_i}f(x)\quad\text{for every }x\in X\setminus N.
  \end{equation}
  Applying \eqref{eq:385} we deduce 
  \begin{equation}
    \label{eq:380}
    \limsup_{n\to \infty}g_{t}^{\pi_n,\sfd_i}(x)\le
    g_{t}^{\sfd_i}(x):=\Big(t^{-1}\sfD^{\sfd_i+}_{t} f\Big)^{q-1}
    \quad\text{for every
    }x\in X\setminus N.
  \end{equation}
  By \eqref{eq:379}, it follows that any weak limit point in
  $L^p(X,\mm)$  of the (uniformly bounded) sequence
  $(\relgradA{\sfQ_t^{\pi_n,\sfd_i}f})_{n\in \N}$ will be bounded by
  $t^{-1}\sfD^{\sfd_i+}_{t} f$. By \eqref{eq:419} we obtain
  \eqref{eq:371}.
  \medskip

  \noindent\textbf{(c)}
  The argument is very similar to the previous one, but now using
  Proposition  \ref{prop:HL2} and the net $i\mapsto \sfd_i$.
\end{proof}
\begin{theorem}
  \label{thm:mainA}
  If $(X,\tau)$ is a compact space and $\AA\subset \Lip(X,\tau,\sfd)$
  is a compatible algebra
  according to definition \ref{def:A-compatibility}, then
  $\Sob^{1,p}(\X)=\Sob^{1,p}(\X,\AA)$
  with equal minimal relaxed gradient (and therefore equal Cheeger
  energy).
  Equivalently,
  for every $f\in \Sob^{1,p}(\X)$ there exists a sequence $f_n\in \AA$
  such that
  \begin{equation}
    \label{eq:422}
    f_n\to f,\quad
    \lip f_n\to \relgrad f\quad\text{strongly in }L^p(X,\mm).
  \end{equation}
\end{theorem}
\begin{proof}
  Let us denote by $\relgradA f$ (resp.~$\relgrad f$) the minimal relaxed gradient induced
  by $\AA$ (resp.~by $\Lip(X,\tau,\sfd)$).
  It is clear that $\Sob^{1,p}(X,\AA)\subset \Sob^{1,p}(X)$ and
  $\relgrad f\le \relgradA f$ for every $f\in \Sob^{1,p}(X,\AA)$;
  in order to prove the Theorem it is sufficient to show that
  every $f\in \Lip(X,\tau,\sfd)$ belongs to $\Sob^{1,p}(\X,\AA)$ with
  $\relgradA f\le \lip f$ $\mm$-a.e.

  We select an arbitrary Borel set $B\subset X$; 
  by using the uniform bound (a consequence of \eqref{eq:412})
  \begin{displaymath}
    \frac{f(x)-\sfQ_t f(x)}t\le \big(q \Lip (f,X)\Big)^p\quad\text{for
      every $x\in X$},
  \end{displaymath}
  the superior limit \eqref{eq:418} and Fatou's Lemma, we obtain
  \begin{equation}
    \label{eq:413}
    \frac 1p\int_B \big(\lip f(x)\big)^p\,\d\mm(x)\ge
    \limsup_{t\down0} \int_B \frac{f(x)-\sfQ_t f(x)}t\,\d\mm(x).
  \end{equation}
  On the other hand, \eqref{eq:401}, the measurability of $\sfD^+ f$
  given by Lemma \ref{le:HL4}, and Fubini's Theorem yield
  \begin{align*}
    \notag
    \int_B \frac{f(x)-\sfQ_t f(x)}t\,\d\mm(x)
    &=
      \frac 1p\int_{B\times (0,1)} \Big(\frac{\sfD^+_{tr}
      f(x)}{tr}\Big)^p
      \,\d (\mm\otimes \LL^1)(x,r)
      \\&=
    \frac 1p \int_0^1 \bigg(\int_{B} \Big(\frac{\sfD^+_{tr}
      f(x)}{tr}\Big)^p\,\d\mm(x)\bigg)
    \,\d r.
  \end{align*}
  A further application of Fatou's Lemma yields
  \begin{equation}
    \label{eq:425}
    \liminf_{t\downarrow0}
    \int_B \frac{f(x)-\sfQ_t f(x)}t\,\d\mm(x)\ge
    \frac 1p \liminf_{s\downarrow0}
    \int_{B} \Big(\frac{\sfD^+_{s}
      f(x)}{s}\Big)^p\,\d\mm(x),
  \end{equation}
  where we used the fact that for every $r\in (0,1)$
  \begin{displaymath}
    \liminf_{t\downarrow0}\int_{B} \Big(\frac{\sfD^+_{tr}
      f(x)}{tr}\Big)^p\,\d\mm(x)=\liminf_{s\downarrow0}
    \int_{B} \Big(\frac{\sfD^+_{s}
      f(x)}{s}\Big)^p\,\d\mm(x).
  \end{displaymath}
  Recalling \eqref{eq:371} and the fact that $\sfQ_s f\to f$ in
  $L^p(X,\mm)$ as $s\downarrow0$ we get
  \begin{equation}
    \label{eq:426}
    \liminf_{s\downarrow0}
    \int_{B} \Big(\frac{\sfD^+_{s}
      f}{s}\Big)^p\,\d\mm
    \ge \liminf_{s\downarrow0}
    \int_{B} \relgradA{\sfQ_s f}^p\,\d\mm
    \ge
    \int_B \relgradA f^p\,\d\mm.
  \end{equation}
  Combining \eqref{eq:413}, \eqref{eq:425} and \eqref{eq:426} we
  deduce that
  \begin{equation}
    \label{eq:427}
    \int_B \big(\lip f(x)\big)^p\,\d\mm(x)\ge
    \int_B \relgradA f^p(x)\,\d\mm(x)
  \end{equation}
  for every Borel subset $B$, so that $\lip f(x)\ge \relgradA f(x)$
  for $\mm$-a.e.~$x\in X$.
\end{proof}
\subsection{Notes}
\label{subsec:notes6}
\begin{notes}
  \Para{\ref{subsec:HL}}
  collects all the basic estimates concerning the Hopf-Lax flow
  in a general extended metric-topological setting.
  We followed the approach of \cite[\S\,3]{AGS14I}
  with some differences:
  we assumed compactness of $(X,\tau)$
  (as in \cite{AGS13,GRS15}),
  we considered general exponents $q=p'\in (1,\infty)$ as
  in \cite{AGS13}, and we
  devoted some effort to study the
  dependence of the
  Hopf-Lax formula on the distance and on the minimizing set,
  a point of view that has also been
  used in \cite{AGS15,AES16}.
  In this respect, compactness plays an essential role.
  Differently from \cite{AGS14I}, the Hopf-Lax
  flow is not used as a crucial ingredient for the so-called Kuwada
  Lemma \cite[\S\,6]{AGS14I}, \cite{Kuwada13},
  but as a powerful approximation tool
  of general functions by elements in the algebra $\AA$.  

  \Para{\ref{subsec:invariance-A-compact}}
  contains the crucial results which justify
  the study of the Hopf-Lax formula in our setting:
  Lemma \ref{le:step1}
  provides a crucial estimate of $\relgradA {\sfQ_t f}$
  for the regularized functions and
  Theorem \ref{thm:mainA} shows that
  compatible algebra are dense (in energy) in
  $\Sob^{1,p}(\X)$.
  Both the results are inspired by
  the techniques of \cite[Theorem 3.12]{AGS15}
  and of \cite[\S\,12]{AES16}.
\end{notes}

\part{$p$-Modulus and nonparametric
  dynamic plans}

\begin{Assumption}
  \label{ass:mainIII}
  As in the previous section, we consider an extended
  metric-topological measure space $\X=(X,\tau,\sfd,\mm)$
  and we fix an exponent $p\in (1,+\infty)$.
  $\RA(X)=\RA(X,\sfd)$ is the space of rectifiable arcs with the
  quotient topology $\tau_\rmA$ studied in \S\,\ref{subsec:arcs}, see
  Proposition \ref{prop:A-extended}.
\end{Assumption}

\section{$p$-Modulus of a family of measures and of
  a family of rectifiable arcs}
\label{sec:Modulus}

\newcommand{\intmu}[1]{\int_X #1\,\d\mu}
\subsection{$p$-Modulus of a family of Radon measures}
\label{subsec:pModmeas}
\index{Modulus of a family of Radon measures}
Given $\Sigma \subset \cMp(X)$ we define (with the usual convention $\inf\emptyset=\infty$)
\begin{equation} \label{eqn:mod2}
\Mdm (\Sigma) := \inf \left\{ \int_X f^p \, \d \mm \, : \, f
  \in \Ldp,\,\, \intmu  f 
  \geq 1 \; \; \text{ for all } \mu \in \Sigma \right\},
\end{equation}
\begin{equation} \label{eqn:mod2c}
\Mdmc (\Sigma) := \inf \left\{  \int_X f^p \, \d \mm \, :
  \, f \in \rmC_b(X),\,\,  \intmu  f 
  \geq 1 \; \; \text{ for all } \mu \in \Sigma \right\}.
\end{equation}
% Equivalently, if $0<{\rm Mod}_{p,\smm}(\Gamma)\leq\infty$, we can say that ${ \rm Mod}_p (\Gamma)^{-1}$ is the least
% number $\xi\in [0,\infty)$ such that 
% \begin{equation}\label{eqn:mod2dual} 
% \left( \inf_{ \mu \in \Gamma} \int_X f \, \d \mu \right)^p
% \leq  \xi \int_X f^p \, \d \mm \quad 
% \text{for all $f \in \Ldp$},
% \end{equation}
% and similarly there is also an equivalent definition for ${ \rm Mod}_{p,\smm,c} (\Gamma)^{-1}$. \\
Since the infimum in \eqref{eqn:mod2c} is unchanged if we restrict the
minimization to nonnegative functions $f\in \rmC_b(X)$ we get 
$\Mdmc (\Sigma) \geq \Mdm (\Sigma)$.
Also, whenever $\Sigma$ contains the null measure, 
we have $\Mdmc (\Sigma) = \Mdm (\Sigma)=\infty$,
whereas $\Mdmc(\emptyset)=0$.
\index{Negligible set of rectifiable arcs}
\begin{definition}[$\Mdm$-negligible
  sets and properties $\Mdm$-a.e.]\label{dfn:modnull} 
A set $\Sigma \subset\cMp(X)$ is said to be 
${\rm Mod}_{p}$-negligible if ${\rm Mod}_{p}(\Sigma)=0$.\\
We say that a property $P$ on $\cMp(X)$ holds ${\rm
  Mod}_{p}$-a.e. if the set 
of measures where $P$ fails is
${\rm Mod}_{p}$-negligible.
\end{definition}
%
% With the above terminology, we can also write
% \begin{equation} \label{eqn:mod2bis}
% { \rm Mod}_{p} (\Sigma) = \inf \left\{  \int_X f^p \, \d \mm \, : \, \int_X  f \, \d \gamma \geq 1 \; \; \text{ for ${\rm Mod}_{p}$-a.e. } \gamma \in \Sigma \right\}.
% \end{equation}
The next result collects various well known properties of the Modulus,
see e.g.~\cite{Bjorn-Bjorn11}, \cite[\S\,5.2]{HKST15}.
\index{Fuglede's Lemma}
\begin{proposition}\label{prop:prop} The set functions $\Sigma\subset\cMp(X)\mapsto\Mdm(\Sigma)$,
$\Sigma\subset\cMp(X)\mapsto\Mdmc(\Sigma)$ satisfy the following properties:
\begin{enumerate}
\item both are monotone and subadditive.
\item If $g\in\Ldp$ then $ \intmu g < \infty$ for
  $\Mdm$-almost every $\mu$; conversely, if 
  $\Mdm(\Sigma)=0$ then there exists $g\in \Ldp$ such that 
  $\intmu g=\infty$ for every $\mu\in \Sigma$.
\item {\upshape [Fuglede's Lemma]} If $(f_n) \subset \Ldp $ converges in $L^p(X,\mm)$ seminorm to $f \in \Ldp$,  
there exists a subsequence $(f_{n(k)})$ such that
\begin{equation} \label{convqo}
  \lim_{k\to\infty}\intmu {|f_{n(k)}-f|}=0
  % \rightarrow \intmu f
  \qquad \Mdm \text{-a.e. in $\cMp(X)$}.
\end{equation}
\item For every $\Sigma\subset\cMp(X)$
  with $\Mdm(\Sigma)<\infty$ there exists $f\in \Ldp$, unique up to $\mm$-negligible sets, 
such that $\intmu f\geq 1$ $\Mdm$-a.e. on $\Sigma$ and
$\|f\|^p_p=\Mdm(\Sigma)$.
% \item %{[Subadditivity]}
%   If $\Sigma_n\subset \cMp(X)$, $n\in \N$, then
%   $\Mdm(\cup_n \Sigma_n)\le \sum_n \Mdm(\Sigma_n)$.
\item If $\Sigma_n$ are nondecreasing subsets of $\cMp(X)$ then $\Mdm (\Sigma_n) \uparrow \Mdm (\cup_n \Sigma_n)$.
\item If $K_n$ are nonincreasing compact subsets of
  $\cMp(X)$ then $\Mdmc (K_n) \downarrow \Mdmc (\cap_nK_n)$.
\end{enumerate}
\end{proposition}
\begin{proof} 
  We repeat almost word by word the arguments of \cite[Proposition 2.2]{ADS15}.
  % Only the proof of the last statement \textbf{(f)} is
  % slightly different, due to the different notion of compactness.
  \medskip

  \noindent\textbf{(a)}
  Monotonicity is an obvious consequence of the definition.
  For the subadditivity, if we take two sets $A,B\subset \cMp(X)$
  and two functions $f,g\in \Ldp$ with $\intmu f\geq 1$ for every
  $\mu\in A$ and
  $\intmu g\geq 1$ for every $\mu\in B$,
  then the function $h:=(f^p+g^p)^{1/p}\ge \max(f,g)$ still satisfies
  $\intmu h \geq 1$ for every  $\mu\in A \cup B$, hence
  $$\Mdm( A \cup B) \leq \int_X h^p\,\d\mm
  =\int_X f^p\,\d\mm+\int_Xg^p\,\d\mm.$$
  Minimizing over $f$ and $g$ we get the subadditivity.
  \medskip

  \noindent\textbf{(b)}
  If we consider the set where the property fails
  $$ \Sigma = \left\{\mu\in \cMp(X): \intmu g = \infty \right\}, $$
  then it is clear that for every $k>0$
  we have 
  $\Sigma = \big\{\mu\in \cMp(X): \intmu {kg} = \infty \big\}$
  so that
  $\Mdm(\Sigma)\leq k^p\|g\|^p_p$ for every $ \kappa >0$ and we deduce
  that $\Mdm(\Sigma)=0$. 

  Conversely, if $\Mdm(A)=0$ for every $n\in\N$ we can find $g_n\in \Ldp$
with
$\intmu {g_n}\ge1$ for every $\mu\in A$ and 
$\int_X g_n^p\,\d\mm\le 2^{-np}$. Thus $g:=\sum_n g_n$ satisfies the 
required properties.
\medskip

\noindent \textbf{(c)}
Let $f_{n(k)}$ be a subsequence such that $\|f- f_{n(k)}\|_p\leq
2^{-k}$.
If we set
$$ g (x) := \sum_{k =1}^{\infty} | f(x) - f_{n(k)} (x) |$$
we have that $g \in \Ldp$ and $\|g\|_{L^p(X,\mm)}\leq 1$;
in particular we have, for Claim (b) above, that $\intmu g $ is finite for $\Mdm$-almost every $\mu$.
For those $\mu$ we get
$$ \sum_{k=1}^\infty \intmu {|f - f_{n(k)} |}  <\infty $$
which yields  \eqref{convqo}.
\medskip

\noindent\textbf{(d)}
Claim (b) shows in particular that
\begin{equation} \label{eqn:mod2bis}
  \Mdm (\Sigma) =
  \inf \left\{  \int_X f^p \, \d \mm \, : \, \intmu  f \geq 1 \;
    \; \text{ for $\Mdm$-a.e. }
    \mu \in \Sigma \right\},
\end{equation}
so that by Claim (c) 
the class of admissible functions $f$ involved in the variational
definition od $\Mdm$ is a convex and closed subset of the Lebesgue space $L^p(X,\mm)$.
Hence, uniqueness follows by the strict convexity of the $L^p$-norm.
\medskip

\noindent\textbf{(e)}
By the monotonicity, it is clear that $\Mdm (A_n)$ is an increasing
sequence and that setting $M:=\lim_{n\to\infty} \Mdm(A_n)$ we have
$M\ge \Mdm(\cup_n A_n) $. 

If $M= \infty$ there is nothing to prove, otherwise, we need to show
that $\Mdm (\cup_nA_n) \leq M$;
let $(f_n) \subset \Ldp$ be a sequence of functions 
such that $ \intmu f_n \geq 1$ on $A_n$ and $\|f_n\|^p_{L^p(X,\mm)} \leq \Mdm (A_n) + \frac 1n$.
In particular we get that $\limsup_n \|f_n\|^p_p =M < \infty$ and so, possibly extracting a subsequence, we can assume that
$f_n$ weakly converge to some $f \in \Ldp$.
By Mazur lemma we can find convex combinations
$$ \hat{f}_n  = \sum_{k=n}^{\infty} \kappa_{k,n} f_k $$
such that $\hat{f}_n$ converge strongly to $f$ in $L^p(X, \mm)$;
furthermore we have that
$\intmu {f_k}\geq 1$ on $A_n$ if $k\geq n$ and so 
$$ \intmu{\hat{f}_n} = \sum_{k=n}^{\infty} \kappa_{k,n} \intmu
f_k\geq 1 \qquad \text{ on }A_n.$$
By Claim (c) above we obtain a subsequence $n(k)$ and a
$\Mdm$-negligible set
$\Sigma \subset \cMp(X)$ such that 
$\intmu{\hat{f}_{n(k)}}\to\intmu f$ outside $\Sigma$;
in particular $\intmu f \geq 1$ on $\cup_nA_n \setminus \Sigma$.

By Claim (b) we can find
$g \in \Ldp$ such that $ \intmu g=\infty$ on $\Sigma$, so that we
have
$ \intmu{(f+\eps g)} \geq 1$ on $\cup_n A_n$ and 
$$ \Mdm ( \cup_nA_n)^{1/p} \leq \|\eps g+ f\|_p \leq  \eps \|g\|_p+ \|f
\|_p \leq \eps \|g\| + \liminf \|f_n\|_p \leq \eps\|g\|_p + M^{1/p}. $$
Letting $\ep \to 0$ and taking the $p$-th. power the inequality $\Mdm(A)\leq\sup_n\Mdm(A_n)$ follows.
\medskip

\noindent \textbf{(f)}
By the monotonicity we get $\Mdmc ( K) \leq \Mdmc(K_n)$;
if $C:=\lim_{n\to\infty}\Mdmc(K_n)$, we only have to prove $\Mdmc(K)
\geq C$ and it is not restrictive to assume $C>0$.
We argue by contradiction: if $\Mdmc(K)<C$
we can find a nonnegative $\psi\in \rmC_b(X)$ such that
$\intmu \psi \ge1$ for every $\mu\in K$ and
$\alpha^{p}=\int_X \psi^p\,\d\mm<C$. Setting $\phi:=\alpha^{-1} \psi$
we obtain a function $\phi\in \rmC_b(X)$
with 
% Let us fix $\theta>0$ such that $\theta^{-p}>\Mdmc(K)$.
% Equivalently, if $0<{\rm Mod}_{p,\smm}(\Sigma)\leq\infty$, we can say that ${ \rm Mod}_p (\Sigma)^{-1}$ is the least
% number $\xi\in [0,\infty)$ such that 
% \begin{equation}\label{eqn:mod2dual} 
% \left( \inf_{ \mu \in \Sigma} \int_X f \, \d \mu \right)^p
% \leq  \xi \int_X f^p \, \d \mm \quad 
% \text{for all $f \in \Ldp$},
% \end{equation}
% and similarly there is also an equivalent definition for ${ \rm Mod}_{p,\smm,c} (\Sigma)^{-1}$. \\
%
% Using 
% \eqref{eqn:mod2dual} we can find $\phi\in \rmC_b(X)$ such that 
$\| \phi \|_p =1$ and
$$\inf_{\mu\in K} \intmu\phi %=\min_K \Phi(\phi_{\ep}) 
>\alpha^{-1}. $$
By the compactness of $K$ the infimum above is a minimum; since $K_n$
is decreasing, 
\begin{displaymath}
  \min\limits_{\mu\in K_n} \intmu \phi
  \to \min\limits_{\mu\in K} \intmu\phi.
\end{displaymath}
It follows that there exists $\bar n\in \N$ such that 
$\min\limits_{\mu\in K_n} \intmu \phi\ge \alpha^{-1}$ for every $n\ge
\bar n$ 
so that using $\psi=\alpha\phi$ we deduce that
$\Mdmc (K_n)\le \alpha^{p}<C$ for $n\ge\bar n$, a contradiction.%  and therefore
 % $C\le \theta^{-p}$. Since $\theta$
 % is arbitrary, we conclude that $C\le \Mdmc(K)$.
\end{proof}

% \begin{remark} \label{rem:saturated} {\rm In connection with Proposition~\ref{prop:prop}(iv), in general the constraint $\int_Xf\,d\mu\geq 1$ is not saturated
% by the optimal $f$, namely the strict inequality can occur for a subset $\Sigma_0$ with positive $(p,\mm)$-modulus. For instance, if
% $X=[0,1]$ and $\mm$ is the Lebesgue measure, then
% $$
% \Mdm\bigl(\{\Leb{1}\res [0,\frac 12],\Leb{1}\res [\frac 12,1],\Leb{1}\res [0,1]\}\bigr)=2^p\quad\text{and}\quad f\equiv 2,
% $$
% but $\int_X f\,\d\mm=2$.
% However, we will prove using the duality formula $\Mdm=C^p_{p}$ that one can always find a subset $\Sigma'\subset\Sigma$ (in the example
% above $\Sigma\setminus\Sigma'=\{\Leb{1}\res [0,1]\}$) with the same
% $(p,\mm)$-modulus satisfying $\int_Xf\,d\mu=1$ for all $\mu\in\Sigma'$, see the comment made after Corollary~\ref{cor:cnes}.
% On the other hand, if the measures in $\Sigma$ are non-atomic, using just the definition of $p$-modulus, one can find instead a family $\Sigma'$ of \emph{smaller}
% measures with the same modulus as $\Sigma$ on which the constraint is saturated: suffices to find, for any $\mu\in\Sigma$, a smaller
% measure $\mu'$ (a sub-curve, in the case of measures associated to curves) satisfying $\int_X f\,d\mu'=1$. In the previous example the two constructions lead to the same result, but the two procedures
% are conceptually quite different.}
% \end{remark}
%
Another important property is the tightness of $\Mdm$ in $\cMp(X)$.
% We adapt the argument of 
% \cite[Lemma 2.3]{ADS15}, with some modification due to the different
% notion
% of compactness.
\index{Tightness of $\Mdm$}
\begin{lemma}[Tightness of $\Mdm$] \label{lem:tightness}
  % Let us suppose that $(X,\sfd)$ is complete
  % and for $\eta>0$
  % let us consider the open subset $\RA_\eta(X):=\Big\{\mu\in \cMp(X):
  % \ell(\mu)>\eta\Big\}$.
  For every $\ep>0$ there exists
  $K_{\ep}\subset\cMp(X)$ compact such that $\Mdm(
  \cMp(X)\setminus K_\ep)\leq\ep$.
\end{lemma}
\begin{proof} Since $\mm$ is a Radon measure
  we can find an nondecreasing family of $\tau$-compact sets
  $K_n\subset X$
  such that $m_n=\mm(X\setminus K_n)>0$, 
  $\lim_{n\to\infty}\mm(X\setminus K_n)=0.$
  We set
  \begin{equation}
    \label{eq:544}
    \delta_n=(\sqrt{m_n}+\sqrt{m_{n+1}})^{1/p},\quad
    a_n:=\delta_n^{-1},
  \end{equation}
  observing that $\delta_n\down0$ and $a_n\up+\infty$ as $n\to\infty$.
  % Let us consider a decreasing vanishing sequence $\delta_n\down0$
  For $k\in \N$ let us now define the sets
  \begin{equation}
    E_k :=\Big\{ \mu \in \cMp(X) :  \mu(X) \leq k,\ \ 
    %\sfe(\mu)\cap K_k\neq\emptyset,\ \
    \mu(X\setminus K) \leq \delta_n \ \text{for every } n \geq k
    \Big\},\label{eq:545}
\end{equation}
  % $\lim_{k\to\infty}\Mdm(E_k^c) = 0$.
  which are compact in $\cMp(X)$ by Theorem \ref{thm:compa-Riesz}.
  
  To evaluate $\Mdm(\cMp(X)\setminus E_k)$
  we introduce the functions
  \begin{equation}
f_k(x) : =\begin{cases} 0  & \text{ if }x \in K_k, \\ a_n & \text{ if
  }x \in K_{n+1} \setminus K_n \text{ and }n\geq k, \\ + \infty
  &\text{otherwise.} \end{cases}\label{eq:596}
\end{equation}
We observe that if $\mu \in \cMp(X)\setminus E_k$
then we have either $\mu(X)>k $
%or
%$\sfe(\mu)\subset X\setminus K_k$
or $\mu(X\setminus
K_n) > \delta_n$ for some $n \geq k$.
In either case the integral of the function $f_k+ \frac 1k$
along $\mu$ is greater or equal to $1$:
\begin{itemize}
\item if $\mu(X)>k$ then
  $$\intmu {\Big(f_k+ \frac 1k \Big)}
  \geq \frac1k \ell(\mu) \geq 1;$$
% \item if $\sfe(\mu)\cap K_k=\emptyset$ then
%   \begin{displaymath}
%     \intmu \left(f_k+ \frac 1k \right)
%     \ge a_k \ell(\mu)\ge \eta a_k\ge 1\quad \text{if}\quad
%     a_k\ge \eta^{-1};
%   \end{displaymath}
\item if $\mu(X\setminus K_n)> \delta_n$
  for some $n\geq k$ we have that
  $$ \intmu {\Big(f_k + \frac 1k \Big)}
  \geq \int_{X\setminus K_n} f_k \, \d \mu
  \geq \int_{X\setminus K_n} a_n \, \d \mu > \delta_n a_n=1.$$
\end{itemize}
So we have that 
$\Mdm(\cMp(X)\setminus E_k)
\leq \|f_k+ \frac 1k\|^p_p \leq ( \|f_k\|_p + \| 1/k \|_p )^p$. But
\begin{align}
  \int_X f_k^p \,\d\mm &=\sum_{n=k}^{\infty}\int\limits_{K_{n+1} \setminus K_n} a_n^p \, \d \mm = 
\sum_{n=k}^{\infty}  \frac{m_{n} - m_{n+1}} { \sqrt{ m_{n} } +
  \sqrt{m_{n+1}}} =  \sum_{n=k}^{\infty} ( \sqrt{m_n} - \sqrt{m_{n+1}}
                         ) = \sqrt{m_k},\label{eq:543}\\
  \|f_k+ \frac 1k\|&\le (m_k)^{1/(2p)} + (\mm(X))^{1/p}/k
                     \label{eq:547}
\end{align}
and therefore we obtain $\Mdm(\cMp(X)\setminus E_k)
\leq \Bigl( (m_k)^{1/(2p)} + (\mm(X))^{1/p}/k \Bigr)^p \to 0$.
\end{proof}

\subsection{$p$-Modulus of a family of rectifiable arcs}

\index{Modulus of a family of rectifiable arcs}
There is a natural way to lift the notion of Modulus
for a family of Radon measures in $\cMp(X)$ to 
a corresponding Modulus for a collection of rectifiable arcs: it is sufficient to
assign a map $M:\RA(X)\to \cMp(X)$ and for every $\Gamma\subset \RA(X)$ set
$\operatorname{Mod}_{p,M} (\Gamma):=\Mdm(M(\Gamma)).$
Clearly such a notion depends on the choice of $M$; in these notes we
will consider two (slightly) different situations:
the first one correspond to the most classic and widely used choice
of the $\mathscr H^1$-measure carried by $\gamma$ of \eqref{eq:232}
\begin{equation}
  \label{eq:593}
  \rmM\gamma:= \nu_\gamma,\quad
  \nu_\gamma:=\ell(\gamma)(\Al_\gamma)_\sharp(\Leb 1\res[0,1]).
\end{equation}
In this case we will keep the standard notation of $\Md (\Gamma),\
\Mdc(\Gamma)$;
e.g.~in the case of $\Md$ \eqref{eqn:mod2} reads
\begin{equation} \label{eqn:mod2gamma}
  \Md(\Gamma) := \Mdm(\rmM(\Gamma))=\inf \left\{ \int_X f^p \, \d \mm \, : \, f
  \in \Ldp,\,\, \int_\gamma  f 
  \geq 1 \; \; \text{ for all } \gamma \in \Gamma \right\},
\end{equation}
with obvious modification for $\Mdc(\Gamma)$.
% \begin{equation} \label{eqn:mod2c}
% { \rm Mod}_{p,c} (\Gamma) := \inf \left\{  \int_X f^p \, \d \mm \, :
%   \, f \in \rmC_b(X),\,\,  \int_\gamma  f 
%   \geq 1 \; \; \text{ for all } \gamma \in \Gamma \right\}.
% \end{equation}
% Equivalently, if $0<{\rm Mod}_{p,\smm}(\Gamma)\leq\infty$, we can say that ${ \rm Mod}_p (\Gamma)^{-1}$ is the least
% number $\xi\in [0,\infty)$ such that 
% \begin{equation}\label{eqn:mod2dual} 
% \left( \inf_{ \mu \in \Gamma} \int_X f \, \d \mu \right)^p
% \leq  \xi \int_X f^p \, \d \mm \quad 
% \text{for all $f \in \Ldp$},
% \end{equation}
% and similarly there is also an equivalent definition for ${ \rm Mod}_{p,\smm,c} (\Gamma)^{-1}$. \\
The second choice corresponds to
\begin{equation}
  \label{eq:594}
  \tilde \rmM\gamma=\tilde\nu_\gamma:=\nu_\gamma+
  \delta_{\gamma_0}+\delta_{\gamma_1},\quad
  \int_X f\,\d(\tilde\nu_\gamma)=f(\gamma_0)+f(\gamma_1)+\int_\gamma f,
\end{equation}
where as usual we write the initial and final points of $\gamma$ as
$\gamma_i=\sfe_i(\gamma)=\Al_\gamma(i)$, $i=0,1$.
We will denote by $\tMd(\Gamma)$ the corresponding modulus,
\begin{equation} \label{eqn:mod2gammabis}
  \begin{aligned}
    &\tMd (\Gamma) :={} \Mdm(\tilde\rmM(\Gamma))
    \\&={}\inf \left\{ \int_X f^p
      \, \d \mm \, : \, f \in \Ldp,\,\,
      f(\gamma_0)+f(\gamma_1)+\int_\gamma f \geq 1 \; \; \text{ for
        all } \gamma \in \Gamma \right\}.
  \end{aligned}
\end{equation}
It is clear that
\begin{equation}
  \label{eq:595}
  \tMd(\Gamma)\le \Md(\Gamma),\quad
  \tMd(\Gamma)\le 2^{-p}\mm(X)\quad
  \text{for every }\Gamma\subset \RA(X).
\end{equation}
One main difference between $\Md$ and $\tMd$
is the behaviour on constant arcs:
if $\Gamma$ contains a constant arc than it is clear that
$\Md(\Gamma)=+\infty$, whereas
for a collection $\Gamma_A=\{\gamma_x:x\in A\}$
of constant arcs parametrized by a Borel set $A\subset X$
we have $\tMd(\Gamma_A)=2^{-p}\mm(A)$.

The notions of $\Md$- or $\tMd$-negligible set of arcs
(and of
properties which hold $\Md$- or $\tMd$-a.e.)
follow accordingly from
Definition \ref{dfn:modnull}.
\index{Negligible set of rectifiable arcs}
% \begin{definition}[${\rm Mod}_{p}$-negligible
%   sets and properties $\Mod_p$-a.e.]\label{dfn:modnull} 
% A set $\Gamma \subset\RA(X)$ is said to be 
% ${\rm Mod}_{p}$-negligible if ${\rm Mod}_{p}(\Gamma)=0$.\\
% We say that a property $P$ on $\RA(X)$ holds ${\rm
%   Mod}_{p}$-a.e. if the set 
% of arcs where $P$ fails is
% ${\rm Mod}_{p}$-negligible.
% \end{definition}
%
% With the above terminology, we can also write
% \begin{equation} \label{eqn:mod2bis}
% { \rm Mod}_{p} (\Gamma) = \inf \left\{  \int_X f^p \, \d \mm \, : \, \int_X  f \, \d \gamma \geq 1 \; \; \text{ for ${\rm Mod}_{p}$-a.e. } \gamma \in \Gamma \right\}.
% \end{equation}
% The next result collects various well known properties of the Modulus,
% see e.g.~\cite{Bjorn-Bjorn11}, \cite[\S\,5.2]{HKST15}.
% \index{Fuglede's Lemma}
Properties (a-e) of Proposition \ref{prop:prop} have an obvious
version for arcs. The only statements that require some care are
Proposition \ref{prop:prop}(f) and Lemma \ref{lem:tightness},
since compactness in $\RA(X)$ for subsets $\Gamma$ of arcs is not equivalent to compactness 
for the corresponding subsets $\rmM(\Gamma),\ \tilde\rmM(\Gamma)$ in
$\cMp(X)$.

Concerning the validity of Proposition \ref{prop:prop}(f),
it is sufficient to note that for every nonnegative $\phi\in
\rmC_b(X)$ the maps
\begin{align*}
  \gamma&\mapsto \int\phi\,\d(\rmM\gamma)=\int_\gamma\phi\\
  \gamma&\mapsto \int\phi\,\d(\tilde\rmM\gamma)=\phi(\gamma_0)+\phi(\gamma_1)+\int_\gamma\phi
\end{align*}
are lower semicontinuous, thanks to Theorem \ref{thm:important-arcs},
so that the argument of the proof works as well.
\begin{corollary}
  \label{cor:Moduno}
  If $\rmK_n$ is a nonincreasing sequence of compact sets in $\RA(X)$ we
  have 
  $\Mdc (K_n) \downarrow \Mdc (\cap_nK_n)$,
  $\tMdc (K_n) \downarrow \tMdc (\cap_nK_n)$.  
\end{corollary}
Concerning the tightness Lemma \ref{lem:tightness} we have:
\index{Tightness of $\Md$ and of $\tMd$}
\begin{lemma}[Tightness of $\Md$ and $\tMd$] \label{lem:tightness2}
  Let us suppose that $(X,\sfd)$ is complete.
  \begin{enumerate}
  \item
    For every $\ep>0$ there exists
    $K_{\ep}\subset\RA(X)$ compact such that $\tMd(
    \RA(X)\setminus K_\ep)\leq\ep$.
  \item
  For every $\eta,\ep>0$ there exists
  $K_{\ep}\subset\RA(X)$ compact such that $\Md(
  \RA_\eta(X)\setminus K_\ep)\leq\ep$ (where $\RA_\eta(X)$
  has been defined in \eqref{eq:RAk}).      
  \end{enumerate}
\end{lemma}
\begin{proof}
  \textbf{(a)}
  We can repeat verbatim the proof of Lemma \ref{lem:tightness}:
  keeping the same notation,
  the main point is that the sets
  %
  %
  % Since $\mm$ is a Radon measure
  % we can find an nondecreasing family of $\tau$-compact sets
  % $K_n\subset X$
  % such that $m_n=\mm(X\setminus K_n)>0$, 
  % $\lim_{n\to\infty}\mm(X\setminus K_n)=0.$
  % We set
  % \begin{equation}
  %   \label{eq:544}
  %   \delta_n=(\sqrt{m_n}+\sqrt{m_{n+1}})^{1/p},\quad
  %   a_n:=\delta_n^{-1},
  % \end{equation}
  % observing that $\delta_n\down0$ and $a_n\up+\infty$ as $n\to\infty$.
  % % Let us consider a decreasing vanishing sequence $\delta_n\down0$
  % For $k\in \N$ let us now define the sets
  \begin{equation}
    \tilde\rmM^{-1}(E_k)
    :=\Big\{ \gamma \in \RA(X) :  2+\ell(\gamma) \leq k,\ \ 
    \tilde\rmM\gamma(X\setminus K_n)
    \leq \delta_n \ \text{for every } n \geq k
    \Big\},\label{eq:545bisbis}
\end{equation}
  % $\lim_{k\to\infty}\Md(E_k^c) = 0$.
are compact in $\RA(X)$ by Theorem \ref{thm:important-arcs}(g):
in fact, since
$\tilde\rmM\gamma=\delta_{\gamma_0}+\delta_{\gamma_1}+\nu_\gamma$,
whenever $\delta_n<1$ condition \eqref{eq:545bisbis} yields
$\gamma_0,\gamma_1\in K$, so that also the
assumption 2. of Theorem \ref{thm:important-arcs}(g)
holds.
\medskip

\noindent\textbf{(b)}
In this case the set $\rmE_k:=\rmM^{-1}(E_k)$ cannot be compact in general,
since it contains all the constant curves, so that there is no hope
to construct inner compact approximations $K_\eps$ such that
$\Md(\RA(X)\setminus K_\eps)\downarrow0$.
We thus replace $\RA(X)$ with $\RA_\eta(X)$, $\eta>0$, and
modify the definition of the sets
$E_k$ by requiring that the support of $\mu$ intersect $K_k$;
in terms of $\rmE_k=\rmM^{-1}(E_k)$
this corresponds to 
\begin{equation}
  \rmE_k :=\Big\{ \gamma \in \RA(X) :  \ell(\gamma) \leq k,\ \ 
  \sfe(\gamma)\cap K_k\neq\emptyset,\ \
  \nu_\gamma(X\setminus K_n)) \leq \delta_n \ \text{for every } n \geq k
  \Big\},\label{eq:545}
\end{equation}
which are compact in $\RA(X)$ by Theorem \ref{thm:important-arcs}(g).

To evaluate $\Md(\RA_\eta(X)\setminus E_k)$
we introduce the functions $f_k$ as in \eqref{eq:596}
and we observe that if $\gamma \in \RA_\eta(X)\setminus E_k$
then we have either $\ell(\gamma) >k $ or
$\sfe(\gamma)\subset X\setminus K_k$ or $\nu_\gamma(X\setminus
K_n) > \delta_n$ for some $n \geq k$.
In either case the integral of the function $f_k+ \frac 1k$
along $\gamma$ is greater or equal to $1$:
we have just to check the case
$\sfe(\gamma)\cap K_k=\emptyset$, for which
%\begin{itemize}
% \item if $\ell(\gamma)>k$ then
%   $$\int_\gamma \left(f_k+ \frac 1k \right)
%   \geq \frac1k \ell(\gamma) \geq 1;$$
% \item if $\sfe(\gamma)\cap K_k=\emptyset$ then
  \begin{displaymath}
    \int_\gamma \left(f_k+ \frac 1k \right)
    \ge a_k \ell(\gamma)\ge \eta a_k\ge 1\quad \text{if}\quad
    a_k\ge \eta^{-1};
  \end{displaymath}
% \item if $\nu_\gamma(X\setminus K_n)> \delta_n$
%   for some $n\geq k$ we have that
%   $$ \int_\gamma \left(f_k + \frac 1k \right)
%   \geq \int_{X\setminus K_n} f_k \, \d \nu_\gamma
%
%  \geq \int_{X\setminus K_n} a_n \, \d \nu_\gamma > \delta_n a_n=1.$$
% \end{itemize}
For sufficiently big $k$ we thus obtain the same estimates \eqref{eq:543} and \eqref{eq:547}.
\end{proof}
\subsection{Notes}
\label{subsec:notes7}
\begin{notes}
  The notion of $p$-Modulus has been introduced by Fuglede \cite{Fuglede57}
  in the natural framework of
  collection of positive 
  measures, as in \cite{ADS15}. Its application to
  the metric theory of Sobolev spaces has been proposed in
  \cite{Koskela-MacManus98} and further studied in
  \cite{Shanmugalingam00},
  where the definition of Newtonian spaces has been introduced.
  We refer to \cite{Bjorn-Bjorn11,HKST15}
  for a comprehensive presentation of this topic.

  The tightness estimate for $\Mdm$ in 
  $\cMp(X)$ has been introduced by \cite{ADS15},
  where it plays a crucial role. Here we 
  used the same approach to derive 
  tightness estimates directly in $\RA(X)$, 
  for the two relevant embeddings of 
  $\RA(X)$ in $\cMp(X)$ giving raise to $\Md$ and $\tMd$.
  
\end{notes}

\section{(Nonparametric) Dynamic plans with barycenter in
  $L^q(X,\mm)$}

\label{sec:DynamicPlans}
Let us keep the main Assumption of page \pageref{ass:mainIII}.
We denote by $q=p'=p/(p-1)\in (1,\infty)$
the conjugate exponent of $p$.

\index{Nonparametric dynamic plans}
\index{Dynamic plans}
\begin{definition}[(Nonparametric) dynamic plans]
  \label{def:dynplan}
  A (nonparametric) dynamic plan is a Radon measure
  $\ppi\in \cMp(\RA(X))$ on $\RA(X)$ such that
  \begin{equation}
    \label{eq:430}
    \ppi(\ell):=\int_{\RA(X)}\ell(\gamma)\,\d\ppi(\gamma)<\infty.
  \end{equation}
\end{definition}
Since $\RA(X)$ is a $F_\sigma$-subset of $\Arc(X)$, a dynamic plan
can also be considered as the restriction of
a Radon measure $\ppi'$ on $\Arc(X)$ satisfying $\int
\ell\,\d\ppi'<\infty$;
in particular $\ppi'$ is concentrated on
$\RA(X)$,
i.e.~$\ell(\gamma)<\infty$ for $\ppi'$-a.e.~$\gamma$.
Using the universally Lusin-measurable map $G:\RA(X)\to
\BVC_c([0,1];X)$
\eqref{eq:560}
we can also lift $\ppi$ to a Radon measure
$\tilde \ppi=\Al_\sharp\ppi$ on $\rmC([0,1];X)$ concentrated on the
set
$\BVC_c([0,1];X)$ \eqref{eq:556}.
% \begin{equation}
%   \label{eq:423}
%   \Lip_c([0,1];X):=\Big\{\gamma\in \Lip([0,1];X): |\dot
%   \gamma|=\ell(\gamma)
%   \ \text{a.e.~in }(0,1)\Big\}.
% \end{equation}
Conversely, any Radon measure $\tilde\ppi$ on $\rmC([0,1];X)$
concentrated on $\BVC([0,1];X)$ yields the Radon measure
$\ppi:=\Quot_\sharp\tilde\ppi$
on $\RA(X)$. Notice that $\Quot_\sharp(\Al_\sharp \ppi)=\ppi$.

If $\ppi$ is a dynamic plan in $\cMp(\RA(X))$,
thanks to Theorem \ref{thm:important-arcs}(e)
and Fubini's Theorem \cite[ Chap.~II-14]{Dellacherie-Meyer78},
we can define
  the Borel measure
  % \begin{equation}
  %   \label{eq:428}
  %   \nnu_\sssigma=\Int\ppi:=\int
  %   \delta_\gamma\otimes\nu_\gamma\,\d\ppi(\gamma)\in \cMp(\RA(X)\times X),
  % \end{equation}
  % a Radon measure whose first marginal is $\ppi$ and
  % whose second marginal is the projected Borel measure 
  $\mu_\sppi:=\proj\ppi\in \cMp(X)$ by the formula
\begin{equation}
  \int f\,\d\mu_\sppi:=
  \iint_\gamma f\,\d\ppi(\gamma)
  \quad\text{for every bounded Borel function }f:X\to \R.
\end{equation}
It is not difficult to show that $\mu_\sppi$ % (and thus $\nnu_\sppi$)
is a Radon
measure with total mass $\ppi(\ell)$ given by \eqref{eq:430}: in fact,
setting $\RA^L(X):=\big\{\gamma\in \RA(X):\ell(\gamma)\le
L\big\}$, for every $\eps>0$ we can find a length $L>0$ such that
$\ppi(\RA(X)\setminus \RA^L(X))\le \eps/2$.
Since $\ppi$ is Radon and $\Al$ is Lusin $\ppi$-measurable, we can also find
a compact set $\KK_\eps\subset
\RA^L(X)$
on which $\Al$ is continuous and $\ppi(\RA(X)\setminus \KK_\eps)\le \eps/(2L)$.
We deduce that
$\mu_\sppi$ is $\eps$-concentrated on the compact $K_\eps:=\{\Al_\gamma(t):t\in [0,1],\ \gamma\in
\KK_\eps\}=\sfe([0,1]\times \Al(\KK_\eps))$, i.e.~$\mu_\sppi(X\setminus
K_\eps)\le \eps$, since
\begin{align*}
  \mu_\sppi(X\setminus
  K_\eps)
  &=\int \Big(\int_\gamma \nchi_{X\setminus K_\eps}\Big)\,\d\ppi(\gamma)
    =\int \Big(\ell(\gamma)\int_0^1 \nchi_{X\setminus K_\eps}(\Al_\gamma(t))\,\d
    t\Big)\,\d\ppi(\gamma)
  \\&\le 
  L \int \Big(\int_0^1 \nchi_{X\setminus K_\eps}(\Al_\gamma(t))\,\d
  t\Big)\,\d\ppi(\gamma)
  = L(\LL^1\otimes \ppi)\{(t,\gamma):\Al_\gamma(t)\not\in K_\eps\}
  \\&\le 
  L \ppi(\{\gamma\in \RA(X):\gamma\not\in \KK_\eps\})\le \eps/2.
\end{align*}
%
% so that $\mu_\sppi$ is Radon as well.
%
Notice that $\mu_\sppi$ can be considered as the integral w.r.t.~$\ppi$
of the Borel family of measures $\nu_\gamma$, $\gamma\in \RA(X)$
\cite[ Chap.~II-13]{Dellacherie-Meyer78}, in
the sense that 
\begin{equation}
  \label{eq:289}
  \int_X f\,\d\mu_\sppi(x)=
  \int_{\RA(X)}\Big(\int_X f\,\d\nu_\gamma\Big)\,\d\ppi(\gamma).
\end{equation}
\index{Barycenter of a dynamic plan}
\begin{definition}
  \label{def:barycenter}
  %Let $q\in (1,+\infty)$. 
  We say that $\ppi\in \cMp(\RA(X))$ has barycenter in $L^q(X,\mm)$ if there exists $h\in
  L^q(X,\mm)$ such that $\mu_\sppi=h\mm$, or, equivalently, if
\begin{equation}
  \label{eq:34}
  \int \int_\gamma f\,\d\ppi(\gamma)=\int fh\,\d\mm\quad \text{for
    every }f\in \rmB_b(X),
\end{equation}
and we call $\brq q (\ppi):=\|h\|_{L^q(X,\mm)}$
the barycentric $q$-entropy of $\ppi$.
We will denote by $\Bar q\mm$ the set of all plans with barycenter in
$L^q(X,\mm)$ and we will set $\brq q(\ppi):=+\infty$ if
$\ppi\not\in\Bar q\mm$.
\end{definition}
$\brq q:\cMp(\RA(X))\to[0,+\infty]$ is
a convex and positively $1$-homogeneous functional.
When $q=1$ $\ppi$ has barycenter in $L^1(X,\mm)$ if and only
if $\mu_\sppi\ll\mm$ and in this case
$\brq 1(\ppi)=\ppi(\ell)=\int \ell\,\d\ppi$.

If $q>1$ (which corresponds to our setting, when $q$ is the dual of $p$)
then $\brq q$ is also lower semicontinuous w.r.t.~the weak topology of
$\cMp(\RA(X))$,
a property which can be easily deduced by the equivalent representation
formula 
\eqref{eq:190} below.
Notice that $\brq q(\ppi)=0$ iff $\ppi$ is concentrated on the set of
constant
arcs in $\RA(X)$.

$\brq q(\ppi)$ has two equivalent representation. 
The first one is related to the $L^q$ entropy of the 
projected measure $\mu_\sppi=\proj\ppi$ with respect to $\mm$:
\begin{equation}
  \label{eq:189}
  \frac 1q\brq q^q(\ppi)=\mathscr L_q(\mu_\sppi|\mm)
  \end{equation}
where for an arbitrary $\mu\in \cMp(X)$ 
\begin{equation}
  \label{eq:290}
  \mathscr L_q(\mu|\mm):=
  \begin{cases}
    \displaystyle \frac 1q \int_X
    \Big(\frac{\d\mu}{\d\mm}\Big)^q\,\d\mm&\text{if $\mu\ll\mm$,}\\
    +\infty&\text{otherwise.}
  \end{cases}
\end{equation}
A second interpretation arises from the dual characterization of
$\LL_q$, since $q=p'\in (1,+\infty)$
\cite[Thm.~2.7, Rem.~2.8]{MLS18}
\begin{equation}
  \label{eq:190}
  \mathscr L_q(\mu|\mm)
  =
  \sup \Big\{\int_X f\,\d\mu-\frac 1p \int_X f^p\,\d\mm:
  f\in \rmC_{b}(X),\quad f\ge 0\Big\}.
\end{equation}
We immediately obtain
\begin{equation}
  \label{eq:291}
  \frac 1q\brq q^q(\ppi)=
  \sup \Big\{\int\int_\gamma f\,\d\ppi(\gamma)-\frac 1p \int_X f^p\,\d\mm:
  f\in \rmC_{b}(X),\quad f\ge 0\Big\}.
\end{equation}
Similarly (see Lemma \ref{le:obvious-2hom} in the Appendix) 
we can easily check that 
\begin{lemma}
  \label{le:bar-check}
  If $q\in (1,\infty)$, $p=q'$, 
a Radon measure $\ppi$ on $\RA(X)$ 
has barycenter
in $L^q(X,\mm)$
if there exists $c\in [0,\infty)$ such that
\begin{equation}\label{eq:defboundcomp}
\int_{\RA(X)} \int_\gamma f  \, \d \ppi(\gamma)  \leq  c \| f
\|_{L^p(X,\mm)}\qquad
\text{for every } f \in \Ldp.
\end{equation}
In this case $\brq q(\ppi)$ is the minimal constant $c$ in
\eqref{eq:defboundcomp}.
Moreover, it is equivalent to check \eqref{eq:defboundcomp} on
nonnegative functions $f\in \rmC_b(X)$.
\end{lemma}
\index{Negligible set of rectifiable arcs}
\begin{definition}[$\Bar q\mm$-negligible sets]
  \label{def:Barqneg}
  We say that a set $\Gamma\subset \RA(X)$ is $\Bar q\mm$-negligible
  if $\ppi(\Gamma)=0$ for every $\ppi\in \Bar q\mm$. Similarly, a
  property $P$ on the set of arcs $\RA(X)$ holds $\Bar q\mm$-a.e.~if
  $\{\gamma\in \RA(X):P(\gamma)\text{ does not hold}\}$ is contained
  in a $\Bar q\mm$-negligible set.
\end{definition}

It is easy to check that for every Borel set $B\subset X$ with
$\mm(B)=0$
the set 
\begin{equation}
  \label{eq:91}
  \big\{\gamma\in \RA(X): \nu_\gamma(B)>0\big\}\quad \text{is $\Bar q\mm$-negligible}.
\end{equation}
There is a simple duality inequality, involving the minimization in
the definition 
\eqref{eqn:mod2} of $\Md$ and a maximization among all 
$\ppi$'s with 
barycenter in $L^q(X,\mm)$. To see it, let's take $f \in \Ldp$ such
that $\int_\gamma f \geq 1$ on 
$\Gamma\subset\mathcal{M}_+(X)$. 
Then, if $\Gamma$ is universally Lusin measurable we may take any plan $\ppi$ with barycenter in $L^q(X,\mm)$ to obtain
\begin{equation}\label{eq:sup1}
  \ppi( \Gamma) \leq \int \int_\gamma f\, \d \ppi(\gamma) \leq \brq
  q(\ppi) \|f \|_p
  \quad\text{if }\int_\gamma f\ge 1\text{ for every }\gamma\in \Gamma.
\end{equation}
By the definition of $\Md$ we obtain
\begin{equation}
  \label{eq:292}
  \ppi( \Gamma) \leq\brq
  q(\ppi) \Md^{1/p} (\Gamma).
\end{equation}
In particular we have
\begin{equation}\label{eq:supp1}
\Md(\Gamma)=0\quad\Longrightarrow\quad \ppi(\Gamma)=0\qquad\text{for
  all $\ppi\in \Bar q\mm$}.
\end{equation}
% In addition, taking in \eqref{eq:sup1} the infimum over all the $f\in\Ldp$ such that $\int f\,\d\mu\geq 1$ on $\Gamma$
% and, at the same time, the supremum with respect to all plans $\eeta$ with barycenter
% in $L^q(X,\mm)$ and $c_q(\eeta)>0$,
% we find 
% \begin{equation}\label{eq:sup}
% \sup_{c(\seeta)>0} \frac{\eeta( \Gamma)}{ c_q(\eeta)} \leq \Md( \Gamma)^{1/p}.
% \end{equation}
%
\begin{lemma}[An equi-tightness criterium]
  \label{le:RA-tight}
  Let us suppose that $(X,\sfd)$ is complete
  and let $\mathcal K$ be
  a subset of $\Bar q\mm$ 
  satisfying the following conditions:
  \begin{enumerate}[\rm (T1)]
  \item There exist constants $C_1,C_2>0$ such that
    \begin{equation}
      \label{eq:546}
      \ppi(\RA(X))\le C_1,\quad
      \brq q(\ppi)\le C_2\quad\text{for every }\ppi\in \cK.
      \end{equation}
    \item
      {For every $\eps>0$ there exists a $\tau$-compact
        set $H_\eps\subset X$ 
        such that}
      \begin{equation}
        \label{eq:539}
      \ppi\big(\{\gamma\in \RA(X):\sfe(\gamma)\cap H_\eps=\emptyset\}\big)\le
      \eps
      \quad\text{for every }\ppi\in \cK.
    \end{equation}
  \end{enumerate}
  Then $\cK$ is relatively compact in $\cMp(\RA(X))$.
\end{lemma}
\begin{proof}
  We want to apply Prokhorov's Theorem \ref{thm:compa-Riesz}
  so that for every $\eps>0$ we have to exhibit a compact set
  $K_\eps\subset\RA(X)$
  such that $\ppi(\RA(X)\setminus K_\eps)\le \eps$.
  
  Let $K_n,\ m_n,\ \ a_n,\ \delta_n,\ f_k$ be defined as in the proof of Lemma
  \ref{lem:tightness} and let us set $\rmE_{k,\xi}:=\rmF_k\cap \rmG_\xi$ where
  \begin{equation}
    \begin{aligned}
      \rmF_{k} :={}&\Big\{ \gamma \in \RA(X) : \ell(\gamma) \leq k,\ \
      \nu_\gamma(X\setminus
      K_n)) \leq \delta_n \ \text{for every } n \geq k \Big\},\\
      \rmG_\xi:={}&\Big\{ \gamma \in \RA(X) : 
      \sfe(\gamma)\cap H_\xi\neq\emptyset \Big\}.
    \end{aligned}
\label{eq:545bis}
  \end{equation}
  For every $k\in \N$ and $\xi>0$
  $\rmE_{k,\xi}$ are compact by Theorem \ref{thm:important-arcs}(g).
  
  Let us estimate $\ppi(\RA(X)\setminus \rmE_{k,\xi})$ for $\ppi\in\cK$.
  By (T2) we know that $\ppi(\RA(X)\setminus \rmG_{\xi})\le \xi$.
  On the other hand, since $\int_\gamma (f_k+1/k)\ge 1$
  for every $\gamma\in \RA(X)\setminus \rmF_k$ we have
  \begin{align*}
    \ppi(\RA(X)\setminus \rmF_k)
    \topref{eq:sup1}\le C_2\|f_k+1/k\|_{L^p(X,\mm)}
    \topref{eq:543}\le C_2 M_k,
  \end{align*}
  {where }
  $  M_k:=(m_k)^{1/(2p)} + (\mm(X))^{1/p}/k\down0\text{ as }k\to\infty.$
  We deduce
  $\ppi(\RA(X)\setminus \rmE_{k,\xi})\le \xi+C_2 M_k$
  and the thesis follows by choosing $K_\eps:=\rmE_{k,\xi}$
  for $\xi+C_2 M_k<\eps$.
\end{proof}
It is easy to check that (T2) is also a necessary condition for
the equi-tightness of $\cK$. In fact, if $\cK$ is equi-tight in
$\cMp(\RA(X))$
then the collection $\cK':=\{(\sfe_0)_\sharp\ppi:\ppi\in \cK\}$ is
equi-tight in $\cMp(X)$ (since $\sfe_0$ is a continuous map from
$\RA(X)$ to $X$). Therefore, for every
$\eps>0$ there exists a compact set $K_\eps\subset X$ such that
$\ppi(\{\gamma\in \RA(X):\gamma_0\not\in K_\eps\})\le \eps$,
which clearly yields \eqref{eq:539}.

It is interesting to notice that if $\cK\subset \cMp(\RA(X))$
satisfies the property (T1) above and
\begin{equation}
  \label{eq:548}
  \ell(\gamma)\ge C_3>0\quad\ppi\text{-a.e.}\quad
  \text{for every }\ppi\in \cK,  
\end{equation}
then (T2) is satisfied as well. In fact, if $K_n\subset X$ is a
compact set with $\mm(X\setminus K_n)\le m_n$
and $\rmG_n:=\{\gamma\in \RA(X):\sfe(\gamma)\cap K_n=\emptyset\}$
for every $\ppi\in \cK$ we have
\begin{align*}
  \ppi\big(\rmG_n\big)
  &\le \frac1{C_3}\int_{\rmG_n}\ell(\gamma)\,\d\ppi(\gamma)
  \le \frac 1{C_3}
  \int\Big(\int_\gamma \nchi_{X\setminus K_n}\Big)\,\d\ppi(\gamma)
  \\&=\frac 1{C_3}\mu_{\sppi}(X\setminus K_n)
  \le \frac 1{C_3}\brq q(\ppi) m_n^{1/p}\le
  \frac{C_2}{C_3}m_n^{1/p}\down0\quad\text{as }n\to\infty.
\end{align*}
%
% It is important to notice that 
% a family $\mathcal F$ of plans $\eeta$ with
% uniformly bounded mass and
% barycenter in $L^q(X,\mm)$ satisfies
% \begin{equation}
%   C:=\sup\big\{\brq q(\eeta)\;:\;\eeta\in\mathcal
%     F\big\}<\infty\quad\Longrightarrow\quad
%   \cK\text{ is tight.}\label{eq:1}
% \end{equation}
% Indeed, $\eeta(E_{\ep^p}^c) \leq \ep \brq q(\eeta) \leq C \ep$ for
% every $\eeta\in \cK$, 
% where the $E_{\ep}\subset \RA(X)$ are the compact sets provided by
% Lemma~\ref{lem:tightness}. 
%
\index{Content of a collection of rectifiable arcs}
The inequality \eqref{eq:sup1} motivates the next definition.
\begin{definition}[$p$-content]
  \label{def:pcontent}
  If $\Gamma\subset \RA(X)$ is a universally measurable set we
say that $\Gamma$ has finite content if there exists
a constant $c\ge0$ such that 
\begin{equation}
  \label{eq:295}
  \ppi(\Gamma)\le c
  \brq   q(\ppi)
  \quad\text{for every }\ppi\in \cMp(\RA(X)).
\end{equation}
In this case, the $p$-content of $\Gamma$ $\Cont p(\Gamma)$ is the minimal constant $c$
satisfying \eqref{eq:295}. If $\Gamma$ has not finite content we set $\Cont p(\Gamma):=+\infty$.
\end{definition}
Notice that if $\Gamma$ contains a constant arc we get
$\Cont p(\Gamma)=+\infty$; conversely, $\Cont p(\Gamma)=0$ if and only if
$\Gamma$ is $\Bar q\mm$-negligible.
We can formulate \eqref{eq:295} in the equivalent form
\begin{equation}
  \label{eq:296}
  \Cont p(\Gamma)=\sup_{\sppi\in \cMp(\RA(X)),\ \brq q(\sppi)>0}\frac{\ppi(\Gamma)}{\brq q(\ppi)},
\end{equation}
and we can also limit the $\sup$ in \eqref{eq:296} and the condition
\eqref{eq:295} to probability plans (i.e.~$\ppi(\RA(X))=1$)
concentrated on $\Gamma$.

By Lemma \ref{le:obvious-2hom} we easily find the equivalent
characterizations of $\Cont p$:
\begin{equation}
  \label{eq:294}
  \frac 1p\Cont p^p(\Gamma)=\sup_{\sppi\in \cMp(\RA(X))} \ppi(\Gamma)-\frac 1q \brq q^q(\ppi)
\end{equation}
showing that $\frac 1p \Cont p^p$ is in fact the Legendre transform of
$\frac 1q\brq q^q$.

Let us now address the question of existence of an optimal dynamic
plan
attaining the supremum in \eqref{eq:296} (or, equivalently, in
\eqref{eq:294}).
The next result corresponds to \cite[Lemma 4.4]{ADS15}, where
however the condition concerning the closure of $\Gamma$
is missing. See also the comments in the Notes \ref{subsec:notes8} at the end of this section.
\begin{lemma} \label{lem:existsppi}
  Let us suppose that $(X,\sfd)$ is complete and let
  $\Gamma\subset \RA(X)$ be a closed set such that
  $0<\Cont{p} (\Gamma)<+\infty $.
  If there exists a compact
  set $K\subset X$ such that
  $\sfe(\gamma)\cap K\neq\emptyset$ for every
  $\gamma\in \Gamma$ (in particular if $\Gamma$ is compact),
  then there exists an optimal plan $\ppi_\Gamma$ with barycenter
  in $L^q(X,\mm)$ attaining the supremum in \eqref{eq:294}.
  $\ppi_\Gamma$ is concentrated on $\Gamma$ and satisfies
  \begin{equation}
    \label{eq:297}
    \ppi_\Gamma(\Gamma)=\Cont p^p(\Gamma)=\brq q^q(\ppi_\Gamma).
  \end{equation}
  In particular, $\tilde
  \ppi_\Gamma:=(\ppi_\Gamma(\Gamma))^{-1}\ppi_\Gamma$ 
  is a probability plan and
  $$ \Cont{p} ( \Gamma) \brq q(\tilde\ppi_\Gamma)=1. $$
\end{lemma}
\begin{proof} 
Taking perturbations of the form $\ppi\to \kappa\ppi$, $\kappa>0$, 
we immediately see that we can restrict the maximization 
to plans satisfying 
\begin{equation}
  \label{eq:299}
  \ppi(\Gamma)=\ppi(\RA(X))=\brq q^q(\ppi)\le \Cont p^p(\Gamma).
\end{equation}
It is also easy to see that it is possible to restrict
the maximization to plans 
concentrated on $\Gamma$, 
since the restriction $\ppi\mapsto \bar \ppi=\nchi_\Gamma\ppi$ 
satisfies $\bar\ppi(\Gamma)=\ppi(\Gamma)$ and 
$\brq q(\bar\ppi)\le \brq q(\ppi)$.
Since $\Gamma$ is closed the functional of \eqref{eq:294} is upper
semicontinuous and therefore it admits a maximum on the compact set
defined by \eqref{eq:299}.
\end{proof}

\subsection{Notes}
\label{subsec:notes8}
\begin{notes}
  The notions of barycentric entropy and content
  have been introduced in \cite{ADS15} for measures in $\cMp(X)$, in order to
  provide
  an equivalent measure-theoretic characterization of the modulus
  $\Mdm$.
  Here we decided to focus mainly on nonparametric
  dynamic plans and to develop the main properties in the more restrictive
  setting  characterized by the embedding $\rmM:\RA(X)\to \cMp(X)$
  of \ref{eq:593}, which is well adapted to the classic modulus $\Md$ on arcs.
  In \,\ref{subsec:link} we will also briefly discuss the notions of
  barycentric entropy and
  content related to $\tMd$ and to the embedding $\tilde\rmM$ of \eqref{eq:594}.
  
  The equi-tightness criterium \ref{le:RA-tight}
  requires slightly more restrictive assumptions than in \cite{AGS15}
  since here compactness is obtained directly in $\cMp(\RA(X))$ instead of
  $\cMp(\cMp(X))$. Notice that the class of constant arcs
  is homeomorphic to $X$ in $\RA(X)$, whereas
  it is identified with the null measure in $\cMp(X)$.

  The existence of an optimal plan attaining \eqref{eq:297} requires
  at least the closure of $\Gamma$: this condition
  should also be added to Lemma 4.4,
  Corollary 5.2(b) and Theorem 7.2 of \cite{AGS15}.
  Notice however that the main consequences
  \cite[Theorem 8.3, Corollary 8.7]{ADS15}  of Theorem 7.2 in \cite{ADS15}
  still hold, since they only require the existence of a nontrivial
  dynamic plan in $\Bar q\mm$ giving positive mass to $\Gamma$
  whenever $\Md (\Gamma)>0$: thanks to Choquet theorem
  this property holds for an arbitrary Souslin set $\Gamma$ and does not require
  its closedness. We will also discuss
  these aspects in the next Section \ref{sec:Equivalence},  see Theorem
  \ref{tmain}.
\end{notes}

\section{Equivalence between $\Cont{p}$ and $\Md$ }
\label{sec:Equivalence}
In this Section we always refer to the main Assumption
of page \pageref{ass:mainIII}.
We have seen that $\Cont p,\Md,\Mdc$ satisfy the property
\begin{equation}
\Cont {p}^p \leq \Md \leq \Mdc\qquad\text{on universally measurable subsets of $\RA(X)$.}\label{eq:300}
\end{equation}
We first prove that \eqref{eq:300} is in fact an identity if $\Gamma$
is compact.
\begin{theorem}
  \label{thm:compact-Mod}
  If $\Gamma$ is a compact subset of $\RA(X)$ we have
  \begin{equation}
    \Cont{p}^p(\Gamma) =\Md(\Gamma) =\Mdc(\Gamma).
    % \qquad\text{on Souslin subsets of ${\mathcal
    %     M}_+(X)$.}
    \label{eq:300bis}
\end{equation}
\end{theorem}
\begin{proof}
  We will set $\cMp(\Gamma):=\{\ppi\in
  \cMp(\RA(X)):\supp(\ppi)\subset \Gamma\}$.
  Since $\ell$ is a lower semicontinuous map,
  the minimum $\ell_0:=\min_{\gamma\in \Gamma}\ell(\gamma)$
  is attained. If $\ell_0=0$ $\Gamma$ contains a constant arc
  and \eqref{eq:300bis} is trivially satisfied since the common value
  is $+\infty$.
  We can thus assume $\ell_0>0$.
  
  We will prove \eqref{eq:300bis} by using a minimax argument
  by applying Von Neumann Theorem \ref{thm:VonNeumann}.

  First of all we observe that for every $f\in \rmC_b(X)$ 
  \begin{equation}
    \label{eq:301}
    \int_\gamma f\ge 1\ \text{for every }\gamma\in \Gamma\quad
    \Leftrightarrow\quad
    \sup\Big\{\int\Big(1-\int_\gamma f\Big)\,\d\ppi(\gamma):
    \ppi\in \cMp(\Gamma)\Big\}=0,
  \end{equation}
  so that 
  \begin{gather*}
    \frac 1p\Mdc(\Gamma)=
    \inf_{f\in \rmC_b(X)}
    \sup_{\sppi\in \cMp(\Gamma)}\cL(\ppi,f),\\
    \cL(\ppi,f):=\frac 1p\int_X f^p\,\d\mm+
    \int\Big(1-\int_\gamma f\Big)\,\d\ppi(\gamma).
  \end{gather*}
  By choosing $f_\star\equiv k\ge 2/\ell_0$ we clearly have
  \begin{displaymath}
    \cL(\ppi,f_*)=\frac 1p\int_X f^p\,\d\mm+
    \int\Big(1-\int_\gamma f\Big)\,\d\ppi(\gamma)\le c_k-\ppi(\Gamma),
    \quad
    c_k:=\frac 1p\mm(X)k^p,
  \end{displaymath}
  so that choosing $D_\star<\frac 1p\Mdc(\Gamma)$
  and $c_k$ sufficiently big,
  the set $\{\ppi\in \cMp(\Gamma):\cL(\ppi,f_*)\ge D_*\}$ 
  is not empty (it contains the null plan) and it is contained 
  in the compact set 
  $\{\ppi\in \cMp(\Gamma): \ppi(\Gamma)\le c_k-D_*\}$. 
  Condition \eqref{eq:115bis} is thus satisfied and we deduce
  \begin{align*}
    \frac 1p\Mdc(\Gamma)&=                          
                          \max_{\sppi\in \cMp(\Gamma)}
                          \inf_{f\in \rmC_b(X)}\cL(\ppi,f)
                          \\&=\max_{\sppi\in \cMp(\Gamma)}
    \ppi(\Gamma)-
                          \sup \int\int_\gamma f\,\d\ppi(\gamma)-\frac 1p\int_X f^p\,\d\mm\\
    &=\max_{\sppi\in \cMp(\Gamma)}
    \ppi(\Gamma)-\frac 1q\brq q^q(\ppi)\topref{eq:294}=\Cont p^p(\Gamma). \qedhere
\end{align*}
\end{proof}
Theorem \ref{thm:compact-Mod} has an important implication
in terms of Choquet capacity; we refer to
\cite[Chap.~III, \S\,2]{Dellacherie-Meyer78} and
to the brief account given in Section \ref{subsec:Choquet}
of the Appendix. Recall that $\BB(Y)$ (resp.~$\KK(Y)$)
will denote the collection of all the Borel (resp.~compact)
subsets of a Hausdorff space $Y$.The definition and the main properties
of Souslin and Analytic sets are briefly recalled in
\S\,\ref{subsec:PLS}.
\index{Analytic sets}
\index{Choquet capacity}
\begin{theorem}\label{tmain}
  Let $\X=(X,\tau,\sfd,\mm)$ be an e.m.t.m.~space.
  \begin{enumerate}
  \item $\Md$ is a Choquet $\KK(\RA(X),\tau_\rmA)$-capacity in
    $\RA(X)$.
      \item For every universally measurable $\Gamma\subset \RA(X)$
    \begin{equation}
    \label{eq:559}
    \Cont p(\Gamma)=\sup\Big\{\Cont p(\rmK):\rmK\subset \Gamma,
    \rmK\text{ compact}\Big\}.
  \end{equation}
  \item
    If $(X,\sfd)$ is complete and $(X,\tau)$ is
    Souslin, then every $\BB(\RA(X),\tau_\rmA)$-analytic set $\Gamma$
    (in particular, every Souslin set) is $\Md$-capacitable
  \item
    If $(X,\sfd)$ is complete and $(X,\tau)$ is
    Souslin, then every $\BB(\RA(X),\tau_\rmA)$-analytic set $\Gamma$
    satisfies $\Md(\Gamma) = \Cont{p}^p(\Gamma)$.
    In particular $\Gamma$
    is $\Md$-negligible if and only if it is $\Bar q\mm$-negligible.
  \end{enumerate}
\end{theorem}
\begin{proof}
  \textbf{(a)}
  Proposition~\ref{prop:prop}(e,f)
  and the fact that $\Mdc = \Md$ if the set is compact by
  Theorem \ref{thm:compact-Mod} give us that $\Md$ is a
  $\KK$-capacity in $(\RA(X),\tau_\rmA)$.
  \medskip

  \noindent \textbf{(b)}
  By \eqref{eq:294} for every $S<\Cont p(\Gamma)$ we can find $\ppi\in\Bar q\mm$
  such that
  \begin{displaymath}
    \frac 1p S^p<\ppi(\Gamma)-\frac 1q\brq q(\ppi). 
  \end{displaymath}
  Since $\Gamma$ is $\ppi$-measurable and $\ppi$ is Radon, we can find
  a compact set
  $\rmK\subset \Gamma$ such that
    \begin{displaymath}
    \frac 1p S^p<\ppi(\rmK)-\frac 1q\brq q(\ppi)\le \frac 1p \Cont p(\rmK),
  \end{displaymath}
  which eventually yields \eqref{eq:559} since $S$ is arbitrary.
 \medskip

 \noindent\textbf{(c)}
 Let us now assume that $(X,\sfd)$ is complete and
  $(X,\tau)$ is Souslin and let us prove that
  every $\BB$-analytic set is capacitable.
  By Choquet's Theorem \ref{thm:Choquet}, it is sufficient to prove
  that every Borel set is capacitable.
  By Corollary \ref{cor:auxiliary-arc}
  we know that $(\RA(X),\tau_\rmA,\sfd_\rmA)$
  admits an auxiliary topology
  $\tau_\rmA'
  $.
  
  Let $\Gamma$ be a
  Borel subset of $\RA(X)$.
  If $\Gamma$ contains a constant arc
  there is nothing to prove, so that we can assume
  that $\Gamma\subset \RA_0(X)=\{\gamma\in \RA(X):\ell(\gamma)>0\}$.
  Recalling the definition of the open sets $\RA_\eta(X)$ given in
  Lemma \ref{lem:tightness} and Proposition~\ref{prop:prop}(e),
  we know that
  $\Md(\Gamma)=\lim_{\eta\down0}\Md(\Gamma_\eta)$,
  where $\Gamma_\eta:=\Gamma\cap \RA_\eta(X)$.
  It is therefore sufficient to prove that every $\Gamma_\eta$ is
  capacitable.
  Let us fix $\eta>0$; by applying Lemma \ref{lem:tightness} for every
  $\eps>0$ we
  can find a compact set $\rmK_\eps\subset \RA(X)$
  such that $\Md(\Gamma_\eta\setminus\rmK_\eps)
  \le \Md(\RA_\eta(X)\setminus \rmK_\eps)\le \eps$.
  Since $\Md$ is subadditive, it remains to prove that
  $\Gamma_\eta\cap \rmK_\eps$ is capacitable.
  Notice that $\rmK_\eps$ is also compact with respect to the
  coarser (metrizable and separable) topology $\tau_\rmA'$ and the
  restriction
  of $\tau_\rmA'$ to $\rmK_\eps$ coincides with $\tau_\rmA$, so
  that $(\rmK_\eps,\tau_\rmA)$ is a Polish space.
  Since $\Gamma_\eta\cap \rmK_\eps$ is Borel in $\rmK_\eps$,
  it is also $\FF$-analytic and therefore (being $\rmK_\eps$ compact)
  $\KK$-analytic. By claim (a) above we deduce that
  $\Gamma_\eta\cap\rmK_\eps$ is capacitable.
  % \noindent {\bf Step 2.}
% Now we will prove that $\Md$ and $\Cont{p}$ are both inner regular, namely their value on Souslin sets is the supremum of their
% value on compact subsets. Inner regularity and equality on compact sets yield $\Cont{p}(B) =\Md(B)^{1/p}$ on every Souslin subset $B$
% of $\mathcal{M}_+ (X)$.
% %
% \medskip
% \noindent
% \textbf{$\Md$ is inner regular.}. 
% For any set $L\subset{\mathcal M}_+(X)$ we have $\Md (L)=\sup_{\ep} \Md(L \cap E_{\ep})$, where
% $E_\ep$ are the compact sets given by Lemma~\ref{lem:tightness}.
% Therefore, suffices to show inner regularity for a Souslin set $B$ contained in $E_\ep$ for some $\ep$. Since $E_\ep$ is compact,
% $B$ is a Souslin-compact set and from Choquet Theorem
% % ~\ref{thm:Choquet} 
% it follows that for every $\delta>0$ 
% there is a compact set $K \subset B$ such that $\Md(K) \geq \Md(B) - \delta$.
% %
   \medskip

  \noindent\textbf{(d)}
  It is an immediate consequence of the previous claims,
  recalling that every $\BB$-analytic set is universally
  measurable.
% \textbf{$\Cont{p}$ is inner regular.} Since Souslin sets are universally measurable and ${\mathcal M}_+(X)$ is Polish, 
% we can apply
% % \eqref{eq:inner1}
% to any Souslin set $B$ with $\sigma=\eeta$ to get
%  $$ \sup_{K \subset B} \Cont{p} (K) =  \sup_{K \subset B} \sup_{\brq
%    q(\seeta)>0} \frac{\eeta (K)}{\brq q(\eeta)} =  
%  \sup_{\brq q(\seeta)>0 } \sup_{K \subset B} \frac{\eeta (K)}{\brq
%    q(\eeta)}  = \sup_{\brq q(\seeta)>0} \frac{\eeta (B)}{\brq q(\eeta)} = \Cont{p} (B).$$
 \end{proof}
\subsection{Notes}
\label{subsec:notes9}
\begin{notes}
  Theorem \ref{thm:compact-Mod} has been proved
  in \cite{ADS15} by using a different argument based on Hahn-Banach
  theorem.
  The proof presented here, based on Von Neumann theorem, shows more
  clearly that the definitions of $\Md$ and of $\Cont p$ rely on
  dual optimization problems, so that their equality is a nice application
  of a min-max argument.

  Theorem \ref{tmain} strongly relies on Choquet's Theorem. It is
  interesting to note that the possibility to separate distance and
  topology in e.m.t.m.~space expands the range of application and
  covers
  the case of general Souslin spaces:
  the use of an auxiliary topology overcomes the difficulty related to 
  the unknown Souslin character of path spaces
  (see \cite[page 46-III]{Dellacherie-Meyer78}).

  Theorem \ref{thm:compact-Mod} and Theorem \ref{tmain}
  could be directly stated at the level of modulus and contents on
  $\cMp(X)$
  instead of $\RA(X)$, see \cite{ADS15}. We will discuss another
  important case 
  in \S\,\ref{subsec:link}.

\end{notes}

\part{Weak upper gradients and
  identification of Sobolev spaces}
\section{(Nonparametric) Weak upper gradients and
weak Sobolev spaces}
\label{sec:Wug}
In this section we introduce a notion of weak upper gradient
modeled on $\calT_q$-test plans, in the usual setting stated
at page \pageref{ass:mainIII}.

\subsection{$\calT_q$-test plans and $\calT_q$-weak upper gradients}
\label{subsec:Tq}
Recall that the (stretched) evaluation maps $\gsfe_t:\RA(X)\to X$ are
defined by $\gsfe_t(\gamma):=\Al_\gamma(t)$, see \eqref{eq:431}.
We also introduce the
restriction maps $\Restr st: \RA(X)\to\RA(X)$,
$0\le s< t\le 1$, given by
\begin{equation}
  \Restr st(\gamma):=\Quot(\Al_\gamma^{s\to t}),
  \quad
  \Al_\gamma^{s\to t}(r):=\Al_\gamma((1-r)s+rt)\quad
  r\in [0,1],
  \label{eq:432}
\end{equation}
where $\Quot$ is the projection map from $\rmC([0,1];X)$ to
$\Arc(X)$.
$\Restr st$ restricts the arc-length parametrization $\Al_\gamma$ of
the arc $\gamma$ to the
interval $[s,t]$ and then ``stretches'' it on the whole of $[0,1]$,
giving back the equivalent class in $\RA(X)$. Notice that
for every $\gamma\in \RA(X)$
\begin{equation}
  \label{eq:445}
  \int_{\Restr st(\gamma)}f=
  (t-s)\ell(\gamma)\int_0^1 f(\Al_\gamma^{s\to t}(r))\,\d r=
  \ell(\gamma)\int_s^t f(\Al_\gamma(r))\,\d r.
\end{equation}
\index{Nonparametric $\calT_q$-test plans}
\begin{definition}[Nonparametric $\calT_q$-test plans]
  \label{def:testplan}
  Let $q=p'\in (1,\infty)$.
  We call $\calT_q=\calT_q(\X)\subset \cMp(\RA(X))$ the collection of
  all
  (nonparametric) dynamic
  plan $\ppi\in \cMp(\RA(X))$ such that 
  \begin{equation}
    \label{eq:436}
    \brq q(\ppi)<\infty,\quad
    \LL_q((\sfe_i)_\sharp\ppi|\mm)<\infty\quad
    i=0,1;
  \end{equation}
  Dynamic plans in $\calT_q$ will be also called
  $\calT_q$-test plans.\\
  We call $\calT_q^*$ the subset of $\calT_q$ whose plans $\ppi$
  satisfy the following property:
  there exists a constant $c>0$ and
  a compact set
  $\rmK\subset \RA(X)$ (depending on $\ppi$) such that
  \begin{equation}
    \label{eq:562}
    (\sfe_i)_\sharp\ppi\le c \mm,\ i=0,1,\qquad
    \ell\restr{\rmK}\text{ is bounded, continuous and strictly positive,}\qquad
    \ppi(\RA(X)\setminus \rmK)=0.
  \end{equation}
  % are concentrated on a compact set $\cK\subset \RA(X)$
  % in which the map $\ell$ is continuous and strictly positive.
  % 
  We say that $\ppi$ is a \emph{stretchable} $\calT_q$-test plan if 
  %there exists a constant $c\ge0$ such that
  \begin{equation}
    \label{eq:437}
    \brq q(\ppi)<\infty,\quad
    \LL_q((\gsfe_t)_\sharp\ppi|\mm)<\infty\quad
    \text{for every }t\in [0,1].
  \end{equation}
  % \begin{equation}
  %   \label{eq:97}
  %   \ppi\in \calT\quad\Longrightarrow\quad
  %   ({\rm
  %     restr}_t^s)_\sharp\ppi\in \calT\quad\text{for every }0
  %   \le t\le s\le 1.
  % \end{equation}
\end{definition}
  Notice that $\ppi$ is a stretchable $\calT_q$-test plan if and only if 
  $\Restr st(\ppi)\in \calT_q$ for every $s,t\in [0,1]$, $s<t$.
  Clearly the class of nonparametric $\calT_q$-test plans depends on the full
  e.m.t.m.~structure of $\X$; however, when there is no risk of
  ambiguity, we will simply write $\calT_q$.
  % 
  % We will often impose additional quantitative assumptions
  % on test plans, besides \eqref{eq:96}. The most important one,
  % which we call \emph{bounded compression},  provides
  % a locally uniform upper bound on the densities of $(\gsfe_t)_\sharp \ppi$.
  % More precisely, a test plan $\ppi$ has bounded compression on the sublevels of
  % $\Wgh$ if for all $M\geq 0$ there exists $C=C(\ppi,M)\in [0,\infty)$
  % satisfying
  % \begin{equation}\label{eq:incompre2}
  %   (\e_t)_\sharp\ppi(B\cap\{\Wgh\leq M\})\leq
  %   C(\ppi,M)\,\mm(B)\qquad\forall B\in\BorelSets{X},\ t\in [0,1].
  % \end{equation}
  % The above condition \eqref{eq:incompre2}
  % depends not only on $\mm$,
  % but also on $V$. For finite measures $\mm$ it will be understood
  % that we take $\Wgh$ equal to a constant, so that \eqref{eq:incompre2}
  % does not depend on the value of the constant.

  % Taking \eqref{eq:incompre2} into account,
  % typical examples of
  % stretchable collections $\calT$
  % are the families of all the test plans with
  % bounded compression which are concentrated on
  % absolutely continuous arcs, or on the curves of finite
  % $2$-energy, or on the geodesics in $X$.
  %
  \index{Negligible set of rectifiable arcs}
  \begin{definition}[$\calT_q$-negligible sets of rectifiable arcs]
  Let $P$ be a property concerning nonparametric arcs in $\RA(X)$.
  We say that $P$ holds $\calT_q$-a.e.~(or for $\calT_q$-almost every 
  arc $\gamma\in \RA(X)$) if for any $\ppi\in\calT_q$
  the set
$$
N:=\left\{\gamma:\ \text{$P(\gamma)$ does not hold }\right\}
$$
is contained in a $\ppi$-negligible Borel set.
% We say that a
%set $N\subset \AC{}{(0,1)}X\sfd$ is $\calT$-negligible provided
%$\ppi(N)=0$ for any test plan $\ppi\in \calT$. 
%A property which
%holds for every curve of some Borel set $A\subset \AC{}{(0,1)}X\sfd$
%except possibly for a negligible subset of $A$ is said to hold for
%$\calT$-almost all curves in $A$.
\end{definition}
Since $\calT_q\subset \Bar q\mm$, it is clear that
for every Borel set $\Gamma\subset \RA(X)$
\begin{equation}
  \label{eq:438}
  \Mod_p(\gamma)=0\quad
  \Rightarrow\quad\Cont p(\Gamma)=0\quad\Rightarrow\quad
  \Gamma\text{ is $\calT_q$-negligible}.
\end{equation}
Notice that we can revert the first implication in \eqref{eq:438}
e.g.~when $(X,\tau,\sfd)$ is a Souslin and complete
e.m.t.~space, see Theorem \ref{tmain}.
% In the next remark we compare our definition with the more classical
% notion of ${\rm Mod}_2$-null set of absolutely continuous curve used
% in \cite{Shanmugalingam00}.

% \begin{remark}\label{rem:comparenullsets}
% \upshape
% Recall that, for a collection $\Gamma$ of absolutely
% continuous curves in $(X,\sfd)$, the $2$-modulus ${\rm
% Mod}_2(\Gamma)$ is defined by
% $$
% {\rm Mod}_2(\Gamma):=\inf\left\{\int_X g^2\,\d\mm:\ \text{$g\geq 0$
% Borel, $\int_\gamma g\geq 1$ for all $\gamma\in\Gamma$}\right\}.
% $$
% If $\calT$ denotes the class of plans with bounded compression
% defined by \eqref{eq:incompre2}, it is not difficult to show that
% Borel and ${\rm Mod}_2$-null sets of curves are $\calT$-negligible.
% Indeed, if $\ppi\in\calT$ has (with no loss of generality) finite
% 2-action and is concentrated on curves contained in $\{V\leq M\}$
% and $\int_\gamma g\geq 1$ for all $\gamma\in\Gamma$, we can
% integrate w.r.t. $\ppi$ and then minimize w.r.t. $g$ to get
% $$
% \bigl[\ppi(\Gamma)\bigr]^2\leq C(\ppi,M){\rm
% Mod}_2(\Gamma)\int\int_0^1|\dot\gamma|^2\,\d s\,\d\ppi(\gamma).
% $$
% Proving equivalence of the two concepts seems to be difficult, also
% because one notion is independent of parameterization, while the
% other one (because of the bounded compression condition) takes into
% account also the way curves are parameterized. \fr
% \end{remark}
\begin{lemma}
  \label{le:nice}
  If a set $N\subset \RA_0(X)$
  is $\ppi$-negligible
  for every $\ppi\in \calT_q^*$  then $N$ is $\calT_q$-negligible.
\end{lemma}
\begin{proof}
  Let us fix $\ppi\in \calT_q$ and let $\pi_i=(\sfe_i)_\sharp\pi=h_i\mm$
  with $h_i\in L^q(X,\mm)$, $h_i$ Borel nonnegative.
  We set
  \begin{displaymath}
    H_{i,k}:=\{x\in X:h_i(x)\le k\},\quad
    \rmH_k:=\{\gamma\in \RA(X):\ell(\gamma)\le k,\ \sfe_i(\gamma)\in H_{i,k}
  \ i=0,1\}.
\end{displaymath}
Clearly
  \begin{displaymath}
    \lim_{k\to\infty}\ppi(\RA(X)\setminus \rmH_k)
  \le \lim_{k\to\infty}\big(\pi_0(X\setminus H_{0,k})+
  \pi_1(X\setminus H_{1,k})\big)=0.
\end{displaymath}
  We can also find an increasing
  sequence of compact sets $\rmK_n\subset \RA_0(X)$ such that
  the restriction of $\ell$ to $\rmK_n$ is continuous and strictly
  positive and
  $\ppi(\RA_0(X)\setminus \rmK_n)\le 2^{-n}$.
  It follows that
  $\ppi_n:=\ppi\restr{\rmK_n\cap \rmH_n}$ belongs to $\calT_q^*$ and we can find
  a Borel set $B_n$ with $\RA_0(X)\supset B_n\supset N$
  such that $\ppi_n(B_n)=0$.
  Setting $B:=\cap_n B_n$ clearly $B\supset N$ and
  $\ppi_n(B)=0$ so that $\ppi(B)=0$ as well.  
\end{proof}
Recall that
a Borel function $g:X\to [0,+\infty]$ is an \emph{upper gradient}
\cite{Cheeger00}
for $f:X\to \R$ if
\begin{equation}
  \label{eq:439}
  \Big|\int_{\partial\gamma}f\Big|\le \int_\gamma g\quad\text{for every
  }\gamma\in \RA(X)\quad\text{where}\quad
  \int_{\partial\gamma}f :=f(\gamma_1)-f(\gamma_0).
\end{equation}
\begin{definition}[$\calT_q$-weak upper gradients]\label{def:weak_upper_gradient}
Given
$f:X\to\R$, a $\mm$-measurable function $g:X\to[0,\infty]$ is a $\calT_q$-weak upper
gradient (w.u.g.) of $f$ if
\begin{equation}
\label{eq:inweak} \biggl|\int_{\partial\gamma}f\biggr|\leq
\int_\gamma g <\infty\qquad\text{for
  $\calT_q$-almost every $\gamma\in \RA(X)$.}
\end{equation}
\end{definition}
\index{Weak upper gradients}
\begin{remark}[Truncations]
  \label{rem:truncations}
  If $T:\R\to\R$ is a $1$-Lipschitz map and $g$ is a
  $\calT_q$-w.u.g.~of $f$, then $g$ is a $\calT_q$-w.u.g.~of $T\circ
  f$ as well. Conversely, if $g$ is a $\calT_q$-w.u.g.~of
  $f_k:=-k\lor f\land k$ for every $k\in \N$, it is easy to see that
  $g$ is a $\calT_q$-w.u.g.~of $f$.
  By this property, in the proof of many statements concerning
  $\calT_q$-w.u.g.~it will not be restrictive to assume $f$ bounded.
\end{remark}
The definition of weak upper gradient enjoys natural
invariance properties w.r.t.~modifications in $\mm$-negligible sets.
We will also show that if $g$ is $\mm$-measurable
the integral in \eqref{eq:inweak}
is well defined for
$\calT_q$-a.e.~arc $\gamma$.
\begin{proposition}[Measurability and invariance under modifications
  in $\mm$-negligible sets]\label{prop:invarianza}
  \
  \begin{enumerate}
  \item If $f,\,\tilde f:X\to\R$ differ in a $\mm$-negligible set
    then
    for $\calT_q$-a.e.~arc $\gamma$
    \begin{equation}
      \label{eq:587}
      f(\gamma_0)=\tilde f(\gamma_0),\quad
      f(\gamma_1)=\tilde f(\gamma_1),\quad
      f\circ\Al_\gamma=\tilde f\circ\Al_\gamma\text{ $\Leb 1$-a.e.~in }(0,1).
    \end{equation}
  \item If $g$ is $\mm$-measurable then the map
    $s\mapsto g(\Al_\gamma(s))$ is $\Leb 1$-measurable for
    $\calT_q$-a.e.~arc $\gamma$.
  \item Let
    $f,\,\tilde f:X\to\R$ and $g,\,\tilde g:X\to[0,\infty]$ be such
    that both $\{f\neq \tilde f\}$ and $\{g\neq \tilde g\}$ are
    $\mm$-negligible.  If $g$ is a $\calT_q $-weak upper gradient of
    $f$ then $\tilde g$ is a $\calT_q $-weak upper gradient of
    $\tilde f$.
  \end{enumerate}
\end{proposition}
\begin{proof}
  \textbf{(a)}
  Let $N\supset\{f\neq\tilde{f}\}$ be a
  $\mm$-negligible Borel set and let $\ppi\in \calT_q$ be a test plan.
  We
  have
  % by
% \eqref{eq:96} that $\ppi(\{\Al_\gamma(t)\in
% A\})=(\gsfe_t)_\sharp\ppi(A)=0$ for every $t\in [0,1]$, so that
%
  \begin{displaymath}
    \int\Big( \int_\gamma \nchi_N\Big)\,\d\ppi(\gamma)
    =\mu_\sppi(N) =0\quad
    \text{since }\mu_\sppi\ll\mm\text{ and }\mm(N)=0,% \int_0^1 \ppi(\{\gamma_t\in A\})\,\d t=
  % \int_0^1 \int \nchi_{\{\gamma_t\in A\}}\,\d\ppi(\gamma)\,\d t=
  % \int \Big(\int_0^1 \nchi_{\{\gamma_t\in A\}}\,\d t\Big)\d\ppi(\gamma).
\end{displaymath}
so that
$\int_\gamma \nchi_N=\ell(\gamma)\int_0^1 \nchi_N(\Al_\gamma(s))\,\d s=0$ for
$\ppi$-a.e. $\gamma$.
For any arc $\gamma$ for which the integral
is null $f(\Al(\gamma_s))$ coincides a.e.~in $[0,1]$ with 
$\tilde{f}(\Al(\gamma_s))$.
The same argument shows 
the sets
$\bigl\{\gamma:\ f(\gamma_t)\neq \tilde f(\gamma_t)\bigr\}\subset
\{\gamma:\gamma_t\in N\}$,
$t=0,\,1$
are $\ppi$-negligible
because
$(\e_t)_\sharp\ppi\ll\mm$, which implies that 
$\ppi(\{\gamma:\ \gamma_t\in N\})=({\e_t})_\sharp\ppi(N)=0$. % For the
\medskip

\noindent\textbf{(b)}
If $\tilde{g}$ is a Borel
modification of $g$ the set $\{g\neq\tilde{g}\}$ is a
$\mm$-negligible set; by the previous Claim (a),
$g(\Al(\gamma_s))$ coincides $\Leb 1$-a.e.~in $[0,1]$ with the Borel map
$\tilde{g}(\Al(\gamma_s))$ and it is therefore $\Leb 1$-measurable.
\medskip

\noindent\textbf{(c)} follows immediately by Claim (a) as well, since
for $\calT_q$-a.e.~arc $\gamma$
$\int_{\partial\gamma} f=\int_{\partial\gamma} \tilde f$ and
$\int_\gamma g=\int_\gamma\tilde g$.
%
% and the set $\bigl\{\gamma:\ \Leb{1}(\{s:
%f(\gamma_s)\neq\tilde{f}(\gamma_s)\})>0\bigr\}$ 
% third one we choose as $A$ a $\mm$-negligible Borel set containing
% $\{G\neq\tilde{G}\}$ and we use the argument described immediately
% after Definition~\ref{def:weak_upper_gradient}. 
%For the fourth one
%we choose a $\mm$-negligible Borel set $A$ containing
%$\{f\neq\tilde{f}\}$ and argue as for the third.
\end{proof}
\begin{remark}[Local Lipschitz constants of
  $\sfd$-Lipschitz functions are weak upper gradients]
  \label{rem:lipweak}\upshape
  If $f\in \Lip_b(X,\tau,\sfd)$ 
  then the local Lipschitz constant
  $\lip f$ is an upper gradient and therefore
  it is also a $\calT_q$-weak upper gradient.
  An analogous property holds for the $\sfd$-slopes (notice that
  the topology $\tau$ does not play any role in the definition)
  \begin{displaymath}
    |\rmD^\pm f|(x):=\lim_{\sfd(y,x)\to
      0}\frac{(f(y)-f(x))_\pm}{\sfd(y,x)},\quad
    |\rmD f|(x):=\lim_{\sfd(y,x)\to 0}\frac{|f(y)-f(x)|}{\sfd(y,x)}% =|\rmD^+ f|(x)\lor |\rmD^- f|(x)
  \end{displaymath}
  of an arbitrary $\sfd$-Lipschitz 
  functions $f\in \rmB_b(X)$: if $ |\rmD f|$ is $\mm$-measurable
  (this property is always satisfied if, e.g., $(X,\tau)$ is Souslin,
  see \cite[Lemma 2.6]{AGS14I})
  then it is a weak upper gradient of $f$.  
\end{remark}
It is easy to check that for every $\alpha,\beta\in \R$
\begin{equation}
  \label{eq:564}
  \text{if $g_i$ is a $\calT_q$-w.u.g.~of $f_i$, $i=0,1$,
    then}\quad
  |\alpha|g_0+|\beta| g_1
  \text{ is a $\calT_q$-w.u.g.~of $\alpha f_0+\beta f_1.$}    
\end{equation}
In particular the set $S:=\{(f,g):\text{ $g$ is a $\calT_q$-w.u.g.~of
  $f$}\}$  is convex.

If we know a priori the integrability of $f$ and the $L^p$-summability
of $g$ then
Definition \ref{def:weak_upper_gradient} can be formulated in a
slightly different way:
\begin{lemma}
  \label{le:equivalent1}
  Let $f\in L^1(X,\mm)$ and $g\in L^p(X,\mm)$, $g\ge 0$.
  $ g$ is a $\calT_q$-weak upper gradient of $f$ if and only if
  \begin{equation}
    \label{eq:440}
    \int \big(f(\gamma_1)-f(\gamma_0)\big)\,\d\ppi(\gamma)\le
    \int\Big(\int_\gamma g\Big)\,\d\ppi(\gamma)
    \quad
    \text{for every $\ppi\in \calT_q^*$.}
  \end{equation}
  Equivalently, setting $\pi_i:=(\sfe_i)_\sharp \ppi$, $i=0,1$,
  \begin{equation}
    \label{eq:441}
    \int_X f\,\d(\pi_1-\pi_0)\le \int_X g\,\d\mu_\sppi \quad\text{for every $\ppi\in \calT_q^*$.}
  \end{equation}
\end{lemma}
\begin{proof}
  It is clear that \eqref{eq:inweak} yields \eqref{eq:440} simply by
  integration w.r.t.~$\ppi$; notice that the integrals in
  \eqref{eq:440} (and in \eqref{eq:441}) are well defined since
  $\pi_i=h_i\mm$ and $\mu_\sppi=h\mm$ for functions $h_i\in
  L^\infty(X,\mm)$ and $h\in
  L^q(X,\mm)$.

  Let us prove the converse implication.
  It is not restrictive to assume that $f,g$ are Borel. 
  By Lemma \ref{le:nice}, if \eqref{eq:inweak} is not
  true,
  there exists a nontrivial test plan $\ppi\in \TT_q^*$ such that
  the Borel set $A:=\Big\{\gamma\in \RA(X):|\int_{\partial\gamma}
  f|>\int_\gamma  g\Big\}$ satisfies
  $\ppi(A)>0$ (notice that \eqref{eq:inweak} is always satisfied
  on constant arcs).
  By possible reducing $A$ we can also find
  $\theta>0$ such that
  $A':=\Big\{\gamma\in \RA(X):\int_{\partial\gamma}
  f\ge \theta+\int_\gamma  g\Big\}$ satisfies
  $\ppi(A')>0$. Thus defining $\ppi':=\ppi\restr {A'}$,
  it is immediate to check that $\ppi'\in \calT_q^*$;
  a further integration with respect to $\ppi'$
  of the previous inequality yields
  \begin{displaymath}
    \int\Big(f(\gamma_1)-f(\gamma_0)\Big)\,\d\ppi'(\gamma)\ge \theta\ppi(A')+
    \int\Big(\int_\gamma g\Big)\,\d\ppi'(\gamma),
  \end{displaymath}
  which contradicts \eqref{eq:440}.  
\end{proof}
\begin{definition}[Sobolev regularity along a rectifiable arc]
\label{def:Sobolev}
We say that a %$\mm$-measurable
map $f:X\to \R$ is 
\emph{Sobolev} (resp.~\emph{strictly Sobolev}) along an
arc $\gamma$ if 
$f\circ\Al_\gamma$ coincides $\Leb 1$-a.e.~in $[0,1]$
(resp.~$\Leb 1$-a.e.~in $[0,1]$ and in $\{0,1\}$) with an
absolutely continuous map $f_\gamma:[0,1]\to\R$.
In this case, we say that a map $g:X\to \R$ is a Sobolev upper
gradient (S.u.g.) for
$f$ along $\gamma$ if $\ell(\gamma)g\circ\Al_\gamma\in L^1(0,1)$
and
\begin{equation}
  \label{eq:592}
  \biggl|\frac{\d}{\dt}f_\gamma\biggr|\leq
      \ell(\gamma) g\circ\Al_\gamma\quad\text{a.e.~in $[0,1]$}.
\end{equation}
\end{definition}
We can give an intrinsic formulation of the (strict) Sobolev
regularity with Sobolev upper gradient $g$ which does not involve
the absolutely continuous representative.
\begin{lemma}
  \label{le:tedious-wug}
  Let us suppose that $\gamma\in \RA(X)$, $f,g:X\to \R$ such that $f\circ\Al_\gamma,\
  \ell(\gamma)g\circ\Al_\gamma\in L^1(0,1)$, and $\calC$
  (resp.~$\calC_c$)
  is a dense subset of $\rmC^1([0,1])$ (resp.~of $\rmC^1_c(0,1))$.
  \begin{enumerate}
  \item
    $f$ is Sobolev along $\gamma$ with Sobolev
    u.g.~$g$ if and only if
    \begin{equation}
      \label{eq:591bis}
      \Big|
      -
      \int_0^1\varphi'(t)f(\Al_\gamma(t))\,\d t\Big|\le \ell(\gamma)\int_0^1
      |\varphi(t)|\,g(\Al_\gamma(t))\,\d t
    \end{equation}
    for every $\varphi\in \calC_c$.
  \item $f$ is strictly Sobolev along $\gamma$ with Sobolev
    u.g.~$g$ if and only if
    \begin{equation}
      \label{eq:591}
      \Big|
      \varphi(1)f(\Al_\gamma(1))-\varphi(0)f(\Al_\gamma(0))-
      \int_0^1\varphi'(t)f(\Al_\gamma(t))\,\d t\Big|\le \ell(\gamma)\int_0^1
      |\varphi(t)|\,g(\Al_\gamma(t))\,\d t
    \end{equation}
    for every $\varphi\in \calC$.
  \end{enumerate}
  In particular, if $f,g$ are Borel maps,
  the sets of curves $\gamma\in \RA(X)$ along which
  $\int_\gamma (|f|+g)<\infty$ and
  $f$ is Sobolev (resp.~strictly Sobolev) with S.u.g.~$g$ is
  Borel in $\RA(X)$.
\end{lemma}
\begin{proof}\textbf{(a)}
  One implication is obvious.
  If \eqref{eq:591bis} holds for every
  $\varphi\in \calC_c$ then it can be extended to every 
  $\varphi\in \rmC^1_c(0,1)$. In particular we have
  \begin{displaymath}
    \Big|
      -
      \int_0^1\varphi'(t)f(\Al_\gamma(t))\,\d t\Big|\le C\sup_{t\in
        [0,1]}|\varphi(t)|\quad\text{where}\quad
      C:=\int_\gamma g
  \end{displaymath}
  for every
  $\varphi\in \rmC^1_c(0,1)$, so that the distributional derivative of
  $f\circ\Al_\gamma$ can be represented by Radon measure
  $\mu\in\cM((0,1))$
  with finite total variation. \eqref{eq:591bis} also yields
  \begin{displaymath}
    \Big|\int_0^1 \varphi\,\d\mu\Big|\le
    \ell(\gamma)\int_0^1 |\varphi|\,g\,\d t\quad\text{for every
    }\varphi\in \rmC^1_c(0,1)
  \end{displaymath}
  so that $\mu=h\Leb 1$ is absolutely continuous w.r.t.~$\Leb 1$ with
  density
  satisfying $|h|\le \ell(\gamma)g\circ\Al_\gamma$ $\Leb 1$-a.e.
  It follows that $f\circ\Al_\gamma\in W^{1,1}(0,1)$ and its
  absolutely continuous representative $f_\gamma$ satisfies
  \eqref{eq:592}.
  \medskip

  \noindent\textbf{(b)}
  follows as in the previous claim (a); from \eqref{eq:591} it is also
  not difficult to check that $f_\gamma(i)=f\circ\Al_\gamma(i)$ for
  $i=0$ or $i=1$.
  \medskip

  \noindent
  Concerning the last statement, arguing as in the proof of Theorem \ref{thm:important-arcs}(e),
  it is not difficult to show that
  every bounded (resp.~nonnegative) function $h$ and for every
  $\psi\in
  \rmC([0,1])$ (resp.~nonnegative) the real maps
  \begin{displaymath}
    \gamma\mapsto \int_0^1 \psi(t)h(\gamma(t))\,\d t\quad\text{are
      Borel in }(\rmC([0,1];X),\tau_\rmC).
  \end{displaymath}
  Since by \ref{thm:important-arcs}(d) the map $\gamma\mapsto
  \Al_\gamma$ is Borel from $(\RA(X),\tau_\rmA)$ to
  $(\rmC([0,1];X),\tau_\rmC)$ and $\calC$, $\calC_c$ are countable, we
  deduce that the sets
  characterized by the family of inequalities \eqref{eq:591bis}
  or \eqref{eq:591} are Borel
  in $\RA(X)$.
\end{proof}

\begin{remark}[Sobolev regularity along $\calT_q$-almost every arc]
  \upshape By Proposition \ref{prop:invarianza}(a),
  it is easy to check that the properties to be Sobolev or strictly Sobolev
  along  $\calT_q $-almost every arc with S.u.g.~$g$ are invariant with respect to
  modification of $f$ and $g$ in $\mm$-negligible sets, and thus they
  make
  sense for Lebesgue classes.
  It is also not restrictive to consider only arcs $\gamma\in \RA_0(X)$.
\end{remark}
In the next Theorem we prove that existence of a $\calT_q $-weak upper gradient
yields strict Sobolev regularity along $\calT_q $-almost every arc.
This property is based on a preliminary lemma which 
provides this property for stretchable plans in $\calT_q^*$.
\begin{lemma} \label{prop:Rovereto}
Assume that $\ppi\in\calT_q^* $ is a stretchable plan 
and that $g:X\to [0,\infty]$ is a $\calT_q $-weak upper gradient of 
a $\mm$-measurable function $f:X\to\R$. Then $f$ is strictly Sobolev
along $\ppi$-almost every arc with S.u.g.~$g$, i.e.
\eqref{eq:591} or equivalently
 \begin{equation}\label{eq:pointwisewug}
   \biggl|\frac{\d}{\dt}f_\gamma\biggr|\leq
      \ell(\gamma) g\circ\Al_\gamma\quad\text{a.e.~in $[0,1]$},\quad
      f_\gamma(i)=f(\Al_\gamma(i))\quad i\in \{0,1\},
    \end{equation}
    hold for $\ppi$-almost every
        $\gamma\in\RA(X)$.
\end{lemma}
\begin{proof}
  Arguing as in Proposition \ref{prop:invarianza} it is not
  restrictive
  to assume that $f,g$ are Borel function. Since $\ppi\in \calT_q^*$
  % Moreover, since $\ppi$ is a Radon measure and $\ell$ is a
  % l.s.c.~(thus Lusin $\ppi$-measurable) map, it is not restrictive
  we know that there exists 
  a compact set $\cK\subset
  \RA(X)$$\ppi$ satisfying \eqref{eq:562}.
  The stretchable condition \eqref{eq:437}
  and an obvious change of variables related to the maps
  $\Al_\gamma^{s\to t}$ of \eqref{eq:432} yields for every $s<t$ in $[0,1]$ and for
  $\ppi$-almost every $\gamma\in \cK$,
  \begin{equation}
      |f(\Al_\gamma(t))-f(\Al_\gamma(s))|\leq
      (t-s)\ell(\gamma)\int_0^1
      g(\Al_\gamma^{s\to t}(r))\,\d
      r =
      \ell(\gamma)\int_s^t g(\Al_\gamma(r))\,\d r,\label{eq:442}
    \end{equation}
    since $\Restr st(\ppi)\in \calT_q.$ 
    We apply Fubini's Theorem 
    to the product measure $\Leb2\otimes\ppi$ in $(0,1)^2\times
    \cK$ and we use the fact that the maps
    characterizing the inequality \eqref{eq:442} are jointly Borel
    with respect to $(s,t,\gamma)\in (0,1)^2\times \cK$: here we use
    the continuity of $\Al$ from $\cK$ to $\BVC_c([0,1];X)$ endowed
    with the topology $\tau_\rmC$.
    It follows that for $\ppi$-a.e. $\gamma$ the function
     $f$ satisfies
    \[
    |f(\Al_\gamma(t))-f(\Al_\gamma(s))|\leq \ell(\gamma)\int_s^t g(\Al_\gamma(r))\,\d
    r \qquad\text{for $\Leb{2}$-a.e. $(t,s)\in (0,1)^2$.}
    \]
    An analogous argument shows that
    for $\ppi$-a.e.~$\gamma$
    \begin{equation}
      \label{eq:222bis}
      \left\{
    \begin{aligned}
      \textstyle |f(\Al_\gamma(s))-f(\gamma_0)|&\textstyle 
      \leq \ell(\gamma)\int_0^s
      g(\Al_\gamma(r))\,\d r\\
      \textstyle |f(\gamma_1)-f(\Al_\gamma(s))|&\textstyle \leq \ell(\gamma)\int_s^1
      g(\Al_\gamma(r))\,\d r
    \end{aligned}\right.
    \qquad\text{for $\Leb{1}$-a.e. $s\in (0,1)$.}
\end{equation}
Since $g\circ \Al_\gamma\in L^1(0,1)$ for
$\ppi$-a.e.\ $\gamma\in \cK$ with $\ell(\gamma)>0$,,  
by the next Lemma \ref{le:dopo}
% \ref{lem:Fibonacci}
it follows that $f\circ\Al_\gamma\in W^{1,1}(0,1)$
for $\ppi$-a.e. $\gamma$ and (understanding the derivative of $f\circ\Al_\gamma$ as the distributional one)
\begin{equation}\label{eq:pointwisewug1}
  \biggl|\frac{\d}{\dt}(f\circ\Al_\gamma)\biggr|\leq
  \ell(\gamma) g\circ\Al_\gamma\quad\text{$\Leb1$-a.e.~in $(0,1)$,\quad for
    $\ppi$-a.e. $\gamma$.}
\end{equation}
We conclude that $f\circ\Al_\gamma\in
W^{1,1}(0,1)$ for $\ppi$-a.e. $\gamma$, and therefore it admits an
absolutely continuous representative $f_\gamma$ for which \eqref{eq:pointwisewug} holds; moreover,
by \eqref{eq:222bis}
, it is immediate to check that
$f(\gamma(t))=f_\gamma(t)$ for $t\in \{0,1\}$
and $\ppi $-a.e.\ $\gamma$.
\end{proof}
\begin{lemma}
  \label{le:dopo}
  Let $f:(0,1)\to \R$ and $g\in L^q(0,1)$ nonnegative satisfy
  \begin{equation}
    \label{eq:561}
    \big|f(t)-f(s)\big|\le \Big|\int_s^t g(r)\,\d r\Big|
    \quad
    \text{for $\Leb 2$-a.e.~$(s,t)\in (0,1)\times (0,1)$.}
  \end{equation}
  Then $f\in W^{1,q}(0,1)$ and $|f'|\le g$ $\Leb 1$-a.e.~in $(0,1)$.
\end{lemma}
We refer to \cite[Lemma 2.1]{AGS13} for the proof.

We can considerably refine Lemma \ref{prop:Rovereto},
by removing the
assumption that $\ppi$ is stretchable and by considering arbitrary
nonparametric dynamic plans in $\calT_q$.
\begin{theorem} \label{thm:Rovereto}
  Assume that 
$g:X\to [0,\infty]$ is a $\calT_q $-weak upper gradient of 
a $\mm$-measurable function $f:X\to\R$. Then $f$ is strictly Sobolev
and satisfies \eqref{eq:pointwisewug}
along $\calT_q $-almost every arc.%  and 
 % \begin{equation}\label{eq:pointwisewug2}
 %      \biggl|\frac{\d}{\dt}f_\gamma\biggr|\leq
 %      \ell(\gamma) g\circ\Al_\gamma\quad\text{a.e.~in $[0,1]$,\quad
 %        for $\calT_q$-almost every
 %        $\gamma\in\RA(X)$.}
 %    \end{equation}
\end{theorem}
\begin{proof}
  By Lemma \ref{le:nice} it is sufficient to prove
  the property for every $\ppi\in \calT_q^*$,
  so that we can also assume that there exists
  a compact set $\cK\subset \RA(X)$ satisfying \eqref{eq:562}.
  
  % we fix $\tau\in
  % (0,1/2)$ and 
  For every $r\in [0,1/3]$ and $s\in [2/3,1]$ we consider the
  rescaled plans
  \begin{equation}
    \label{eq:443}
    \ppi^+_r:=(\Restr r1)_\sharp \ppi,\quad
    \ppi^-_s:=(\Restr 0s)_\sharp (\ppi)
  \end{equation}
  which form two continuous (thus Borel) collections depending on
  $r\in [0,1/3],\ s\in [2/3,1]$. We then set
  \begin{equation}
    \label{eq:444}
    \ppi^+:=3\int_0^{1/3} \ppi_r^+\,\d r,\quad
    \ppi^-:=3\int_{2/3}^1\ppi_s^-\,\d s.
  \end{equation}
  Notice that we can 
  equivalently characterize $\ppi^+,\ppi^-$ as the push forward measures
  of
  \begin{displaymath}
    \text{$\ssigma^+:=3\Leb1\restr{(0,1/3)} \otimes \ppi$
  and $\ssigma^-:=3\Leb 1\restr{(2/3,1)}\otimes \ppi$}
\end{displaymath}
  through the continuous maps $(r,\gamma)\mapsto \Restr
  r1(\gamma)$ and  $(s,\gamma)\mapsto \Restr{0}s(\gamma)$
 respectively:
  \begin{equation}
  \ppi^+=(\Restr \cdot 1)_\sharp \big(3\Leb1\restr{(0,1/3)} \otimes \ppi\big),\quad
  \ppi^-=(\Restr 0\cdot)_\sharp\big(3\Leb 1\restr{(2/3,1)}\otimes \ppi\big).
  .\label{eq:446}
\end{equation}
  Let us check that $\ppi^\pm$ belong to $\calT_q^*$ and are
  stretchable.
  We only 
  consider $\ppi^+$, since the argument for $\ppi^-$ is completely
  analogous.
  Recalling Lemma \ref{le:bar-check}, for every nonnegative $f\in \rmC_b(X)$ we have
  \begin{align}
    \notag
    \int\int_\gamma f\,\d\ppi^+(\gamma)
    &      \topref{eq:444}=
      3\int_0^{1/3}\Big(
      \int \int_\gamma f\,\d\ppi^+_r(\gamma)\Big)\,\d r
            \topref{eq:445}=
      3\int_0^{1/3} \Big(
      \int \ell(\gamma)\int_r^1 f(\Al_\gamma(s))\,\d
      s\,\d\ppi(\gamma)\Big)\,\d r
    \\&\ \ \le
    3\int_0^{1/3} \Big(
    \int \int_\gamma f\,\d\ppi(\gamma)\Big)\,\d r
    % =3\int_0^{1/3}\Big(\int_X f\,\d\mu_\sppi\Big)\,\d r
    % \\&=
    % \int_X f\,\d\mu_\sppi
    \topref{eq:defboundcomp}\le \|f\|_{L^p}\brq q(\ppi),
    \label{eq:uno}
  \end{align}
  which shows that $\brq q(\ppi^+)\le \brq q(\ppi)$.
  
  Let us now prove that $(\gsfe_s)_\sharp \ppi^+\ll\mm$ with density
  in $L^q(X,\mm)$ for every $s\in [0,1]$;
  setting $\ell_{o}:=\min_\cK\ell>0$,  if $s<1$ we have
  \begin{align}
    \notag
    \int f\, \d (\gsfe_s)_\sharp \ppi^+
    &=\int f(\Al_\gamma(s))\,\d\ppi^+(\gamma)
      \topref{eq:444}=
      3\int_0^{1/3} \Big(
      \int f(\Al_\gamma(s))\,\d\ppi^+_r(\gamma)\Big)\,\d r
      \\    \notag&\topref{eq:446}=
      3\int_0^{1/3} \Big(\int
      f(\Al^{r\to 1}_\gamma(s))\d\ppi(\gamma)\Big)\,\d r
    =3 \int\Big(\int_0^{1/3}
    f(\Al_\gamma(r(1-s)+s))
    \,\d r\Big)\d\ppi(\gamma)
    \\    \notag&=\frac{3}{1-s}\int
    \Big(\int_s^{s+(1-s)/3}f(\Al_\gamma(\theta))\,\d \theta
    \Big)\,\d \ppi(\gamma)
    \\&\le \frac{3}{\ell_o(1-s)}
    \int\int_\gamma f\,\d\ppi(\gamma)
    \le
    \frac{3}{\ell_o(1-s)}\|f\|_{L^p}\brq q(\ppi).
    \label{eq:due}
    % \int_X f\,\d\mu_\sppi\le \|f\|_{L^p}\|h_\sppi\|_{L^q},\quad
    % \mu_\sppi=h_\sppi\mm.
  \end{align}
  On the other hand, if $s=1$, we can use the fact that
  $(\sfe_1)_\sharp\ppi_r=(\sfe_1)_\sharp\ppi$
  \begin{align*}
    \int f(\gamma_1)\,\d\ppi^+(\gamma)&=
    3\int_0^{1/3} \Big(
    \int f(\gamma(1))\,\d\ppi^+_r(\gamma)\Big)\,\d r
    =
    3\int_0^{1/3} \Big(
    \int f(\gamma(1))\,\d\ppi(\gamma)\Big)\,\d r
                                        \\&=\int f\,\d(\sfe_1)_\sharp \ppi,
  \end{align*}
  so that $(\sfe_1)_\sharp\ppi^+=(\sfe_1)_\sharp\ppi$ which has an
  $L^q$ density w.r.t.~$\mm$.
  
  Let us now select a Borel representative of $f$.
  Applying Lemma \ref{prop:Rovereto}
  we know that $f$ is Sobolev along $\ppi^+$ and $\ppi^-$-almost every
  arc. Recalling the representation result
  \eqref{eq:446} and applying Fubini's Theorem, we can find
  a $\ppi$-negligible set $N\subset \cK$
  such that for every $\gamma\in \cK\setminus N$
  the map $f$ is Sobolev along the arcs 
  $\Restr r1(\gamma)$ and $\Restr 0s(\gamma)$ for $\Leb 1$-a.e.~$r\in
  [0,1/3]$
  and $\Leb 1$-a.e.~$s\in [2/3,1]$ and \eqref{eq:pointwisewug} holds.
  Choosing arbitrarily $r\in [0,1/3]$ and $s\in [2/3,1]$ so that
  such a property holds, since the absolutely continuous
  representative $f_\gamma$ should
  coincide along the curve $t\mapsto \Al_\gamma(t)$ in the interval $[r,s]$,
  one immediately sees
  that $f$ is Sobolev along $\gamma$ and \eqref{eq:pointwisewug} holds
  as well. We conclude that $f$ is Sobolev along $\ppi$-a.e.~arc
  and since $\ppi$ is arbitrary in $\cT_q^*$ we get the thesis.
\end{proof}

\begin{remark}[Equivalent formulation]
  \label{re:restr}
  \upshape
  By a similar argument we obtain an equivalent formulation of the
  weak upper gradient property
  when $f$ is Sobolev along
  $\calT_q $-almost every arc: a function $g$ satisfying
  $\int_\gamma g<\infty$ for $\calT_q $-almost every
  arc $\gamma$ is a $\calT_q$-weak upper gradient
  of $f$ if and only if \eqref{eq:591} holds for
  every $\varphi$ in a dense subset of $\rmC^1([0,1])$ and $\calT_q $-almost every
  arc $\gamma$, or, equivalently, 
  the function $f_\gamma$ of Definition~\ref{def:Sobolev} satisfies 
  \eqref{eq:pointwisewug} $\calT_q $-almost everywhere.
\end{remark}
\subsection{The link with $\Md$-weak upper gradients}
\label{subsec:link}
In this section we will show that the definition of $\calT_q$-weak
upper gradient can be equivalently stated in terms of $\Md$,
as in the Newtonian approach to metric Sobolev spaces.
Part of the results stated here could also be derived as
a consequence of the identification Theorem
of Section \ref{sec:Identification}, so we will just sketch the main
ideas.

First of all we can associate a $q$-barycentric entropy to plans in
$\calT_q$: we consider the measure $\tilde\mu_\sppi\in \cMp(X)$ defined by
\begin{equation}
  \label{eq:598}
  \int_Xf\,\d\tilde\mu_\sppi:=\int\Big(f(\gamma_0)+f(\gamma_1)+\int_\gamma
  f\Big)\,\d\ppi(\gamma)=
  \int\Big(\int f\,\d\tilde\rmM\gamma\Big)\,\d\ppi(\gamma)
\end{equation}
where $\tilde\rmM:\RA(X)\to\cMp(X)$ has been defined in
\eqref{eq:594}.
We then set
\begin{equation}
  \label{eq:599}
  \frac 1q\tbrq q^q(\ppi):=\LL^q(\tilde\mu_\sppi|\mm)=\frac 1q\int_X
  h^q\,\d\mm\quad
  \text{if }\tilde\mu_\sppi=h\mm.
\end{equation}
and it is easy to check that
\begin{equation}
\text{$\ppi\in \calT_q$ if and only if $\tbrq
  q(\ppi)<\infty$.}
\label{eq:603}
\end{equation}
It is clear that $\tbrq q(\ppi)\ge \brq q(\ppi)$ for every
dynamic plan $\ppi$. Arguing as in the proof of Lemma
\ref{le:RA-tight}
one can also see that for every $k\ge0$
\begin{equation}
  \label{eq:600}
  \text{if $(X,\sfd)$ is complete then the set
    $\Big\{\ppi\in\cMp(\RA(X)):\tbrq q(\ppi)\le k\Big\}$ is compact.}
\end{equation}
By duality we obtain the corresponding notion of content
\begin{equation}
  \label{eq:601}
  \frac 1p\tCont p(\Gamma):=
  \sup\Big\{\ppi(\Gamma)-\frac 1q\tbrq q^q(\ppi):\ppi\in \calT_q\Big\},
\end{equation}
and we can obtain an important characterization of $\tMd$, as
for Theorems \ref{thm:compact-Mod} and \ref{tmain}:
\begin{theorem}
  \label{thm:pourri}
  \begin{enumerate}
  \item If $\Gamma$ is a compact subset of $\RA(X)$ then
    \begin{equation}
      \label{eq:602}
      \tCont p^p(\Gamma)=\tMd(\Gamma)=\tMdc(\Gamma).
    \end{equation}
  \item $\tMd$ is a $\KK(\RA(X),\tau_\rmA)$-Choquet capacity in
    $\RA(X)$.
  \item For every universally measurable $\Gamma\subset \RA(X)$
    \begin{equation}
      \label{eq:597}
      \tCont p(\Gamma)=\sup\Big\{\tCont p(\rmK):\rmK\subset \Gamma,\
      \rmK\text{ compact}\Big\}.
    \end{equation}
  \item If $(X,\sfd)$ is complete and $(X,\tau)$ is Souslin then
    every $\BB(\RA(X),\tau_\rmA)$-analytic set $\Gamma$ is
    $\tMd$-capacitable
    and satisfies
    $\tMd(\Gamma)=\tCont p^p(\Gamma)$.
    In particular $\Gamma$ is $\tMd$-negligible if and only if it is $\calT_q$-negligible.
  \end{enumerate}  
\end{theorem}
We leave the proof to the reader:
Claim (a) is based on the same min-max argument of
Theorem \ref{thm:compact-Mod} (replacing integration
w.r.t.~$\nu_\gamma$
with integration w.r.t.~$\tilde\nu_\gamma$),
Claims (b-d) can be obtained by arguing as in 
Theorem \ref{tmain}. In fact, the proofs would be slightly
easier,
since compactness of sublevels of $\tbrq q$ and tightness of $\tMd$ behave
better
than the corresponding properties for $\brq q $ and $\Md$.
It would also be possible to derive the proofs by a general duality
between $\Mdm$ and a corresponding notion of content in $\cMp(X)$,
see \cite[\S 5]{ADS15}.
\medskip

\noindent
Let us now observe that 
if we consider only Sobolev regularity along arcs, we can
improve Theorem \ref{thm:Rovereto}.
\begin{proposition}\
  \label{prop:Sobreg}
  \begin{enumerate}
    \item
    If $g\in L^p(X,\mm)$, $g\ge0$, is a $\calT_q $-weak upper gradient of a
    $\mm$-measurable function $f:X\to\R$, then $f$ is Sobolev
    with S.u.g.~$g$ along
    $\Bar q\mm$-almost every arc; \eqref{eq:592} holds
    % \begin{equation}\label{eq:pointwisewug3}
    %   \biggl|\frac{\d}{\dt}f_\gamma\biggr|\leq
    %   \ell(\gamma) g\circ\Al_\gamma\quad\text{a.e.~in $[0,1]$},
    % \end{equation}
    for $\Bar q\mm$-almost every
    $\gamma\in\RA(X)$.
  \item
    If moreover $(X,\sfd)$ is complete and $(X,\tau)$ is Souslin, then
    $f$ is Sobolev with S.u.g.~$g$ along $\Md$-almost every arc and
    \eqref{eq:592} holds $\Md$-a.e.
  % \item Let us suppose that $(X,\sfd)$ is complete and $(X,\tau)$ is Souslin
  %   and let $B\subset \RA(X)$ be a Borel set such that
  %   $f$ is Sobolev along every arc $\gamma\in B$.
  %   Then the map $F:(\gamma,t)\mapsto f_\gamma(t)$ is
  %   universally measurable
  %   in
  %   $B\times [0,1]$.
  \end{enumerate}
\end{proposition}
  \begin{proof}
    \textbf{(a)}
    It is not restrictive to assume that $f$ and $g$ are Borel.
    Let us show that
    for every plan $\ppi\in \Bar q\mm$
    $f$ is Sobolev along $\ppi$-a.e.~arc $\gamma$ and
    \eqref{eq:592}
    holds $\ppi$-a.e. Since $\ppi$ is Radon and both the properties
    trivially holds along constant arcs,
    it is not restrictive to assume that $\ppi$ it is concentrated
    on a compact set $\rmK\subset \RA_0(X)$ where $\ell$ is
    continuous. In particular 
    the map $T:(r,\gamma)\mapsto \Restr {1-r}r(\gamma)$ is continuous
    in $[0,1/3]\times \rmK$. Arguing as in the proof of Theorem
    \ref{thm:Rovereto} we define
    \begin{equation}\label{eq:641}
      \ppi_r:=(\Restr{1-r}r)_\sharp\ppi,\quad
      \tilde\ppi:=3\int_0^{1/3} \ppi_r\,\d r=
      T_\sharp(3\Leb 1\restr{(0,1/3)}\otimes \ppi),
    \end{equation}
    and by calculations similar to \eqref{eq:uno} and \eqref{eq:due}
    we can check that $\tilde\ppi\in \calT_q$.
    By Theorem \ref{thm:Rovereto}
    we deduce that
    $f$ is Sobolev along $\tilde \ppi$-a.e.~arc and
    \eqref{eq:592}
    holds for $\tilde\ppi$-a.e.~$\gamma$. Applying Fubini's Theorem
    we can find a $\ppi$-negligible Borel set $\rmN\subset \RA(X)$
    such that for every $\gamma\in \RA(X)\setminus \rmN$ 
    $f$ is Sobolev with S.u.g.~$g$
    along the arcs $\Restr{1-r}r(\gamma)$ for $\Leb 1$-a.e.~$r\in
    (0,1/3)$. For every $\gamma\in \RA(X)\setminus \rmN$
    we can thus find a vanishing sequence $r_n\down0$
    such that $f$ is Sobolev and \eqref{eq:592} holds along
    $\Restr{1-r_n}{r_n}(\gamma)$. We can thus pass to the limit and
    obtain
    the same properties along $\gamma$.
    \medskip

    \noindent\textbf{(b)}
    As in the previous Claim, it is not restrictive to assume
    $f,g$ Borel; by Remark \ref{rem:truncations}
    we can also suppose that $f$ is
    bounded.
    Let us consider a countable dense subset $\calC_c$ of $\rmC^1_c(0,1)$
    and let us define the sets
    \begin{align}
      \label{eq:588}
      A_0:={}&\Big\{\gamma\in \RA_0(X):\int_\gamma
             g<\infty\Big\},
             % \quad
      % B:={}\Big\{\gamma\in A: f\text{ is Sobolev along
      %        $\gamma$}\Big\},
      \\\label{eq:516}
      B_0:={}&\Big\{\gamma\in A_0: \Big|\int_0^1
             \varphi'(t)f(\Al_\gamma(t))\,\d t\Big|\le
             \ell(\gamma)\int_0^1|\varphi(t)|\,g(\Al_\gamma(t))\,\d t
             \quad\text{for every }\varphi\in \calC_c\Big\}.
    \end{align}
    By Theorem \ref{thm:important-arcs}(e) $A_0$ is a Borel set;
    Proposition \ref{prop:prop}(b) shows that $\Md(\RA_0(X)\setminus
    A_0)=0$. By Lemma \ref{le:tedious-wug} for every arc $\gamma\in A_0$,
    $f$ is Sobolev along $\gamma$ with S.u.g.~$g$ if and
    only
    if $\gamma\in B_0$. 
    Lemma \ref{le:tedious-wug} also shows that $A_0\setminus B_0$ is
    Borel. Since by Claim (a) we know that $\Cont p(A_0\setminus
    B_0)=0$, we get $\Md (A_0\setminus B_0)=0$
    by Theorem \ref{tmain}(d).
  \end{proof}
  According to the Definition \ref{def:Sobolev}, Proposition
  \ref{prop:Sobreg} ensures that for $\Md$-a.e.~arc $\gamma$
  a function $f$ with $\calT_q$-w.u.g.~in $L^p(X,\mm)$
  coincides $\Leb 1$-a.e.~with an absolutely continuous function
  $f_\gamma$.
  We can in fact prove a much better result, which establishes
  a strong connection with the theory of Newtonian Sobolev spaces.
  \begin{theorem}[Good representative]
    \label{thm:representative}
    Let us suppose that $(X,\sfd)$ is complete and $(X,\tau)$ is
    Souslin.
    Every $\mm$-measurable function $f$ with a $\calT_q$-w.u.g.~$g\in
    L^p(X,\mm)$ admits a Borel $\mm$-representative $\tilde f$ such
    that
    $\tilde f\circ\Al_\gamma$ is absolutely continuous with S.u.g.~$g$
    along $\Md$-a.e.~arc $\gamma$ (and a fortiori along $\calT_q$-a.e.~arc).
    %   \item $\tilde f\circ\Al_\gamma$ is absolutely continuous and satisfies
    %   \eqref{eq:592} along $\Md$-a.e.~arc $\gamma$.
    % \end{enumerate}
  \end{theorem}
  \begin{proof}
    As usual, it is not restrictive to assume
    that $f,g$ are Borel maps, $f$ bounded.
    We will also denote by $f_\gamma$ 
    the absolutely continuous representative of $f\circ \Al_\gamma$
    whenever $f$ is Sobolev along $\gamma$.    
    \medskip

    \noindent
    % \textbf{(a)}, 
    \underline{Claim 1}: {\em There exists 
      $h\in \cL_+^p(X,\mm)$,
      such that
      \begin{equation}
      \text{$f$ is Sobolev with S.u.g.~$g$ along all the arcs of
        $H:=\Big\{\gamma\in \RA_0(X):\int_\gamma
        h<\infty\Big\}$.}\label{eq:604}
    \end{equation}
     \begin{equation}
        \label{eq:605}
        \begin{aligned}
          &\text{$f$ is \underline{strictly} Sobolev with S.u.g.~$g$
            along all the
            arcs of}\\
          &\qquad H_0:=\Big\{\gamma\in \RA_0(X):
          h(\gamma_0)+h(\gamma_1)+\int_\gamma h<\infty\Big\}.
        \end{aligned}
      \end{equation}
    }%
    Notice that $H_0\subset H$, $\Md(\RA_0(X)\setminus H)=0$, and
    $\tMd(\RA_0(X)\setminus H_0)=0$.\\
    By Proposition \ref{prop:Sobreg} and Proposition
    \ref{prop:prop}(b)
    we can find
    a Borel function $h'\in \Ldp$ such that
    $f$ is Sobolev with S.u.g.~$g$
    along all the arcs of
    $H':=\Big\{\gamma\in \RA_0(X):\int_\gamma
    h'<\infty\Big\}$.
    Notice that $\Md(\RA_0(X)\setminus H')=0$.

    In order to get \eqref{eq:605}
    we argue as in the proof of Proposition \ref{prop:Sobreg}(b):
    we fix a countable set $\calC$ dense
    in $\rmC^1([0,1])$ and
    we consider the sets
    \begin{align}
      \label{eq:588bis}
      A:={}&\Big\{\gamma\in
             \RA(X):
             % |f(\gamma_0)|+|f(\gamma_1)|+\int_\gamma
             %   f<\infty,\
             \int_\gamma
             g<\infty\Big\},
             % \quad
             % B:={}\Big\{\gamma\in A: f\text{ is Sobolev along
      %        $\gamma$}\Big\},
      \\\notag
      B:={}&\Big\{\gamma\in A: \Big|
             \varphi(1)f(\Al_\gamma(1))-\varphi(0)f(\Al_\gamma(0))-\int_0^1
             \varphi'(t)f(\Al_\gamma(t))\,\d t\Big|\le
             \\&\label{eq:516bis}\qquad\qquad
             \ell(\gamma)\int_0^1|\varphi(t)|\,g(\Al_\gamma(t))\,\d t
      \quad\text{for every }\varphi\in \calC\Big\},
    \end{align}
    By Theorem \ref{thm:important-arcs}(e) $A$ is a Borel set;
    Proposition \ref{prop:prop}(b) shows that $\tMd(\RA(X)\setminus
    A)=0$. By Lemma \ref{le:tedious-wug} for every arc $\gamma\in A$,
    $f$ is Sobolev along $\gamma$ and \eqref{eq:pointwisewug} holds if and
    only
    if $\gamma\in B$, so that $\tCont p(A\setminus B)=0$.
    Lemma \ref{le:tedious-wug} also shows that $B$ is Borel,
    so that $\tMd(\RA(X)\setminus B)=0$
    by Theorem \ref{thm:pourri}(d). We can eventually apply
    Proposition \ref{prop:prop} to find $h_0'\in \Ldp$ such that
    $f$ is Sobolev with S.u.g.~$g$ along all the arcs of
    $H_0':=\big\{\gamma\in \RA_0(X):
    h_0'(\gamma_0)+h_0'(\gamma_1)+\int_\gamma h_0'<\infty\big\}$
    We can eventually set $h:=h'+h_0'$ and define the sets $H$ and
    $H_0$ accordingly.
    % \medskip
    % \noindent    
    % % Since $A\setminus B$ is Borel and $\calT_q$-negligible,
    % % by Theorem \ref{thm:pourri}
    % % we deduce that $\tMd(\RA(X)\setminus B)=0$
    % Thanks to Claim 1, we We can then reduce $B$ by setting
    % $B':=\{\gamma\in B:h(\gamma_0)+h(\gamma_1)+\int_\gamma h<\infty\}$, still obtaining a
    % Borel set
    % with $\tMd$-negligible complement. 
    \medskip

    \noindent    
    \underline{Claim 2}: \emph{If $\gamma,\gamma'\in H$ and
      $\Al_\gamma(r)=\Al_{\gamma'}(r')$ for some $r,r'\in [0,1]$ then
      $f_\gamma(r)=f_\gamma(r')$.} \\
    Let us argue by contradiction assuming that there exist
    $\gamma,\gamma'\in H$ and $r,r'\in [0,1]$ such that
    $\Al_\gamma(r)=\Al_{\gamma'}(r')=x$ but
    $f_\gamma(r)\neq f_{\gamma'}(r')$. Up to a possible inversion of 
    the orientation of $\gamma$ or $\gamma'$ it is not restrictive
    to assume that $r>0$ and $r'<1$.
    % Let us distinguish some cases:
    % \begin{enumerate}[i.]
    % \item If $\ell(\gamma)=\ell(\gamma')=0$ then $\Al_\gamma$ and
    %   $\Al_{\gamma'}$ are constants and coincides with $x$, so that
    %   $f_\gamma= f_{\gamma'}\equiv f(x)$.
    % \item $\ell(\gamma)=0$ and $\ell(\gamma')>0$.
    %   $\Al_\gamma\equiv x$ so that $f_\gamma(r)=f(x)$. If $r'=0$ then
    %   also $f_{\gamma'}(r')=f(x)$. If $r'>0$ we can then consider the
    %   curve $\gamma'':=\Restr{r'}1(\gamma')$ noticing
    %   that $\int_{\gamma''}h\le \int_{\gamma'} h<\infty$,
    %   $h(\Al_{\gamma''}(0))=h(\Al_{\gamma'}(r'))=h(\Al_\gamma(0))<\infty$
    %   (since $\gamma\in H$) and
    %   $h(\Al_{\gamma''}(1))=h(\Al_{\gamma'(1)})<\infty $ (since
    %   $\gamma'\in H$).
    %   It follows that $\gamma''\in H$ so that
    %   $f$ admits an absolutely continuous representative
    %   $f_{\gamma''}$ along $\gamma''$ which should coincide (up to an obvious
    %   rescaling) with $f_{\gamma'}$. In particular, since
    %   $\Al_{\gamma''}(0)=\Al_{\gamma'}(r')=x$ we have
    %   $f(x)=f_{\gamma''}(\Al_{\gamma''}(0))=f_{\gamma'}(r')$.
    % \item If $\ell(\gamma)>0$ and $\ell(\gamma')=0$ we can argue as in
    %   the case ii.
  %$\ell(\gamma)\cdot\ell(\gamma')>0$.
    We can then consider the curve
    $\gamma''$ obtained by gluing $\gamma_-=\Restr 0r(\gamma)$ and
    $\gamma_+=\Restr {r'}1(\gamma')$, with
    $\ell(\gamma'')=r\ell(\gamma)+
    (1-r')\ell(\gamma')$. 
    Clearly $\gamma''\in \RA(X)$ and
    $\int_{\gamma''}h=\int_{\gamma_-}h+\int_{\gamma_+}h<\infty$,
    % moreover $h(\Al_{\gamma''}(0))=h(\Al_{\gamma}(0))<\infty$ (since
    % $\gamma\in H$) and
    % $h(\Al_{\gamma''}(1))=h(\Al_{\gamma'}(1))<\infty$ (since
    % $\gamma'\in H$)
    so
    that $\gamma''\in H$ as well. Moreover, if
    $r''=r\ell(\gamma)/\ell(\gamma'')$ we have $R_{\gamma''}^{0\to
      r''}(t)=
    R_\gamma^{0\to r}(t)$ and
    $R_{\gamma''}^{r''\to1}(t)=R_{\gamma'}^{r'\to1}(t)$
    for every $t\in [0,1]$.
    It follows that
    $f_{\gamma''}(t)=f_\gamma(r t/r'')$ for $t\in [0,r'']$
    and $f_{\gamma''}(t)=f_{\gamma'}(r'+(1-r')(t-r'')/(1-r''))$
    so that
    $\lim_{t\up r_2}f_{\gamma''}(t)=f_\gamma(r)\neq
    \lim_{t\down r_2}f_{\gamma''}(t)=f_{\gamma'}(r')$,
    which conflicts with the fact that $f_{\gamma''}$ is absolutely
    continuous.  
    %\end{enumerate}
    \medskip

    \noindent
    \underline{Claim 3:}
    {\em Let us set
      \begin{equation}\label{eq:589}
        \tilde f(x):=
        \begin{cases}
          f_\gamma(r)&\text{if $x=\Al_\gamma(r)$ for some $\gamma\in
            H$ and $r\in [0,1]$,}\\
          f(x)&\text{otherwise}
        \end{cases}
      \end{equation}
      Then $\tilde f$ is well defined, $\tilde f(\Al_\gamma)\equiv
      f_\gamma$ for every $\gamma\in H$, and
    $\tilde f(x)=f(x)$ in $\{x\in X:h(x)<\infty\}$. In particular
    $\{\tilde f\neq f\}$ is $\mm$-negligible and $\tilde f=f_\gamma$
    along $\Md$-a.e.~arc (and a fortiori along $\Bar q\mm$ and $\calT_q$-a.e.~arc).} \\
    The facts that $\tilde f$ is well defined and
    $\tilde f(\Al_\gamma)\equiv
    f_\gamma$ for every $\gamma\in H$ follow directly from the
    previous claim.
    Let us now argue by contradiction and let us suppose that there
    exists
    $x\in X$ with$\tilde f(x)\neq f(x)$ and $h(x)<\infty$.
    By definition of $\tilde f$ there exists an arc $\gamma\in H$ and
    $r\in [0,1]$
    such that $\Al_\gamma(r)=x$.
    Since $\gamma\in H$ we know that $\int_\gamma h<\infty$: we can
    thus find $s\in [0,1]\setminus r$ such that
    $h(\Al_\gamma(s))<\infty$.
    Assuming that $r<s$ (otherwise we switch the order of $r$ and
    $s$),
    we can consider the arc $\gamma':=\Restr rs(\gamma)$
    which satisfies $\int_{\gamma'} h\le \int_\gamma h<\infty$ and
    $h(\Al_{\gamma'}(0))=h(\Al_\gamma(r))=h(x)<\infty$ and
    $h(\Al_{\gamma'}(1))=h(\Al_\gamma(s))<\infty$. We deduce that
    $\gamma'\in H_0$ so that $f$ is \emph{strictly} Sobolev along $\gamma'$
    and therefore $\tilde f(x)=\tilde f(\Al_{\gamma'}(0))=
    f_{\gamma'}(\Al_{\gamma'}(0))=f(\Al_{\gamma'}(0))=f(x)$,
    a contradiction.
    % \medskip
    %
    % \noindent\textbf{(b)} 
    % We can find 
    %
    % and
    %     \begin{align}
    %   \label{eq:588}
    %       A_0:={}&\Big\{\gamma\in \RA_0(X):\int_\gamma
    %          g<\infty\Big\},
    %          % \quad
    %   % B:={}\Big\{\gamma\in A: f\text{ is Sobolev along
    %   %        $\gamma$}\Big\},
    %   \\\label{eq:516}
    %   B_0:={}&\Big\{\gamma\in A: \Big|\int_0^1
    %          \varphi'(t)f(\Al_\gamma(t))\,\d t\Big|\le
    %          \ell(\gamma)\int_0^1|\varphi(t)|\,g(\Al_\gamma(t))\,\d t
    %          \quad\text{for every }\varphi\in \calC_c\Big\}.
    % \end{align}
  \end{proof}
  Let us apply the previous representation Theorem to prove the
  equivalence of the notion of $\calT_q$-w.u.g.~with the ``Newtonian''
  one
  introduced in \cite{Shanmugalingam00}.
\begin{definition}[Newtonian weak upper gradient] Let $f\in
  \cL^p(X,\mm)$. We say that $f$ belongs to the Newtonian space
  $N^{1,p}(\X)$ if 
  $f$ is absolutely continuous along $\Md$-a.e.~arc $\gamma\in
  \RA_0(X)$  and there exists a nonnegative $g\in L^p(X,\mm)$ such
  that 
  \label{def:newtonian}
  \begin{equation}
  \label{eq:606}
  \Big|\int_{\partial\gamma} f\Big|\le \int_\gamma g\quad\text{for
    $\Md$-a.e.~arc $\gamma\in \RA_0(X)$}.
\end{equation}
In this case, we say that $g$ is a $N^{1,p}$-weak upper gradient of $f$.
\end{definition}
Functions with $\Md$-weak upper gradient 
have the important Beppo-Levi property of being absolutely continuous
along $\Md$-a.e.~arc $\gamma$.
% (more precisely,
% this means $f\circ\sfj\pgamma\in\AC([0,1];X)$), see
% \cite[Proposition~3.1]{Shanmugalingam00}.
Because of the implication \eqref{eq:438},
%\eqref{eq:bl2},
functions with $\Md$-weak upper gradient have also $\calT_q$-weak upper gradient. 
A priori there is an important difference between the two definitions,
since Definition \ref{def:newtonian} is not invariant
w.r.t.~modifications of 
$f$ in a $\mm$-negligible set.
However, as an application of Theorem~\ref{thm:representative},
% \ref{teo:popen},
we can show that these two notions are essentially equivalent modulo the choice
of a representative in the equivalence class:
\begin{corollary}
  \label{cor:equivalence}
  Let us suppose that $\X$ is a complete Souslin e.m.t.m.~space.
  A function $f\in L^p(X,\mm)$ 
  admits a $\calT_q$-weak upper gradient $g\in L^p(X,\mm)$ 
  if and only if there is a Borel representative $\tilde f:X\to \R$ 
  with $\mm(\{\tilde f\neq f\})=0$ 
  which belongs to the Newtonian space
  $N^{1,p}(\X)$.
  Equivalently, $\tilde f$ is absolutely continuous along
  $\Md$-a.e.~arc and
  $g$ satisfies \eqref{eq:606} $\Md$-a.e.
  In particular, the class of $\calT_q$-w.u.g.~for 
  $f$ coincides with the class of $N^{1,p}$-w.u.g.~for
  a suitable Borel representative $\tilde f$ of $f$.
\end{corollary}

\subsection{Minimal $\calT_q$-weak upper gradient and the
  Sobolev space $W^{1,p}(\X,\calT_q)$.}
\label{subsec:calculus_weak}
\index{Minimal weak upper gradients}
We want now to characterize
the minimal $\calT_q$-w.u.g.~of a function and the
corresponding notion of Sobolev space.
% Let $f:X\to \R$ a function with 
% a $\calT_q$-weak upper gradient.
% In order to select the
% minimal element among $\calT_q$-weak upper gradients
We first prove two important properties.
The first one directly involves the characterization of Theorem
\ref{thm:Rovereto}.

\begin{proposition}[Locality]\label{prop:locweak}
Let $f:X\to\R$ be $\mm$-measurable
and let $g_1,\,g_2$ be weak upper gradients of $f$
 w.r.t.\ $\calT_q $. Then $\min\{g_1,g_2\}$ is a $\calT_q $-weak upper gradient of
 $f$.
\end{proposition}
\begin{proof} We know from
    Theorem \ref{thm:Rovereto} that $f$ is Sobolev along
$\calT_q $-almost every arc.
Then, the claim is a direct consequence of Remark~\ref{re:restr} and \eqref{eq:pointwisewug}.
\end{proof}
Another important property of weak upper gradients is their
stability w.r.t.~weak $L^p$ convergence.
\begin{theorem}[Stability w.r.t.~weak %and $\mm$-a.e.~
  convergence]\label{thm:stabweak}
% Let us suppose that $\calT_q $ is a stretchable collection of test
% plans concentrated on $\AC p{(0,1)}X\sfd$ for some $p\in (1,\infty]$
% such that for all $\ppi\in \calT_q $ and all $M\geq 0$ there exists
% $C=C(\ppi,M)\in [0,\infty)$ satisfying
% \begin{equation}\label{eq:incompre}
%    \int_0^1(\e_t)_\sharp\ppi(B\cap\{\Wgh\leq M\})\,\dt\leq
%    C(\ppi,M)\,\mm(B)\qquad\forall B\in\BorelSets{X}.
% \end{equation}
  Assume that $f_n\in L^1(X,\mm)$ and that $g_n\in L^p(X,\mm)$
  are $\calT_q $-weak upper gradients of $f_n$.
  If 
  $f_n\weakto f$ weakly in $L^1(X,\mm)$ and
  $g_n\weakto g$ weakly in $L^p(X,\mm)$ as $n\to\infty$,
then $f$ is Sobolev along $\calT_q$-a.e.~arc and
$g$ is a $\calT_q$-weak upper gradient of $f$.
% Assume furthermore that $f_n(x)\to f(x)\in\R$ for $\mm$-a.e. $x\in
% X$ and that $(G_n)$ weakly converges to $G$ in $L^q(\{\Wgh\leq
% M\},\mm)$ for all $M\geq 0$, where $q\in [1,\infty)$ is the
% conjugate exponent of $p$. Then $G$ is a $\calT_q $-weak upper gradient
% of $f$.
\end{theorem}
\begin{proof}
  We can apply Lemma \ref{le:equivalent1}:
  we fix a test plan $\ppi\in \calT_q^*$ and set $h_0,h_1
  \in L^\infty(X,\mm)$, $h\in L^q(X,\mm)$
    such that
    \begin{displaymath}
      (\sfe_0)_\sharp\ppi=\pi_0=h_0\mm,\quad
      (\sfe_1)_\sharp\ppi=\pi_1=h_1\mm,\quad
      \mu_\sppi=h\mm.
    \end{displaymath}
    Since $g_n$ is a $\calT_q$-weak upper gradient for $f_n$ we know
    that
    \begin{displaymath}
      \int_X f_n\,\d(\pi_1-\pi_0)=
      \int_X f_n(h_1-h_0)\,\d\mm\le
      \int_X g_n\,\d\mu_\sppi=
      \int_X g_n h\,\d\mm.
    \end{displaymath}
    Passing to the limit by weak convergence in $L^1$ and $L^p$ we immediately
    get
    \begin{equation}
      \label{eq:447}
      \int_X f\,\d(\pi_1-\pi_0)\le \int_X g\,\d\mu_\sppi.
    \end{equation}
    Since $\ppi\in \calT_q$ is arbitrary, we conclude.
 \end{proof}
We can now formalize the notion of
$\calT_q$-minimal weak upper gradient.
For the sake of simplicity, here we will consider only the case
of functions with $\calT_q$-w.u.g.~in $L^p(X,\mm)$.
\begin{definition}[Minimal $\calT_q$-weak upper gradient]
  \label{def:minimalwug}
  Let  $f\in L^1(X,\mm)$ be a $\mm$-measurable function with a
  $\calT_q $-weak upper gradient in $L^p(X,\mm)$.
  The $\calT_q $-minimal weak upper gradient $\weakgradTq f$ of $f$
  is the $\calT_q $-weak upper gradient characterized, up to
  $\mm$-negligible sets, by the property
\begin{equation}\label{eq:defweakgrad}
  \weakgradTq f\leq g\qquad\text{$\mm$-a.e. in $X$, for every $\calT_q $-weak upper gradient $g$ of $f$.}
\end{equation}  
 \end{definition}
Uniqueness of the minimal weak upper gradient is obvious. For
existence, let us consider a minimizing sequence $(g_n)_{n\in
  \N}\subset L^p(X,\mm)$ for the problem
$$
\inf\left\{\int_X g^p\,\d\mm:\ \text{$g$ is a $\calT_q $-weak upper
    gradient of $f$}
\right\}.
$$
We immediately see, thanks to Theorem~\ref{thm:stabweak}, that
we can assume with no loss of generality that $g_{n}\weakto g_\infty$
in $L^p(X,\mm)$ and $g_\infty$ is the 
$\calT_q $-weak upper gradient of $f$ of minimal $L^p$-norm.
This
minimality, in conjunction with Proposition~\ref{prop:locweak},
gives \eqref{eq:defweakgrad} for $\weakgradTq f:=g_\infty$.
\index{Cheeger energy (weak)}
\index{Metric Sobolev space $W^{1,p}(\X,\calT_q)$}
\begin{definition}[The weak $(\calT_q,p)$-energy and
  the Sobolev space
  $W^{1,p}(\X,\calT_q)$]
  \label{def:weakCheeger}
  Let $f\in L^1(X,\mm)$ with
  a $\calT_q$-weak upper gradient $g\in L^p(X,\mm)$.
  The weak $(\calT_q,p)$-energy of $f$ is defined by
  \begin{equation}
    \label{eq:563}
    \wCE_{p,\calT_q}(f)=
    \wCE_p(f):=
    \int_X \weakgradTq f^p\,\d\mm.
  \end{equation}
  If moreover $f\in L^p(X,\mm)$ we say that 
  $f$ belongs to the space $W^{1,p}(\X,\calT_q) $.
  $W^{1,p}(\X,\calT_q)$ is a Banach space endowed with the norm
  \begin{equation}
    \label{eq:471}
    \|f\|_{W^{1,p}(\X,\calT_q)}^p:=\int_X
    \Big(f^p+\weakgradTq f^p\Big)\,\d\mm=
    \|f\|_{L^p(X,\mm)}^p+\wCE_{p,\calT_q}(f).
  \end{equation}
\end{definition}
\begin{remark}[The $\calT_q$-notation]
  \label{rem:notation}
  Even if we will mainly use minimal w.u.g.~induced by $\calT_q$-test
  plan,
  we will keep the explicit occurrence of $\calT_q$ in the notation
  $\weakgradTq f$ and $\wCE_{p,\calT_q}$
  in order to distinguish these notions from other definitions
  of weak upper gradients based on different class of test plan
  (also on parametric arcs), which usually
  share the symbol $|\rmD f|_w$.
  We will use the shorter notation $\wCE_p$ only when
  no risk of confusion will be possible.
\end{remark}
By using the same approach, the construction of the minimal $p$-weak
upper gradient
$|\rmD f|_{w,N^{1,p}}$ can also be
performed
for functions in the Newtonian space $N^{1,p}(\X)$, and gives raise to
the (semi)norm
\begin{equation}
  \label{eq:607}
  \|f\|_{N^{1,p}(\X)}^p:=\int_X \Big(|f|^p+|\rmD f|_{w,N^{1,p}}^p\Big)\,\d\mm.
\end{equation}
Taking Corollary \ref{cor:equivalence} into account we easily have:
\begin{corollary}[The link with the Newtonian space $N^{1,p}(\X)$]
  \label{cor:isomorphism}
  Let us suppose that $\X$ is a complete Souslin e.m.t.m.~space.
  (The Lebesgue equivalence class of) every function $f\in N^{1,p}(\X)$
  belongs to $W^{1,p}(\X,\calT_q)$. Conversely, every function $f\in
  W^{1,p}(\X,\calT_q)$ has an equivalent representative $\tilde f$ in
  $N^{1,p}(\X)$ with
  \begin{equation}
    \label{eq:608}
    \weakgradTq f=|\rmD \tilde f|_{w,N^{1,p}}
    \quad\text{a.e.},\quad
    \|f\|_{W^{1,p}(\X,\calT_q)}=\|\tilde f\|_{N^{1,p}(\X)}.
  \end{equation}
\end{corollary}
It is easy to check using \eqref{eq:564} and Theorem
\ref{thm:stabweak}
that the weak Cheeger energy $\wCE_{p,\calT_q}$ is
a convex, $p$-homogeneous, weakly lower-semicontinuous functional
in $L^1(X,\mm)$.
It is also easy to state a first comparison with the strong Cheeger
energy $\CE_{p}$ (the corresponding
inequalities for $\CE_{p,\AA}$ and $\relgradA f$
follow trivially by \eqref{eq:360} and \eqref{eq:361}).
\begin{lemma}
  \label{le:obvious-comparison}
  Every function $f\in \Sob^{1,p}(\X)$ belongs to
  $W^{1,p}(\X,\calT_q)$ and
  \begin{equation}
    \label{eq:565}
    \CE_p(f)\ge \wCE_{p,\calT_q}(f),\quad
    \relgrad f\ge \weakgradTq f\quad\text{$\mm$-a.e.~in $X$}.
  \end{equation}
\end{lemma}
\begin{proof}
  We already notice that for a Lipschitz function $f\in
  \Lip_b(X,\tau,\sfd)$ $\lip f$ is a $\calT_q$-w.u.g.~so that
  \begin{equation}
    \label{eq:566}
    \pCE_p(f)\ge \wCE_{p,\calT_q}(f).
  \end{equation}
  It is then sufficient to take an optimal sequence $f_n\in
  \Lip_b(X,\tau,\sfd)$ as in \eqref{eq:16bis} and
  to apply the stability Theorem \ref{thm:stabweak}.
\end{proof}

\begin{proposition}[Chain rule for minimal weak upper gradients]\label{prop:chainweak}
If  $f\in L^1(X,\mm)$ 
has a $\calT_q $-weak upper gradient in $L^p(X,\mm)$,
the following properties hold:
\begin{itemize}
\item[(a)] for any $\Leb{1}$-negligible Borel set $N\subset\R$ it holds
$\weakgradTq f=0$ $\mm$-a.e. on $f^{-1}(N)$.
\item[(b)] $\weakgradTq{\phi(f)}=\phi'(f)\weakgradTq f$ $\mm$-a.e. in $X$, with the convention $0\cdot\infty=0$,
for any nondecreasing function $\phi$, Lipschitz on an
interval containing the image of $f$.
\end{itemize}
\end{proposition}
\begin{proof}
  We use the equivalent formulation of Remark~\ref{re:restr}
and the well-known fact that both (a) and (b) are true when $X=\R$
endowed with Euclidean distance and Lebesgue measure
and $f$ is absolutely continuous.
We can prove (a) setting 
$$
G(x):=
\begin{cases}
\weakgradTq f(x) &\text{if $f(x)\in\R\setminus N$};\\
0 &\text{if $f(x)\in N$}
\end{cases}
$$
and noticing the validity of (a) for real-valued absolutely continuous maps gives that $G$ is $\calT_q $-weak upper gradient of $f$.
Then, the minimality of $\weakgradTq f$ gives $\weakgradTq f\leq G$ $\mm$-a.e. in $X$.

By a similar argument based on \eqref{eq:pointwisewug} we can prove that $\weakgradTq{\phi(f)}\leq\phi'(f)\weakgradTq f$
$\mm$-a.e. in $X$. Then, the same subadditivity argument of
Theorem \ref{le:useful}(c)
%\ref{prop:chainrule}(d)
provides the equality $\mm$-a.e. in $X$.
%First we prove (b) in the case when $\phi$ is everywhere
%differentiable. By the same minimality argument in
%Proposition~\ref{prop:chainrule}(d), it suffices to show that
%$\weakgradTq {\phi(f)}\leq\phi'(f)\weakgradTq f$. This inequality is a
%direct consequence of \eqref{eq:pointwisewug}: indeed, if $f_\gamma$
%is the absolutely continuous function equal to $f\circ\gamma$ a.e.
%in $[0,1]$ and in $\{0,1\}$, we have $\phi\circ f_\gamma=\phi\circ
%f$ on $\{0,1\}$ and $|(\phi\circ f_\gamma)'|=
%\phi'(f_\gamma)|f_\gamma'|\leq\phi'(f_\gamma)\weakgradTq
%f\circ\gamma|\dot\gamma|$. Since $f_\gamma=f\circ\gamma$ a.e. in
%$[0,1]$, by integration we get that $\phi'(f)\weakgradTq f$ is a weak
%upper gradient.
%
%Having established the chain rule when $\phi$ is differentiable,
%the proof of (a) follows by the stability of weak gradients, as in
%the proof of Proposition~\ref{prop:chainrule}(a). Eventually we
%extend (b), which now makes sense defining arbitrarily $\phi'(f)$ at
%points where $x$ such that $\phi$ is not differentiable at $f(x)$,
%by a further approximation, as in Proposition~\ref{prop:chainrule}.
\end{proof}

\subsection{Invariance properties of weak Sobolev spaces}
\label{subsec:invariance-weak}
In this section we will state a few useful results on the behaviour of
weak Sobolev spaces with respect to some basic operations.
\subsubsection*{Restriction}
\begin{lemma}
  \label{le:simple-but-crucial}
  Let $\X=(X,\tau,\sfd,\mm)$ be an e.m.t.m.~space
  and let $Y\subset X$ be a $\sfd$-closed set
  such that $\mm(X\setminus Y)=0$.
  Then every dynamic plan $\ppi\in \Bar q\mm$ is concentrated on
  $\RA(Y)$.
\end{lemma}
\begin{proof}
  Let $\ppi\in \cMp(\RA(X))$ with $\brq
  q(\ppi)<\infty$. Setting $Z:=X\setminus Y$ we have
  \begin{align*}
    \int \int_\gamma \nchi_Z\,\d\ppi(\gamma)=
    \int_X \nchi_Z\,\d\mu_{\sppi}=0
  \end{align*}
  since $\mmu_\sppi\ll\mm$ and $\mm(Z)=0$. We deduce that
  for $\ppi$-a.e.~$\gamma$
  $\int_\gamma \nchi_Z =0$, i.e.~$\LL^1(\{t\in [0,1]:\Al_\gamma(t)\in
  Z\})=0$. Since $Z$ is $\sfd$-open, it follows that
  $\Al_\gamma([0,1])\subset Y$, i.e.~$\ppi$-a.e.~$\gamma$ belongs to $\RA(Y)$.
\end{proof}
\begin{corollary}[Invariance of $W^{1,p}$ by restriction]
  \label{cor:restriction}
  Let $Y\subset X$ be a $\sfd$-closed set
  such that $\mm(X\setminus Y)=0$.
  Setting $\Y:=(Y,\tau,\sfd,\mm)$, we have
  $W^{1,p}(\X,\calT_q)=W^{1,p}(\Y;\calT_q)$
  and for every Sobolev function $f$
  the minimal $\calT_q(\X)$-weak upper gradient
  coincides with the minimal $\calT_q(\Y)$-weak upper gradient.
\end{corollary}
\subsubsection*{Measure preserving isometric embeddings}
Let $\X=(X,\tau,\sfd,\mm)$ and $\hhat\X=(\hhat X,\hhat \tau,\hhat \sfd,\hhat\mm)$
be e.m.t.m.~spaces and let suppose that
$\iota:X\to\hhat X$ is a measure-preserving embedding according to
Definition \ref{def:embeddings}.
We will call $\hhat{\calT_q}=\calT_q(\hhat \X)$ the class of nonparametric test plans
in $\cMp(\RA(\hhat X))$.

Starting from $\iota$ we can define a continuous injective map $J:\rmC([0,1];X)\to\rmC([0,1];\hhat X)$ by setting
$J(\gamma):=\iota\circ \gamma$. Thanks to the isometric property of $\iota$,
  $J(\BVC([0,1];X)\subset \BVC([0,1];\hhat X)$ and
  clearly $J$ is preserves equivalence classes of curves, so that $J$
  induces a continuous injective map from $\RA(X)$ to $\RA(\hhat X)$
    satisfying
  \begin{equation}
    \label{eq:498}
    \int_{J\gamma}\hhat f=
    \int_\gamma \hhat f\circ \iota,\quad
    \ell(J\gamma)=\ell(\gamma)\quad\text{for every $\gamma\in \RA(X),\
      \hhat f\in \rmB_b(\hhat X)$.}
  \end{equation}
  It is interesting to notice that
  \begin{equation}
    \label{eq:502}
    \text{$\iota$ is surjective}\quad\Rightarrow\quad
    J\text{ is surjective.}
  \end{equation}
  In fact, given an arc $\hhat\gamma\in \RA(\hhat X)$ we can
  consider the curve $R:=\iota^{-1}\circ\Al_{\hhat\gamma}$ which satisfies
  \begin{displaymath}
    \sfd(R(s),R(t))=\hhat\sfd(\Al_{\hhat \gamma}(s),\Al_{\hhat
      \gamma}(t))=
    \ell(\hhat\gamma)|t-s|
  \end{displaymath}
  so that $R\in \Lip_c([0,1];(X,\sfd))\subset \BVC([0,1];(X,\sfd))$
  and
  $\gamma=\sfq (R)\in \RA(X)$ with $J\gamma=\hhat \gamma$.
  \begin{lemma}
    \label{le:equivalent-dynamic}
  For every dynamic plan $\ppi\in \cMp(\RA(X))$ the push forward 
  $\hhat\ppi:=J_\sharp \ppi$ is a dynamic plan in $\cMp(\RA(\hhat X))$
  satisfying
  \begin{equation}
    \label{eq:500}
    \mu_{\hhat \sppi}=\iota_\sharp \mu_\sppi,\quad
    \int\ell(\gamma)\,\d\ppi(\gamma)=
    \int \ell(\hhat \gamma)\,\d\hhat\ppi(\hhat \gamma),\quad
    (\sfe_i)_\sharp \hhat\ppi=
    \iota_\sharp(    (\sfe_i)_\sharp\ppi)\quad
    i=0,1.
  \end{equation}
  In particular
  \begin{equation}
    \label{eq:501}
    \brq q(\hhat\ppi)=\brq q(\ppi)
  \end{equation}
  and $\hhat\ppi$ belongs to
  $\hhat{\calT_q}$ if and only if $\ppi$ belongs to
  ${\calT_q}$.
\end{lemma}
\begin{proof}
  For every nonnegative $\hhat f\in \rmC_b(\hhat X)$ we have by \eqref{eq:498}
  \begin{align*}
    \int f'\,\d\mu_{\hhat \sppi}&=\int \int_{\hhat\gamma} \hhat f\,\d\hhat\ppi(\hhat\gamma)=
                                  \int \int_{J\gamma} \hhat f\,\d\ppi(\gamma)
                                      \topref{eq:498}=
    \int \int_{\gamma} \hhat f\circ\iota\,\d\ppi(\gamma)
    =
    \int \hhat f\circ\iota \,\d\mu_\sppi,
  \end{align*}
  which shows the first identity of \eqref{eq:500}. The second follows
  easily by choosing $f\equiv 1$ and the third identity
  is a consequence of the relation
  $\sfe_i\circ J=\iota\circ\sfe_i$, $i=0,1$.
  Since $\iota$ is injective, \eqref{eq:501} is a consequence of the
  general properties
  of relative entropy functionals
  \begin{equation}
    \label{eq:499}
    \brq q^q(\hhat\ppi)=\LL^q(\mu_{\hhat\sppi}|\hhat\mm)=
    \LL^q(\iota_\sharp\mu_{\hhat\sppi}|\iota_\sharp\hhat\mm)=
    \LL^q(\mu_{\sppi}|\mm)=
    \brq q^q(\ppi).
  \end{equation}
  A similar argument shows that $\LL^q((\sfe_i)_\sharp\hhat\ppi|\hhat
  \mm)=
  \LL^q((\sfe_i)_\sharp\ppi|
  \mm).$
\end{proof}
A simple but important application of the previous two Lemma yields the
following result.
\begin{theorem}
  \label{thm:invariance-weak}
  Let $\iota:X\to X'$ be a measure-preserving isometric imbedding of
  $\X$ into $\X'$.
  For every $\hhat f\in W^{1,p}(\hhat \X,\calT_q')$ the function
  $f:=\iota^* \hhat f$ belongs to $ W^{1,p}(\X,\calT_q)$
  and
  \begin{equation}
    \label{eq:497}
    \weakgradTq f\le \iota^*(\weakgradTqhat{\hhat f})\quad
    \text{$\mm$-a.e.~in $X$}.
  \end{equation}
  If moreover $\iota$ is surjective or $(X,\sfd)$ is complete then
  $\iota^*$ is an isomorphism between
  $W^{1,p}(\hhat \X,\hhat{\calT_q})$ and
  $ W^{1,p}(\X,\calT_q)$ whose inverse is $\iota_*$ and
  \begin{equation}
    \label{eq:497is}
    \weakgradTq f= \iota^*(\weakgradTqhat{\hhat f})\quad
    \text{$\mm$-a.e.~in $X$}.
  \end{equation}
\end{theorem}
\begin{proof}
  Let $\hhat g\in L^p(\hhat X,\hhat \mm)$ be a $\hhat{\calT_q}$-weak upper gradient
  of $\hhat f$ in $\hhat X$ and let $\ppi\in \calT_q$.
  We want to show that $g:=\iota^*\hhat g\in L^p(X,\mm)$ is a
  $\calT_q$-weak upper gradient for $f$: we use
  the equivalent characterization \eqref{eq:441} of Lemma \ref{le:equivalent1}.
  
  For every plan $\ppi\in \calT_q$ 
  Lemma \ref{le:equivalent-dynamic} shows that
  $\hhat\ppi:=J_\sharp\ppi\in \hhat{\calT_q}$
  with $\mu_{\hhat \sppi}=\iota_\sharp \mu_\sppi$ and
  $(\sfe_i)_\sharp\hhat\ppi=\iota_\sharp (\sfe_i)_\sharp\ppi$.
  We then obtain
  \begin{align*}
    \int_X f\,\d\pi_0-\int_X f\,\d\pi_1
    &=
      \int_X \hhat f\circ\iota\,\d\pi_0-\int_X \hhat
      f\circ\iota\,\d\pi_1
      =
      \int_{\hhat X} \hhat f\,\d(\iota_\sharp\pi_0)-\int_{\hhat X} 
      \hhat f\,\d(\iota_\sharp \pi_1)
    \\&=
    \int_{\hhat X} \hhat f\,\d\hhat \pi_0-\int_{\hhat X}
    \hhat f\,\d\hhat\pi_1\le
    \int_{\hhat X} \hhat g\,\d\mu_{\hhat \sppi}
    \\&=
    \int_X \hhat g\,\d(\iota_\sharp \mu_{\sppi})=
    \int \hhat g\circ\iota \,\d\mu_{\sppi}=
    \int g \,\d\mu_{\sppi}.
  \end{align*}
  If $\iota$ is surjective, then $J$ is surjective by \eqref{eq:502},
  so that the very same argument shows that any
  $\calT_q$-weak upper gradient for $f\in L^p(X,\mm)$ yields
  a weak $\hhat{\calT_q}$ weak upper gradient $\hhat g:=\iota_* g$ for
  $\iota_* f$ in $L^p(\hhat X,\hhat \mm)$ with
  $\|\hhat g\|_{L^q(\hhat X,\hhat \mm)}=\|g\|_{L^q(X,\mm)}$
  thus showing \eqref{eq:497is}.

  When $(X,\sfd)$ is complete then $(\iota(X),\hhat \sfd)$
  is complete (and therefore $\hhat \sfd$-closed) in $\hhat X$,
  so that by Corollary \ref{cor:restriction}
  $W^{1,p}(\iota(X),\hhat \tau,\hhat\sfd,\hhat \mm;\hhat{\calT_q})=
  W^{1,p}(\hhat X,\hhat \tau,\hhat\sfd,\hhat \mm;\hhat{\calT_q})$
  with equality of minimal $\hhat {\calT_q}$-weak upper gradients.
  On the other hand, $\iota:X\to \iota(X)$ is a measure preserving
  surjective embedding and we can apply the previous statement.
\end{proof}
\subsubsection*{Length distances and conformal invariance}
\index{Length distance}
We refer to the definitions and notation of Section \ref{sec:length-Finsler}.
\begin{lemma}
  \label{le:invariance-length}
  Let $\X=(X,\tau,\sfd,\mm)$ be an e.m.t.m.~space
  and let $\delta:X\times X\to [0,+\infty]$ an extended distance
  such that
  \begin{equation}
    \label{eq:504}
    (X,\tau,\delta)\text{ is an extended metric-topological space},\quad
    \sfd\le \delta\le \sfd_\ell\quad\text{in }X\times X.
  \end{equation}
  Then $W^{1,p}(X,\tau,\sfd,\mm;\calT_q)=
  W^{1,p}(X,\tau,\delta,\mm;\calT_q)$ and the corresponding minimal
  weak $\calT_q$-upper gradients coincide.
\end{lemma}
\begin{proof}
  We know that the class of rectifiable arcs $\RA(X,\sfd)$ and
  $\RA(X,\delta)$ coincide, since $\sfd_\ell=\delta_\ell$,
  with the same length. Therefore, the corresponding classes of dynamic plans in
  $\calT_q$ coincide.
\end{proof}
By the previous result, we can always replace $\sfd$ with
$\sfd_\ell'$, $\sfd_\ell''$ or $\sfd_\ell'''$ in the definition of the
Sobolev spaces.
We can also use $\sfd_\ell$
whenever $\sfd_\ell$ is $\tau$-continuous or when $(X,\tau)$ is
compact.
\begin{remark}
  \label{rem:pedante}
  The (easy) proof of the previous Lemma shows that the
  definition of the Sobolev space $W^{1,p}(X,\tau,\delta,\mm;\calT_q)$
  can be extended to a slightly more general setting:
  in fact, the condition
  that
  $(X,\tau,\delta)$ is an e.t.m.~space
  can be relaxed by asking that there exists an extended distance
  $\sfd:X\times X\to[0,+\infty]$ such that
  \begin{equation}
    \label{eq:637}
    (X,\tau,\sfd)\text{ is an e.t.m.~space\quad and\quad}\sfd\le \delta\le \sfd_\ell.
  \end{equation}  
\end{remark}

We now discuss the case of a conformal distance $\sfd_g$
induced by a continuous function $g\in \rmC_b(X)$ with $\inf_X g>0$.
\begin{proposition}
  \label{prop:invariance-conformal}
  Let $\X=(X,\tau,\sfd,\mm)$ be an e.m.t.m.~space,
  let $g\in \rmC_b(X)$ with $0<m_g\le g\le M_g<\infty$,
  and let $\delta:X\times X\to [0,+\infty]$ be an extended distance
  such that
  \begin{equation}
    \label{eq:504bis}
    \X'=(X,\tau,\delta,\mm)\text{ is an e.m.t.m.~space},\quad
    \sfd_g'\le \delta\le \sfd_g\quad\text{in }X\times X.
  \end{equation}
  Then $\calT_q(\X)$ coincides with $\calT_q':=\calT_q(\X')$,
  a function $f\in L^p(X,\mm)$ belongs to the
  Sobolev space
  $W^{1,p}(\X';\calT_q')$ if and only if
  $f\in W^{1,p}(\X,\calT_q)$,
  and the corresponding minimal
  $\calT_q$-weak upper gradients in $\X$ and in $\X'$
  (which we call  $|\rmD f|_{w,\X}$ and $|\rmD f|_{w,\X'}$ respectively) satisfy
  \begin{equation}
    \label{eq:506}
    |\rmD f|_{w,\X}=g^{-1}|\rmD f|_{w,\X'}.
  \end{equation}
\end{proposition}
\begin{proof}
  Denoting by $\int_{\gamma'}$ the integration of a function $f$ along an arc $\gamma$ with respect to
  the $\delta$ arc-length, we can easily check that
  \begin{equation}
    \label{eq:505}
    \int_{\gamma'}f =\int_\gamma gf.
  \end{equation}
  It follows that if 
  $\ppi\in \calT_q$ 
  \begin{align*}
    \int\int_{\gamma'}f\,\d\ppi(\gamma)
    =\int\int_{\gamma}gf\,\d\ppi(\gamma)=
    \int gf\,\d\mu_\sppi=
    \int f\,\d\mu_\sppi',\quad\mu_\sppi':=g\mu_\sppi.
  \end{align*}
  We deduce that
  \begin{displaymath}
    m_g^q\LL^q(\mu_\sppi|\mm) \le \LL^q(\mmu_\sppi'|\mm)\le M_g^q\LL^q(\mu_\sppi|\mm) 
  \end{displaymath}
  so that $\calT_q=\calT_q'$.
  If $h\in L^p(X,\mm)$ is a weak $\calT_q$-upper gradient for $f$ in
  $\X$ we have
  \begin{align*}
    \int f\,\d(\pi_0-\pi_1)\le
    \int h\,\d\mu_\sppi=
    \int g(g^{-1}h)\,\d\mu_\sppi
    =\int g^{-1}h\,\d\mu_\sppi',
  \end{align*}
  which shows that $g^{-1}h$ is a $\calT_q'$-weak upper gradient for $f$ in $\X'$.
\end{proof}
\subsection{The approach by parametric dynamic plans}
\label{subsec:parametric}
Let us give a brief account of the definition of the
Sobolev space $W^{1,p}$ by parametric dynamic plans
\cite{AGS14I,ADS15}, i.e.~Radon
measures
on suitable subsets of $\rmC([0,1];(X,\tau))$, and their relations
with the notions we introduced in the previous Sections.

We first define the space $\mathrm{AC}^q([0,1];X)$ as the
collection of curves $\gamma\in \BVC([0,1];(X,\sfd))$ such that
$V_\gamma$ is absolutely continuous with derivative $|\dot \gamma|:=
V_\gamma'\in L^q(0,1)$. The $q$-energy of a curve $\gamma$ is defined by
\begin{equation}
  \label{eq:629}
  \rmE_q(\gamma):=\int_0^1|\dot\gamma|^q\,\d t\quad\text{if }
  \gamma \in
  \mathrm{AC}^q([0,1];X)\qquad
  \rmE_q(\gamma):=+\infty\text{ otherwise};
\end{equation}
it defines a $\tau_\rmC$-lower semicontinuous map. It follows in
particular that
$\mathrm{AC}^q([0,1];X)$ is a $F_\sigma$ (thus Borel) subset of
$\rmC([0,1];(X,\tau))$.

Recall that $\sfe_t: \rmC([0,1];(X,\tau))\to X$ is the
evaluation map $\sfe_t(\gamma)=\gamma(t)$.
\begin{definition}[Parametric $q$-test plan, \cite{AGS14I,AGS13}]
  We denote by $\rmT_q$ 
  the collection of all the Radon probability
  measures $\ssigma$ on
  $\rmC([0,1];(X,\tau))$ satisfying the following two properties:
  \begin{enumerate}[(T1)]
  \item there exists $M_\sssigma>0$ such that
  \begin{equation}
    \label{eq:628}
    (\sfe_t)_\sharp\ssigma\le M_\sssigma\,\mm\quad
    \text{for every }t\in [0,1].
  \end{equation}
\item $\ssigma$ is concentrated on
  $\mathrm{AC}^q([0,1];X)$,
  i.e.~$\ssigma\big(\rmC([0,1];(X,\tau)\setminus
  \mathrm{AC}^q([0,1];X)\big)=0$;
\end{enumerate}
We will call $\rmT_q^*$ the subset of dynamic plans in $\rmT_q$
with finite $q$-energy:
\begin{equation}
  \label{eq:631}
  \cE_q(\ssigma):=\int \rmE_q(\gamma)\,\d\ssigma(\gamma)<\infty.
\end{equation}
We will say that a set $\Sigma\subset \rmC([0,1];(X,\tau))$
is $\rmT_q$-negligible (resp.~$\rmT_q^*$-negligible) if $\ssigma(\Sigma)=0$ for every $\ssigma\in
\rmT_q$ (resp.~$\ssigma\in \rmT_q^*$).
\end{definition}
As usual, we will say that a property $P$ on curves of
$\rmC([0,1];(X,\tau))$
holds $\rmT_q$-a.e.~if the set where $P$ does not hold is $\rmT_q$-negligible.

Notice that if a set $\Sigma$ is $\rmT^*_q$-negligible 
then it is also $\rmT_q$ negligible: it is sufficient to approximate
every plan $\ssigma\in \rmT_q$ by an increasing sequence of 
plans satisfying \eqref{eq:631}.

Starting from the notion of $\rmT_q$-exceptional sets, we can
introduce
the corresponding definition of $\rmT_q$-weak upper gradient and
Sobolev space.
\begin{definition}[$\rmT_q$-weak upper gradient]
  We say that a function $f\in L^p(X,\mm)$ belongs to
  the Sobolev space
  $W^{1,p}(\X,\rmT_q)$ if there exists a function $g\in \Ldp$ such
  that
  \begin{equation}
    \label{eq:632}
    |f(\gamma(1))-f(\gamma(0))|\le \int_0^1 g(\gamma(t))|\dot
    \gamma|(t)\,\d t
  \end{equation}
  for $\rmT_q$-a.e.~$\gamma\in \mathrm{AC^q}([0,1];X)$.
  Every function $g$ with the stated property is called a
  $\rmT_q$-w.u.g.~of $f$.
\end{definition}
The properties of Sobolev functions  in $W^{1,p}(\X,\rmT_q)$
can be studied by arguments similar to the ones we presented in
\S\,\ref{subsec:Tq} and \S\,\ref{subsec:calculus_weak},
obtaining corresponding results adapted to the parametric
$\rmT_q$-setting: we refer to \cite{AGS14I,AGS13} for the
precise statements and proofs.
% Here we just quote the following
% natural property, 
% \begin{proposition}
  
% \end{proposition}

However, by adapting the arguments of \cite{ADS15}, it is possible to
prove
directly that the notions of $\calT_q$ and $\rmT_q$ weak upper
gradient coincide, obtaining the equivalence of the
corresponding Sobolev spaces
$W^{1,p}(\X,\calT_q)$ and $W^{1,p}(\X,\rmT_q)$.

First of all, it is not difficult to check that
 for every $f\in L^p(X,\mm)$ and $g\in \Ldp$
\begin{equation}
  \label{eq:633}
  \text{$g$ is a $\calT_q$-w.u.g. of $f$}\quad
  \Rightarrow\quad
  \text{$g$ is a $\rmT_q$-w.u.g. of $f$},
\end{equation}
since
for every parametric dynamic plan $\ssigma\in \rmT_q^*$
the corresponding nonparametric version
$\ppi:=\Quot_\sharp \ssigma$ belongs to $\calT_q$;
recall that we denoted by $\Quot:\rmC([0,1];(X,\tau))\to
\Arc(X,\tau)$ the quotient map.
In fact
\begin{displaymath}
  (\sfe_i)_\sharp\ppi=(\sfe_i)_\sharp\ssigma\le M_\sssigma\mm\quad i=0,1,
\end{displaymath}
and for every bounded Borel function $f:X\to \R$
\begin{align*}%
  \int \int_\gamma f\,\d\ppi(\gamma)
  &=
    \int \int_{\sfq(\eta)} f\,\d\ssigma(\eta)=
    \int \int_0^1  f(\eta(t))|\dot \eta|(t)\,\d\ssigma(\eta)
    \\&\le
    \int \Big(\int_0^1  f^p(\eta(t))\,\d
    t\Big)^{1/p}\rmE_q^{1/q}(\eta)\,\d\ssigma(\eta)
  \le
  \cE_{p}^{1/p}
  \Big(\int \int_0^1  f^p(\eta(t))\,\d t\,\d\ssigma(\eta)\Big)^{1/p}
  \\&\le
  \cE_{p}^{1/p}
  \Big(\int_0^1 \int_0^1  f^p(\sfe_t(\eta))\,\d\ssigma(\eta)\,\d
  t\Big)^{1/p}
  \le (M_\sssigma\,\cE_{p})^{1/p}\Big(\int_X f^p\,\d\mm\Big)^{1/p}.
\end{align*}
We can deduce that for every $\Gamma\subset \RA(X)$
\begin{equation}
  \label{eq:630}
  \Gamma\ \ 
  \text{is $\calT_q$-negligible}
  \quad\Rightarrow\quad
  \Quot^{-1}(\Gamma)\ \ \text{is $\rmT_q$-negligible}
\end{equation}
and therefore we get \eqref{eq:633}.
In order to prove the converse property we introduce the notion of
parametric barycenter of a Radon measure $\ssigma\in
\cMp(\rmC([0,1];(X,\tau)))$: it is the image measure
$\varrho_\sssigma:=\sfe_\sharp(\ssigma\otimes \Leb 1)\in \cMp(X)$, which
satisfies
\begin{equation}
  \label{eq:634}
  \int_X f\,\d\varrho_\sssigma=
  \int\int_0^1 f(\sfe_t(\gamma))\,\d t\,\d\ssigma(\gamma)
  \quad\text{for every $f\in \rmB_b(X)$}.
\end{equation}
We say that $\ssigma$ has \emph{parametric barycenter in $L^q(X,\mm)$}
if $\varrho_\sssigma=h_\sssigma \mm\ll\mm$ for a density
$h_\sssigma\in L^q(X,\mm)$.

The proof of the converse implication of \eqref{eq:633} is based on
the following two technical Lemmata.
\begin{lemma}
  \label{le:technical1}
  Let us suppose that $g\in \Ldp$ is a $\rmT_q$-w.u.g.~of $f\in
  L^p(X,\mm)$ and let $\ssigma\in \cMp(\rmC([0,1];(X,\tau)))$
  be a dynamic plan satisfying \eqref{eq:631} and
  %concentrated on $\mathrm{AC^q([0,1];X)}$ 
  %satisfying
  \begin{equation}
    \label{eq:635}
    (\sfe_i)_\sharp \ssigma\le M\mm\ll\mm\quad i=0,1,\qquad
    \varrho_\sssigma\le M\mm\qquad\text{for a constant }M>0.
  \end{equation}
  Then \eqref{eq:632} holds for $\ssigma$-a.e.~$\gamma$.
\end{lemma}
\begin{proof}
  The argument is similar (but simpler) than the one used for
  the proof of Theorem \ref{thm:Rovereto}.
  For $0\le r<s\le 1$ we consider the Borel maps
  $D^+_r,D^-_s:\rmC([0,1],X)\times [0,1]\to \rmC([0,1];(X,\tau))$ defined by
  \begin{displaymath}
    D^+[\gamma,r](t):=\gamma((r+t)\land 1),\quad
    D^-[\gamma,s](t):=\gamma((t-s)\lor 0).
  \end{displaymath}
  We then set $\lambda:=3\Leb 1\res(1/3,2/3)$ and
  % \begin{displaymath}
  %   \ssigma^+_r:=(D^+_r)_\sharp \ssigma,\quad
  %   \ssigma^-_s:=(D^-_s)_\sharp \ssigma,\quad
  %   \ssigma^+:=3\int_0^{1/3}\ssigma^+_r\,\d r\quad
  %   \ssigma^-:=3\int_{2/3}^{1}\ssigma^-_s\,\d r
  % \end{displaymath}
  which can also be characterized as
  \begin{displaymath}
    \ssigma^+=(D^+)_\sharp(\ssigma\otimes \lambda),\quad
    \ssigma^-=(D^-)_\sharp(\ssigma\otimes \lambda).
  \end{displaymath}
  We easily get for every $t\ge 2/3$
  $(\sfe_t)_\sharp\ssigma^+=(\sfe_1)_\sharp\ssigma\le
  M\mm$, whereas for every $t\in [0,2/3)$ and every nonnegative Borel $f:X\to\R$
 \begin{align*}
   \int f(\sfe_t(\gamma))\,\d\ssigma^+(\gamma)
     &=
         3\int_{1/3}^{2/3}\Big(
      \int f(\gamma((t+r)\land 1))\,\d\ssigma(\gamma)\Big)\,\d r
    \\&  =
   3\int_{1/3}^{(1-t)\land 2/3} \Big(
   \int f(\gamma(r+t))\,\d\ssigma(\gamma)\Big)\,\d r
   +3(1/3-t)_+\int f(\gamma(1))\,\d\ssigma(\gamma)
   \\&\le 
   3\int_X f\,\d\varrho_\sssigma+\int_X f\,\d(\sfe_1)_\sharp\ssigma
   \le 4M \int_X f\,\d\mm
    % =3\int_0^{1/3}\Big(\int_X f\,\d\mu_\sppi\Big)\,\d r
    % \\&=
    % \int_X f\,\d\mu_\sppi
   %\label{eq:uno}
 \end{align*}
 so that $(\sfe_t)_\sharp\ssigma^+\le 4M\mm$ for every $t\in [0,1]$.
 An analogous calculation holds for $\ssigma^-$, so that both satisfy
 \eqref{eq:268}.
 Since $\rmE_q(D^\pm(\gamma,r))\le \rmE_q(\gamma)$ for every $r\in
 [0,1]$ we also get \eqref{eq:631}.
 We deduce that $\ssigma^+,\ssigma^-$ belong to $\rmT_q$
 so that \eqref{eq:632} holds for $\ssigma^+$ and
 $\ssigma^-$-a.e.~curve $\gamma$. Applying Fubini's theorem, we can
 find
 a common Borel and $\ssigma$-negligible set $N\subset \rmC([0,1];(X,\tau))$
 such that \eqref{eq:632} for every $\gamma\in \AC^q([0,1];X)\setminus
 N$
 \begin{align*}
   |f(\gamma(1-s))-f(\gamma(0)|
   &=|f(D^-[\gamma,s](1))-f(D^-[\gamma,s](0))|
   \\&  \le \int_0^1
     g(\gamma((t-s)\lor 0))|\dot\gamma((t-s)\lor 0)|\,\d t
   \\&=
   \int_0^{1-s}g(\gamma(t))|\dot \gamma|(t)\,\d t\quad
   \text{for a.e.~$s\in (1/2,3/2)$}
 \end{align*}
 and similarly
  \begin{align*}
   |f(\gamma(1))-f(\gamma(r)|
   &=|f(D^+[\gamma,r](1))-f(D^+[\gamma,r](0))|
   \\&  \le \int_0^1
    g(\gamma((t+r)\land 1))|\dot\gamma((t-s)\land 1)|\,\d t
   \\&=
    \int_r^{1}g(\gamma(t))|\dot \gamma|(t)\,\d t\quad
    \text{for a.e.~$r\in (1/2,3/2)$}
  \end{align*}
  For every $\gamma\in \mathrm{AC}^q([0,1];X)$
  we can thus find a common value $r=1-s\in
  (1/2,3/2)$ such that the previous inequality hold, obtaining
  \begin{displaymath}
    |f(\gamma(1))-f(\gamma(0)|\le \int_0^{1-s}g(\gamma(t))|\dot
    \gamma|(t)\,\d t
    +\int_r^{1}g(\gamma(t))|\dot \gamma|(t)\,\d t=
    \int_0^{1}g(\gamma(t))|\dot \gamma|(t)\,\d t,
  \end{displaymath}
  which yields \eqref{eq:632}.
\end{proof}
The second Lemma is a reparametrization technique
taken from \cite[Theorem 8.5]{ADS15}.
\begin{lemma}
  \label{le:technical2}
  For every nonparametric dynamic plan
  $\ppi\in \calT_q^*$ 
  there exists a parametric dynamic plan $\ssigma$
  satisfying \eqref{eq:635} such that
  \begin{equation}
    \label{eq:636}
    \ppi\ll \Quot_\sharp\ssigma.
  \end{equation}
\end{lemma}
Combining Lemma \ref{le:technical1} and \ref{le:technical2} we obtain
the following result, which shows the
equivalence of the parametric and nonparametric approaches.
\begin{corollary}
  \label{cor:paravsnonpara}
  For every $f\in L^p(X,\mm)$ and $g\in \Ldp$
\begin{equation}
  \label{eq:633bis}
  \text{$g$ is a $\calT_q$-w.u.g. of $f$}\quad
  \Longleftrightarrow\quad
  \text{$g$ is a $\rmT_q$-w.u.g. of $f$}.
\end{equation}
In particular $W^{1,p}(\X,\calT_q)=W^{1,p}(\X,\rmT_q)$.
\end{corollary}
\begin{proof}
  We have only to prove the converse implication of \eqref{eq:633}.
  Let $g\in \Ldp$ be a $\rmT_q$-w.u.g.~and let $\ppi\in \calT_q^*$.
  By Lemma \ref{le:technical2} there exists
  a parametric dynamic plan $\ssigma$ satisfying \eqref{eq:635}
  such that $\ppi\ll\Quot_\sharp\ssigma$.
  By Lemma \ref{le:technical1} we know that
  \eqref{eq:632} holds for $\ssigma$-a.e.~curve, i.e.
  \begin{displaymath}
    |f(\gamma_1)-f(\gamma_0)|\le \int_{\gamma}g
    \quad\text{for $\Quot_\sharp\ssigma$-a.e.~$\gamma\in \RA(X)$}.
  \end{displaymath}
  Since $\ppi\ll\Quot_\sharp\ssigma$ we deduce that \eqref{eq:440}
  holds as well, so that we can apply Lemma \ref{le:equivalent1}.
\end{proof}

\begin{remark}\upshape As for Cheeger's energy and the relaxed gradient, if no additional
assumption on $(X,\tau,\sfd,\mm)$ is made, it is well possible that
the weak upper gradient is trivial. We will discuss this issue in the
next Theorem \ref{thm:triviality}.
\end{remark}

\subsection{Notes}
\label{subsec:notes10}
\begin{notes}
  \Para{\ref{subsec:Tq}} and \Para{\ref{subsec:calculus_weak}}
  contain new definitions of weak upper gradient and weak Sobolev spaces based on the
  class of $\calT_q$-weak upper gradients. It has some useful
  characteristics:
  \begin{itemize}[-]\itemsep-3pt
  \item it involves measures on nonparametric arcs; notice that
    the notion of upper gradient is inherently invariant
    w.r.t.~parametrization,
    so arcs provide a natural setting;
  \item it is invariant w.r.t.~modification on $\mm$-negligible sets;
  \item it seems quite close to the class $\Bar q\mm$: one has only
    to
    add the control of the initial and final points of the arcs
  \item the corresponding Modulus $\tMd$ is strictly related to $\Md$,
    so that via the selection of a ``good representative'' the
    Sobolev class $W^{1,p}(\X,\calT_q)$ coincides with $N^{1,p}(\X)$;
  \item it is directly connected with the dual of the Cheeger energy.
  \end{itemize}
  Of course, the study of the properties of the $\calT_q$
  w.u.g.~retains
  many ideas of the corresponding analysis based on
  Radon measures on parametric curves \cite{AGS14I,AGS13}
  as the stability, the Sobolev property along $\calT_q^*$-a.e.~arc,
  the chain rule.
  The rescaling technique of Theorem \ref{thm:Rovereto}
  has been also used in \cite{ADS15}.\\
  It is worth noticing that Corollary \ref{cor:isomorphism}
  could also be derived as a consequence of
  Theorem \ref{thm:main-identification-complete},
  as in \cite{AGS14I,AGS13}. Here we followed the
  more direct approach of \cite{ADS15}, which shows
  the closer link between $W^{1,p}$ and $N^{1,p}$.

  \Para{\ref{subsec:link}} combines various methods introduced by
  \cite{ADS15}: apart from some topological aspects,
  Theorem \ref{thm:pourri} is a particular case of the identity
  between Modulus and Content at the level of collection of Radon
  measures,
  Proposition \ref{prop:Sobreg} uses the invariance of the Sobolev
  property
  by restriction and Theorem \ref{thm:representative}
  is strongly inspired by \cite[Theorem 10.3]{ADS15}.

  \Para{\ref{subsec:invariance-weak}} contains
  natural invariance properties of weak Sobolev spaces:
  the most important one is \eqref{eq:497is}
  of Theorem \ref{thm:invariance-weak},
  which will play a crucial role in the final part of the proof
  of the identification Theorem \ref{thm:main-identification-complete}.

  \Para{\ref{subsec:parametric}}
    contains a brief discussion of the
    equivalence between the nonparametric and
    parametric approaches to weak upper gradients and weak Sobolev
    spaces.
    It uses some of the arguments of \cite{ADS15}
    to show that the two approaches lead to equivalent definitions.
\end{notes}

\section{Identification of Sobolev spaces}
\label{sec:Identification}
\GGG
\newcommand{\Path}[1]{\operatorname{ArcL}(#1)}
\newcommand{\ArcL}[1]{\operatorname{ArcL}(#1)}

In this Section we will prove the main identification Theorem for
the Sobolev spaces $\Sob^{1,p}(\X,\AA)$ and $W^{1,p}(\X,\calT_q)$
when $(X,\sfd)$ is complete.
As a first step we study a dual characterization of the weak
$(\calT_q,p)$-energy.
\index{Dual Cheeger energy}
\subsection{Dual Cheeger energies}
\label{subsec:dual-Cheeger}
For every $\mu_0,\mu_1\in \cMp(X)$ we
will introduce the (possibly empty) set
\begin{equation}
  \label{eq:480}
  \Pi(\mu_0,\mu_1):=\Big\{\ppi\in \cMp(\RA(X)),\quad
  (\sfe_{i})_\sharp\ppi=\mu_i\Big\},
\end{equation}
and we define the % energy functional $\EE_q:\cMp(\RA(X))\to[0,+\infty]$ as
% \begin{equation}
%   \label{eq:187}
%   \EE_q(\ppi):=\brq q^q(\ppi)+\LL_q((\sfe_1)_\sharp \ppi|\mm)
% \end{equation}
cost functional 
\begin{equation}
  \label{eq:191}
  \DD_q(\mu_0,\mu_1):=\inf\Big\{\brq q^q(\ppi):\ppi\in
  \Pi(\mu_0,\mu_1)\Big\},
  \quad
  \DD_q(\mu_0,\mu_1)=+\infty\quad\text{if }\Pi(\mu_0,\mu_1)=\emptyset.
\end{equation}
Notice that $\Pi(\mu_0,\mu_1)$ is surely empty if $\mu_0(X)\neq
\mu_1(X)$.

Let us check that if $\DD_q(\mu_0,\mu_1)<+\infty$
and $(X,\sfd)$ is complete, then the infimum in \eqref{eq:191} is
attained. Notice that $\Pi(\mu_0,\mu_1)$ is 
a closed convex subset of $\cMp(\Arc(X))$.
\begin{lemma}
  \label{le:attained}
  Let us suppose that $(X,\sfd)$ is complete.
  For every $\mu_0,\mu_1\in \cMp(X)$, if $\DD_q(\mu_0,\mu_1)<\infty$ then
  there exists a minimizer $\ppi_{\rm min}\in \Pi(\mu_0,\mu_1)$ which realizes
  the infimum
  in \eqref{eq:191}.
  The set $\Pi_o(\mu_0,\mu_1)$ of optimal plans is a compact convex
  subset of $\cMp(\RA(X))$ and for every $\ppi\in \Pi_o(\mu_0,\mu_1)$ 
  the induced measure $\mu_\sppi$ 
  is uniquely determined and is independent of the choice of the minimizer.
  % If $\mu=h\mm$ for some $h\in L^p(X,\mm)$ then
  % $\DD_q(\mu)\le \int_X h^p\,\d\mm<\infty$.
\end{lemma}
\begin{proof}
  Let $\ppi'\in \Pi(\mu_0,\mu_1)$ with $\brq q\ppi'=E<\infty$ and
  define $\cK:=\big\{\ppi\in \Pi(\mu_0,\mu_1):
  \brq q(\ppi)\le E\big\}$. We can apply Lemma \ref{le:RA-tight}:
  for every $\ppi\in \Pi(\mu_0,\mu_1)$
  $\ppi(\RA(X))=\mu_0(X)$ so that condition (T1) is satisfied.
  Concerning (T2) it is sufficient to we use the tightness of $\mu_0$ to find
  compact sets $H_\eps\subset X$ such that $\mu_0(X\setminus
  H_\eps)\le \eps$; clearly
  \begin{displaymath}
    \ppi(\{\gamma:\sfe(\gamma)\cap H_\eps=\emptyset\})
    \le \ppi(\{\gamma:\sfe_0(\gamma)\cap H_\eps=\emptyset\})=
    \mu_0(X\setminus H_\eps)\le \eps.
  \end{displaymath}
  Since the functional $\brq q$
  is lower semicontinuous with respect to weak convergence,
  we conclude that the minimum is attained. The convexity and the
  compactness of $\Pi_o(\mu_0,\mu_1)$ are also immediate; 
  the uniqueness of $\mu_\sppi$ 
  when $\ppi$ varies in $\Pi_o(\mu_0,\mu_1) $ depends on the strict
  convexity of the $L^q(X,\mm)$-norm and on the convexity of
  $\Pi_o(\mu_0,\mu_1) $.
  % Finally, if $\mu=h\mm$ for some nonnegative $h\in L^p(X,\mm)$ then
  % the choice $\ppi=T_\sharp\mu$ where
  % $T:X\to \RA(X)$ is the map which associate to every point $x\in X$
  % the constant curve $\gamma\equiv x$ is a competitor in $\rmE_0(\mu)$
  % for
  % \eqref{eq:191} with $(\sfe_1)_\sharp \ppi=\mu$ and $\brq q(\ppi)=0$.
\end{proof}
We want to compare $\DD_q$ with the dual of the pre-Cheeger energy:
\begin{equation}
  \label{eq:448}
  \frac 1q\pCE_{p}^*(\mu):=\sup\Big\{\int_X f\,\d\mu-\frac 1p
  \pCE_{p}(f):f\in \Lip_b(X,\tau,\sfd)\Big\},\quad
  \mu=\mu_0-\mu_1\in \cM(X).
\end{equation}
Notice that by Lemma \ref{le:obvious-2hom} we have the equivalent representation
\begin{equation}
  \label{eq:449}
  \pCE_{p}^*(\mu)=
  \sup\Big\{\int_X f\,\d\mu:
  f\in \Lip_b(X,\tau,\sfd),\ 
      \pCE_{p}(f)\le 1\Big\}.
    \end{equation}
    Whenever $\mu=h\mm$ with $h\in L^q(X,\mm)$,
    we can also consider the dual of
    the Cheeger energy
    \begin{equation}
      \label{eq:450}
      \frac 1q\CE_{p}^*(h):=\sup\Big\{\int_X f\,h\,\d\mm-\frac 1p
      \CE_{p}(f):f\in \Sob^{1,p}(\X)\Big\},
    \end{equation}
    and of the weak $(\calT_q,p)$-energy $\wCE_{p}$
    (defined by a formula analogous to \eqref{eq:450})
    that we will denote
    by $\wCE_p^*$.
    An obvious necessary condition for the finiteness of $\pCE_p^*$
    and of $\CE_{p}^*$ is given by
    \begin{equation}
      \label{eq:323}
      \pCE_p^*(\mu)<+\infty\quad\Rightarrow\quad
      \mu(X)=0;\qquad
      \CE_{p}^*(h)<+\infty\quad\Rightarrow\quad
      \int_X h\,\d\mm=0.
    \end{equation}
    Since $\wCE_{p}(f)\le \CE_{p}(f)$ for every $f\in L^p(X,\mm)$
    and
    $\CE_{p}(f)\le \pCE_{p}(f)$ for every $f\in \Lip_b(X,\tau,\sfd)$, 
    it is clear that
    \begin{equation}
      \label{eq:453}
      \wCE_{p}^*(h)\ge\CE_{p}^*(h)\ge
      \pCE_{p}^*(h\mm)\quad
      \text{for every }h\in L^p(X,\mm).
    \end{equation}
\begin{lemma}
  For every $\mu_0,\mu_1\in \cMp(X)$ we have
  \begin{equation}
    \label{eq:192}
    \DD_q(\mu_0,\mu_1)\ge \pCE_{p}^*(\mu_0-\mu_1).
  \end{equation}
  If moreover $\mu_i=h_i\mm$ with $h_i\in L^p(X,\mm)$, $h_i\ge0$, then
  \begin{equation}
    \label{eq:192bis}
    \DD_q(h_0\mm,h_1\mm)\ge \wCE_{p}^*(h)\quad
    h=h_0-h_1.
  \end{equation}
\end{lemma}
\begin{proof}
  We observe that for every $\ppi\in \Pi(\mu_0,\mu_1)$ 
  we have
  \begin{equation}
    \label{eq:451}
    \frac 1q\brq q^q(\ppi)=
    \sup_{g\in \cL^p_+(X,\mm)} \iint_\gamma g\,\d\ppi-
    \frac 1p \int g^p\,\d\mm.
  \end{equation}
  Restricting the supremum to the functions $g:=\lip f$
  for some $f\in \Lip_b(X,\sfd,\mm)$ % such that 
  % \begin{align*}
  %   \exists f\in \Lip_b(X,\sfd,\mm),\quad
  %   0\le g\le \lip f,
  % \end{align*}
  and observing that 
  in this case {for every }$\gamma\in \RA(X)$,
  \begin{equation}
    \label{eq:452}
    f(\gamma_0)-f(\gamma_1)\le 
    \int_\gamma g
      \end{equation}
  we get
  \begin{align*}
    &\iint_\gamma g\,\d\ppi
      -
      \frac 1p \int g^p\,\d\mm
      \ge
      \int \big(f(\gamma_0)-f(\gamma_1)\big)\,\d\ppi-
      \frac1p \pCE_p(f)
    \\&
        =
    \int_X f\,\d(\mu_0-\mu_1)-\frac1p \pCE_p(f)
  \end{align*}
  so that 
  \eqref{eq:192} follows by taking the supremum w.r.t.~$f$
  and the infimum w.r.t.~$\ppi$.

  When $\mu_i=h_i\mm$ with $h_i\in L^p(X,\mm)$ nonnegative,
  any dynamic plan $\ppi\in \Pi(\mu_0,\mu_1)$ with
  $\brq q(\ppi)<\infty$ belongs to $\calT_q$.
  Restricting the supremum of \eqref{eq:451} to (the Borel
  representative of) functions $g
  =\weakgradTq
    f$ for some $f\in W^{1,p}(\X,\calT_q) $
  % \begin{displaymath}
  %   \exists\,,\quad 0\le g\le \weakgradTq
  %   f\quad\text{$\mm$-a.e.~in $X$},
  % \end{displaymath}
  it follows that \eqref{eq:452} holds for $\calT_q$-a.e.~curve,
  in particular for $\ppi$-a.e.~curve $\gamma$.
  We can then perform the same integration with respect to $\ppi$
  and obtain \eqref{eq:192bis}.
  %and recalling Lemma \ref{le:obvious-2hom}.
\end{proof}
\subsection[$H=W$]{\texorpdfstring{$\boldsymbol{H=W}$}{H=W}}
\label{subsec:H=W}
\subsubsection*{The compact case}
Let us first consider the case when $(X,\tau)$ is compact.
For every strictly positive function $g\in \rmC_b(X)$
(we will still use the notation $\rmC_b(X)$ even if
the subscript $_b$ is redundant, being $X$ compact)
we will denote
by $\sfd_g$ the conformal distance we studied in \S\,\ref{subsec:Finsler}
and by $\sfK_{\sfd_g}$ the Kantorovich-Rubinstein distance
induced by $\sfd_g$, see
\S\,\ref{subsec:KR}. Notice that $(X,\tau,\sfd_g)$ is
a geodesic e.m.t.~space thanks to Theorem \ref{thm:Finsler}.
\begin{theorem}
  \label{thm:identification-compact1}
  Let us suppose that $(X,\tau)$ is compact; then for every $\mu_0,\mu_1\in
  \cMp(X)$ with $\mu_0(X)=\mu_1(X)$ 
  we have
  \begin{equation}
    \label{eq:177pre}
    \DD_q(\mu_0,\mu_1)=\sup\Big\{
    \sfK_{\sfd_g}(\mu_0,\mu_1)-\frac 1p\int_X g^p\,\d\mm:
    g\in \rmC_b(X),\ g>0\Big\}.
    % \min\Big\{\EE(\ppi)
    % %\br ^2(\ppi)+\LL_2((\sfe_1)_\sharp\ppi)
    % :
    % (\sfe_0)_\sharp\ppi=\mu\Big\}.
  \end{equation}
\end{theorem}
\begin{proof}
  Let us introduce the convex set
  \begin{equation}
    \label{eq:458}
    \calC:=\Big\{(g,\psi_0,\psi_1)\in \big(\rmC_b(X)\big)^3:
    g(x)>0\quad\text{for every }x\in X\Big\}
  \end{equation}
  and the dual representation of the convex set $\Pi(\mu_0,\mu_1)$
  given by two Lagrange multipliers $\psi_0,\psi_1\in
  \rmC_b(X)$:
  $\ppi\in \Pi(\mu_0,\mu_1)$ if and only if (here
  $\gamma_i=\sfe_i(\gamma)$, $i=0,1$)
  \begin{equation}
    \label{eq:459}
    \sup_{\psi_0,\psi_1\in \rmC_b(X)}
    \int_X\psi_0\,\d\mu_0-\int \psi_0(\gamma_0)\,\d\ppi(\gamma)
    -
    \Big(    \int_X\psi_1\,\d\mu_1-\int \psi_1(\gamma_1)\,\d\ppi(\gamma)\Big)<+\infty;
  \end{equation}
  Notice that whenever the supremum in \eqref{eq:459} is finite, it vanishes.
  We first observe that 
  \begin{displaymath}
    \frac 1q\DD_q(\mu_0,\mu_1)=
    \inf_{\sppi\in \cMp(\RA(X))}\sup_{(g,\psi_0,\psi_1)\in \calC}
    \cL((g,\psi_0,\psi_1);\ppi)
  \end{displaymath}
  where
  the Lagrangian function $\cL$ is given by
  \begin{equation}
    \label{eq:454}
    \begin{aligned}
      \cL((g,\psi_0,\psi_1);\ppi):=&\int\Big(\int_\gamma g+
      \psi_1(\gamma_1)-\psi_0(\gamma_0)\Big)\,\d\ppi(\gamma)
      \\&\qquad+ \int_X
      \psi_0\,\d\mu_0-\int_X \psi_1\,\d\mu_1-\frac 1p\int_X
      g^p\,\d\mm,
    \end{aligned}
  \end{equation}
  and it is clearly convex w.r.t.~$\ppi$ and concave
  w.r.t.~$(g,\psi_0,\psi_1)$.
  We want to apply Von Neumann Theorem
  \ref{thm:VonNeumann} and to invert the order of
  $\inf$ and $\sup$.

  Selecting $\ g_\star\equiv 1,\ \psi_{1,\star}\equiv 1,\
  \psi_{0,\star}\equiv 0$
  we see that for every $C\ge0$ the sublevel
  \begin{equation}
    \label{eq:455}
    \cK_C:=\Big\{\ppi\in \RA(X):\cL((g_\star,\psi_{0,\star},\psi_{1,\star});\ppi)\le C\Big\}
  \end{equation}
  is not empty (it contains the null plan) and compact, since
  for every $\ppi\in \cK_C$ we have
  \begin{equation}
    \label{eq:456}
    \ppi(\RA(X))+
    \int \ell(\gamma)\,\d\ppi\le C+\frac 1p\mm(X)+\mu_1(X),
  \end{equation}
  so that $\cK_C$ is equi-tight, thanks to Theorem
  \ref{thm:important-arcs}(g)
  (here we use the compactness of $(X,\tau)$).

  We therforetherefore obtain
  \begin{equation}\label{eq:462}
    \DD_q(\mu_0,\mu_1)=
    \sup_{(g,\psi_0,\psi_1)\in \calC}\,\inf_{\sppi\in \cMp(\RA(X))}
    \cL((g,\psi_0,\psi_1);\ppi).
  \end{equation}
  We can introduce the conformal (extended) distance generated by $g$
  \begin{equation}
    \label{eq:460}
    \sfd_g(x_0,x_1):=\inf \Big\{\int_\gamma g:\gamma\in \RA(X),\
    \gamma_0=x_0,\ \gamma_1=x_1\Big\}
  \end{equation}
  observing that if the triple $(g,\psi_0,\psi_1)$ does not belong to
  the subset
  of $\calC$
  \begin{equation}
    \label{eq:461}
    \Sigma:=\Big\{(g,\psi_0,\psi_1)\in \rmC_b(X)^3% \times\rmC(X)
    :
    g>0, \ \psi_0(x_0)-\psi_1(x_1)\le \sfd_g(x_0,x_1)
    \quad\text{for every
    }x_0,x_1\in X\Big\}
  \end{equation}
  we would have
  \begin{displaymath}
    \inf_{\sppi}\cL((g,\psi_0,\psi_1);\ppi)=-\infty.
  \end{displaymath}
  On the other hand, if $(g,\psi_0,\psi_1)\in \Sigma$
  the infimum in \eqref{eq:462} is attained at $\ppi=0$ so that
  \begin{displaymath}
    \inf_{\sppi}\cL((g,\psi_0,\psi_1);\ppi)=
    \int_X
      \psi_0\,\d\mu_0-\int_X \psi_1\,\d\mu_1-\frac 1p\int_X
      g^p\,\d\mm    
  \end{displaymath}
  and therefore \eqref{eq:462} reads
  \begin{equation}
    \label{eq:463}
    \DD_q(\mu_0,\mu_1)=\sup\Big\{
    \int_X \psi_0\,\d\mu_0-
    \int_X \psi_1\,\d\mu_1-\frac 1p \int_X g^p\,\d\mm:
    (g,\psi_0,\psi_1)\in \Sigma,\ g>0\Big\},
  \end{equation}
  which coincides with \eqref{eq:177pre} thanks to
  \eqref{eq:490dg}.
\end{proof}

\begin{theorem}
  \label{thm:identification-compact2}
  Let us suppose that $(X,\tau)$ is compact; then for every $\mu_0,\mu_1\in
  \cMp(X)$ 
  we have
  \begin{equation}
    \label{eq:177}
    \DD_q(\mu_0,\mu_1)=\pCE_{p}^*(\mu_0-\mu_1).
    % \min\Big\{\EE(\ppi)
    % %\br ^2(\ppi)+\LL_2((\sfe_1)_\sharp\ppi)
    % :
    % (\sfe_0)_\sharp\ppi=\mu\Big\}.
  \end{equation}
\end{theorem}
\begin{proof}
  Combining \eqref{eq:177pre} with
  \eqref{eq:493dg} we easily get
  \begin{equation}
    \label{eq:322}
    \begin{aligned}
      \DD_q(\mu_0,\mu_1)= \sup\Big\{& \int_X \varphi\,\d(\mu_0-\mu_1)- \frac 1p \int_X
      g^p\,\d\mm:\\& g\in \rmC_{b}(X),\ g>0,\ \varphi\in
      \Lip_b(X,\tau,\sfd),
      \ \lip_{\sfd}\varphi\le g\Big\},
    \end{aligned}
  \end{equation}
  so that
  \begin{align*}
    \DD_q(\mu_0,\mu_1)&\topref{eq:322}\le 
                        \sup_{\varphi\in
                        \Lip_b(X,\tau,\sfd)}\int\varphi\,\d(\mu_0-\mu_1)
                        -\frac 1p
              \int \lip^p( \varphi)\,\d\mm=
                        \pCE^*_{p}(\mu_0-\mu_1).
  \end{align*}
  Since we already proved that $\DD_q(\mu_0,\mu_1)\ge \pCE_{p}^*(\mu_0-\mu_1)$ we conclude.
\end{proof}
\begin{corollary}
  \label{cor:identification-compact2}
  Let us suppose that $(X,\tau)$ is compact. For every
  $h\in L^q(X,\mm)$ with $\int_X h\,\d\mm=0$ we have
  \begin{equation}
    \label{eq:470}
    \DD_q(h_+\mm,h_-\mm)=\CE_{p}^*(h)=\wCE_{p}^*(h)=\pCE_p^*(h\mm).
  \end{equation}
\end{corollary}
\begin{proof}
  Combining \eqref{eq:192bis} and \eqref{eq:453} we know that for every
  $h\in L^q(X,\mm)$
  \begin{displaymath}
    \DD_q(h_+\mm,h_-\mm)\ge \wCE_{p}^*(h)\ge \CE_{p}^*(h)\ge \pCE_{p}^*(h\mm).
  \end{displaymath}
  Equality then follows by Theorem \ref{thm:identification-compact2}.
\end{proof}
By Fenchel-Moreau duality we can now
recover for every $f\in L^p(X,\mm)$ 
\begin{align*}
  \frac 1p\CE_p(f)&=\sup_{h\in L^q(X,\mm)}\int_X hf\,\d\mm-\frac 1q
                    \CE_p^*(h)
                    \\&=\sup_{h\in L^q(X,\mm)}\int_X hf\,\d\mm-\frac 1q
  \wCE_p^*(h)
  \\&=\frac 1p\wCE_p(f),
\end{align*}
and we obtain the identification of the strong and weak Cheeger energy and
of the Sobolev spaces, including the case of a compatible algebra
$\AA$,
thanks to Theorem \ref{thm:mainA}.
\begin{corollary}
  \label{cor:main-identification-compact}
  Let us suppose that $(X,\tau)$ is compact.
  Then for every algebra $\AA$ compatible with $\X$
  we have
  \begin{equation}
    \label{eq:467}
    \Sob^{1,p}(\X,\AA)=\Sob^{1,p}(\X)=W^{1,p}(\X,\calT_q)
  \end{equation}
  with equality of norms; in particular 
  \begin{equation}
    \label{eq:468}
    \CE_{p,\sAA}(f)=\CE_{p}(f)=\wCE_{p,\calT_q} (f)\quad
    \text{for every }f\in L^p(X,\mm),
  \end{equation}
  and for every $f\in W^{1,p}(\X,\calT_q)$
  \begin{equation}
    \label{eq:469}
    \relgradA f=\relgrad f=\weakgradTq f\quad\text{$\mm$-a.e.~in $X$}.
  \end{equation}
\end{corollary}

\nc
\subsubsection*{The complete case}
Let us now extend the previous result to the case when
$(X,\sfd)$ is complete, by removing the compactness assumption.
%We start from the identification theorem.
\begin{theorem}
  \label{thm:main-identification-complete}
  Let us suppose that $(X,\sfd)$ is complete
  and let $\AA$ be an algebra compatible with $\X$.
  Then the same conclusions \eqref{eq:467}, \eqref{eq:468}
  and \eqref{eq:469} hold.
  % Then for every algebra $\AA$ we have
  % \begin{equation}
  %   \label{eq:467bis}
  %   \Sob^{1,p}(\X,\AA)=\Sob^{1,p}(\X)=W^{1,p}(\X,\calT_q)
  % \end{equation}
  % with equality of norms; in particular 
  % \begin{equation}
  %   \label{eq:468bis}
  %   \CE_{p,\sAA}(f)=\CE_{p}(f)=\int_X \weakgradTq f^p\,\d\mm\quad
  %   \text{for every }f\in L^p(X,\mm)
  % \end{equation}
  % and for every $f\in W^{1,p}(\X,\calT_q)$
  % \begin{equation}
  %   \label{eq:469bis}
  %   \relgrad f=\weakgradTq f\quad\text{$\mm$-a.e.~in $X$}.
  % \end{equation}
\end{theorem}
\begin{proof}
  Let us consider the Gelfand compactification $\hat \X=(\hat X,\hat
  \tau,\hat \sfd,\hat \mm)$ of Theorem
  \ref{thm:G-compactification}
  induced by $\AA$.
  Since $(X,\sfd)$ is complete, we can apply Theorem
  \ref{thm:invariance-weak}
  and we obtain that
  $\iota_*$ induces an isomorphism of
  $W^{1,p}(\X,\calT_q)$ onto $W^{1,p}(\hat \X,\hat \calT_q)$
  with
  \begin{equation}
    \label{eq:507}
    \weakgradTq f=\iota^*(|\rmD \hat f|_{w,\hat \calT_q}),\quad
    \hat f:=\iota_* f.
  \end{equation}
  Since $(\hat X,\hat \tau)$ is compact,
  by Corollary \ref{cor:main-identification-compact} we know that
  $\hat f\in H^{1,p}(\hat X,\hat \AA)$ with
  \begin{equation}
    \label{eq:508}
    |\rmD \hat f|_{\star,\hat \sAA}=|\rmD \hat f|_{w,\hat\calT_q}\quad \text{$\hat \mm$-a.e.}
  \end{equation}
  Finally, applying Lemma \ref{le:preliminary-embedding} we obtain
  that $f=\iota^*\hat f$ belongs to $H^{1,p}(\X,\AA)$ with
  \begin{equation}
    \label{eq:509}
    |\rmD f|_{\star,\sAA}\le \iota^*\big(|\rmD \hat f|_{\star,\hat \sAA}\big).
  \end{equation}
  Combining the previous inequalities we obtain
  \begin{equation}
    \label{eq:510}
    |\rmD f|_{\star,\sAA} \le \weakgradTq f \quad \text{$\mm$-a.e.}
  \end{equation}
  Recalling \eqref{eq:565} we conclude.
\end{proof}
We can also extend to the complete case the dual characterizations of
Theorem \ref{thm:identification-compact2} and Corollary
\ref{cor:identification-compact2}.
\begin{theorem}
  \label{thm:main-duality-complete}
  Let us suppose that $(X,\sfd)$ is complete. Then
  for every $\mu_0,\mu_1\in
  \cMp(X)$ 
  we have
  \begin{equation}
    \label{eq:177bis}
    \DD_q(\mu_0,\mu_1)=\pCE_{p}^*(\mu_0-\mu_1).
    % \min\Big\{\EE(\ppi)
    % %\br ^2(\ppi)+\LL_2((\sfe_1)_\sharp\ppi)
    % :
    % (\sfe_0)_\sharp\ppi=\mu\Big\}.
  \end{equation}
  and whenever $\mu_i=h_i\mm$ with $h_i\in L^q(X,\mm)$
  and $h=h_0-h_1$
  \begin{equation}
    \label{eq:470bis}
    \DD_q(h_0\mm,h_1\mm)=\CE_{p}^*(h)=\wCE_{p}^*(h)=\pCE_p^*(h\mm).
  \end{equation}
\end{theorem}
\begin{proof}
  It is sufficient to prove that
  $\DD_q(\mu_0,\mu_1)\le\pCE_{p}^*(\mu_0-\mu_1).$
  Keeping the same notation of the previous proof and
  using the compactification $\hat \X$ induced by
  the canonical algebra $\AA=\Lip_b(X,\tau,\sfd)$, we consider the Radon measures
  $\hat\mu_i:=\iota_\sharp\mu_i\in \cMp(\hat X)$.
  It is easy to check that for every plan $\ppi\in \Pi(\mu_0,\mu_1)$
  the push forward $J_\sharp \ppi$
  (where $J(\gamma)=\iota\circ\gamma$)
  belongs to
  $\Pi(\hat\mu_0,\hat\mu_1)$
  in $\RA(\hat X)$, so that $\DD_q(\mu_0,\mu_1)\le
  \DD_q(\hat \mu_0,\hat\mu_1)$.
  On the other hand, by Lemma \ref{le:simple-but-crucial} and the
  completeness
  of $(\iota(X),\hat \sfd)$,
  every plan $\hat\ppi\in \Pi(\hat\mu_0,\hat\mu_1)$ is concentrated
  on curves in $J(\RA(X))$ so that
  $\DD_q(\mu_0,\mu_1)=
  \DD_q(\hat \mu_0,\hat\mu_1)$.
  Recalling that
  for every $f\in \Lip_b(X,\tau,\sfd)$ we have
  $\hat f=\Gamma(f)\in \Lip_b(\hat X,\hat \tau,\hat \sfd)$
  with
  $\lip_{\hat\sfd} \hat f(\iota( x))\ge \lip_\sfd f(x)$ we get
  \begin{align*}
    \DD_q(\mu_0,\mu_1)
    &=\DD_q(\hat \mu_0,\hat\mu_1)
      =\sup_{\hat f\in \Lip_b(\hat X,\hat \tau,\hat\sfd)}
    \int_{\hat X}\hat f\,\d(\hat\mu_0-\hat \mu_1)-
    \frac 1p\int_{\hat X}\big(\Lip_{\hat \sfd} \hat f(\hat
      x)\big)^p\,\d\hat\mm(\hat x)
    \\&=
    \sup_{\hat f\in \Lip_b(\hat X,\hat \tau,\hat\sfd)}
    \int_{X}\hat f\circ\iota\,\d(\mu_0- \mu_1)-
    \frac 1p\int_{X}\big(\Lip_{\hat \sfd} \hat f(\iota(x)\big)^p\,\d\mm(x)
    \\&\le
    \sup_{f\in \Lip_b(X,\tau,\sfd)}
    \int_{X}f\,\d(\mu_0-\mu_1)-
    \frac 1p\int_{X}\big(\Lip_{\sfd}  f(x)
    \big)^p\,\d\mm(x)=\pCE^*_p(\mu_0-\mu_1).\qedhere
  \end{align*}
  \end{proof}
  As a consequence of the above result, we can
  prove a simple characterization of nontriviality
  for the Cheeger energy.
  \begin{theorem}
    \label{thm:triviality}
    Let us suppose that $(X,\sfd)$ is complete. The following
    properties are equivalent:
    \begin{enumerate}
    \item The Cheeger energy is trivial: $\CE_p(f)\equiv 0$ for every
      $f\in L^p(X,\mm)$.
    \item The Cheeger energy $\CE_p(f)$ is finite for every $f\in
      L^p(X,\mm)$.
    \item $\RA_0(X)$ is $\calT_q$-negligible.
    \item $\RA_0(X)$ is $\Bar q\mm$-negligible
      (equivalently, if $(X,\tau)$ is Souslin, $\RA_0(X)$ is $\Md$-negligible).
    \end{enumerate}
  \end{theorem}
  \begin{proof}
    The implication $(a)\Rightarrow (b)$ is obvious.
    \medskip

    \noindent$\boldsymbol{(b)\Rightarrow(a)}$
    If the Cheeger energy is always finite then the Sobolev
    norm of $\Sob^{1,p}(\X)\subset L^p(X,\mm)$ is equivalent to
    the $L^p$-norm \cite[Corollary 2.8]{Brezis11}, so that there
    exists a constant $C>0$ such that
    \begin{equation}
      \label{eq:638}
      \CE_p(f)\le C\|f\|_{L^p(X,\mm)}^p\quad\text{for every }f\in L^p(X,\mm).
    \end{equation}
    Let us show that \eqref{eq:638} implies $\CE_p(f)\equiv 0$ for
    every $f\in \Sob^{1,p}(\X)$. We consider the $2$ periodic
    Lipschitz function
    $\phi:\R\to\R$ satisfying $\phi(r)=|r|$ for $r\in [-1,1]$
    and we set 
    \begin{displaymath}
      \phi_n(r):=\phi(n r),\quad
      f_n(x):=\phi_n(f(x)).
    \end{displaymath}
    Thanks to the locality of the minimal relaxed gradient we have
    $
      \relgrad {f_n}(x)= n\relgrad {f}(x)
    $
    so that
    \begin{displaymath}
      \CE_p(f)=\frac 1{n^p}\CE_p(f_n)\le
      \frac C{n^p}\|f_n\|_{L^p(X,\mm)}^p\le \frac
      C{n^p}\mm(X)\to0\quad\text{as }n\to\infty.
    \end{displaymath}
    $\boldsymbol{(c)\Rightarrow(a)}$
    If (c) holds then for every nonvanishing
    $h\in L^p(X,\mm)$
    the class $\Pi(h_+\mm,h_-\mm)$ is empty so that
    $\CE_p^*(h)=\DD_q(h_+\mm,h_-\mm)=+\infty$. By duality we obtain
    $\CE_p\equiv 0$.
    \medskip

    \noindent
    $\boldsymbol{(a)\Rightarrow(c)}$
    Let $\ppi\in \calT_q$ with $\ppi(\RA(X))>0$ and let
    $(\sfd_i)_{i\in I}$ be a directed family of continuous
    semidistances as in {\upshape (\ref{eq:monotone}a,b,c,d)}.
    It is not restrictive to assume that
    $\ppi$ is concentrated on a compact set $\Gamma$
    on which $\ell$ is continuous.
    The image $K=\sfe(\Gamma)$ is compact in $(X,\tau)$: for every
    $i\in I$ we can find a countable set $K_i\subset K$
    such that for every $x\in K$ $\inf_{y\in K_i}\sfd_i(x,y)=0$.
    For every $y\in K_i$ the minimal relaxed gradient of the function  $x\mapsto \sfd_i(x,y)$
    vanishes, so that
    there exists a $\ppi$-negligible set $N_i\subset \Gamma$ such that
    the function $t\mapsto \sfd_i(\Al_\gamma(t),y)$ is constant for
    every $y\in K_i$ and $\gamma\in \Gamma\setminus N_i$. By
    continuity we deduce that
    $t\mapsto \sfd_i(\Al_\gamma(t),y)$ is constant for every $y\in K$
    so that $\sfd_i(\Al_\gamma(t),\gamma_0)=0$ for every $\gamma\in
    \Gamma\setminus N_i$;
    by integration we obtain
    \begin{equation}
      \label{eq:639}
      \int\Big(\int_\gamma \sfd_i(x,\gamma_0)\Big)\,\d\ppi(\gamma)=0.
    \end{equation}
    On the other hand, for every $\gamma\in \Gamma$
    Beppo Levi's Monotone Convergence Theorem yields $\lim_{i\in I}\int_\gamma \sfd_i(x,\gamma_0)=
    \int_\gamma \sfd(x,\gamma_0)$. A further application of
    the same theorem  thanks to the fact that the function $\gamma\mapsto \int_\gamma
    \sfd_i(x,\gamma_0)$ is continuous on $\Gamma$ with respect to the
    $\tau_\rmA$ topology yields
    \begin{equation}
      \label{eq:640}
      0=\lim_{i\in I}\int\Big(\int_\gamma
      \sfd_i(x,\gamma_0)\Big)\,\d\ppi(\gamma)=
      \int\Big(\int_\gamma \sfd(x,\gamma_0)\Big)\,\d\ppi(\gamma)
    \end{equation}
    which shows that $\ppi$-a.e.~$\gamma$ is constant, a contradiction.
    \medskip

    \noindent
    $\boldsymbol{(d)\Leftrightarrow(c)}$
    The implication $(d)\Rightarrow (c)$ is obvious.
    In order to prove the converse one, we argue by contradiction and
    we suppose that
    there exists a plan $\ppi\in \Bar q\mm$ with $\ppi(\RA_0(X))>0$.
    We can then argue as in the proof of Proposition \ref{prop:Sobreg}
    and
    define a new plan $\tilde\ppi\in \calT_q$ according to
    \eqref{eq:641}.
    It is clear that $\tilde\ppi(\RA_0(X))>0$ as well.
    \end{proof}
\subsection{Notes}
\label{subsec:notes11}
\begin{notes}
  The representation theorems \ref{thm:identification-compact1},
  \ref{thm:identification-compact2}, and \ref{thm:main-duality-complete} are new.
  The proof of Theorem $H=W$ has been given in \cite{Shanmugalingam00}
  in the case of doubling, $p$-Poincar\'e spaces
  \cite[Theorem 5.1]{Bjorn-Bjorn11} and
  in \cite{AGS14I,AGS13} for general spaces by
  a completely different
  method: it relies on three basic ingredients:
  \begin{itemize}[-]\itemsep-3pt
  \item the properties of the $L^2$-gradient flow of the Cheeger
    energy
    (in particular the comparison principle),
  \item the estimate of the Wasserstein velocity of the evolution
    curve, by means of a suitable version of the Kuwada's Lemma,
  \item the representation of the solution as the evaluation at time
    $t$
    of a dynamic plan concentrated on curves with finite $q$-energy,
  \item the derivation of the Shannon-Reny entropy along
  the flow, by using the weak upper gradients of the solutions.
  \end{itemize}
  It is curious that the refined estimates of the
  Hopf-Lax flow play a crucial role in the second step.

  A different proof of Theorem \ref{thm:triviality} in the context of
  Newtonian spaces can be found in \cite[Prop.~7.1.33]{HKST15}.
\end{notes}
\section{Examples and applications}
\label{sec:Examples}
\GGG
\subsection{Refined invariance of the (strong) Cheeger energy}

\subsubsection*{Invariance w.r.t.~the algebra $\AA$}
\begin{theorem}[Invariance of the Cheeger energy w.r.t.~$\AA$]
  \label{thm:trivialA}
  For every e.m.t.m.~space $\X$
  and every compatible algebra $\AA$
  the Sobolev space $\Sob^{1,p}(\X,\AA)$
  is independent of the compatible algebra $\AA$
  and coincides with $\Sob^{1,p}(\X)$.
\end{theorem}
\begin{proof}
  It is sufficient to combine 
  Theorem \ref{thm:main-identification-complete}
  with Corollary \ref{cor:completionA}.
\end{proof}
We can rephrase the previous statement as a density result:
if $\AA$ is a compatible algebra for $\X$,
\begin{equation}
\begin{gathered}
  \text{for every $f\in \Sob^{1,p}(\X)$
%    $\exists \,$
 there exists a sequence
  $f_n\in \AA$
  such that}\\
  f_n\to f,\quad
  \lip f_n\to \relgrad f\quad\text{strongly in }L^p(X,\mm).
\end{gathered}\label{eq:569}
\end{equation}
We can also slightly relax the assumption that $\AA$ is unital.
\begin{proposition}
  \label{prop:nonunital}
  Let $\X=(X,\tau,\sfd,\mm)$ be an e.m.t.m.~space and
  let $\AA\subset \Lip_b(X,\tau,\sfd)$ be an algebra of functions
  satisfying \eqref{eq:214bis} (we do not assume that $\unit\in\AA$).
  If there exists a sequence
  of compact sets $K_n\subset X$ and functions $f_n\in \AA$ such that
  \begin{equation}
    \label{eq:570}
    f_n(x)\ge 1\quad\text{for every $x\in K_n$},\quad
    \lim_{n\to\infty} \int_{X\setminus K_n}\Big(1+|\lip f_n(x)|^p\Big)\,\d\mm(x)=0
  \end{equation}
  then $\AA$ satisfies \eqref{eq:569}.
\end{proposition}
Since $\mm$ is tight, \eqref{eq:570} is surely satisfied if
for every compact $K\subset X$ there exists a function $f\in \AA$
such that
\begin{equation}
  \label{eq:571}
  f(x)\ge 1\quad\text{for every $x\in K$},\quad
  \Lip(f,X)\le C\quad
  \text{for a constant $C$ independent of $K$.}
\end{equation}
\begin{proof}[Proof of Proposition \ref{prop:nonunital}]
  Let $K_n,f_n$ be satisfying \eqref{eq:570} and let
  $[-c_n,c_n]\supset f_n(X)$.
  Choosing a sequence $i\mapsto \eps_i\downarrow0$
  % and $j\mapsto r_j\uparrow 1$
  we can consider the
  polynomial
  $P_{n,i}=2%r_j^{-1}
  P_{\eps_i}^{c_n,-1/2,1/2}$
  where $P_\eps^{c,\alpha,\beta}$ is given by Corollary
  \ref{cor:trunc1}.
  Notice that
  $P_{n,i}(0)=0$ 
  so that the functions $h_{n,i}:=P_{n,i}\circ f_n$ belong to $\AA$.
  It is easy to check that
  \begin{gather}
    \label{eq:572}
    \lim_{i\to\infty} h_{n,i}= f_{n}':= (-1\lor 2f_n\land 1),\quad
    \lip h_{n,i}\le 2\lip f_n,\\
    \label{eq:574}
    \lim_{i\to\infty} \lip h_{n,i}(x)=
    \lim_{i\to\infty} P_{n,i}'(f_n(x))\lip f_n(x)=0\quad
    \text{for every $x$ in a neighborhood of }K_n.
  \end{gather}
  By Lebesgue Dominated Convergence Theorem we get
  \begin{equation}
    \label{eq:576}
    \begin{gathered}
      \lim_{i\to\infty} \int_{K_n}|h_{n,i}-1|^p\,\d\mm=0,\quad
      \lim_{i\to\infty} \int_{K_n}|\lip h_{n,i}|^p\,\d\mm=0, \\
      \int_{X\setminus K_n}\Big(|h_{n,i}-1|^p +|\lip h_{n,i}|^p
      \Big)\,\d\mm \le 2^p
      \int_{X
        \setminus K_n}\Big(1+|\lip f_n|^p\Big)\,\d\mm.
    \end{gathered}
  \end{equation}
  We can now introduce the algebra $\tilde \AA=\AA\oplus \{a\unit\}=
  \{\tilde f=f+c\unit:f\in \AA,\ c\in \R\}$ 
  which is clearly unital and compatible with $\X$ according to
  definition \ref{def:A-compatibility}.
  Applying \eqref{eq:569} to $\tilde\AA$, for every $f\in
  \Sob^{1,p}(\X)$
  we can find a sequence $\tilde f_k=f_k+a_k\unit\in \tilde \AA$,
  $k\in \N$,
  such that
  \begin{displaymath}
      \tilde f_k\to f,\quad
      \lip \tilde f_k\to \relgrad f\quad\text{strongly in }L^p(X,\mm)
      \quad\text{as }k\to\infty.
    \end{displaymath}
    For every $k>0$, by \eqref{eq:576} and \eqref{eq:570}
    we can find $i=i(k)$ and $n=n(k)$ sufficiently big,
    such that $u_k:=h_{n,i}\in \AA$ such that
  \begin{equation}
    \label{eq:575}
    -1\le u_k\le 1,\quad
    a_k^p \int_X\Big(|u_k-1|^p+|\lip u_k|^p\Big)\,\d\mm\le 1/k^p
  \end{equation}
  We can then consider $f_k':=f_k+a_k u_k\in \AA$, observing that
  \begin{align*}
    \|f_k'-\tilde f_k\|_p&\le a_k\|u_k-\unit\|_{L^p(X,\mm)}\le
    1/k,\\
    \|\lip f_k'\|_{L^p(X,\mm)}&\le \|\lip \tilde
    f_k \|_{L^p(X,\mm)}+
                           a_k\|\lip u_k\|_{L^p(X,\mm)}\le
                           \|\lip \tilde
    f_k \|_{L^p(X,\mm)}+1/k.
  \end{align*}
  We conclude that the sequence $(f_k')_{k\in \N}\subset \AA$ satisfies
  \begin{displaymath}
    \lim_{k\to\infty}\|f_k'-f\|_{L^p(X,\mm)}=0,\quad
    \lim_{k\to\infty}\int_X |\lip f_k'|^p\,\d\mm=\CE_p(f).\qedhere
  \end{displaymath}
\end{proof}
Let us show two simple examples of applications of
Proposition \ref{prop:nonunital} and condition \eqref{eq:571}:
\begin{enumerate}[1.]
\item If $\sfd$ is $\tau$-continuous, one can always consider the algebra
  \begin{equation}
    \label{eq:609}
    \Lip_{bs}(X,\tau,\sfd):=\Big\{f\in \Lip(X,\tau,\sfd):\text{ $f$ has
      $\sfd$-bounded support}\Big\}
  \end{equation}
\item If $(X,\sfd)$ is locally compact (and thus $\tau$ is the
  topology induced by $\sfd$) then the algebra
  \begin{equation}
    \label{eq:609}
    \Lip_c(X,\tau,\sfd):=\Big\{f\in \Lip(X,\tau,\sfd):\text{ $f$ has
      compact support}\Big\}
  \end{equation}
  satisfies \eqref{eq:569}.
\end{enumerate}
\subsubsection*{Invariance w.r.t.~measure-preserving embeddings}
Let us now consider the invariance
of the strong Cheeger energy w.r.t.~measure preserving
embeddings.
Thanks to the previous Theorem \ref{thm:trivialA} it is
sufficient to consider the case of the canonical algebra.
\begin{theorem}[Invariance of the (strong) Cheeger energy
  w.r.t.~measure preserving embeddings]
  \label{thm:invariance-i}
  Let $\X=(X,\tau,\sfd,\mm)$ and $\X'=(X',\tau',\sfd',\mm')$ be
  two e.m.t.m.~spaces and let $\iota:X\to X'$ be a measure preserving
  embedding of $\X$ into $\X'$ according to Definition
  \ref{def:embeddings}.
  Then $\iota^*$ is an isomorphism between $\Sob^{1,p}(\X')$
  onto $\Sob^{1,p}(\X)$ and 
  \begin{equation}
    \label{eq:567}
    \text{for every $f=\iota^* f'\in
      \Sob^{1,p}(\X)$}\quad
    \relgrad f=\iota^*\big(\relgrad f'\big).
  \end{equation}
\end{theorem}
\begin{proof}
  Let $\bar \X$ and $\bar \X'$ be the completion of $\X$
  and $\X'$ (where $X$ and $X'$ can be identified as
  $\sfd$ and $\sfd'$ dense subsets of $\bar X$ and $\bar X'$
  respectively, see Remark \ref{rem:better-interpretation}).
  Since $\iota:X\to X'$ is an isometry and $X$ is $\sfd$-dense in
  $\bar X$,
  $\iota$ can be extended to an isometric embedding $\bar \iota$ of $\bar X$ into
  $\bar X'$.
  Using property (C4) of Corollary \ref{cor:completion}
  one can check that $\bar\iota$ is also continuous from $(\bar X,\bar \tau)$ to
  $(\bar X',\bar \tau'$) and since $\bar X\setminus X$ and
  $\bar X'\setminus X'$ are $\bar\mm$ and $\bar \mm'$ negligible
  subsets
  respectively, we also see that $\bar \iota$ is
  measure-preserving. We conclude that $\iota$ is a measure-preserving
  imbedding of $\bar X$ into $\bar X'$.

  Since the Cheeger energy is invariant w.r.t.~completion
  by Corollary \ref{cor:completionA}, the above argument
  shows that it is not restrictive to assume that $\X$ and $\X'$ are
  complete.
  By Theorem \ref{thm:main-identification-complete}
  out thesis follows by the property
  for the spaces $W^{1,p}(\X,\calT_q)$
  and $W^{1,p}(\X',\calT_q')$ and the corresponding weak upper
  gradients, proved in Theorem \ref{thm:invariance-weak}.
\end{proof}
Recalling the examples of \ref{ex:embeddings}, we obtain two useful
properties:
\begin{corollary}[Invariance w.r.t.~the topology]
  \label{cor:topology-I}
  Let $\X=(X,\tau,\sfd,\mm)$ be an e.m.t.m.~space and let $\tau'$
  be a coarser topology such that $(X,\tau',\sfd)$ is an e.m.t.~space.
  Then $\Sob^{1,p}(X,\tau,\sfd,\mm)$ is isomorphic to
  $\Sob^{1,p}(X,\tau',\sfd,\mm)$ with equal minimal relaxed gradients.
\end{corollary}
\begin{corollary}[Restriction]
  \label{cor:restriction-I}
  Let $\X=(X,\tau,\sfd,\mm)$ be an e.m.t.m.~space and let $Y\subset X$
  be a $\mm$-measurable subset of $X$ with $\mm(X\setminus Y)=0$.
  If $\Y$ is the associated e.m.t.m.~space according to Example
  \ref{ex:embeddings}(d), $\Sob^{1,p}(\X)$ is isomorphic to
  $\Sob^{1,p}(\Y)$ with equal minimal relaxed gradients.
  In particular, $\Sob^{1,p}(X,\tau,\sfd,\mm)$ is always isomorphic to
  $\Sob^{1,p}(\supp(\mm),\tau,\sfd,\mm)$.
\end{corollary}
\subsubsection*{Invariance w.r.t.~the length and the conformal constructions}
Thanks to Theorem \ref{thm:main-identification-complete},
we can extend the results of Lemma \ref{le:invariance-length} and
Proposition \ref{prop:invariance-conformal} to the Cheeger energy.
\begin{corollary}
  \label{cor:invariance-length}
  Let $\X=(X,\tau,\sfd,\mm)$ be a \emph{complete} e.m.t.m.~space
  and let $\delta:X\times X\to [0,+\infty]$ be an extended distance
  such that
  \begin{equation}
    \label{eq:504bis}
    (X,\tau,\delta)\text{ is an extended metric-topological space},\quad
    \sfd\le \delta\le \sfd_\ell\quad\text{in }X\times X.
  \end{equation}
  Then $\Sob^{1,p}(X,\tau,\sfd,\mm)=
  \Sob^{1,p}(X,\tau,\delta,\mm)$ and the corresponding minimal
  relaxed gradients coincide.
\end{corollary}
\begin{corollary}
  \label{cor:invariance-conformal}
  Let $\X=(X,\tau,\sfd,\mm)$ be a \emph{complete} e.m.t.m.~space,
  let $g\in \rmC_b(X)$ with $0<m_g\le g\le M_g<\infty$,
  and let $\delta:X\times X\to [0,+\infty]$ be an extended distance
  such that
  \begin{equation}
    \label{eq:504bis}
    \X'=(X,\tau,\delta,\mm)\text{ is an e.m.t.m.~space},\quad
    \sfd_g'\le \delta\le \sfd_g\quad\text{in }X\times X.
  \end{equation}
   A function $f\in L^p(X,\mm)$ belongs to the
  Sobolev space
  $\Sob^{1,p}(\X')$ if and only if
  $f\in \Sob^{1,p}(\X)$,
  and the corresponding minimal
  relaxed gradients in $\X$ and in $\X'$
  (which we call  $|\rmD f|_{\star,\X}$ and $|\rmD f|_{\star,\X'}$ respectively) satisfy
  \begin{equation}
    \label{eq:506bis}
    |\rmD f|_{\star,\X}=g^{-1}|\rmD f|_{\star,\X'}.
  \end{equation}
\end{corollary}
\subsection{Examples}
\label{subsec:examples-space}
\begin{example}[Sobolev spaces in $\R^d$ or in a Finsler-Riemannian manifold]
  Let us consider the space $X:=\R^d$
  with the usual topology $\tau$, the distance $\sfd$ induced by a norm
  $\|\cdot\|$ with dual norm $\|\cdot\|_*$, and a finite positive
  Borel measure
  $\mm$.

  Being $(\R^d,\|\cdot\|)$ complete the weak and strong Sobolev spaces
  coincide. By Proposition \ref{prop:nonunital} we can choose
  \begin{equation}
    \label{eq:496}
    \AA:=\rmC^\infty_c(\R^d),\quad
    \lip f(x)=\|\rmD f(x)\|_*\quad\text{for every }f\in \AA.
  \end{equation}
  We thus obtain
  \begin{equation}
    \label{eq:273}
    \begin{aligned}
      \Sob^{1,p}(\R^d,\tau,\|\cdot\|,\mm)=
      {}\Big\{&f\in L^p(\R^d,\mm):
      \exists\, f_n\in \rmC^\infty_c(\R^d)\\
      & f_n\to f\text{ in
      }L^p(X,\R^d),\ \sup_n \int_{\R^d} \|\rmD
      f_n\|_*^p\,\d\mm<\infty\Big\}.
    \end{aligned}
  \end{equation}
  It is not difficult to check that this space is always reflexive
  (see also \cite{Ambrosio-Colombo-DiMarino15} and Corollary \ref{cor:reflexive})
  and it is an Hilbert space if $\|\cdot\|$ is induced by a scalar
  product and $p=2$, since $\pCE_2$ is a quadratic form on $\AA$.
  In this case we obtain the Sobolev space introduced by
  \cite{Bouchitte-Buttazzo-Seppecher97}.
  At least when the gradient operator is closable in $L^p(X,\mm)$,
  the present metric approach also coincides with
  the definition of weighted Sobolev spaces given in
  \cite{HKM93} (a proof of the equivalence
  under doubling and Poincar\'e assumptions has been given in
  \cite[Appendix 2]{Bjorn-Bjorn11})
  
  A completely analogous approach can be used in a complete
  Finsler or Riemannian manifold.
\end{example}
\begin{example}[Sobolev space on a separable Banach space]
  Let $(B,\|\cdot\|)$ be a separable Banach space
  endowed with the strong topology $\tau_s$ and the distance
  $\sfd$ induced by the norm. Let $\mm$ be a finite positive Borel
  measure in $B$ and $\B:=(B,\tau_s,\sfd,\mm)$.
  We can consider the algebra $\AA=\operatorname{Cyl}(B)$
  of smooth cylindrical functions
  (see Example \ref{ex:cylindrical}) so that
  \index{Cylindrical functions}
  \begin{equation}
    \label{eq:568}
    \lip f(x)=\|\rmD f(x)\|_*\quad \text{for every }f\in \AA
  \end{equation}
  and therefore
  \begin{equation}
    \label{eq:273-Banach}
    \begin{aligned}
      \Sob^{1,p}(\B)=
      {}\Big\{&f\in L^p(H,\mm):
      \exists\, f_n\in \AA\\
      & f_n\to f\text{ in
      }L^p(B,\mm),\ \sup_n \int_{B} \|\rmD
      f_n\|_*^p\,\d\mm<\infty\Big\}.
    \end{aligned}
  \end{equation}
  We can give an equivalent intrinsic characterization in terms of
  vector valued Sobolev differentials, in the case $B$ is also
  reflexive. Some of the results below could be extended
  to the case when $B$ has the Radon-Nikodym property
  \cite{Diestel-Uhl77}.
  %
  % We denote by $\tau_s'$ (resp.~$\tau_{w*}'$) the strong Polish
  % (resp.~weak$^*$ Lusin)
  % topology of $B'$.
  % The class of Borel sets of $\tau_{w*}'$ and $\tau_s'$
  % coincide.
  % 
  % Let $(Y,\tau_Y)$ be a Hausdorff topological space. 
  %
  If $\hh:B\to B'$ is a Borel map
  (recall the
  definition
  given in \S\,\ref{subsec:measure} and \S\,\ref{subsec:B-valued} in
  the Appendix) % then it is also
  % Lusin $\mm$-measurable, since $B'$ is metrizable;
  % in particular, $\hh$ admits a Borel representative $\tilde \hh$
  % such that $\mm(\tilde\hh\neq \hh)=0$.
  % If $\int_B\|\hh\|_*\,\d\mm<\infty$ then it is also
  % Bochner integrable, i.e.~there exists a sequence
  % $\hh_n:B\to B'$ of simple Borel functions such that
  % $\lim_{n\to\infty} \|\hh_n-\hh\|_*\,\d\mm=0$
  % %
  % % and it is it is sufficient
  % % to define the $\mm$-measurable set
  % % $H_n:=\{x\in B:\|\hh(x)\|_*\le n\}$ observing that
  % % $\lim_{n\to\infty}\mm(B\setminus H_n)=0$.
  % % Since $\mm$ is a Radon measure on $H_n$
  % % and the restriction of $\hh$ to $A_n$ is a Borel $\mm$-measurable
  % % map
  % % with values in the metrizable space $\{x'\in B':\|x'\|_*\le n\}$,
  % % we deduce that there exists a compact subset $K_n\subset A_n$
  % % such that $\mm(A_n\setminus K_n)\le \mm(X\setminus A_n)+1/n$ and
  % % the restriction of $\hh$ to $K_n$ is continuous.
  it is not difficult to check that
  % Since the duality pairing map $(z,y)\to \la z,y\ra$ is
  % $\tau_{w*}\times \tau_s$ continuous from
  % $B'\times B$ to $\R$,
  % the map 
  % $(x,y)\to \la \tilde\hh(x),y\ra$ is Borel from $B\times B$
  % to $\R$. Since the Borel $\sigma$ algebra of $B\times B$ coincides
  % with the product of the Borel $\sigma$-algebra,
  for every $\gamma\in \RA(B)$ we have
  $t\mapsto \la \hh(\Al_\gamma(t)),\Al_\gamma'(t))$ 
  is Lebesgue-measurable.
  If $\int_\gamma\|\hh\|_*<\infty$ we can thus consider the
  curvilinear integral
  \begin{equation}
    \label{eq:610}
    \int_\gamma \la\hh,\dot\gamma\ra:=
    \int_0^1 \la \hh(\Al_\gamma(t)),\Al_\gamma'(t)\ra\,\d t
  \end{equation}
  We will denote by $L^p(B,\mm;B')$
  the Bochner space of Borel $\mm$-measurable maps $\hh:B\to B'$
  such that 
  \begin{equation}
    \label{eq:573}
    \int_B \|\hh(x)\|_*^p\,\d\mm(x)<\infty,
  \end{equation}
  which is the dual of the Bochner space $L^q(B,\mm;B)$
  \cite[Theorem 8.20.3]{Edwards95}.
  
  Given a function $f\in L^p(B,\mm)$ we say
  that a Borel map $\gg\in L^p(B,\mm;B')$
  is a $\calT_q$-weak gradient of $f$ if
  \begin{equation}
    \label{eq:578}
    f(\gamma_1)-f(\gamma_0)=
    \int_\gamma \la \gg,\dot\gamma\ra
    %\int_0^1 \la \gg(\Al_\gamma(t)),\dot \Al_\gamma(t)\ra\,\d t
    \quad\text{for $\calT_q$-a.e.~$\gamma\in \RA(B)$}.
  \end{equation}
  Notice that the integral in \eqref{eq:578} is well defined since
  the fact that $\|\gg\|_*\in L^q(B;\mm)$ yields
  $\int_\gamma \|\gg\|_*<\infty$ for $\calT_q$-a.e.$\gamma\in \RA(B)$.
  Arguing as in \ref{prop:invarianza}, we
  can show that the class of weak gradients
  is invariant w.r.t.~modifications in a $\mm$-negligible subset.
  We will use the symbols
  \begin{equation}
    \label{eq:577}
    \begin{aligned}
      \mathrm{WG}_p(f):={}&\Big\{\gg\in L^p(B,\mm;B'):
      \gg\text{ is a weak gradient of }f\Big\},\\
      \mathrm{WG}_p:={}&\Big\{(f,\gg)\in L^p(B,\mm)\times L^p(B,\mm;B'):
      \gg\in \mathrm{WG}_p(f)\Big\}.
    \end{aligned}
  \end{equation}
  Every curve $\gamma\in \RA(B)$ induces a vector measure
  $\nnu_\gamma\in \cM(B;B)$
  defined by
  \begin{displaymath}
    \int_B f\,\d\nnu_\gamma:=\int_0^1 f(\Al_\gamma(t))\Al_\gamma'(t)\,\d t
    \quad\text{for every }f\in \rmB_b(B),
  \end{displaymath}
  whose total variation is bounded by $\nu_\gamma$:
  $|\nnu_\gamma|\le \nu_\gamma$.
  If $\ppi\in \cMp(\RA(B))$ is a dynamic plan
  we can then consider the vector measure
  \begin{displaymath}
    \mmu_\sppi:=\int_B\nnu_\gamma\,\d\ppi(\gamma),\quad
    |\mmu_\sppi|\le \mu_\sppi.
  \end{displaymath}
  If $\ppi\in \Bar q\mm$ then there exists a function $\hh_\sppi\in
  L^q(B,\mm;B)$ such that
  \begin{equation}
    \label{eq:580}
    \mmu_\sppi=\hh_\sppi \mm,\quad
    \int_B f\,\d\mmu_\sppi=\int_B f(x)\hh_\sppi(x)\,\d\mm.
  \end{equation}
  % Let $\frd:[0,1]\times \RA(B)\to B$ the map
  % $(t,\gamma)\to \Al_\gamma'(t)$. Whenever $\ppi$ is a dynamic plan,
  % $\frd$ is $\Leb 1\times \ppi$-measurable. If $\ppi\in \Bar q\mm$
  \begin{theorem}
    \label{prop:B-equivalent}
    Let us suppose that $B$ is a separable and reflexive Banach space
    and let $f\in L^p(B,\mm)$.
    \begin{enumerate}
    \item If $f\in \rmC^1(B)\cap \Lip(B)$ then $\rmD f\in \mathrm{WG}_p(f)$.
    \item A function $\gg\in L^p(B,\mm;B')$ belongs to
      $\mathrm{WG}_p(f)$ if and only if
      for every $\ppi\in \Bar q\mm$ with $\mmu_\sppi=
      \hh_\sppi\mm$
      \begin{equation}
        \label{eq:581}
        \int_B f\,\d(\pi_1-\pi_0)=
        \int_B \la \gg(x),\hh_\sppi(x)\ra\,\d\mm.
      \end{equation}
    \item The set $\mathrm{WG}_p$ is a (weakly) closed
      linear
      space of $L^p(B,\mm)\times L^p(B,\mm;B')$.
    \item If $(f,\gg)\in \mathrm{WG}_p$ then
      $g:=\|\gg\|_*$ is a $\calT_q$-weak upper gradient of $f$.
      Conversely, if $g$ is a $(p,\operatorname{Cyl}(B))$-relaxed gradient of $f$
      then there exists $\gg\in \mathrm{WG}_p(f)$ such that
      $\|\gg\|_*\le g$.
      \item A function $f$ belongs to the Sobolev space
        $\Sob^{1,p}(\B)$
        if and only if there exists a weak gradient $\gg\in
        L^p(B,\mm;B')$.
        In this case $\mathrm{WG}_p(f)$ has a unique
        element of minimal norm $\rmD_\smm f$,
        \begin{equation}
          \label{eq:579}
          \relgrad f=\|\rmD_\smm f\|_*\quad\text{$\mm$-a.e.},\quad
          \CE_p(f)=\int_B \|\rmD_\smm f\|_*^p\,\d\mm,
        \end{equation}
        and there exists a sequence $f_n\in \operatorname{Cyl}(B)$
        such that
        \begin{equation}
          \label{eq:582}
          \lim_{n\to\infty}\|f_n-f\|_{L^p(B,\mm)}+
          \|\rmD f_n-\rmD_\smm f\|_{L^p(B,\mm;B')}=0.
        \end{equation}
    \end{enumerate}
    % Let $f\in L^p(B,\mm)$. A weak gradient
    % $\hh:B\to B'$ is a weakly$^*$-measurable map
    % with $\int_B \|hh(x)\|_*^p\,\d\mm(x)<\infty$
    % such that for
    % $\calT_q$ a.e.~curve $\gamma$
    % belongs to $\Sob^{1,p}(\B)$
    % if and only if
    % it admits a Borel representative $\tilde f$ which
    % is
    % Sobolev along  
  \end{theorem}
  \begin{proof}
    \textbf{(a)} is an obvious consequence of the chain rule of $f$ along a
    Lipschitz curve.
    \medskip

    \noindent\textbf{(b)} follows by the same argument of the proof
    of Lemma \ref{le:equivalent1}.
    \medskip

    \noindent\textbf{(c)}
    is an immediate consequence of \eqref{eq:581}.
    \medskip

    \noindent\textbf{(d)}
    The first statement is a consequence of \eqref{eq:581}, which
    yields
    \begin{align*}
      \int_B f\,\d(\pi_1-\pi_0)\le
      \int_B \|\gg\|_*\|\hh_\sppi\|\,\d\mm
      =
      \int_B \|\gg\|_*\,\d(|\mmu_\sppi|)
      \le \int_B \|\gg\|_*\,\d\mu_\sppi
    \end{align*}
    so that $\|\gg\|_*$ is a $\calT_q$-weak upper gradient by
    Lemma \ref{le:equivalent1}.

    Conversely, let $g$ be a $(p,\operatorname{Cyl}(B))$-relaxed
    gradient of $f$. By definition, there exists
    a sequence $f_n$ of cylindrical functions such that
    $f_n\to f$ in $L^p(B,\mm)$ and
    $\lip f_n\weakto \tilde g$ in $L^p(B,\mm)$ with $\tilde g\le g$.
    Since $f_n$ are cylindrical, $\lip f_n(x)=\|\rmD f_n(x)\|_*$;
    since $L^p(B,\mm;B')$ is reflexive, there exists a subsequence
    (still denoted by $f_n$) such that
    $\rmD f_n\weakto \gg$ in $L^p(B,\mm;B')$.
    Thanks to claim (c), $(f,\gg)$ belongs to $\mathrm{WG}_p$ and
    the weak lower semicontinuity of continuous convex functionals
    in a reflexive space
    yields for every Borel set $A\subset B$
    \begin{align*}
      \int_A \|\gg\|_*\,\d\mm
      &\le
        \liminf_{n\to\infty}\int_A \|\rmD f_n\|_*\,\d\mm
        =\int_A \tilde g\,\d\mm\le \int_A g\,\d\mm                                 
    \end{align*}
    so that $\|\gg\|_*\le g$ $\mm$-a.e.
    \medskip

    \noindent\textbf{(e)}
    The first statement follows by Claim (d) and the identification
    Theorem
    \ref{thm:main-identification-complete} between $\Sob^{1,p}(\B)$
    and $W^{1,p}(\B,\calT_q)$.
    Claim (d) and the strict convexity of the $L^p(B,\mm;B')$ norm
    yields \eqref{eq:579}. The proof of Claim (d) also shows
    that there exists a sequence $f_n\in \operatorname{Cyl}(B)$
    such that $(f_n,\rmD f_n)$ weakly converges to $(f,\rmD_\smm f)$
    in
    $L^p(B,\mm)\times L^p(B,\mm;B')$. We can now apply Mazur Theorem.
  \end{proof}
  \begin{corollary}
  \label{cor:reflexive}
  If $B$ is a reflexive Banach space then $\Sob^{1,p}(\B)$ is
  reflexive.
  If moreover $B$ is an Hilbert space then $\Sob^{1,2}(\B)$ is an
  Hilbert space.
\end{corollary}
\begin{proof}
  The proof that $\Sob^{1,p}(\B)$ is reflexive is standard:
  we first notice that $\mathrm{WG}_p$ is a weakly
  closed subset of the reflexive space $L^p(B,\mm)\times
  L^p(B,\mm;B')$ and
  the projection on the first component $\rmp:(f,\gg)\to f$
  is a continuous and surjective map from $\mathrm{WG}_p$ onto
  $\Sob^{1,p}(\B)$ satisfying
  $\|f\|_{\Sob^{1,p}(\B)}=\min\big\{\|(f,\gg)\|_{\mathrm{WG}_p}:
  \rmp(f,\gg)=f\big\}.$
  If $L$ is a bounded linear functional on $\Sob^{1,p}(\B)$
  then $L\circ \rmp$ belongs to $\mathrm{WG}_p'$.
  If $f_n$ is a bounded sequence in $\Sob^{1,p}(\B)$
  then there exists a subsequence $k\mapsto f_{n(k)}$
  and limits $(f,\gg)\in \mathrm{WG}_p$ such that
  $(f_{n(k)},\rmD_\smm f_{n(k)})\weakto (f,\gg)$ in
  $L^p(B,\mm)\times
  L^p(B,\mm;B')$. It follows that
  \begin{displaymath}
    \lim_{k\to\infty} L(f_{n(k)}) 
    =\lim_{k\to\infty}L\circ\rmp (f_{n(k)},\rmD_\smm f_{n(k)}))=
    \lim_{k\to\infty}L\circ\rmp (f,\gg)=
    L(f).\qedhere
\end{displaymath}
\end{proof}
\begin{remark}
  The same conclusion of the previous Corollary holds
  even if $X$ is a \emph{closed subset} of a reflexive and separable Banach (or
  Hilbert) space $\B$ endowed with the induced length distance
  $\sfd_\ell$ (and, e.g., the strong topology $\tau_s$). In this case
  we have
  $W^{1,p}(X,\tau_s,\sfd_\ell,\mm)=
  W^{1,p}(X,\tau_s,\sfd,\mm)$ by Lemma \ref{le:invariance-length} (see
  also \ref{rem:pedante}),
  $W^{1,p}(X,\tau_s,\sfd,\mm)=\Sob^{1,p}(X,\tau_s,\sfd,\mm)$
  by Theorem \ref{thm:main-identification-complete}, and eventually
  $\Sob^{1,p}(X,\tau_s,\sfd,\mm)=\Sob^{1,p}(\B,\tau_s,\sfd,\mm)$
  by Corollary \ref{cor:restriction-I}. We can then apply
  Corollary \ref{cor:reflexive}.
\end{remark}
\begin{remark}
  If we consider the closed subspace
  \begin{equation}
    \label{eq:583}
    \mathrm{WG}_{p,o}:=\{0\}\times
    \mathrm{WG}_p(0)=\Big\{(0,\gg):\gg\in L^p(B,\mm;B'):
    \int_B \la \gg,\hh_\sppi\ra\,\d \mm=0
    \quad\text{for every }\ppi\in \calT_q\Big\}
  \end{equation}
  it would not be difficult to see that $\Sob^{1,p}(\B)$ is isomorphic
  to the quotient space
  $\mathrm{WG}_p/\mathrm{WG}_{p,o}$.
  The operator $f\mapsto \rmD f$ from $\mathrm{Cyl}(B)$ to
  $L^p(B,\mm;B')$ is closable if and only if $\mathrm{WG}_{p,o}=0$.
  As typical example one can consider the case of an Hilbert space $H$
  endowed with a log-concave probability measure $\mm$
  (in particular a Gaussian measure), see
  e.g.~\cite{Ambrosio-Savare-Zambotti09},
  \cite{DPZ02}.
\end{remark}
\end{example}
\begin{example}[Wiener space]
  Let $(X,\|\cdot\|)$ be a separable Banach space
  endowed with its strong topology $\tau$
  and let $\mm$ be a centered non-degenerate Radon Gaussian measure.
  For every bounded linear functional $v\in X'$ let us set
  \begin{equation}
    \label{eq:584}
    R_\mm(v):=\int_X |\la v,x\ra|^2\,\d\mm(x)
  \end{equation}
  $R_\mm$ is a nondegenerate continuous quadratic form on $X'$, whose
  dual
  characterizes the Cameron-Martin space $H(\mm)$
  as the subset of $X$ where the functional 
  \begin{equation}
    \label{eq:585}
    |x|_{H(\mm)}=\sup\{\la v,x\ra:v\in X',\ R_\mm(v)\le 1\},
  \end{equation}
  is finite, and thus defines a Hilbertian norm. We also set $\sfd(x,y):=|x-y|_{H(\mm)}$.
  As we have seen in Example \ref{ex:Wiener}, $\X=(X,\tau,\sfd,\mm)$
  is an Polish e.m.t.m.~space.
  By using the algebra $\AA=\mathrm{Cyl}(X)$
  of smooth cylindrical functions it is not difficult to see that
  for every $f\in \AA$ we have $\rmD f(x)\in X'$ and
  \begin{equation}
    \label{eq:586}
    \lip_\sfd f(x)=\sup_{v\in H(\mm),\, |v|\le 1}\la \rmD f(x),v\ra=
    (R_\mm(\rmD f(x))^{1/2}=|\rmD f(x)|_{H(\mm)'}
  \end{equation}
  so that the metric Sobolev space $H^{1,p}(\X)$ coincides with
  the usual Sobolev space $W^{1,p}(\mm)$ \cite{Bogachev98} defined as the completion of
  the cylindrical functions with respect to the norm
  \begin{displaymath}
    \|f\|_{W^{1,p}(\mm)}^p:=\int_X \Big(|f(x)|^p+|\rmD f(x)|_{H(\mm)'}^p\Big)\,\d\mm(x)
  \end{displaymath}

\end{example}
\subsection{Distinguished representations of metric Sobolev spaces}
\label{subsec:representation}
We have already seen that the strong Cheeger energy is invariant
w.r.t.~completion of the underlying space.
We can now use Theorem \ref{thm:invariance-i}
to obtain isomorphic realizations of the Sobolev space
$\Sob^{1,p}(\X)$
with special e.m.t.m.~space $\X$.
Let us first fix the property we are interested in.
\begin{definition}[Isomorphic representations of Sobolev spaces]
  \label{def:iso}
  Let $\X$, $\X'$ be two e.m.t.m.~spaces. We say that
  $\Sob^{1,p}(\X')$ is an isomorphic representation of
  $\Sob^{1,p}(\X)$ if there exists
  a linear isomorphism $\iota^*: \Sob^{1,p}(\X')\to \Sob^{1,p}(\X) $
  satisfying \eqref{eq:567} induced by
  a measure preserving embedding $\iota:X\to X'$
  from $\X$ into $\X'$.  
\end{definition}
All the statements below refers to an arbitrary e.m.t.m.~space
$\X=(X,\tau,\sfd,\mm)$
and to the strong Sobolev space $\Sob^{1,p}(\X)$. Starting from
a complete space, they also provide equivalent
representations for the weak Sobolev space $W^{1,p}(\X,\calT_q)$
thanks to Theorem \ref{thm:main-identification-complete}.

A first example has already been used in the proof of Theorem
\ref{thm:main-identification-complete}. It is sufficient to
use the compactification Theorem \ref{thm:G-compactification}.
\begin{corollary}[Compact representation]
  \label{cor:compact-r}
  Every Sobolev space $\Sob^{1,p}(\X)$ admits an isomorphic
  representation $\Sob^{1,p}(\hat \X)$ where
  $\hat X$ is a compact e.m.t.m.~space.
\end{corollary}
\begin{corollary}
  Suppose that $(X,\tau)$ is a Souslin space.
  Then there exists a separable Banach space $(B,\|\cdot\|_B)$
  and a weakly$^*$ compact convex subset $\Sigma$ of the dual unit ball of
  $B'$ such that 
  $\Sob^{1,p}(\X)$ admits an isomorphic
  representation as 
  $\Sob^{1,p}(\Sigma,\tau_{w*},\sfd_{B'},\mm_B)$ where
  $\tau_{w*}$ is the weak$^*$ topology of $B'$ ($(\Sigma,\tau_{w*})$ is a
  compact geodesic metric space) and $\sfd_{B'}(v,w):=\|v-w\|_{B'}$.
  Moreover, we can choose the compatible algebra $\AA$
  of the smooth cylindrical functions generated by the 
  elements of $B$ (as linear functional on $B'$).
\end{corollary}
\begin{proof}
  Since $(X,\tau)$ is Souslin, we can find a metrizable and separable
  auxiliary topology $\tau'$ and a compatible 
  algebra $\AA\subset \Lip(X,\tau',\sfd)$ which is countably
  generated.
  We can then apply the Gelfand compactification
  Theorem \ref{thm:G-compactification}
  with the construction described by Proposition
  \ref{prop:iota-description}. Since $B$ is the closure of $\AA$ in
  $\rmC_b(X,\tau')$, $B$ is a separable Banach space and $\Sigma$ is a
  compact convex subset of the unit ball of $B'$.  
\end{proof}

\appendix
\GGG
\section{Appendix}
\label{sec:appendis}
\subsection{Nets}
\label{subsec:nets}
We recap here a few basic facts about nets
(see e.g.~\cite[p.187-188]{Munkres00}).
Let $I$ be a directed set, i.e.~a set endowed with a partial order
$\preceq$ satisfying
\begin{equation}
  \label{eq:325}
  i\preceq i;\quad
  i\preceq j,\ j\preceq k\quad\Rightarrow
  \quad i\preceq k\quad\text{for every }i,j,k\in I,
\end{equation}
\begin{equation}
  \label{eq:227}
  \forall\,i,j\in I\quad \exists\,k\in I:\quad
  i\preceq k,\ j\preceq k.
\end{equation}
As subset $J\subset I$ is called cofinal if for every $i\in I$ there
exists $j\in J$ such that $i\preceq j$.

If $(Y,\tau_Y)$ is a Hausdorff topological space,
a net in $Y$ is a map $y:I\to Y$ 
defined in some directed set $I$; the notation $(y_i)_{i\in I}$
(or simply $(y_i)$) is often used to denote a net.

The net $(y_i)_{i\in I}$
converges 
to an element $y\in Y$ and we write $y_i\to y$ or $\lim_{i\in I}y_i=y$
if for every neighborhood $U$ of $y$ 
there exists $i_0\in I$ such that $i_0\preceq i\ \Rightarrow\ y_i\in
U$.

$y$ is an accumulation point of $(y_i)$ if 
for every neighborhood $U$ of $y$ the set of indexes
$\{i\in I:y_i\in U\}$ is cofinal.

A \emph{subnet} $(y_{\mathfrak i(j)})_{j\in J}$ of $(y_i)$ is obtained by a
composition $y\circ \mathfrak i$
where $\fri :J\to I$ is a map defined in a directed set $J$ satisfying
\begin{displaymath}
  j_1\preceq j_2\ \Rightarrow\ \fri(j_1)\preceq \fri(j_2),\quad
  \fri(J)\text{ is cofinal in $I$}.
\end{displaymath}
Nets are a useful substitution of the notion of sequences, when
the topology $\tau_Y$ does not satisfy the first countable axiom.
In particular we have the following properties:
\begin{enumerate}
\item A point $y$ belongs to the closure of a subset $A\subset Y$ if
  and only if there exists a net of points of $A$ converging to $y$.
\item A function $f:Y\to Z$ between Hausdorff topological spaces
  is continuous if and only if for every
  net $(y_i)_{i\in I}$ converging to $y$ in $Y$ 
  we have $\lim_{i\in I}f(y_i)=f(y)$.
\item $y$ is an accumulation point of $(y_i)$ if and only if there
  exists
  a subnet $(y_{\fri(j)})_{j\in J}$ such that 
  $\lim_{j\in J}y_{\fri(j)}=y$.
\item $(Y,\tau_Y)$ is compact if and only if
  every net in $Y$ has a convergent subnet.
\end{enumerate}
\subsection{Initial topologies}
\label{subsec:initial}
\index{Initial topology}
Let $(Y,\tau_Y)$ be a Hausdorff topological space and 
let $\cF\subset \rmC(Y)$ be a collection of real continuous functions
separating the points of $Y$. 
We say that $\tau_Y$ is generated by $\cF$ 
if it is the coarsest topology for which all the functions
of $\cF$ are continuous 
(thus $\tau_Y$ coincides with the \emph{initial} or \emph{weak}
topology induced by $\cF$).
A basis for the topology $\tau_Y$ is generated by the 
finite intersections of sets of the form 
$\big\{f^{-1}(U): f\in \cF,\ U\text{ open in
  $\R$}\big\}$.

An important property of topologies generated by a separating family
of functions is the characterization of convergence:
for every net $(y_i)_{i\in I}$ in $Y$
\begin{equation}
  \label{eq:327}
  \lim_{i\in I}y_i=y\text{ in $Y$}\quad\Leftrightarrow\quad
  \lim_{i\in I}f(y_i)=f(y)\quad\text{for every }f\in \cF.
\end{equation}
It is also easy to check that such topologies are completely regular:
if $F$ is a closed set and $y\in Y\setminus F$, we can find 
$f_1,\cdots,f_N\in \cF$ and open sets $U_1,\cdots U_N\in \R$ 
such that $y\in\cap_{n=1}^Nf_n^{-1}(U_n)\subset Y\setminus F$.
Up to compositions with affine maps, it is not restrictive to assume that 
$f_n(y)=1$ and $U_n\supset (0,2)$ so that the function
$f(x):=0\lor \min_{1\le n\le N}f_n(x)(2-f_n(x))$ satisfies $f(y)=1$
and $f\restr{Y\setminus F}\equiv0$.

\subsection{Polish, Lusin, Souslin and Analytic sets.}
\label{subsec:PLS}
\index{Polish space}
\index{Lusin space}
\index{Souslin space}
\index{Analytic sets}
Denote by $\N^\infty$ the collection of all infinite sequences of natural numbers and by $\N^\infty_0$
the collection of all finite sequences $(n_0,\ldots,n_i)$,
with $i\geq
0$ and $n_i$ natural numbers.
Let $\cA\subset \mathfrak P(Y)$ containing the empty set (typical examples
are, in a topological space $(Y,\tau_Y)$, the classes
$\FF(Y)$, $\mathscr K(Y)$, $\BB(Y)$ of
closed, compact, and Borel sets respectively). 
We call \emph{table of sets} (or \emph{Souslin scheme}) in $\cA$
\cite[Definition 1.10.1]{Bogachev07}
a 
map $\rmA$ associating to each
finite sequence $(n_0,\ldots,n_i)\in\N_0^\infty$ a set
$\rmA_{(n_0,\ldots,n_i)}\in \cA$. 
 \begin{definition} [$\cA$-analytic sets] 
 $S\subset Y$ is said to be 
  $\cA$-analytic
if there exists a table $\rmA$ of sets in $\cA$ such that
$$S= \bigcup_{(n)\in\N^\infty } \bigcap_{i=0}^{\infty}
\rmA_{(n_0,\ldots, n_i)}. $$
The collection of all the $\cA$-analytic sets will be denoted by $\rmS(\cA)$.
  \end{definition}
  Let us recall a list of useful properties (see \cite[\S\,1.10]{Bogachev07})
  \begin{enumerate}[({A}1)]
  \item Countable unions and countable intersections of elements of
    $\cA$
    belongs to $\rmS(\cA)$.
  \item $\rmS(\rmS(\cA))=\rmS(\cA)$
  \item If the complement of every set of $\cA$ belongs to $\rmS(\cA)$
    then $\rmS(\cA)$ contains the $\sigma$-algebra generated
    by $\cA$. In particular, in a metrizable space $Y$
    $\BB(Y)$-analytic sets are $\FF(Y)$-analytic.
  \item In a topological space $(E,\tau)$, $\BB(E)$-analytic sets are \emph{universally measurable}
\cite[Theorem~1.10.5]{Bogachev07}, i.e.~they are $\mu$-measurable 
for any finite Borel measure $\mu$.
  \end{enumerate}
\begin{definition}[{\cite[Chap.~II]{Schwartz73}}]
  An Hausdorff topological space $(Y,\tau_Y)$
  (in particular, a subset of a topological space $(X,\tau)$ with the
  relative topology) is
  \emph{a Polish space} if it is separable and
  $\tau_Y$ is induced by a complete 
  metric $\sfd_Y$ on $Y$.\\
  $(Y,\tau_Y)$ is said to be \emph{Souslin} (resp.~\emph{Lusin})
  if it is the image of a Polish space
  under a continuous (resp.~injective and continuous) map.
\end{definition}
Differently from the Borel property,
notice that the Souslin and Lusin properties for subsets of a topological space are intrinsic, i.e. they depend only
on the induced topology.

We recall a few important properties
of the class of Souslin and Lusin sets.
\begin{proposition}\label{proplusin}
The following properties hold:
\begin{enumerate}
\item  In a Hausdorff topological space $(Y,\tau_Y)$,
  Souslin sets are $\FF(Y)$-analytic; if $\SS(Y)$ denotes the class
  of Souslin sets, $\rmS(\SS(Y))=\SS(Y)$.
\item if $(Y,\tau_Y)$ is a Souslin space (in particular 
  if it is a Polish or a Lusin space),
  the notions of Souslin and $\FF(Y)$-analytic sets coincide
and in this case Lusin sets are Borel and Borel sets are Souslin; 
\item if $Y$, $Z$ are Souslin spaces and $f:Y\to Z$
  is a Borel injective map, then $f^{-1}$ is Borel;
\item if $Y$, $Z$ are Souslin spaces and $f:Y\to Z$ is a Borel map,
  then $f$ maps Souslin sets to Souslin sets.
\item If $(Y,\tau_Y)$ is Souslin then every
  finite nonnegative Borel measure in $Y$ is Radon.
\end{enumerate}
\end{proposition}
\begin{proof}
  \textbf{(a)}
  is proved in
  \cite[Theorems~6.6.6, 6.6.8]{Bogachev07}.
  In connection with
  \textbf{(b)}, the equivalence
  between Souslin and $\FF(E)$-analytic sets is proved in
  \cite[Theorem~6.7.2]{Bogachev07},
  the fact that Borel sets are Souslin
  in \cite[Corollary~6.6.7]{Bogachev07} and the fact that Lusin sets are Borel in
  \cite[Theorem~6.8.6]{Bogachev07}.
\textbf{(c)} and \textbf{(d)} are proved in
\cite[Theorem~6.7.3]{Bogachev07}.
For \textbf{(e)} we refer to \cite[Thm.~9 \& 10,
p.~122]{Schwartz73}.
\end{proof}
Since in Souslin
spaces $(Y,\tau_Y)$
we have at the same time tightness of finite Borel measures
and coincidence of Souslin and $\FF(E)$-analytic sets, the
measurability of $\BB(E)$-analytic sets yields in particular that
for every $\mu\in \cMp(Y)$
\begin{equation}\label{eq:inner1}
  \mu(B)=\sup\big\{\mu(K)\;:\; K\in\mathscr{K}(Y),\,\,K\subset B\big\}
  \quad\text{for every $B\in \SS(Y)$.}
\end{equation}
We will also recall another useful property \cite[Pages 103-105]{Schwartz73}.
\begin{lemma}
  \label{le:Souslin-useful}
  Let us suppose that $(Y,\tau_Y)$ is a Souslin space.
  \begin{enumerate}
  \item $Y$ is strongly Linde\"of, i.e.~every open cover of an open
    set
    has a countable sub-cover.
  \item Every family $\cF$ of lower semicontinuous 
    real functions defined in $Y$ 
    has a countable subfamily
    $(f_n)_{n\in \N}\subset \cF$ such that $\sup_{f\in\cF}f(x)=\sup_{n\in
      \N}f_n(x)$ for every $x\in Y$.
  \item If $Y$ is regular, every open set is an $F_\sigma$
    (countable intersection of closed set), thus in particular is
    $\FF(Y)$-analytic.
    \item If $Y$ is completely regular, there exists a 
      metrizable and separable topology $\tau'$ coarser than $\tau_Y$.
  \end{enumerate}
\end{lemma}
\subsection{Choquet capacities}
\label{subsec:Choquet}
\index{Choquet capacity}
Let us recall the definition of a \emph{Choquet capacity} in
related to a collection
$\cA$ of subsets of $Y$ containing the empty set and
closed under 
finite unions and countable intersections.
\cite[Chap.~III, \S\,2]{Dellacherie-Meyer78}.
\begin{definition}
  A function $\cI:\mathfrak P(Y)\to [0,+\infty]$ is a Choquet $\cA$-capacity
  if it satisfies the properties
  \begin{enumerate}[(C1)]
  \item $\cI$ is increasing:
    $A\subset B\quad\Rightarrow\quad \cI(A)\le \cI(B)$.
  \item For every increasing sequence $A_n\subset Y$:
    $\cI\big(\cup_nA_n\big)=\lim_{n\to\infty}\cI(A_n)$.
  \item For every decreasing sequence 
    $K_n\in \cA$: $\cI\big(\cap_n K_n\big)=\lim_{n\to\infty}\cI(K_n)$.
  \end{enumerate}
  A subset $A\subset Y$ is called \emph{capacitable} if
  $\cI(A)=\sup\big\{\cI(K):K\subset A,\ K\in \cA\big\}$.
\end{definition}
\begin{theorem}[Choquet, {\cite[Chap.~III, 28]{Dellacherie-Meyer78}}]
  \label{thm:Choquet}
  If $\cI$ is a $\cA$-capacity then
  every $\cA$-analytic set is capacitable.
\end{theorem}
% Choquet's Theorem 
% guarantees that every $\KK(Y)$-analytic set is capacitable.
% If $Z\subset Y$ is compact, then any Borel (or Souslin) subset of $Z$
% is $\KK(Y)$-analytic and therefore it is capacitable.
\subsection{Measurable maps with values in separable Banach spaces}
\label{subsec:B-valued}
Let $(Y,\tau_Y)$ be a Hausdorff topological space
endowed with a Radon measure $\mu\in \cMp(Y)$
and let $(V,\|\cdot\|_V)$ be a separable Banach space with dual $V'$.
Since $V$ is a Polish space, the classes of strong 
and weak Borel sets coincide.

A map $\hh:Y\to V$ is Borel $\mu$-measurable (recall the
definition
given in \S\,\ref{subsec:measure}) then it is also
Lusin $\mu$-measurable, since $V$ is metrizable;
in particular, $\hh$ admits a Borel representative $\tilde \hh$
such that $\mm(\tilde\hh\neq \hh)=0$.
If $\int_Y\|\hh\|\,\d\mm<\infty$ then $\hh$ is also
Bochner integrable, i.e.~there exists a sequence
$\hh_n:Y\to V$ of simple Borel functions such that
$$\lim_{n\to\infty} \int_Y\|\hh_n-\hh\|\,\d\mm=0.$$
We can then define its Bochner integral $\int_Y\hh\,\d\mu$
as the limit $\lim_{n\to\infty}\int_Y \hh_n\,\d\mm$ 
and the corresponding vector measure $\mmu_\shh:=\hh\mu$
defined by
\begin{displaymath}
  \mmu_\shh(A):=\int_A \hh\,\d\mmu\quad
  \text{for every $\mu$-measurable set }A\subset Y.  
\end{displaymath}

\subsection{Homogeneous convex functionals}
Let us first recall a simple property of $p$-homogeneous convex functionals.
\begin{lemma}[Dual of $p$-homogeneous functionals]
  \label{le:obvious-2hom}
  Let $C$ be a convex cone of some vector space $V$, $p>1$, and
  $\phi,\psi:C\to [0,\infty]$ with $\psi=\phi^{1/p}$, $\phi=\psi^p$. 
  We have the following properties:
  \begin{enumerate}
  \item $\phi$ is convex and $p$-homogeneous (i.e.~$\phi(\kkappa
    v)=\kkappa^p\phi(v)$ for every $\kkappa\in \R$ and $v\in C$) in $C$
    if and only if $\psi$ is
    convex and $1$-homogeneous on $C$ (a seminorm, if $C$ is a vector
    space and $\psi$ is finite).
  \item Under one of the above equivalent assumptions, setting for
    every linear functional $z:V\to\R$
    \begin{displaymath}
      \frac 1q\phi^*(z):=\sup_{v\in C}\la z,v\ra-\frac 1p\phi(v),\quad
      \psi_*(z):=\sup\Big\{\la z,v\ra:v\in C,\ \psi(v)\le 1\Big\},
    \end{displaymath}
    we have 
    \begin{equation}
      \label{eq:19}
      \psi_*(z)=\inf\Big\{c\ge 0: \la z,v\ra\le c\,\psi(v)\quad
      \text{for every }v\in C\Big\},
      \quad
      \phi^*(z)=(\psi_*(z))^q,
    \end{equation}
    where in the first infimum we adopt the convention $\inf
    A=+\infty$ if $A$ is empty.
  \end{enumerate}
\end{lemma}
\begin{proof}
  By setting $\phi(v)=\psi(v)=+\infty$ if $v\in V\setminus C$, it is
  not restrictive to assume that $C=V$.
  1. Let us assume that $\phi$ is convex and $p$-homogeneous: we want
  to prove that $\psi$ is a seminorm (this is the only nontrivial
  implication). Since $\psi$ is $1$-homogeneous, it is sufficient to
  prove that it is convex.
  Let $v_i\in V$, $i=0,1$, with $r_i:=\psi(v_i)+\eps$ for $\eps>0$, so that $\tilde
  v_i:=v_i/r_i$ satisfies
  $\psi(\tilde v_i)<1$. We fix $\alpha_i\ge 0$ with $\sum_i
  \alpha_i=1$
  and we set $r:=\sum_i \alpha_i r_i$ and
  $\beta_i:=\alpha_i r_i/r$
  which still satisfy $\beta_i\ge0$ and $\sum_i \beta_i=1$.
  Since the set $K:=\{\psi(v)\le
  1\}=\{\phi(v)\le 1\}$ is convex we have $\sum \beta_i \tilde v_i\in
  K$.
  It follows that $\psi(\sum_i \beta_i\tilde v_i)\le 1$; 
  on the other hand 
  $\sum_i \beta_i\tilde v_i=\frac 1 r\sum_i \alpha_i v_i$ and therefore
  $\psi(\sum_i \alpha_i v_i)=r \psi(\sum_i\beta_i\tilde v_i)\le r=\eps+\sum_i\alpha_i\psi(v_i)$. 
  Since $\eps>0$ is arbitrary, we conclude.
  \medskip

  \noindent
  2. We set $K_a:=\{v\in V: \psi(v)=
  a\}$, $a\in \{0,1\}$, and observe that 
  \begin{displaymath}
    \psi_*(z)=\delta_{K_0}(z)+\sup_{v\in K_1} \la z,v\ra\quad
  \end{displaymath}
  where
  \begin{displaymath}
    \delta_{K_0}(z)=\sup_{v\in K_0}\la z,v\ra=
    \begin{cases}
      0&\text{if }\la z,v\ra\equiv0\ \forall\, v\in K_0\\
      +\infty&\text{otherwise}.
    \end{cases}
  \end{displaymath}
  Similarly
  \begin{displaymath}
    \frac 1q\phi^*(z)=\delta_{K_0}(z)+\sup_{v\in V\setminus K_0}\Big(\la
    z,v\ra-\frac 1p\phi(v)\Big).
  \end{displaymath}
  Since $V\setminus K_0=\bigcup_{\kkappa\in \R}\kkappa K_1$ we have
  \begin{align*}
    \frac 1q\phi^*(z)
    &=\delta_{K_0}(z)+
      \sup_{v\in K_1,\kkappa \in \R}
      \kkappa \la z,v\ra-\frac {\kkappa^p}p\phi(v)
      =
      \delta_{K_0}(z)+
      \sup_{v\in K_1}\sup_{\kkappa \in \R}
      \Big(\kkappa \la z,v\ra-\frac {\kkappa^p}p\Big)
      \\&=\delta_{K_0}(z)+\frac 1p\sup_{v\in K_1}\Big(\la
          z,v\ra\Big)^p=
          \big(\psi_*(z)\big)^p. \qedhere
  \end{align*}
\end{proof}

%\subsection{Measurable family of measures}

\subsection{Von Neumann theorem}

Let $\A,\B$ be convex sets of some vector spaces and let $\cL:\A\times \B\to \R$
be
a saddle function satisfying
\begin{align}
  \label{eq:112}
  a\mapsto \cL(a,b)&\quad\text{is concave in $\A$ for every
                     $b\in \B$},\\
  \label{eq:113}
  b\mapsto\cL(a,b)&\quad\text{is convex in $\B$ for every $a\in \A$.}
\end{align}
It is always true that
\begin{equation}
  \label{eq:114}
  \adjustlimits\inf_{b\in \B}\sup_{a\in \A}\cL(a,b)\ge
  \adjustlimits\sup_{a\in \A}\inf_{b\in \B }\cL(a,b).
\end{equation}

The next result provides an important sufficient condition to
guarantee the equality in \eqref{eq:114}:
we use a formulation which is slightly more general than the statement
of \cite[Thm. 3.1]{Simons98}, but it follows by the same argument.
\begin{theorem}[Von Neumann]
  \label{thm:VonNeumann}
  Let us suppose that \eqref{eq:112}, \eqref{eq:113} hold,
  that
  $\B$ is endowed with some Hausdorff topology and that
  there exists $a_\star\in \A$ and $C_\star>\adjustlimits\sup_{a\in \A}\inf_{b\in \B}\cL(a,b)$ such that
    \begin{gather}
          \label{eq:115}
      \B_\star:=\Big\{b\in \B:\cL(a_\star,b)\le
      C_\star\Big\}\quad\text{is not empty and compact in $\B$,}\\
      \label{eq:140}
      \text{$b\mapsto\cL(a,b)$ is lower semicontinuous in $\B_\star$ for
        every $a\in \A$}.
    \end{gather}
  Then
  \begin{equation}
  \label{eq:114bis}
  \adjustlimits\min_{b\in \B}\sup_{a\in \A}\cL(a,b)=
  \adjustlimits\sup_{a\in \A}\inf_{b\in \B }\cL(a,b).
\end{equation}
Similarly, if  $\A$ is endowed with a Hausdorff topology and
there exists $b_\star\in \B$ and $D_\star<\adjustlimits\inf_{b\in \B}\sup_{a\in \A}\cL(a,b)$ such that
    \begin{gather}
          \label{eq:115bis}
          \A_\star:=\Big\{a\in \A:\cL(a,b_\star)\ge
          D_\star\Big\}\quad\text{is not empty and compact in $\A$,}\\
      \label{eq:140bis}
      \text{$a\mapsto\cL(a,b)$ is upper semicontinuous in $\A_\star$ for
        every $b\in \B$}.
    \end{gather}
  Then
  \begin{equation}
  \label{eq:114tris}
  \adjustlimits\inf_{b\in \B}\sup_{a\in \A}\cL(a,b)=
  \adjustlimits\max_{a\in \A}\inf_{b\in \B }\cL(a,b).
\end{equation}
\end{theorem}
We reproduce here the main part of the proof of \eqref{eq:114bis};
\eqref{eq:114tris} follows simply by considering the Lagrangian
$\tilde\cL(b,a):=-\cL(a,b)$ in $\B\times \A$ and inverting the
role of $\A$ and $\B$.
\begin{proof}
  Let $s:=\adjustlimits\sup_{a\in \A}\inf_{b\in \B }\cL(a,b)$ and
  let $\B_a:=\{b\in \B:\cL(a,b)\le s\}$,
  {$\B_{a\star}:=\{b\in \B:\cL(a_\star,b)\le s\}$}.
  We notice that $\B_{a\star}\subset \B_\star$ and that
  for every $a\in \A$ the set $\B_{a} \cap \B_{\star}=
  \{b\in \B_\star:\cL(a,b)\le s\}$ is compact thanks to \eqref{eq:115}
  and \eqref{eq:140}. \nc
  If $A\subset \A$ is a collection
  containing $a_\star$ then 
  \begin{displaymath}
    \B_A=\bigcap_{a\in A}\B_a=
    {
      \bigcap_{a\in A} \big(\B_{a} \cap \B_{a\star}\big)=
      \bigcap_{a\in A} \big(\B_{a} \cap \B_{\star}\big)\nc}
  \end{displaymath}
  so that $\B_A$ is a (possibly empty) compact set.
  The thesis follows if we check that
  {$\B_\A$} contains a point $\bar
  b$, since in that case
  $  \adjustlimits
\inf_{b\in \B}\sup_{a\in \A}\nc\cL(a,b)\le 
  \sup_{a\in \A}\cL(a,\bar b)\le s$ by construction;
  on the other
  hand, \eqref{eq:114} shows that
  $\sup_{a\in \A}\cL(a,\bar b)=s\le \sup_{a\in \A}\cL(a,b)$
  for every $b\in \B$, so that the minimum in the left-hand side of
  \eqref{eq:114bis} is attained at $\bar b$. \nc

  Since $\B_A$ are compact whenever $a_\star\in A$, it is sufficient
  to prove that for every finite collection $A=\{a_1,\cdots,a_n\}$
  containing $a_\star$ the intersection $\B_A$ is not empty.
  To this aim, since $b\mapsto \cL(a_k,b)$ are convex functions,
  \cite[Lemma 2.1]{Simons98} yields
  \begin{align*}
    {\adjustlimits\inf_{b\in \B}\sup_{1\le k\le n} \nc \cL(a_k,b)=}
    \inf_{b\in \B}\nc\sum_{k=1}^N \nchi_k \cL(a_k,b)
  \end{align*}
  for a suitable choice of nonnegative coefficients $\nchi_k\in
  [0,1]$ with $\sum_{k=1}^n\nchi_k=1$. We thus get by concavity
  \begin{align*}
    {\inf_{b\in \B}\sum_{k=1}^N\nchi_k\cL(a_k,b)\le }
    \inf_{b\in \B}\cL(\sum_{k=1}^N\nchi_k
    a_k,b)\le s,
  \end{align*}
   so that $\inf_{b\in \B}\sup_{1\le k\le n} \nc \cL(a_k,b)\le
  s$. 
  On the other hand, since $\sup_{1\le k\le n} \cL(a_k,b) \ge
  \cL(a_\star,b)$,
  every $b\in \B$ such that $\sup_{1\le k\le n} \cL(a_k,b)\le C_\star$
  belongs to $\B_\star$ so that $C_\star>s$ yields
  \begin{displaymath}
    s=\adjustlimits\inf_{b\in \B}\sup_{1\le k\le n} \cL(a_k,b)=
    \adjustlimits\inf_{b\in \B_\star}\sup_{1\le k\le n} \cL(a_k,b)=
    \adjustlimits\min_{b\in \B_\star}\sup_{1\le k\le n} \cL(a_k,b),
  \end{displaymath}
  where in the last identity we used the fact that $\B_\star$ is
  compact and that the restriction of the function $b\mapsto \sup_{1\le k\le n} \cL(a_k,b)$
  to $\B_\star$ is lower semicontinuous. We conclude that  
  $\bigcap_{k=1}^N \{ b \in \B: \cL(a_k,b) \leq s \}$ is not empty.
\end{proof}

\printindex
%\addcontentsline{toc}{Chapter}{References}
% \bibliographystyle{abbrv}
% \bibliography{bibliografia-CIME2019}

\def\cprime{$'$}

\end{document}